\documentclass[11pt,reqno]{amsart}
\usepackage{mathrsfs}
\usepackage{url}
\usepackage{mathtools}
\usepackage{latexsym,epsfig,amssymb,amsmath,amsthm,color,url,bm}
\usepackage[inline,shortlabels]{enumitem}
\usepackage{hyperref}
\usepackage[foot]{amsaddr}
\usepackage{amsmath,amsbsy}
\usepackage{mwe}
\RequirePackage[numbers]{natbib}
\usepackage{mathptmx}
\usepackage[text={16cm,24cm}]{geometry}
\usepackage{multirow}
\usepackage{subcaption}

\allowdisplaybreaks 
\newcommand\Item[1][]{%
  \ifx\relax#1\relax  \item \else \item[#1] \fi
  \abovedisplayskip=0pt\abovedisplayshortskip=0pt~\vspace*{-\baselineskip}}

\numberwithin{equation}{section}
\newtheorem{theorem}{Theorem} 
\newtheorem{lemma}{Lemma}[section] 
\newtheorem{corollary}{Corollary}[section] 

\theoremstyle{definition}

\theoremstyle{definition}
\newtheorem{remark}{Remark}[section] 

\DeclareMathOperator{\Prob}{\mathbf{P}}
\DeclareMathOperator{\G}{\widehat{G}}
\DeclareMathOperator{\Out}{Out}
\DeclareMathOperator{\E}{\widehat{E}}
\DeclareMathOperator{\F}{\widehat{F}}

\title[Generalized percolation games on $\mathbb{Z}^{2}$, and ergodicity of associated PCA]{Generalized percolation games on the $2$-dimensional square lattice, and ergodicity of associated probabilistic cellular automata}
\date{}
\author{Dhruv Bhasin, Sayar Karmakar, Moumanti Podder, Souvik Roy}
\address{Dhruv Bhasin, Indian Institute of Science Education and Research (IISER) Pune, Dr.\ Homi Bhabha Road, Pashan, Pune 411008, Maharashtra, India.}
\address{Sayar Karmakar, University of Florida, 230 Newell Drive, Gainesville, Florida 32605, USA.}
\address{Moumanti Podder, Indian Institute of Science Education and Research (IISER) Pune, Dr.\ Homi Bhabha Road, Pashan, Pune 411008, Maharashtra, India.}
\address{Souvik Roy, Indian Statistical Institute, 203 Barrackpore Trunk Road, Kolkata 700108, West Bengal, India.}
\email{bhasin.dhruv@students.iiserpune.ac.in}
\email{sayarkarmakar@ufl.edu}
\email{moumanti@iiserpune.ac.in}
\email{souvik.2004@gmail.com}

\begin{document}
\bibliographystyle{plainnat}

\begin{abstract}
We consider a highly generalized set-up in which each vertex of the infinite $2$-dimensional square lattice graph (whose set of vertices is $\mathbb{Z}^{2}$, with each vertex $(x,y)$ adjacent to each of $(x+1,y)$ and $(x,y+1)$) is assigned, independent of all else, a label that reads \emph{trap} with probability $p$, \emph{target} with probability $q$, and \emph{open} with the remaining probability $(1-p-q)$, and additionally, each edge is assigned, independent of all else, a label that reads \emph{trap} with probability $r$ and \emph{open} with probability $(1-r)$. This model encompasses the \emph{seemingly} more general model where, in addition to all the vertex-labels and edge-labels described above, an edge can also be labeled as a \emph{target}, since assigning the label of target to an edge going from $(x,y)$ to either $(x+1,y)$ or $(x,y+1)$ is equivalent to marking the \emph{vertex} $(x,y)$ as a trap. A \emph{percolation game} is played on this random board, involving two players and a token. The players take turns to make \emph{moves}, where a \emph{move} involves relocating the token from where it is currently located, say some vertex $(x,y) \in \mathbb{Z}^{2}$, to any one of $(x+1,y)$ and $(x,y+1)$. A player wins if she is able to move the token \emph{to} a vertex labeled as a target, or force her opponent to either move the token \emph{to} a vertex labeled as a trap or \emph{along} an edge labeled as a trap. We seek to find a \emph{regime}, in terms of values of the parameters $p$, $q$ and $r$, in which the probability of this game resulting in a draw equals $0$. We further consider special cases of this game, such as when each edge is assigned, independently, a label that reads \emph{trap} with probability $r$, \emph{target} with probability $s$, and \emph{open} with probability $(1-r-s)$, but the vertices are left unlabeled, and various regimes of values of $r$ and $s$ are explored in which the probability of draw is guaranteed to be $0$. We show that the probability of draw in each such game equals $0$ if and only if a suitably defined \emph{probabilistic cellular automaton} (PCA) is \emph{ergodic}, following which we implement the technique of \emph{weight functions} or \emph{potential functions} to investigate the regimes in which said PCA is ergodic. We mention here that one of the main results of \cite{holroyd2019percolation} follows as a special case of our main result. Moreover, our result shows that a phase transition happens at the origin (i.e.\ at $(p,q,r)=(0,0,0)$ in case of generalized percolation games, and at $(r,s)=(0,0)$ in case of bond percolation games) in the sense that, the probability of draw equals $1$ at $(p,q,r)=(0,0,0)$ (respectively, at $(r,s)=(0,0)$), whereas in \emph{every} neighbourhood around $(0,0,0)$ (respectively, $(0,0)$), there exists some value of $(p,q,r)$ (respectively, $(r,s)$) for which the probability of draw equals $0$.
\end{abstract}

\subjclass[2020]{91A46, 91A43, 60K35, 82B43, 05C57, 37B15, 37A25, 68Q80} 

\keywords{percolation; percolation games on lattices; two-player combinatorial games; probabilistic cellular automata; ergodicity; probability of draw; weight function; potential function}

\maketitle

\section{Introduction}\label{sec:intro}
The study of \emph{oriented Bernoulli bond percolation} on $\mathbb{Z}^{d}$ (i.e.\ the infinite graph whose vertex set is $\mathbb{Z}^{d}$ and in which two vertices $\mathbf{x}=(x_{1},x_{2},\ldots,x_{d})$ and $\mathbf{y}=(y_{1},y_{2},\ldots,y_{d})$ are adjacent if and only if the Euclidean distance between them is $1$) is now more than six decades old. For each $i \in \{1,2,\ldots,d\}$, let $\mathbf{e}_{i} \in \mathbb{Z}^{d}$ denote the vertex in which the $i$-th coordinate equals $1$, while all other coordinates equal $0$. For each $\mathbf{x}\in \mathbb{Z}^{d}$, let the edge between $\mathbf{x}$ and $\mathbf{x}+\mathbf{e}_{i}$, denoted henceforth as $\big(\mathbf{x},\mathbf{x}+\mathbf{e}_{i}\big)$, be directed \emph{from} $\mathbf{x}$ \emph{towards} $\mathbf{x}+\mathbf{e}_{i}$. Each such directed edge is now labeled \emph{closed} with probability $r$, and \emph{open} with probability $(1-r)$, for some $r \in [0,1]$. The question of interest is: for what values of $r$ does there exist, with positive probability, an infinite path that begins from the origin $\mathbf{0}$ and consists only of open, directed edges? In other words, for which values of $r$ does \emph{percolation happen}? Introduced in \cite{broadbent1957percolation}, this problem has since been studied extensively in \cite{griffeath1981basic}, \cite{durrett1983supercritical}, \cite{durrett1984oriented} etc. The existence of a \emph{critical value} $r_{c}(d) \in (0,1)$, such that percolation happens for $r < r_{c}(d)$ and percolation does \emph{not} happen when $r > r_{c}(d)$, has been well-known. However, the exact value of $r_{c}(d)$ is seldom known. For instance, in \cite{bollabas1997approximate}, it has been shown, via rigorous methods, that $r_{c}(2) < 0.647$, while Monte Carlo simulations suggest that $r_{c}(2) \approx 0.6445$. 

The work in this paper began in an attempt to understand what happens when we incorporate an \emph{adversarial element} into the set-up of oriented bond percolation mentioned above. In other words, we now imagine a \emph{percolation game} on $\mathbb{Z}^{2}$ with two players and a token that is placed at the origin, $(0,0)$, at the beginning of the game. The two players take turns to make \emph{moves}, and each such move constitutes a \emph{round} of the game. Each player, when it is her turn to make a move, tries to relocate the token from where it is currently located, say $(x,y) \in \mathbb{Z}^{2}$, to one of $(x+1,y)$ and $(x,y+1)$, but relocation is only allowed along edges that are open. A player loses if she is unable to move, i.e.\ her opponent has succeeded in moving the token to some vertex $(x,y)$ such that both the directed edges $\big((x,y),(x+1,y)\big)$ and $\big((x,y),(x,y+1)\big)$ are closed. The game may continue indefinitely (i.e.\ neither of the two players is able to clinch the victory within a finite number of rounds of the game), in which case we say that the game has resulted in a draw. Here, we ask the question: for what values of $r$ does this game result in a draw with positive probability? Note how closely \emph{this} question ties in with the question regarding the occurrence of percolation that we posed in the previous paragraph. The existence of an \emph{infinite, open path} starting at the origin, i.e.\ an infinite path consisting of open, directed edges, is necessary for even the remotest possibility of a draw, but the existence of an \emph{arbitrary} infinite, open path does not, necessarily, ensure the occurrence of a draw. In particular, for draw to happen, there must exist an infinite, open path that acts as an `equilibrium path' -- i.e.\ a path from which neither of the two players is willing to deviate under the common belief of rationality. We show, in this paper, that for $r$ as low as $0.157176$, the probability of draw in the game described above is $0$. A comparison of this result with the estimate of $r_{c}(2)$ mentioned in the previous paragraph reveals how fascinating the study of such games is, and how vastly the inclusion of an adversarial flavour impacts the original percolation process. 

Even though the set-up described in the previous paragraph is what the work in this paper commenced with, the contents of this paper do not remain limited to this one-parameter percolation game. The first step towards generalizing this game can be brought about by classifying the edges of our graph into the following \emph{three} categories, instead of two: 
\begin{enumerate}
\item a player is \emph{penalized} if she moves the token along a directed edge that has been labeled as a \emph{trap} (this is the equivalent of the label `closed' in the previous set-up), which happens with probability $r$,
\item a player is \emph{rewarded} if she moves the token along a directed edge that has been labeled as a \emph{target}, which happens with probability $s$, 
\item and a player is neither rewarded nor punished if she moves the token along a directed edge that has been labeled \emph{open}, which happens with the remaining probability $(1-r-s)$.
\end{enumerate}
In other words, a player, when it is her turn to move, relocates the token, as before, from where it is currently located, say $(x,y) \in \mathbb{Z}^{2}$, to one of $(x+1,y)$ and $(x,y+1)$, but \emph{this} game can come to an end in one of \emph{two} different ways: either one of the players succeeds in moving the token along an edge that has been labeled as a target, and thereby wins the game, or one of the players is forced to move the token along an edge that has been labeled as a trap, and thereby loses the game. Note that the previously described one-parameter percolation game is a special case of this two-parameter percolation game, obtained by setting $s=0$, and, as mentioned in the previous paragraph, we have shown in this paper that the probability of draw in the one-parameter percolation game equals $0$ whenever $r \geqslant 0.157176$. For the two-parameter set-up, we establish two different regimes of values of the parameter-pair $(r,s)$ for which the probability of draw is guaranteed to be $0$: 
\begin{enumerate}
\item where $r=s > 0.10883$, and
\item \label{simplified_regime_2_r,s} where each of $r$ and $s$ is sufficiently small (how small they need to be comes out as a consequence of the various steps involved in the main, technical part of the proof, outlined in \S\ref{sec:bond_1_wt_fn_steps} -- but it suffices for us to take each of $r$ and $s$ to be less than or equal to $1/50$), and $3(1-s)(2s-s^{2})^{2}(1-6s+3s^{2}) \geqslant 4r$.
\end{enumerate}
These results have been stated formally in Theorem~\ref{thm:two-parameter} (the regime of values of $(r,s)$ covered by \eqref{simplified_regime_2_r,s} above is slightly smaller, i.e.\ the inequality is stronger, than what we actually deduce in \ref{bond_regime_1} of Theorem~\ref{thm:two-parameter} -- the stronger, and simpler, inequality has been stated here for the sake of simplicity and lucid reading).

Our endeavour to generalize the notion of percolation games does not end here -- we go a step further and consider a \emph{three-parameter} set-up which, in some sense, combines the notion of \emph{site percolation} and \emph{bond percolation}, as follows. Each vertex $(x,y) \in \mathbb{Z}^{2}$ is now assigned, independent of all else, a label that reads \emph{trap} with probability $p$, \emph{target} with probability $q$, and \emph{open} with the remaining probability $(1-p-q)$. Additionally, each edge is assigned, independent of all else, a label that reads \emph{trap} with probability $r$ and \emph{open} with the remaining probability $(1-r)$. The permitted moves by the players are the same as those described in the previous paragraph, but now, a player can clinch the victory in one of \emph{three} possible ways:
\begin{enumerate} 
\item either she succeeds in moving the token to a vertex labeled as a target,
\item or she is able to force her opponent to move the token to a vertex labeled as a trap,
\item or her opponent is compelled to move the token along an edge labeled as a trap.
\end{enumerate}
As before, the game continues for as long as one of the players does not emerge the victor, and the result is a draw if the game continues indefinitely. We establish a regime of values of $(p,q,r)$, described in detail in Theorem~\ref{thm:three-parameter}, for which the probability of draw in this game is guaranteed to equal $0$. 

It is straightforward to see that draw is the only outcome possible when $p=q=r=0$, and it happens with probability $0$ when $p+q=1$ or when $r=1$. As indicated in Chapter 7 of \cite{PCA_survey_old}, an important and challenging question concerns itself with what happens to the probability of draw (and, as we shall introduce subsequently, to the ergodicity of a suitably defined probabilistic cellular automaton) for arbitrarily small values of the probabilities of trap and target (for both vertices and edges), i.e.\ when each of $p$, $q$ and $r$ is arbitrarily small. The crucial contribution of Theorem~\ref{thm:three-parameter} (as also discussed in \S\ref{subsubsec:contribution_PCA}) lies in its coverage of a significant subset of values of $(p,q,r)$ in a small neighbourhood of $(0,0,0)$, and pictorial illustrations of what this subset looks like have been provided in Figure~\ref{fig:3D-Plots} following the statement of Theorem~\ref{thm:three-parameter}. Theorem~\ref{thm:three-parameter} provides four such regions, or \emph{regimes}, each specified by some inequalities involving $p,q$, and $r$, such that the probability of draw equals $0$ whenever $(p,q,r)$ belongs to one of these regions. As these inequalities are somewhat involved, we provide here some simpler and stronger inequalities (in other words, some \emph{sub-regimes} of the regimes covered by Theorem~\ref{thm:three-parameter}) to elucidate said contribution of the theorem to the reader. \footnote{The actual conditions differ only with respect to second or higher order terms (i.e.\ terms of the form $p^{i}q^{j}r^{k}$ where $i, j, k \in \mathbb{N}_{0}$ and $i+j+k\geqslant 2$).} A simpler (and stronger) version of the inequalities in \eqref{three_cond_universal} and \eqref{three_cond_1} of Theorem~\ref{thm:three-parameter} is given by
\begin{equation}\label{three_cond_1_simplified}
p \leqslant \frac{q}{2} \quad \text{and} \quad p+q \geqslant 2r, 
\end{equation}
and likewise, a simpler and stronger version of the inequalities in \eqref{three_cond_universal} and \eqref{three_cond_4} of Theorem~\ref{thm:three-parameter} is given by
 \begin{equation}\label{three_cond_4_simplified}
q + r \leqslant p \quad \text{and} \quad 5q \geqslant 4r.
\end{equation}
It follows from \eqref{three_cond_1_simplified} that the probability of draw is $0$ when $q = \epsilon$, $p = \alpha \epsilon$ and $r \leqslant (1+\alpha)\epsilon/2$, where $\alpha \leqslant 1/2$ and $\epsilon$ is an arbitrarily small positive real. Similarly, \eqref{three_cond_4_simplified} reveals that the probability of draw is $0$ when $p=\epsilon$, $q = \alpha\epsilon$ and $r = \beta \epsilon$ where $\alpha+\beta \leqslant 1$ and $5 \alpha \geqslant 4 \beta$, and $\epsilon$ is an arbitrarily small positive real. To demonstrate, pictorially, the $3$-dimensional region that all $(p,q,r)$ satisfying one of \eqref{three_cond_1_simplified} and \eqref{three_cond_4_simplified} cover, we refer the reader to Figure~\ref{fig:simplified} (in each of (A) and (B) of Figure~\ref{fig:simplified}, we take each of $p$, $q$ and $r$ to be bounded above by $1/50=0.02$, although this could be optimized further, since the claims made in Theorem~\ref{thm:three-parameter} assuredly hold when $p$, $q$ and $r$ are this small). 
Furthermore, the inequalities \eqref{three_cond_universal} and \eqref{three_cond_2} of Theorem~\ref{thm:three-parameter} together imply, as a special case, that draw happens with probability $0$ when $p=q=r=\epsilon$ for $\epsilon > 0$ arbitrarily small. Yet another contribution of Theorem~\ref{thm:three-parameter} stems from the fact that, setting $r=0$, the inequalities in \eqref{three_cond_universal}, \eqref{three_cond_1} and \eqref{three_cond_4} together imply that the probability of draw equals $0$ whenever at least one of $p$ and $q$ is strictly positive -- thereby allowing Theorem 1 of \cite{holroyd2019percolation} to be derived as a special case of Theorem~\ref{thm:three-parameter}. These instances go to show how Theorem~\ref{thm:three-parameter} establishes that the probability of draw is $0$ for a fairly wide region of values of $(p,q,r)$ arbitrarily close to $(0,0,0)$. Since the probability of draw is $1$ at $(0,0,0)$, our result seems to indicate (and proves it under certain (relative) restrictions on $p$, $q$ and $r$, such as when $p=q=r$) that a phase transition happens exactly at $p=q=r=0$. 

\begin{figure}[htbp]
    \centering
    \begin{subfigure}[b]{0.45\textwidth}
        \centering
        \includegraphics[width=\textwidth]{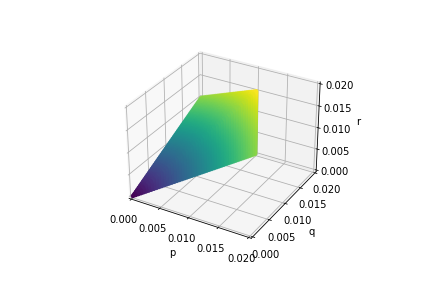}
        \caption{3D plot of \eqref{three_cond_1_simplified}
        \label{subfig:simplified-1}}
    \end{subfigure}
    \hfill
    \begin{subfigure}[b]{0.45\textwidth}
        \centering
        \includegraphics[width=\textwidth]{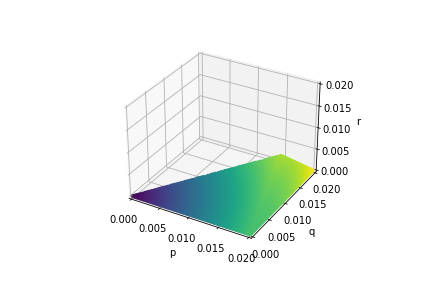}
        \caption{3D plot of \eqref{three_cond_4_simplified}
        \label{subfig:simplified-2}}
    \end{subfigure}
    \caption{Regions covered by \eqref{three_cond_1_simplified} and \eqref{three_cond_4_simplified}}
    \label{fig:simplified}
\end{figure}

\subsection{The probabilistic cellular automata we study in this paper} Percolation games form one of the two chief components of this paper, the other being a class of \emph{probabilistic cellular automata} (henceforth abbreviated as \emph{PCAs}). Informally speaking, each game considered in this paper gives rise to a set of game rules, as follows. Once an assignment of labels to the vertices and / or edges of $\mathbb{Z}^{2}$ has been realized, each vertex $(x,y) \in \mathbb{Z}^{2}$ can be classified into one of three categories or \emph{states} depending on the outcome of the game that begins from $(x,y)$: either this game is won by the player who plays the first round, or this game is lost by the player who plays the first round, or this game results in a draw. The state of each $(x,y)$ can be determined from the states of $(x+1,y)$ and $(x,y+1)$, the labels assigned to the directed edges $\big((x,y),(x+1,y)\big)$ and $\big((x,y),(x,y+1)\big)$, and the label assigned to $(x,y)$ itself. This gives rise to a collection of recurrence relations that are referred to as the \emph{game rules}. If we form an \emph{alphabet} consisting of the above-mentioned three possible states that a vertex can belong to, and we let $\{(x+1,y),(x,y+1)\}$ be, loosely speaking, the \emph{neighbourhood} of $(x,y)$, for each $(x,y) \in \mathbb{Z}^{2}$, then the game rules can be interpreted as the \emph{stochastic update rules} associated with a PCA (see \S\ref{subsubsec:general_PCAs} for a discussion on PCAs in general). Moreover, it can be shown that the probability of a percolation game resulting in a draw equals $0$ if and only if the PCA arising from its game rules is \emph{ergodic} (see \S\ref{subsubsec:general_PCAs} for the definition of ergodicity). Establishing ergodicity of PCAs via rigorous, analytical proofs is usually a difficult feat to achieve, but we are able to accomplish this for some regimes of values of the underlying parameter(s) using the relatively novel and relatively-less-explored technique of \emph{weight functions} or \emph{potential functions}, introduced in \cite{holroyd2019percolation}.

It is worthwhile to note here, even though we have not yet formally introduced PCAs and the definition of ergodicity, that one of the main contributions of this paper lies in establishing the ergodicity of certain classes of \emph{elementary PCAs} in regimes (defined in terms of the values of the underlying parameter(s)) where existing, older methods for proving ergodicity do \emph{not} turn out to be useful. In particular, in both Theorem~\ref{thm:three-parameter} and in the regime described by \ref{bond_regime_1} of Theorem~\ref{thm:two-parameter}, each of the parameter(s) is allowed to be as close to $0$ as one desires, as long as at least one of them is strictly positive and the relevant inequality / inequalities is / are satisfied. Without going into technical details (which we present in \S\ref{subsubsec:contribution_PCA}), we highlight the importance of these regimes to the reader at this point in an intuitive manner. From our earlier, informal description of the two-parameter percolation game, one may intuit that proving the probability of draw to be equal to $0$ ought to become harder as the values of $r$ and $s$ become smaller -- this is because, small values of $r$ and $s$ would ensure that each edge remains open with a higher probability, and the game does not come to an end as long as the token keeps being moved along open edges. Likewise, in case of the three-parameter percolation game, proving the probability of draw to be equal to $0$ ought to be harder for smaller values of $p$, $q$ and $r$, as these ensure that each vertex as well as each edge remains open with a higher probability. Since the ergodicity of the two elementary PCAs studied in this paper ties in with the probability of draw being equal to $0$ in the two-parameter and three-parameter percolation games respectively, it is no surprise that the true challenge lies in considering $(r,s)$ (respectively, $(p,q,r)$) in as small a neighbourhood of $(0,0)$ (respectively, $(0,0,0)$) as one desires (while making sure that $r+s > 0$ in the former situation, and $p+q+r > 0$ in the latter). This challenge cannot be addressed by making use of existing methods for establishing ergodicity of elementary PCAs (such as those discussed in Chapter 7 of \cite{PCA_survey_old}), but \emph{can} be addressed to a \emph{significant} extent using the method of weight functions that we implement in this paper.

\subsection{Why the three-parameter percolation game is an all-encompassing set-up}\label{subsec:three_parameter_suffices} An obvious question to ask at this point is, why, in the set-up for the three-parameter percolation game, \emph{target} is not included in the list of labels assigned to the edges of our graph. For any $(x,y) \in \mathbb{Z}^{2}$, if at least one of its outgoing edges, $\big((x,y),(x+1,y)\big)$ and $\big((x,y),(x,y+1)\big)$, has been labeled as a target, a player knows that she will win the game the moment her opponent moves the token to $(x,y)$, \emph{unless} $(x,y)$ has already been labeled as a target. Consequently, seen from her opponent's perspective, except for the case where $(x,y)$ has been labeled as a target, it is as if $(x,y)$ \emph{acts} like a vertex that has been labeled as a trap. This equivalence has been formalized in the following paragraph.

We mention here that the suggestion for investigating generalized percolation games (where \emph{both} the vertices and the directed edges of an infinite lattice graph may receive the labels of trap, target and open), and the fact that such a game on the $2$-dimensional infinite square lattice may be reduced to one that is governed by only three parameters, not four, came from an anonymous referee, and we are deeply grateful for these suggestions.

Consider the following policy of (random) assignment of labels to our graph: each vertex is labeled, independently, a trap with probability $p'$, a target with probability $q'$, and open with probability $(1-p'-q')$, and each edge is labeled, independently, a trap with probability $r'$, a target with probability $s'$, and open with probability $(1-r'-s')$. Let $(\sigma,\eta)$ indicate a realized assignment of labels obtained by implementing this policy, with $\sigma=(\sigma((x,y)): (x,y) \in \mathbb{Z}^{2})$ denoting the tuple of vertex labels, and $\eta$ denoting the tuple of edge labels, $\eta\big((x,y),(x+1,y)\big)$ and $\eta\big((x,y),(x,y+1)\big)$, for each $(x,y) \in \mathbb{Z}^{2}$. We now come up with a modified labeling, $(\sigma',\eta')$, as follows: if, for any $(x,y) \in \mathbb{Z}^{2}$, the label $\eta\big((x,y),(x+1,y)\big)$ equals a target, we set 
\[
 \eta'\big((x,y),(x+1,y)\big) = 
  \begin{cases} 
   &\text{trap with probability } r'(1-s')^{-1} \\
   &\text{open with probability } (1-r'-s')(1-s')^{-1}
  \end{cases}
\]
\big(an analogous re-labeling is applied if $\eta\big((x,y),(x,y+1)\big)$ is a target, instead\big), and we set  
\[
 \sigma'(x,y) = 
  \begin{cases} 
   \sigma(x,y) & \text{if } \sigma(x,y) \in \{\text{trap}, \text{target}\} \\
   \text{ trap} & \text{if } \sigma(x,y) = \text{open}.
  \end{cases}
\]
As has been explained in \S\ref{subsubsec:bond_to_generalized}, this modification reduces what was initially a four-parameter percolation game into a three-parameter one in which $p$ indicates the probability of a vertex being labeled as a trap, $q$ indicates the probability of a vertex being labeled as a target, and $r$ indicates the probability of an edge being labeled as a trap, with
\begin{equation}
p=p'+(1-p'-q')\left(2s'-{s'}^{2}\right),\quad q=q' \quad \text{and} \quad r=\frac{r'}{1-s'}.\nonumber
\end{equation}  
This modification allows the reader to see why, when considering the most generalized framework for percolation games on the $2$-dimensional infinite square lattice, it suffices to work with only \emph{three} categories of labels instead of four (such as we do in this paper, namely, 
\begin{enumerate*}
\item trap for vertices,
\item target for vertices,
\item and trap for edges).
\end{enumerate*}

\subsection{Organization of the rest of the paper}\label{subsec:organization} We describe here the organization of the rest of this paper. In \S\ref{sec:formal_defns_main_results}, we formally describe the two games, \emph{generalized percolation games} and \emph{bond percolation games}, that we study in this paper, although the reader has already been accorded an informal introduction to these. The two main results of this paper, pertaining to the probability of draw in each of these games, have been stated in Theorems~\ref{thm:three-parameter} and \ref{thm:two-parameter}. The probabilistic cellular automata (PCAs) arising as a result of the game rules have been formally described in \S\ref{subsec:formal_PCA}, and the main result pertaining to their ergodicity properties has been stated in Theorem~\ref{thm:PCAs}. We draw the reader's attention to \S\ref{subsubsec:bond_to_generalized}, which shows how the bond percolation game (and correspondingly, the PCA, $\E_{r',s'}$, arising out of its game rules) can be obtained as a special case of the generalized percolation game (and correspondingly, the PCA, $\G_{p,q,r}$, arising out of its game rules). One of this paper's most significant contributions (see also the less formal discussion preceding \S\ref{subsec:three_parameter_suffices}) has been outlined in detail in \S\ref{subsubsec:contribution_PCA}. We have dedicated \S\ref{subsec:motivations_literature} and \S\ref{subsec:motivation_PCA} to discussions on the motivations propelling the study of these games and these PCAs, respectively. The game rules, and how they give rise to the PCAs, $\G_{p,q,r}$ and $\E_{r',s'}$, have been detailed in \S\ref{subsec:recurrence}. The two most technical results of this paper, namely, Theorems~\ref{thm:generalized_envelope_PCA_D} and \ref{thm:bond_envelope_PCA_D}, have been stated in \S\ref{subsec:relation}, along with the proofs of Theorems~\ref{thm:three-parameter}, \ref{thm:two-parameter} and \ref{thm:PCAs} assuming the claims made in Theorems~\ref{thm:generalized_envelope_PCA_D} and \ref{thm:bond_envelope_PCA_D} to be true. The proof of Theorem~\ref{thm:generalized_envelope_PCA_D} has been carried out in \S\ref{sec:generalized_weight_function}. We mention here that while the proof \emph{specifically} of Theorem~\ref{thm:generalized_envelope_PCA_D} has been outlined in detail, step by step, in \S\ref{sec:generalized_wt_fn_steps} (the final weight function, the corresponding final weight function inequality, and the deduction of the claim made in Theorem~\ref{thm:generalized_envelope_PCA_D}, have been included in \S\ref{subsec:final_wt_fn_generalized}), a very general idea on how to construct a suitable weight function that serves our desired purpose has been outlined in \S\ref{subsec:central_ideas_weight_functions}. Likewise, the final weight functions, the corresponding weight function inequalities, and the deduction of the claims stated in Theorem~\ref{thm:bond_envelope_PCA_D}, have been included in \S\ref{subsec:final_wt_fn_bond}, while the details of the proofs have been covered in \S\ref{sec:bond_1_wt_fn_steps}, \S\ref{sec:bond_2_wt_fn_steps} and \S\ref{sec:bond_3_wt_fn_steps}.

\section{Formal description of our games and the main results}\label{sec:formal_defns_main_results}
We shall use $\mathbb{Z}^{2}$ to indicate both the infinite $2$-dimensional square lattice graph and its set of vertices. Each edge of this graph is either of the form $\big((x,y),(x+1,y)\big)$ (directed from $(x,y)$ towards $(x+1,y)$) or of the form $\big((x,y),(x,y+1)\big)$ (directed from $(x,y)$ towards $(x,y+1)$), for $(x,y) \in \mathbb{Z}^{2}$. To each $(x,y) \in \mathbb{Z}^{2}$, we assign, independently, a label that reads \emph{trap} with probability $p$, \emph{target} with probability $q$, and \emph{open} with probability $(1-p-q)$, while to each directed edge, we assign, independently, a label that reads \emph{trap} with probability $r$ and \emph{open} with probability $(1-r)$. We assume that at least one of $p$, $q$ and $r$ is strictly positive, so that our parameter-space $\Theta$ is given by $\Theta=\{(p,q,r) \in [0,1]^{3}: p+q+r>0, p+q \leqslant 1\}$. The resulting graph, bearing (random) vertex-labels and (random) edge-labels, is referred to as a \emph{random board}, on which our \emph{generalized percolation games} are played. Any realization of vertex-labels and edge-labels, assigned as described above, is referred to as a \emph{configuration} on $\mathbb{Z}^{2}$, and we \emph{fix} a configuration \emph{before} the game begins. A generalized percolation game is a \emph{two-player combinatorial game}, which tells us that it is a game of \emph{perfect information} -- both players are fully aware of the environment (in this case, the environment is the configuration fixed before the game begins) on which the game is being played.

A token is placed at a vertex of $\mathbb{Z}^{2}$, referred to as the \emph{initial vertex}, just before the game starts. Two players take turns to make moves, where a move involves relocating the token from where it is currently located, say $(x,y) \in \mathbb{Z}^{2}$, to one of $(x+1,y)$ and $(x,y+1)$. A player wins if she 
\begin{enumerate*}
\item succeeds in moving the token to a vertex labeled as a target,
\item or forces her opponent to move the token to a vertex labeled as a trap,
\item or forces her opponent to move the token along an edge labeled as a trap.
\end{enumerate*}
The game continues for as long as the token stays \emph{on} vertices labeled open and keeps being moved \emph{along} edges labeled open, and this may happen indefinitely, leading to a draw. Of primary interest to us is the investigation of \emph{regimes} of values of the parameter-triple $(p,q,r)$ for which the probability of draw equals $0$. 

We assume that, if the game is destined to end in a finite number of rounds, then the player who is destined to win tries to achieve her victory as quickly as possible, while her opponent attempts to stall and prolong the game as much as possible. It is also crucial to note that the roles of the players are symmetric in this game -- a fact that heavily impacts the simplicity of the recurrence relations that we shall study in \S\ref{subsec:recurrence}.


\begin{theorem}\label{thm:three-parameter}
There exist $\epsilon_{1}$, $\epsilon_{2}$ and $\epsilon_{3}$, in $(0,1)$, such that whenever $(p,q,r) \in \Theta$ with $p \leqslant \epsilon_{1}$, $q \leqslant \epsilon_{2}$ and $r \leqslant \epsilon_{3}$, the inequality in \eqref{three_cond_universal} 
holds, and $(p,q,r)$ satisfies the constraints described in precisely one of \eqref{three_cond_1}, \eqref{three_cond_2}, \eqref{three_cond_3} and \eqref{three_cond_4}, the probability of the event that the generalized percolation game described above, with underlying parameters $p$, $q$ and $r$, results in a draw, equals $0$. Here, \eqref{three_cond_universal}, \eqref{three_cond_1}, \eqref{three_cond_2}, \eqref{three_cond_3} and \eqref{three_cond_4} are as follows:
\begin{equation}\label{three_cond_universal}
2(p+q)+6r^{2}+3r(p+2q) \geqslant 4r+(p+q)^{2},
\end{equation}
\begin{equation}\label{three_cond_1}
    \begin{aligned}
      &p(1-p) \leqslant q\{1-q-r(1-p-q)\}, 
    \end{aligned}
\end{equation} 
\begin{equation}\label{three_cond_2}
  \begin{rcases}
    \begin{aligned}
      &q\{1-q-r(1-p-q)\} < p(1-p) \leqslant \{q+r(1-p-q)\}\{1-q-r(1-p-q)\}, \\
      &p-2q+pq+qr-3pr-2p^{2}+3q^{2} \leqslant 0,
    \end{aligned}
  \end{rcases}
\end{equation}
\begin{equation}\label{three_cond_3}
  \begin{rcases}
    \begin{aligned}
      &q\{1-q-r(1-p-q)\} < p(1-p) \leqslant \{q+r(1-p-q)\}\{1-q-r(1-p-q)\}, \\
      &p-2q+pq+qr-3pr-2p^{2}+3q^{2} > 0,\\
      &p+4q+9pr+5qr+6r^{2} \geqslant 4r+p^{2}+4q^{2}+5pq,
    \end{aligned}
  \end{rcases}
\end{equation}
and
\begin{equation}\label{three_cond_4}
  \begin{rcases}
    \begin{aligned}
      &q+r(1-p-q) < p, \\
      &6q+10pr+4qr+p^{2}+6r^{2} \geqslant 4r+6pq+7q^{2}.
    \end{aligned}
  \end{rcases}
\end{equation}
\end{theorem}

\begin{figure}[htbp]
    \centering
    \begin{subfigure}[b]{0.45\textwidth}
        \centering
        \includegraphics[width=\textwidth]{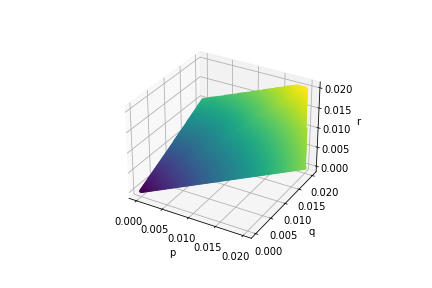}
        \caption{3D plot of \eqref{three_cond_universal} and \eqref{three_cond_1}}
        \label{subfig:3D-region-1}
    \end{subfigure}
    \hfill
    \begin{subfigure}[b]{0.45\textwidth}
        \centering
        \includegraphics[width=\textwidth]{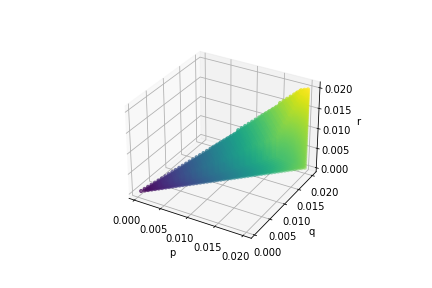}
        \caption{3D plot of \eqref{three_cond_universal} and \eqref{three_cond_2}}
        \label{subfig:3D-region-2}
    \end{subfigure}
    \vskip\baselineskip
    \begin{subfigure}[b]{0.45\textwidth}
        \centering
        \includegraphics[width=\textwidth]{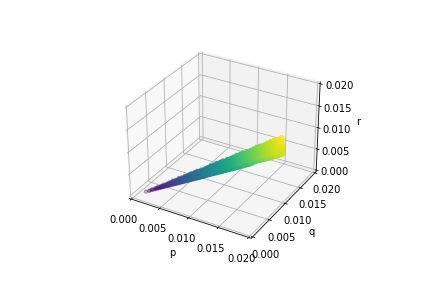}
        \caption{3D plot of \eqref{three_cond_universal} and \eqref{three_cond_3}}
        \label{subfig:3D-region-3}
    \end{subfigure}
    \hfill
    \begin{subfigure}[b]{0.45\textwidth}
        \centering
        \includegraphics[width=\textwidth]{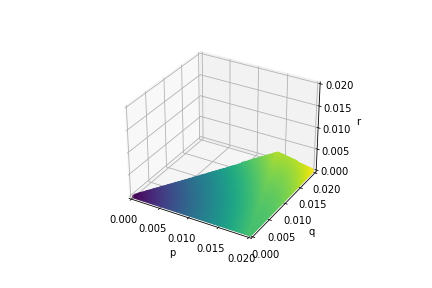}
        \caption{3D plot of \eqref{three_cond_universal} and \eqref{three_cond_4}}
        \label{subfig:3D-region-4}
    \end{subfigure}
    \vskip\baselineskip
    \begin{subfigure}[b]{0.5\textwidth}
        \centering
        \includegraphics[width=\textwidth]{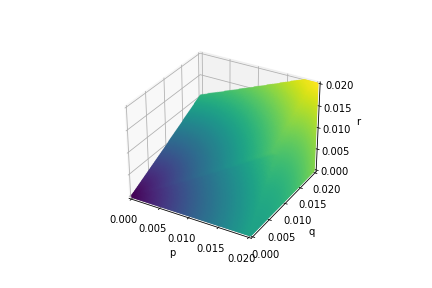}
        \caption{3D plot for the complete region covered by Theorem~\ref{thm:three-parameter}}
        \label{subfig:3D-region-union}
    \end{subfigure}
    \caption{Regions covered by Theorem~\ref{thm:three-parameter}}
    \label{fig:3D-Plots}
\end{figure}

\begin{figure}[htbp]
    \centering
    \begin{subfigure}[b]{0.33\textwidth}
        \centering
        \includegraphics[width=\textwidth]{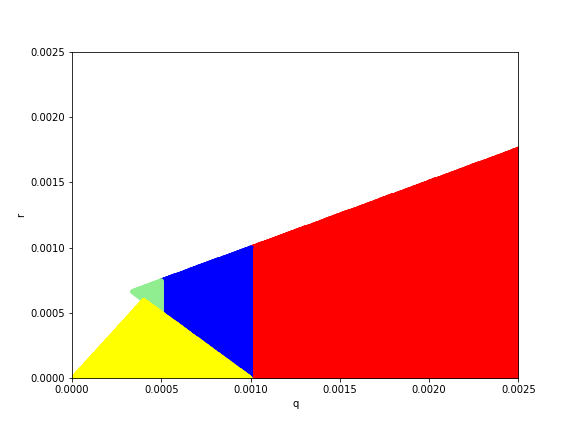}
        \caption{2D plot for $p = 0.001$}
        \label{subfig:region-1-p}
    \end{subfigure}
    \hfill
    \begin{subfigure}[b]{0.33\textwidth}
        \centering
        \includegraphics[width=\textwidth]{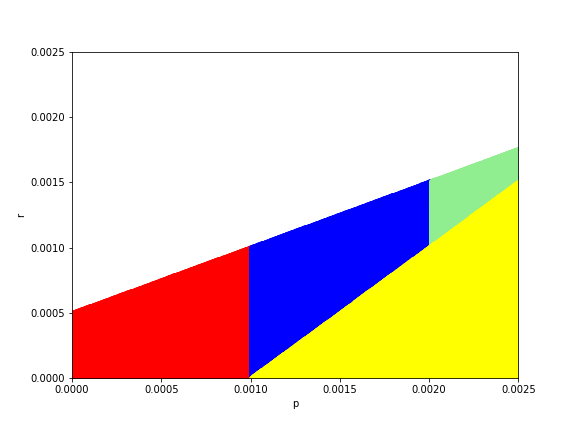}
        \caption{2D plot for $q = 0.001$}
        \label{subgfig:region-2-p}
    \end{subfigure}
    \hfill
    \begin{subfigure}[b]{0.33\textwidth}
        \centering
        \includegraphics[width=\textwidth]{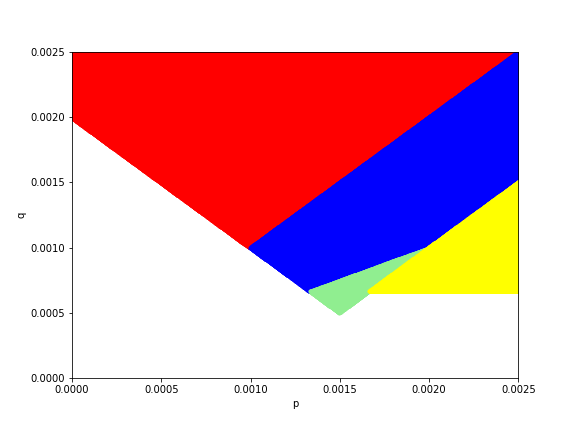}
        \caption{2D plot for $r = 0.001$}
        \label{subfig:region-3-p}
    \end{subfigure}

    \caption{2D plots for regions covered by Theorem~\ref{thm:three-parameter}}
    \label{fig:2D-plots}
\end{figure}

The values of $\epsilon_{1}$, $\epsilon_{2}$ and $\epsilon_{3}$ (i.e.\ how small $p$, $q$ and $r$ need to be for Theorem~\ref{thm:three-parameter} to hold) can be deduced from the various steps, outlined in \S\ref{sec:generalized_wt_fn_steps}, of the proof of Theorem~\ref{thm:generalized_envelope_PCA_D} (which is instrumental in proving Theorem~\ref{thm:three-parameter}). Although better bounds on these are possible, we do not put in much effort to optimize such bounds, and we simply comment here that assuming each of $\epsilon_{1}$, $\epsilon_{2}$ and $\epsilon_{3}$ to be equal to $1/50$ is enough.

In Figure~\ref{fig:3D-Plots}, we plot the regions (of $\Theta$) covered by the inequalities given by \eqref{three_cond_universal}, \eqref{three_cond_1}, \eqref{three_cond_2}, \eqref{three_cond_3} and \eqref{three_cond_4}. As the captions suggest, (A) of Figure~\ref{fig:3D-Plots} shows the coverage of \eqref{three_cond_universal} and \eqref{three_cond_1}, i.e.\ the set of all values of $(p,q,r)$ that satisfy \eqref{three_cond_universal} and \eqref{three_cond_1}, (B) of Figure~\ref{fig:3D-Plots} shows the coverage of \eqref{three_cond_universal} and \eqref{three_cond_2}, and so on. The colours in each of these plots have been applied according to the distance (of the parameter-tuple under consideration) from the origin, with darker colours applied to those near the origin, and lighter colours applied to those farther away. Observe that we consider only the segment from $0$ to $1/50=0.02$ for each of the axes, since Theorem~\ref{thm:three-parameter} assuredly holds when each of $p$, $q$ and $r$ lies in $[0,1/50]$ (along with satisfying \eqref{three_cond_universal} and precisely one of \eqref{three_cond_1}, \eqref{three_cond_2}, \eqref{three_cond_3} and \eqref{three_cond_4}). (E) of Figure~\ref{fig:3D-Plots} shows the union of all four regions demonstrated in (A), (B), (C) and (D) of Figure~\ref{fig:3D-Plots}. In each of these five images, we see that the coloured region has a non-empty intersection with \emph{every} neighbourhood around $(0,0,0)$. Since the 3D plots show only one `surface' of the feasible regions, we demonstrate 
\begin{enumerate}
    \item in (A) of Figure~\ref{fig:2D-plots}, the region covered by our result in Theorem~\ref{thm:three-parameter} when the value of $p$ is kept fixed at $p=0.001$,
    \item in (B) of Figure~\ref{fig:2D-plots}, the region covered by our result in Theorem~\ref{thm:three-parameter} when the value of $q$ is kept fixed at $q=0.001$,
    \item and finally, in (C) of Figure~\ref{fig:2D-plots}, the region covered by our result in Theorem~\ref{thm:three-parameter} when the value of $r$ is kept fixed at $r=0.001$.
\end{enumerate}
We emphasize to the reader that in (A) of Figure~\ref{fig:2D-plots}, the four different colours used are meant to indicate the four distinct regions that are covered by Theorem~\ref{thm:three-parameter}, i.e.\ for $p=0.001$, the region coloured red indicates the set of all values of $(q,r)$ such that $(0.001,q,r)$ satisfies \eqref{three_cond_universal} and \eqref{three_cond_1}, the region coloured blue indicates the set of all values of $(q,r)$ such that $(0.001,q,r)$ satisfies \eqref{three_cond_universal} and \eqref{three_cond_2}, the region coloured green indicates the set of all values of $(q,r)$ such that $(0.001,q,r)$ satisfies \eqref{three_cond_universal} and \eqref{three_cond_3}, and the region coloured yellow indicates the set of all values of $(q,r)$ such that $(0.001,q,r)$ satisfies \eqref{three_cond_universal} and \eqref{three_cond_4}. Similar interpretations apply to (B) and (C) of Figure~\ref{fig:2D-plots}. Notice that the four colour-coded regions in each of (A), (B) and (C) of Figure~\ref{fig:2D-plots} are pairwise non-overlapping -- this is as it should be, because no $(p,q,r) \in \Theta$ can satisfy all the constraints of more than one of \eqref{three_cond_1}, \eqref{three_cond_2}, \eqref{three_cond_3} and \eqref{three_cond_4}.

\begin{remark}\label{rem:three-parameter}
We draw the reader's attention to an important point here. The proof of Theorem~\ref{thm:three-parameter} happens via 
Theorem~\ref{thm:generalized_envelope_PCA_D}, which, in turn, is proved via the technique of weight functions. The construction of a suitable weight function for proving Theorem~\ref{thm:generalized_envelope_PCA_D}, and thereby, Theorem~\ref{thm:three-parameter}, has been outlined in \S\ref{sec:generalized_wt_fn_steps}. As has been emphasized in \S\ref{sec:generalized_wt_fn_steps}, this weight function is not unique, and furthermore, small tweaks here and there in the long and involved process of constructing this weight function, step by step, could lead to very marginal improvements -- we say ``marginal" because such improvements are likely to happen via terms of the form $p^{i}q^{j}r^{k}$, for some $i, j, k \in \mathbb{N}_{0}$, such that $(i+j+k)$ equals $3$ or more, and since we have already assumed that each of $p$, $q$ and $r$ is small, such terms make little difference to the regime obtained in Theorem~\ref{thm:three-parameter}.
\end{remark}

As explained in \S\ref{sec:intro}, the above-described three-parameter set-up is the most generalized of all the games studied in this paper. Setting $r=0$ reduces this set-up to that of the site percolation game studied in \cite{holroyd2019percolation}. In fact, the main result (pertaining to the probability of draw in the site percolation game) of \cite{holroyd2019percolation} follows as a corollary of Theorem~\ref{thm:three-parameter}:
\begin{corollary}
Assuming that Theorem~\ref{thm:three-parameter} holds, the generalized percolation game, with underlying parameters $p$, $q$ and $r$, results in a draw with probability $0$ whenever $p+q > 0$ and $r=0$.
\end{corollary}
\begin{proof}
When $r=0$, the first case to consider is where $p \leqslant q$, which, assuming both $p$ and $q$ to be bounded above by $1/2$, is equivalent to $p(1-p) \leqslant q(1-q)$ (since the function $f(x)=x(1-x)$ is strictly increasing for $x \in [0,1/2)$ and strictly decreasing for $x \in (1/2,1]$). Moreover, since $r=0$, the inequality in \eqref{three_cond_universal} reduces to $2(p+q) \geqslant (p+q)^{2}$, which is indeed true since $p+q \leqslant 1$. Thus, in this case, $(p,q,0)$ satisfies the constraints of \eqref{three_cond_universal} and \eqref{three_cond_1}, leading to the conclusion, from Theorem~\ref{thm:three-parameter}, that the probability of draw in the corresponding generalized percolation game equals $0$.

The second case to consider is where $p > q$. As above, it can be seen that \eqref{three_cond_universal} holds since $r=0$. Furthermore, the second inequality of \eqref{three_cond_4} reduces to 
\begin{align}
    {}&6q+p^{2} \geqslant 6pq+7q^{2} \Longleftrightarrow 6q(1-p-q)+(p-q)(p+q) \geqslant 0,\nonumber
\end{align}
which is indeed true since $p+q \leqslant 1$ and $p > q$. Therefore, the constraints of \eqref{three_cond_universal} and \eqref{three_cond_4} are satisfied by $(p,q,0)$ in this case, and by Theorem~\ref{thm:three-parameter}, we conclude that the probability of draw in the corresponding generalized percolation game equals $0$.

This completes the proof of the claim that the main result of \cite{holroyd2019percolation} follows as a special case of our Theorem~\ref{thm:three-parameter}.
\end{proof}

Next, we come to the following two-parameter set-up: each directed edge of $\mathbb{Z}^{2}$ is assigned, independent of all else, a label that reads \emph{trap} with probability $r'$, \emph{target} with probability $s'$, and \emph{open} with the remaining probability $(1-r'-s')$, with $(r',s')$ belonging to the parameter-space $\Theta'=\{(r',s') \in [0,1]^{2}: 0 < r'+s' \leqslant 1\}$. We refer to the game played on this random board as the \emph{bond percolation game}. The moves permitted in this game are the same as those described above for the generalized percolation game, and a player wins if she is 
\begin{enumerate*}
\item either able to move the token along an edge labeled as a target,
\item or force her opponent to move the token along an edge labeled as a trap.
\end{enumerate*}
The game continues for as long as the token keeps getting moved along edges that have been labeled open, which may happen indefinitely, leading to a draw. Once again, our objective is to find regimes, in terms of values of the parameter-pair $(r',s')$, for which the probability of draw in this game equals $0$, and in Theorem~\ref{thm:two-parameter}, we list three such regimes.
\begin{theorem}\label{thm:two-parameter}
Let $(r',s') \in \Theta'$, where $\Theta'$ is as described above.
\begin{enumerate}[label={(B\arabic*)},leftmargin=*]
\item \label{bond_regime_1} There exist $\epsilon_{1}, \epsilon_{2} \in (0,1)$ such that the probability of draw in the bond percolation game, with underlying parameters $r'$ and $s'$, equals $0$ whenever $r' \leqslant \epsilon_{1}$, $s' \leqslant \epsilon_{2}$ and 
\begin{equation}
\frac{3(2s'-{s'}^{2})^{2}}{2}+\frac{8(2s'-{s'}^{2})r'}{1-s'}+\frac{11{r'}^{2}}{2(1-s')^{2}}\geqslant \frac{2r'}{1-s'}+\frac{9(2s'-{s'}^{2})^{3}}{2}+\frac{16(2s'-{s'}^{2})^{2}r'}{1-s'}+\frac{43(2s'-{s'}^{2}){r'}^{2}}{2(1-s')^{2}}+\frac{5{r'}^{3}}{(1-s')^{3}}.\label{two_regime_1_eq}
\end{equation}
\item \label{bond_regime_2} As long as $s'=0$ and $r'>0.157175$, the probability of draw in the bond percolation game, with underlying parameters $r'$ and $s'$, equals $0$.
\item \label{bond_regime_3} As long as $r'=s' \geqslant 0.10883$, the probability of draw in the bond percolation game, with underlying parameters $r'$ and $s'$, equals $0$. 
\end{enumerate}
\end{theorem}
We draw the reader's attention to the fact that, between them, the regimes described in \ref{bond_regime_1}, \ref{bond_regime_2} and \ref{bond_regime_3} cover a significant subset of the parameter-space $\Theta'$. In particular, \ref{bond_regime_1} covers a significant portion of $\Theta'$ that intersects the neighbourhood around $(r',s')=(0,0)$ -- a region where older techniques in the literature fail to establish that the probability of draw equals $0$. Furthermore, as is the case with the three-parameter generalized percolation game, \ref{bond_regime_1} seems to indicate (and proves it under the restriction that $(r',s')$ satisfies the inequality given by \eqref{two_regime_1_eq}) that a phase transition happens exactly at $(0,0)$ (in the sense that, at $(r',s')=(0,0)$, the probability of draw equals $1$, whereas for $r'$, $s'$ sufficiently small, at least one of them strictly positive and $(r',s')$ satisfying \eqref{two_regime_1_eq}, the probability of draw equals $0$). 

\section{Formal description of the probabilistic cellular automata we are concerned with}\label{subsec:formal_PCA}
As explained towards the end of \S\ref{sec:intro}, the game rules arising from each game considered in this paper can be represented in the form of a probabilistic cellular automaton (PCA), and the question as to whether such a PCA is \emph{ergodic} or not ties in inexorably with the question of whether the corresponding percolation game has a positive chance of resulting in a draw or not. It is important that we recall some of the basic definitions pertaining to PCAs in general, and for this reason, \S\ref{subsubsec:general_PCAs} is a necessary digression.

\subsection{A brief discussion on PCAs in general}\label{subsubsec:general_PCAs} A $d$-dimensional \emph{deterministic cellular automaton} (CA) is a discrete dynamical system with the following components:
\begin{enumerate*}
\item the \emph{universe} $\mathbb{Z}^{d}$ consisting of \emph{cells},
\item a finite set of states $\mathcal{A}$ referred to as the \emph{alphabet},
\item a finite set of indices $\mathcal{N} = \left\{\mathbf{y}_{1}, \mathbf{y}_{2}, \ldots, \mathbf{y}_{m}\right\} \subset \mathbb{Z}^{d}$, referred to as the \emph{neighbourhood-marking set},
\item and a \emph{local update rule} $f: \mathcal{A}^{\mathcal{N}} \rightarrow \mathcal{A}$.
\end{enumerate*}

A \emph{probabilistic cellular automaton} (PCA) $F$ is, as the name suggests, a CA with a \emph{random} or \emph{stochastic} update rule, represented by a stochastic matrix $\varphi: \mathcal{A}^{m} \times \mathcal{A} \rightarrow [0,1]$. A PCA can be viewed as a discrete-time Markov chain on the \emph{state space} $\Omega = \mathcal{A}^{\mathbb{Z}^{d}}$. Given any element $\eta = \left(\eta(\mathbf{x}): \mathbf{x} \in \mathbb{Z}^{d}\right)$ of the state space $\Omega$, referred to as a \emph{configuration}, the state $\eta(\mathbf{x})$ of the cell $\mathbf{x}$, for each $\mathbf{x} \in \mathbb{Z}^{d}$, is updated, independent of the updates happening at all other cells, to the \emph{random} state $F\eta(\mathbf{x})$ whose distribution is given by  
\begin{align}\label{general_update_rule_eq}
\Prob[F\eta(\mathbf{x}) = b\big|\eta(\mathbf{x}+\mathbf{y}_{i}) = a_{i} \text{ for all } 1 \leqslant i \leqslant m] = \varphi(a_{1}, a_{2}, \ldots, a_{m}, b) \text{ for all } b \in \mathcal{A},
\end{align}  
\sloppy for any $a_{1}, a_{2}, \ldots, a_{m} \in \mathcal{A}$. We may imagine that the updates happen at discrete time-steps, starting with an initial configuration $\eta_0$ at time-step $t=0$, and recursively defining, for each $t \in \mathbb{N}$, $\eta_{t} = \left(\eta_{t}(\mathbf{x}): \mathbf{x} \in \mathbb{Z}^{d}\right) = F \eta_{t-1} = \left(F\eta_{t-1}(\mathbf{x}): \mathbf{x} \in \mathbb{Z}^{d}\right)$. This yields $\eta_{t} = F^{t}\eta_{0}$, where $F^{t}$ refers to the update (i.e.\ the application of $F$) having happened $t$ times. 

Given a random configuration $\pmb{\eta}$ which follows a probability distribution $\mu$ (measurable with respect to the $\sigma$-field generated by all cylinder sets of $\Omega$), we let $F^t\mu$ denote the probability distribution of $F^{t}\pmb{\eta}$ (obtained by starting from $\pmb{\eta}$ and performing $t$ updates sequentially via the stochastic update rules of $F$, as shown in \eqref{general_update_rule_eq}). We call a probability distribution $\mu$ \textit{stationary} for $F$ if $F\mu = \mu$. A PCA $F$ is said to be \textit{ergodic} if 
\begin{enumerate}
    \item It has a unique stationary measure, say $\mu$.
    \item $F^t\nu \rightarrow \mu$ weakly, as $t\rightarrow \infty$, for each probability distribution $\nu$ on $\Omega$.
\end{enumerate}


\subsection{The \emph{specific} PCAs we work with in this paper}\label{subsubsec:specific_PCAs}
We begin by describing the PCA that represents the game rules arising from the generalized percolation games described in \S\ref{sec:formal_defns_main_results}. This PCA, denoted $\G_{p,q,r}$ (in order to emphasize its dependence on the parameter-triple $(p,q,r)$), is a $1$-dimensional one (i.e.\ its universe is $\mathbb{Z}$), with
\begin{enumerate}
\item the alphabet $\hat{\mathcal{A}}=\{W,L,D\}$ (already, this choice of notations hints at the intimate connection between the game and this PCA: if $(x,y) \in \mathbb{Z}^{2}$ serves as the initial vertex for a generalized percolation game, we say that $(x,y) \in W$ if this game is won by the player who plays the first round, $(x,y) \in L$ if this game is lost by the player who plays the first round, and $(x,y) \in D$ if this game results in a draw),
\item the neighbourhood-marking set $\mathcal{N}=\{0,1\}$,
\item and stochastic update rules that can be described via the stochastic matrix $\widehat{\varphi}_{p,q,r}: \hat{\mathcal{A}}^{2} \times \hat{\mathcal{A}} \rightarrow [0,1]$, defined as follows (see Figure~\ref{fig:envelope_generalized} for a pictorial illustration):
\begin{equation}\label{GPCA_rule_1}
 \widehat{\varphi}_{p,q,r}(W,W,b) =
  \begin{cases} 
   p & \text{if } b = W, \\
   (1-p) & \text{if } b = L,
  \end{cases}
\end{equation}
\begin{equation}\label{GPCA_rule_2}
 \widehat{\varphi}_{p,q,r}(a_{0},a_{1},b) =
  \begin{cases} 
   p+(1-p-q)(1-r) & \text{if } b = W, \\
   q+(1-p-q)r & \text{if } b = L,
  \end{cases} \quad \text{where } (a_{0}, a_{1}) \in \{(W,L), (L,W)\},
\end{equation}
\begin{equation}\label{GPCA_rule_3}
 \varphi_{p,q,r}(L,L,b) =
  \begin{cases} 
   p+(1-p-q)(1-r^{2}) & \text{if } b = W, \\
   q+(1-p-q)r^{2} & \text{if } b = L,
  \end{cases}
\end{equation}
\begin{equation}\label{GPCA_rule_4}
 \widehat{\varphi}_{p,q,r}(a_{0}, a_{1}, b) =
  \begin{cases} 
   p & \text{if } b = W, \\
   (1-p-q)(1-r) & \text{if } b = D, \\
   (1-p-q)r+q & \text{if } b = L,
  \end{cases} \quad \text{where } (a_{0}, a_{1}) \in \{(W,D), (D,W)\},
\end{equation}
\begin{equation}\label{GPCA_rule_5}
 \widehat{\varphi}_{p,q,r}(a_{0}, a_{1}, b) =
  \begin{cases} 
   p+(1-p-q)(1-r) & \text{if } b = W, \\
   (1-p-q)r(1-r) & \text{if } b = D,\\
   q+(1-p-q)r^{2} & \text{if } b = L,
  \end{cases} \quad \text{where } (a_{0}, a_{1}) \in \{(L,D), (D,L)\},
\end{equation}
\begin{equation}\label{GPCA_rule_6}
 \widehat{\varphi}_{p,q,r}(D,D,b) =
  \begin{cases} 
   p & \text{if } b = W,\\
   (1-p-q)(1-r^{2}) & \text{if } b = D, \\
   q+(1-p-q)r^{2} & \text{if } b = L.
  \end{cases}
\end{equation}
\end{enumerate}
\begin{figure}[h!]
  \centering
    \includegraphics[width=0.85\textwidth]{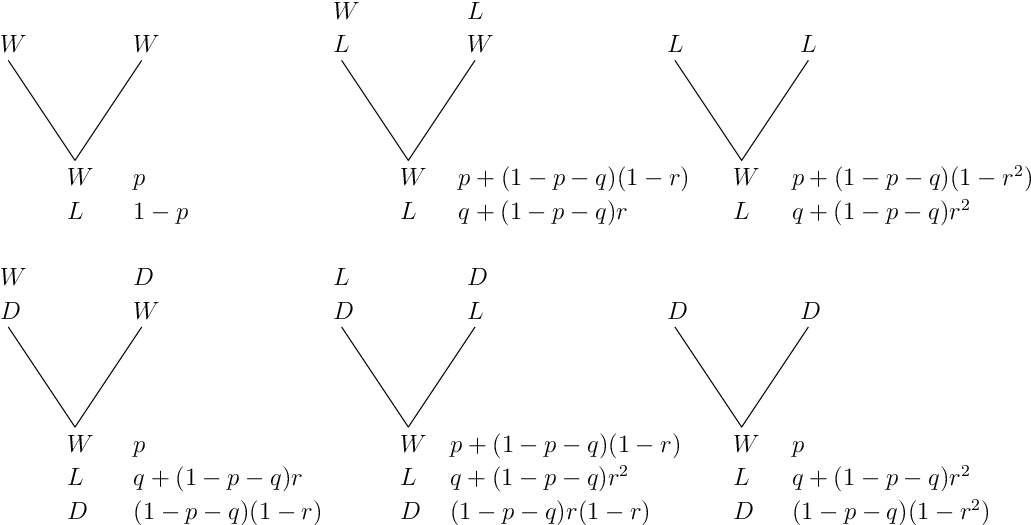}
\caption{Illustrating the stochastic update rules for the PCA $\G_{p,q,r}$}
  \label{fig:envelope_generalized}
\end{figure}
A detailed explanation as to why these stochastic update rules are equivalent to the game rules governing the generalized percolation game, has been provided in \S\ref{subsec:recurrence}. When we restrict $\hat{\mathcal{A}}$ to the smaller alphabet, $\mathcal{A}=\{W,L\}$, we obtain yet another PCA, denoted $G_{p,q,r}$, with stochastic update rules described via the stochastic matrix $\varphi_{p,q,r}: \mathcal{A}^{2} \times \mathcal{A} \rightarrow [0,1]$, where $\varphi_{p,q,r}(a_{0},a_{1},b)=\widehat{\varphi}_{p,q,r}(a_{0},a_{1},b)$ for all $a_{0}, a_{1}, b \in \mathcal{A}$ (in other words, $\varphi_{p,q,r}$ can be defined using \eqref{GPCA_rule_1}, \eqref{GPCA_rule_2} and \eqref{GPCA_rule_3}). We refer to $\G_{p,q,r}$ as the \emph{envelope} to $G_{p,q,r}$ (see \S\ref{subsec:relation}, where we include a discussion on the importance of invoking the notion of envelope PCAs).

We now describe the PCA, denoted $\E_{r',s'}$, that represents the game rules arising from our bond percolation game, with the underlying parameters, $r'$ and $s'$, described in \S\ref{sec:formal_defns_main_results}. Endowed with $\mathbb{Z}$ as its universe, $\E_{r',s'}$ is, just like $\G_{p,q,r}$, a $1$-dimensional PCA with alphabet $\hat{\mathcal{A}}$ and neighbourhood-marking set $\mathcal{N}$, and its stochastic update rules are captured by the stochastic matrix $\widehat{\varphi}_{r',s'}: \hat{\mathcal{A}}^{2} \times \hat{\mathcal{A}} \rightarrow [0,1]$ defined as follows (see Figure~\ref{fig:bond_generalized} for a pictorial illustration):
\begin{equation}\label{envelope_PCA_rule_1}
 \widehat{\varphi}_{r',s'}(W,W,b) =
  \begin{cases} 
   2s'-{s'}^2 & \text{if } b = W, \\
   (1-s')^2 & \text{if } b = L,
  \end{cases}
\end{equation}
\begin{equation}\label{envelope_PCA_rule_2}
 \widehat{\varphi}_{r',s'}(a_{0},a_{1},b) =
  \begin{cases} 
   1-r'+r's' & \text{if } b = W, \\
   r'-r's' & \text{if } b = L,
  \end{cases} \quad \text{where } (a_{0}, a_{1}) \in \{(W,L), (L,W)\},
\end{equation}
\begin{equation}\label{envelope_PCA_rule_3}
 \varphi_{r',s'}(L,L,b) =
  \begin{cases} 
   1-{r'}^{2} & \text{if } b = W, \\
   {r'}^{2} & \text{if } b = L,
  \end{cases}
\end{equation}
\begin{equation}\label{envelope_PCA_rule_4}
 \widehat{\varphi}_{r',s'}(a_{0}, a_{1}, b) =
  \begin{cases} 
   (1-s')(1-r'-s') & \text{if } b = D, \\
   r'(1-s') & \text{if } b = L, \\
   2s'-{s'}^2 & \text{if } b = W,
  \end{cases} \quad \text{where } (a_{0}, a_{1}) \in \{(W,D), (D,W)\},
\end{equation}
\begin{equation}\label{envelope_PCA_rule_5}
 \widehat{\varphi}_{r',s'}(a_{0}, a_{1}, b) =
  \begin{cases} 
   1-r'+r's' & \text{if } b = W, \\
   r'(1-r'-s') & \text{if } b = D,\\
   {r'}^{2} & \text{if } b = L,
  \end{cases} \quad \text{where } (a_{0}, a_{1}) \in \{(L,D), (D,L)\},
\end{equation}
\begin{equation}\label{envelope_PCA_rule_6}
 \widehat{\varphi}_{r',s'}(D,D,b) =
  \begin{cases} 
   (1-r'-s')(1+r'-s') & \text{if } b = D,\\
   {r'}^{2} & \text{if } b = L, \\
   2s'-{s'}^2 & \text{if } b = W.
  \end{cases}
\end{equation}
\begin{figure}[h!]
  \centering
    \includegraphics[width=0.85\textwidth]{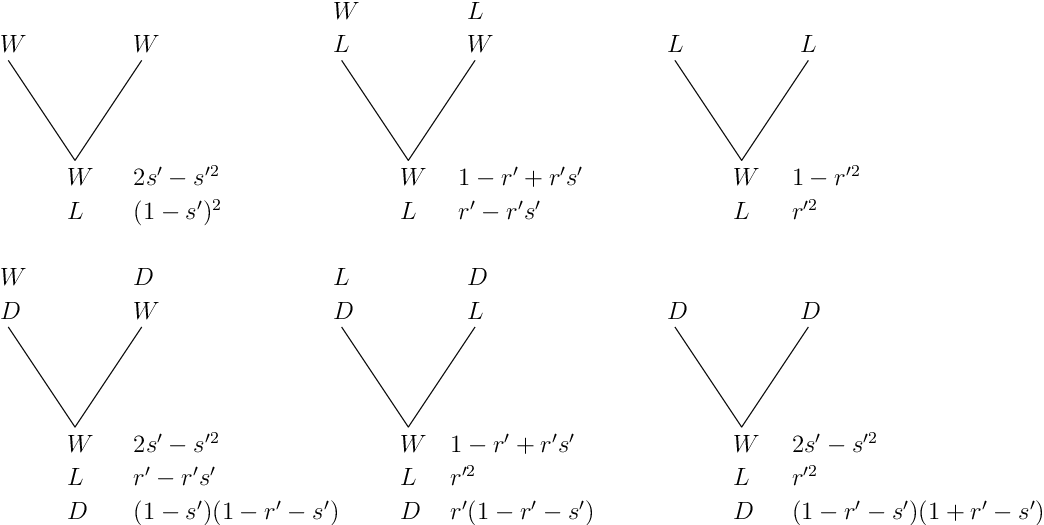}
\caption{Illustrating the stochastic update rules for the PCA $\E_{r',s'}$}
  \label{fig:bond_generalized}
\end{figure}
Once again, we refer the reader to \S\ref{subsec:recurrence} for a detailed explanation as to how the game rules associated with the bond percolation game give rise to the above-mentioned stochastic update rules. If we restrict $\hat{\mathcal{A}}$ to $\mathcal{A}$, we obtain the PCA $E_{r',s'}$ with the same neighbourhood-marking set $\mathcal{N}$, and with stochastic update rules captured by the stochastic matrix $\varphi_{r',s'}: \mathcal{A}^{2} \times \mathcal{A} \rightarrow [0,1]$ such that $\varphi_{r',s'}(a_{0},a_{1},b)=\widehat{\varphi}_{r',s'}(a_{0},a_{1},b)$ for all $a_{0}, a_{1}, b \in \mathcal{A}$ (in other words, the equations defining $\varphi_{r',s'}$ are the same as \eqref{envelope_PCA_rule_1}, \eqref{envelope_PCA_rule_2} and \eqref{envelope_PCA_rule_3}). As is the case with $G_{p,q,r}$ and $\G_{p,q,r}$, we refer to $\E_{r',s'}$ as the envelope for the PCA $E_{r',s'}$. 

We end \S\ref{subsec:formal_PCA} with the main results of this paper pertaining to the PCAs described in \S\ref{subsubsec:specific_PCAs}:
\begin{theorem}\label{thm:PCAs}
The PCA $G_{p,q,r}$, as well as its envelope $\G_{p,q,r}$, is ergodic whenever the underlying parameter-triple $(p,q,r)$ belongs to the parameter-space $\Theta$, where $\Theta$ is as defined in \S\ref{sec:formal_defns_main_results}, and satisfies the constraints stated in Theorem~\ref{thm:three-parameter} (in other words, when $p$, $q$ and $r$ are sufficiently small, the inequality \eqref{three_cond_universal} holds, and $(p,q,r)$ satisfies one of \eqref{three_cond_1}, \eqref{three_cond_2}, \eqref{three_cond_3} and \eqref{three_cond_4}). 

Likewise, the PCA $E_{r',s'}$, as well as its envelope $\E_{r',s'}$, is ergodic whenever the underlying parameter-pair $(r',s')$ belongs to the parameter-space $\Theta'$, where $\Theta'$ is as defined in \S\ref{sec:formal_defns_main_results}, and $(r',s')$ satisfies any one of the constraints stated in Theorem~\ref{thm:two-parameter}, i.e.\ either \ref{bond_regime_1} or \ref{bond_regime_2} or \ref{bond_regime_3}.
\end{theorem}

\subsection{How the PCA $\E_{r',s'}$ is obtained as a special case of the PCA $\G_{p,q,r}$}\label{subsubsec:bond_to_generalized} Let us consider a bond percolation game with the underlying parameter-pair $(r',s')$, i.e.\ where $r'$ indicates the probability of an edge being labeled as a trap and $s'$ indicates that of it being labeled as a target. Letting $\eta$ denote this random assignment of labels (i.e.\ $\eta$ is the infinite tuple made up of $\eta\big((x,y),(x+1,y)\big)$ and $\eta\big((x,y),(x,y+1)\big)$ for all $(x,y)\in\mathbb{Z}^{2}$), we tweak $\eta$ as follows:
\begin{enumerate}
\item if $\eta\big((x,y),(x+1,y)\big)$ (respectively, $\eta\big((x,y),(x,y+1)\big)$) equals a target, we re-label the edge $\big((x,y),(x+1,y)\big)$ (respectively, the edge $\big((x,y),(x,y+1)\big)$), independent of all else, to $\eta'\big((x,y),(x+1,y)\big)$ (respectively, to $\eta'\big((x,y),(x,y+1)\big)$), where
\[
 \eta'\big((x,y),(x+1,y)\big) = 
  \begin{cases} 
   &\text{trap with probability } r'(1-s')^{-1} \\
   &\text{open with probability } (1-r'-s')(1-s')^{-1}
  \end{cases}
\]
(the same distribution is adopted for $\eta'\big((x,y),(x,y+1)\big)$ if $\eta\big((x,y),(x,y+1)\big)$ equals a target), and we label the vertex $(x,y)$ a trap, i.e.\ we set $\sigma'(x,y)$ to be equal to a trap;
\item if $\eta\big((x,y),(x+1,y)\big)$ (respectively, $\eta\big((x,y),(x,y+1)\big)$) is either open or a trap, we keep its label unchanged, i.e.\ we set $\eta'\big((x,y),(x+1,y)\big)=\eta\big((x,y),(x+1,y)\big)$ (respectively, $\eta'\big((x,y),(x,y+1)\big)=\eta\big((x,y),(x,y+1)\big)$); 
\item for any $(x',y') \in \mathbb{Z}^{2}$ for which neither $\eta\big((x,y),(x+1,y)\big)$ nor $\eta\big((x,y),(x,y+1)\big)$ equals a target, we label $(x,y)$ open, i.e.\ we set $\sigma'(x,y)$ to be open.
\end{enumerate}
The modified labeling is now $(\sigma',\eta')$, where $\sigma'$ is the tuple comprising $\sigma'(x,y)$ for all $(x,y) \in \mathbb{Z}^{2}$, and $\eta'$ is the tuple comprising $\eta'\big((x,y),(x+1,y)\big)$ and $\eta'\big((x,y),(x,y+1)\big)$ for all $(x,y)\in\mathbb{Z}^{2}$. From the definition of $(\sigma',\eta')$, it is evident that the labels $\sigma'(x,y)$ are i.i.d.\ over all $(x,y) \in \mathbb{Z}^{2}$, with
\begin{align}
p={}&\Prob[\sigma'(x,y)=\text{trap}]\nonumber\\
={}&\Prob\left[\text{at least one of } \eta\big((x,y),(x+1,y)\big) \text{ and } \eta\big((x,y),(x,y+1)\big) \text{ equals a target}\right]\nonumber\\
={}& 1-(1-s')^{2}=2s'-{s'}^{2},\label{transform_1}
\end{align}
and likewise, the labels $\eta'(e)$ are i.i.d.\ over all directed edges $e$ of our graph, with
\begin{align}
r={}&\Prob[\eta'(e)=\text{trap}]\nonumber\\
={}&\Prob[\eta(e)=\text{trap}]+\Prob[\eta'(e)=\text{trap}\big|\eta(e)=\text{target}]\Prob[\eta(e)=\text{target}]\nonumber\\
={}&r'+\frac{r'}{1-s'} \cdot s' = \frac{r'}{1-s'}.\label{transform_2}
\end{align}
Moreover, it is easily verified that the tuples $\sigma'$ and $\eta'$ are independent of each other as well: for instance, we have
\begin{align}
{}&\Prob\left[\sigma'(x,y)=\text{trap}, \eta'\big((x,y),(x+1,y)\big)=\text{trap}\right]\nonumber\\
={}&\Prob\left[\eta\big((x,y),(x+1,y)\big)=\text{target},\eta'\big((x,y),(x+1,y)\big)=\text{trap}\right]\nonumber\\&+\Prob\left[\eta\big((x,y),(x+1,y)\big)=\text{trap},\eta\big((x,y),(x,y+1)\big)=\text{target}\right]\nonumber\\
={}&s' \cdot \frac{r'}{1-s'} + r' \cdot s' = \frac{r' s' (2-s')}{1-s'} = \left(2s'-{s'}^{2}\right) \frac{r'}{1-s'} = pr,\nonumber
\end{align}
as desired. It is now straightforward to see that by choosing $q=0$, $p$ as in \eqref{transform_1} and $r$ as in \eqref{transform_2}, we can reduce \eqref{GPCA_rule_1} to \eqref{envelope_PCA_rule_1}, \eqref{GPCA_rule_2} to \eqref{envelope_PCA_rule_2} and so on.

This transformation -- from a generalized percolation game with the underlying parameter-triple $(p,0,r)$ to a bond percolation game with the underlying parameter-pair $(r',s')$, where $p$ is as given by \eqref{transform_1} and $r$ is as given by \eqref{transform_2} -- as established above, lies at the heart of the proof that in the regime covered by \ref{bond_regime_1} of Theorem~\ref{thm:two-parameter}, the probability of draw equals $0$.  

\subsection{A major contribution of this paper in establishing ergodicity of elementary PCAs}\label{subsubsec:contribution_PCA} A one-dimensional PCA is said to be \emph{elementary} if the cardinality of its alphabet as well as its neighbourhood-marking set equals $2$. The definitions of $G_{p,q,r}$ and $E_{r',s'}$ in \S\ref{subsubsec:specific_PCAs} immediately reveal that each of them is an elementary PCA. In Chapter 7 of \cite{PCA_survey_old}, two fundamental results pertaining to ergodicity of elementary PCAs have been proposed. We first state these two results in terms of the notation used in this paper. We let $\mathcal{A}=\{W,L\}$ (as has been defined in \S\ref{subsubsec:specific_PCAs}) denote the alphabet of \emph{any} elementary PCA $F$, and we let
\begin{equation}
\theta_{i,j}=\Prob\left[F\eta(x)=L\big|\eta(x+y_{1})=i,\eta(x+y_{2})=j\right], \text{ for } i,j \in \mathcal{A},\nonumber
\end{equation}
where $\{y_{1},y_{2}\}$ forms the neighbourhood-marking set of $F$. Thus, the PCA $F$ is completely specified by the parameters $\theta_{W,W}$, $\theta_{W,L}$, $\theta_{L,W}$ and $\theta_{L,L}$. We call $F$ \emph{symmetric} when $\theta_{W,L}=\theta_{L,W}$, in which case $F$ is specified by only three parameters, namely, $\theta_{W,W}$, $\theta_{W,L}$ and $\theta_{L,L}$. The following are two well-known results when it comes to ergodicity properties of symmetric elementary PCAs, stated in Chapter 7 of \cite{PCA_survey_old}:
\begin{enumerate}[label=(\alph*)]
\item \label{ergodic_criterion_1} the PCA $F$ is ergodic when $\theta_{W,W}$, $\theta_{W,L}$ and $\theta_{L,L}$ satisfy the inequalities:
\begin{align}
{}& 0 < \theta_{W,W}, \theta_{W,L}, \theta_{L,L} < 1,\label{gen_erg_1}\\
{}& \theta_{L,L} > \theta_{W,W} - 2\theta_{W,L},\label{gen_erg_2}\\
{}& \theta_{L,L} > \theta_{W,W} - 2(1-\theta_{W,L});\label{gen_erg_3}
\end{align}
\item \label{ergodic_criterion_2} the PCA $F$ is ergodic when $\theta_{W,W}$, $\theta_{W,L}$ and $\theta_{L,L}$ satisfy the inequality:
\begin{align}
\max\{|\theta_{i,j}-\theta_{k,\ell}|:i,j,k,\ell \in \{W,L\}\}+2\max\{|\theta_{L,L}-\theta_{W,L}|,|\theta_{W,W}-\theta_{W,L}|\} < 2.\label{gen_erg_4}
\end{align}
\end{enumerate}
It has been mentioned in Chapter 7 of \cite{PCA_survey_old} that these two regimes together cover more than $90\%$ of the unit cube $[0,1]^{3}$ defined by the parameter-triple $(\theta_{W,W},\theta_{W,L},\theta_{L,L})$, and that the only region where no method of proving or disproving ergodicity for symmetric elementary PCAs, in general, is known is the union of the neighbourhoods of the points $(\theta_{W,W},\theta_{W,L},\theta_{L,L})=(1,0,0)$ and $(\theta_{W,W},\theta_{W,L},\theta_{L,L})=(1,1,0)$.

We now come to the two symmetric elementary PCAs that are of interest to us in this paper, namely $G_{p,q,r}$ and $E_{r',s'}$. We see, from \eqref{GPCA_rule_1}, \eqref{GPCA_rule_2} and \eqref{GPCA_rule_3}, that $\theta_{W,W}=(1-p)$, $\theta_{W,L}=q+(1-p-q)r$ and $\theta_{L,L}=q+(1-p-q)r^{2}$ for $G_{p,q,r}$, so that we have $(\theta_{W,W},\theta_{W,L},\theta_{L,L})=(1,0,0)$ if and only if $p=q=r=0$. Thus, the neighbourhood of $(1,0,0)$ is obtained when we consider as small a value of each of $p$, $q$ and $r$ as we desire. We then show, in Theorem~\ref{thm:three-parameter}, that if, in addition to $(\theta_{W,W},\theta_{W,L},\theta_{L,L})$ being in the neighbourhood of $(1,0,0)$ (equivalently, $(p,q,r)$ being arbitrarily close to $(0,0,0)$), at least one of $p$, $q$ and $r$ is strictly positive and the constraints specified in Theorem~\ref{thm:three-parameter} are satisfied, the PCA $G_{p,q,r}$ is ergodic -- thereby rigorously proving the ergodicity of $G_{p,q,r}$ in a \emph{considerable} chunk of the neighbourhood of $(1,0,0)$. Likewise, from \eqref{envelope_PCA_rule_1}, \eqref{envelope_PCA_rule_2} and \eqref{envelope_PCA_rule_3}, we see that $\theta_{W,W}=(1-s')^{2}$, $\theta_{W,L}=r'(1-s')$ and $\theta_{L,L}={r'}^{2}$ for the PCA $E_{r',s'}$, so that we have $(\theta_{W,W},\theta_{W,L},\theta_{L,L})=(1,0,0)$ if and only if $r'=s'=0$. Thus, the neighbourhood of $(1,0,0)$ is obtained when we consider as small a value of each of $r'$ and $s'$ as we desire. Part \ref{bond_regime_1} of Theorem~\ref{thm:two-parameter} then shows that, in addition to $(\theta_{W,W},\theta_{W,L},\theta_{L,L})$ being in the neighbourhood of $(1,0,0)$ (equivalently, $(r',s')$ being arbitrarily close to $(0,0)$), if at least one of $r'$ and $s'$ is strictly positive and \eqref{two_regime_1_eq} holds, the PCA $E_{r',s'}$ is ergodic -- thereby rigorously proving the ergodicity of $E_{r',s'}$ in a considerable chunk of the neighbourhood of $(1,0,0)$.

In fact, even the regimes described by \ref{bond_regime_2} and \ref{bond_regime_3} of Theorem~\ref{thm:two-parameter} go significantly beyond those that are obtained via \eqref{ergodic_criterion_1} and \eqref{ergodic_criterion_2} for the PCA $E_{r',s'}$. For instance, when $s'=0$, the criterion in \eqref{gen_erg_1} fails to hold for $E_{r',s'}$ since $\theta_{W,W}=1$, so that \ref{bond_regime_2} falls outside of the regime described in \eqref{ergodic_criterion_1}. On the other hand, $\max\{|\theta_{i,j}-\theta_{k,\ell}|:i,j,k,\ell \in \mathcal{A}\} = 1-{r'}^{2}$ and $\max\{|\theta_{L,L}-\theta_{W,L}|,|\theta_{W,W}-\theta_{W,L}|\} = 1-r'$ when $s'=0$, so that for \eqref{gen_erg_4} to hold, we must have $1-{r'}^{2}+2(1-r') = 3-2r'-{r'}^{2} < 2 \Longleftrightarrow 2r'+{r'}^{2} > 1 \Longleftrightarrow r' > 0.414214$. This tells us that the regime described in \ref{bond_regime_2}, i.e.\ $s'=0$ and $r' > 0.157175$, goes \emph{far beyond} the regime obtained by implementing \eqref{ergodic_criterion_2}.

When $r'=s'$ (obviously, the common value for $r'$ and $s'$ in this case must be bounded above by $0.5$, since $r'+s' \leqslant 1$), we have $\theta_{W, W} = (1-r')^{2}$, $\theta_{W,L} = r'-{r'}^{2}$ and $\theta_{L,L} = {r'}^{2}$. It is immediate that these values of $\theta_{W,W}$, $\theta_{L,W}$ and $\theta_{L,L}$ satisfy \eqref{gen_erg_1} when $r' \in (0,1)$. When examining if the inequality in \eqref{gen_erg_2} is satisfied or not, we observe that
\begin{equation*}
    \theta_{L,L} > \theta_{W,W} - 2\theta_{W,L}
    \Longleftrightarrow {r'}^2 > (1-r')^2 - 2(r'-{r'}^2) 
    \Longleftrightarrow 2{r'}^2 - 4r' +  1 < 0 
    \Longleftrightarrow r' > 1 - \frac{1}{\sqrt{2}} \approx 0.293.
\end{equation*}
When checking whether the inequality in \eqref{gen_erg_3} is satisfied or not, we obtain
\begin{equation*}
    \theta_{L,L} > \theta_{W,W} - 2(1 - \theta_{W,L})
    \Longleftrightarrow {r'}^2 > (1-r')^2 - 2(1-r'+{r'}^2)
    \Longleftrightarrow 2{r'}^2 + 1 >0, 
\end{equation*}
which is indeed true for all $(r',s')$ belonging to the regime in \ref{bond_regime_3}. We thus conclude that when $r'=s'$, we require $r' > 0.293$ for the inequalities in \eqref{ergodic_criterion_1} to hold simultaneously, whereas our result in this paper is able to establish ergodicity for $E_{r',s'}$ whenever $r' = s' > 0.10883$. We now focus on the regime covered by \eqref{ergodic_criterion_2} when $r'=s'$. We observe that
\begin{equation}
        \max\{|\theta_{W, W} - \theta_{W, L}|, |\theta_{W, W} - \theta_{L, L}, |\theta_{W, L} - \theta_{L, L}|\} = \max\{|r'(2r'-1)|, |2r'-1|, |(1-r')(2r'-1)|\} = 1-2r' \nonumber
    \end{equation}
	and 
    \begin{equation}
        \max\{|\theta_{W, W} - \theta_{W, L}|, |\theta_{W, L} - \theta_{L, L}|\} = \max\{|r'(2r'-1)|, |(1-r')(2r'-1)|\} =  (1-r')(1-2r'), \nonumber
    \end{equation} 
so that \ref{gen_erg_4} reduces to
\begin{equation}
    1-2r' + 2(1-r')(1-2r') < 2 \Longleftrightarrow 4{r'}^2-8r'+1<0 \Longleftrightarrow r'>\frac{2-\sqrt{3}}{2} \approx 0.13397.\nonumber
\end{equation}
Once again, our result, in \ref{bond_regime_3}, establishes ergodicity for $E_{r',s'}$ whenever $r' = s' > 0.10883$, a significant improvement on the lower bound that is provided by \eqref{ergodic_criterion_2}.

\section{Motivations for studying generalized / bond percolation games, and a brief review of pertinent literature}\label{subsec:motivations_literature}

\subsection{In oligopolistic competitions} The bond percolation game can be interpreted as a stylized model of \emph{oligopolistic competition} between two firms. At each stage, an action labeled as a \emph{trap} represents a significant loss that forces a firm to exit the market, while an action labeled as a \emph{target} represents a sufficient gain that secures monopoly. The environment (nature) randomly and independently selects the set of available actions at each stage, and these are revealed to the firms at the outset, modeled via a labeled directed graph. Each firm aims to strategically navigate the graph to reach a winning situation.
A natural extension of this model would be one in which the environment selects available actions based on the history of previous moves, introducing a dependency structure and dynamic information revelation.

Analogous dynamics arise in courtroom settings, where the prosecution and defense take sequential actions. A \emph{target} corresponds to decisive evidence in favor of a party, while a \emph{trap} denotes conclusive evidence leading to defeat. The sequential nature of legal argumentation and information disclosure closely mirrors the game-theoretic framework described above.

\subsection{As an adversarial avatar of percolation} Another usefulness of our generalized / bond percolation game lies in incorporating an \emph{adversarial} element into the notion of ordinary site and / or bond percolation on directed infinite graphs. Percolation, by itself, is a vast and ever-expanding area of research that spans a multitude of disciplines, including statistical physics, chemistry, computer science, biology and sociology (for instance, in the understanding of phenomena such as \emph{polymeric gelation}, the spread of a forest fire or the propagation of an epidemic, the transportation of a fluid through a porous material etc.). Percolation games (more specifically, \emph{site percolation games}) were first introduced in \cite{holroyd2019percolation}, where each \emph{vertex} of the infinite $2$-dimensional square lattice $\mathbb{Z}^{2}$ was assigned, independently, a label that read \emph{trap} with probability $p$, \emph{target} with probability $q$, and \emph{open} with probability $(1-p-q)$. The two players were allowed the same moves as those permitted in our game in \S\ref{sec:formal_defns_main_results}, and a player won if she could move the token to a vertex labeled as a target or force her opponent to move the token to a vertex labeled as a trap. It was shown that the probability of draw in this game is $0$ whenever $(p+q) > 0$. The set-up described for our generalized percolation game in \S\ref{sec:formal_defns_main_results} is a natural and important generalization of this model, as it additionally considers a random assignment of labels to the \emph{edges} of $\mathbb{Z}^{2}$. The result presented in \cite{holroyd2019percolation} was extended to the set-up considered in \cite{bhasin2022class}, where the token was permitted to be moved from where it was currently located, say the vertex $(x,y)$, to any one of $(x,y+1)$, $(x+1,y+1)$ and $(x+2,y+1)$. As in \cite{holroyd2019percolation}, it was shown in \cite{bhasin2022class} that the probability of draw is $0$ whenever at least one of $p$ and $q$ is strictly positive. A different direction of generalization of the above-mentioned site percolation game on $\mathbb{Z}^{2}$ was pursued in \cite{bhasin2022ergodicity}, by allowing the token to be moved from $(x,y)$ to one of $(x,y+1)$ and $(x+1,y+1)$ if $x$ is even, and from $(x,y)$ to one of $(x+1,y+1)$ and $(x+2,y+1)$ if $x$ is odd. 

It is worthwhile to note here that when $s'=0$ and $r' > 0$ in the set-up described for our bond percolation game in \S\ref{sec:formal_defns_main_results}, the game is essentially a \emph{normal} game that is being played on the infinite oriented $2$-dimensional square lattice, $\mathbb{Z}^{2}$, once the process of bond percolation, with edge-deletion probability $r'$, has been performed on it. This can be justified as follows: each edge of $\mathbb{Z}^{2}$ is, independently, retained with probability $(1-r')$ and deleted with probability $r'$. A move in the normal game played on this premise involves relocating the token from where it is currently located, say the site $(x,y)$, to either $(x+1,y)$ (if the directed edge $\big((x,y),(x+1,y)\big)$ has not been deleted) or to $(x,y+1)$ (if the directed edge $\big((x,y),(x,y+1)\big)$ has not been deleted). A player loses if she is unable to make a move (i.e.\ there are no outgoing edges from the site where the vertex is currently located). The game continues indefinitely if neither player ever encounters a site from which both outgoing edges have been deleted, and this is possible only if there exists an infinite path starting from the initial vertex, i.e.\ if percolation happens. Thus, in such a scenario, the probability of draw is bounded above by the probability of occurrence of percolation, i.e.\ the probability of the existence of an infinite path that begins from the initial vertex. In this context, we refer the reader to \cite{holroyd2021galton}, where the normal game, as well as the \emph{mis\`{e}re} and the \emph{escape} games have been studied on rooted Galton-Watson trees.  

\subsection{A brief discussion on Maker-Breaker percolation games} A different class of \emph{two-player combinatorial games} pertaining to percolation, broadly referred to as the \emph{Maker-Breaker percolation games}, was introduced in \cite{day2021maker}, and further explored in \cite{day2021makerescaping}, \cite{dvovrak2021maker}, \cite{wallwork2022maker} and \cite{dvovrak2024maker}. Each such game is played on a graph referred to as a \emph{board}, and two players, titled \emph{Maker} and \emph{Breaker}, take turns, with Maker claiming $m$ (as yet unclaimed) edges of the board during each of her turns, and Breaker deleting $b$ (as yet unclaimed) edges of the board during each of hers. In \cite{day2021maker}, the board considered is an $m \times n$ rectangular grid, and Maker wins the \emph{$(m,b)$-crossing game} if she manages to claim all edges constituting a \emph{crossing path} joining the left boundary of the grid to its right boundary. In \cite{day2021makerescaping}, the board considered is an infinite connected graph $\Lambda$ (such as the $2$-dimensional infinite square lattice $\mathbb{Z}^{2}$ and the $d$-regular tree in which each vertex has degree $d$), with a vertex $v_{0} \in \Lambda$ specified, and Breaker wins if at any point of time during the game, the connected component of $\Lambda$ containing $v_{0}$ becomes finite. In \cite{dvovrak2021maker}, a critical lower threshold for $b$ in terms of $m$ was found such that if $b$ were to exceed this value, Breaker would win the Maker-Breaker percolation game on $\mathbb{Z}^{2}$; it was also shown that when the board is $\mathbb{Z}^{2}$ \emph{after} the usual bond percolation process with parameter $p$ has been performed on it, with $p$ not too large compared to $1/2$, Breaker almost surely wins the \emph{unbiased} Maker-Breaker percolation game (i.e.\ where $m=b=1$). In \cite{dvovrak2024maker}, the same random board as in \cite{dvovrak2021maker} was considered, and it was shown that 
\begin{enumerate*}
\item when $p < 1$, Breaker almost surely wins the unbiased Maker-Breaker percolation game, 
\item and when $m=2$ and $b=1$, Maker almost surely wins whenever $p > 0.9402$, while Breaker almost surely wins whenever $p < 0.5278$.
\end{enumerate*}
In \cite{wallwork2022maker}, the $(m,b)$-crossing game was studied on a triangular grid that was $m$ vertices across and $n$ vertices high. 

\section{Motivation for studying the PCAs we consider in this paper}\label{subsec:motivation_PCA}
Recall from \S\ref{subsubsec:general_PCAs} that PCAs are obtained as \emph{random perturbations} of CAs. PCAs find applications in various domains like fault-tolerant computation, classifying CAs based on robustness, connections to Gibbs potentials in statistical mechanics, and modeling complex systems in physics, chemistry, and biology. For a comprehensive survey on PCAs and their development, we refer the reader to \cite{PCA_survey} and \cite{PCA_survey_old}. The applications of PCAs are extensive, spanning probability, statistical mechanics, computer science, natural sciences, dynamical systems, and computational cell biology, as showcased in \cite{louis_nardi}, \cite{stat_mech_PCA}, \cite{DNA_PCA}, \cite{crystal_plasticity_PCA}, and many others.

But why are the \emph{specific} PCAs we study, i.e.\ $G_{p,q,r}$, $\G_{p,q,r}$, $E_{r',s'}$ and $\E_{r',s'}$, of interest (independent of the significance they have in the context of our bond percolation games)? Focusing on $E_{r',s'}$, we see that it can be interpreted as a \emph{learning model} for \emph{social learning} in a system of \emph{interacting particles} (these particles are often thought of as \emph{players} or \emph{agents}). The notion of social learning was introduced in \cite{ellison1993rules}, which studied how the speed of learning and market equilibrium were impacted by social networks and other institutions governing communication among market participants. In \cite{bala1998learning}, a general framework in which agents, unaware of the payoffs from different actions, use their own past experience and the experience of their neighbours to guide their decision making. In \cite{bala2000noncooperative}, an approach to network formation is studied in which it is assumed that a link formed by one agent with another allows access, in part and in due course, to the benefits available to the latter agent via their own links. In \cite{chatterjee2004technology}, a model of social learning in a population of myopic, memoryless individuals is studied in which the agents are placed on $\mathbb{Z}$ and at each time-step, each agent performs an experiment using the technology they currently possess, then takes into account the outcome of their own experiment as well as the outcomes of the experiments performed by their neighbours. 

Coming back to our PCA $E_{r',s'}$, we refer to Figure~\ref{fig_1}. Let each vertex or site on the integer line $\mathbb{Z}$ be inhabited by an agent, and at the beginning of each epoch (here, the epochs are indexed by the set $\mathbb{N}_{0}$ of non-negative integers), each agent can avail one of two technologies: $W$ and $L$. During each epoch, each agent performs, using the technology it has chosen to adopt at the beginning of that epoch, an experiment that has two possible outcomes: success and failure. An agent using technology $W$ has probability $s'$ of achieving success, while an agent equipped with technology $L$ has probability $(1-r')$ of achieving success. It is assumed that these outcomes occur independently for all agents on $\mathbb{Z}$, over all epochs in $\mathbb{N}_{0}$. At the beginning of each epoch, an agent looks at itself and its nearest neighbour to the right, and 
\begin{enumerate*}
\item updates its technology to $L$ if and only if \emph{both} of them suffered failures at the previous time-step, 
\item or else, it updates its technology to $W$.
\end{enumerate*}
Figure~\ref{fig_1} reveals that this is exactly how updates happen when we apply the stochastic update rules corresponding to $E_{r',s'}$ (recall these rules from \S\ref{subsubsec:specific_PCAs}). An understanding of the ergodicity, and subsequently, limiting distribution(s) of $E_{r',s'}$ will, therefore, reveal how the diffusion of the technologies, $W$ and $L$, happens throughout this system of interacting agents as time approaches $\infty$.
\begin{figure}[h!]
  \centering
    \includegraphics[width=0.7\textwidth]{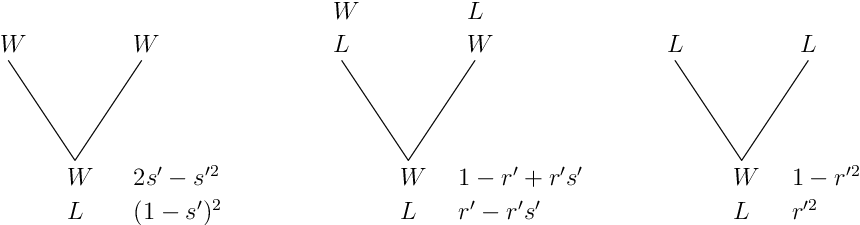}
\caption{Illustrating the stochastic update rules of the PCA $E_{r',s'}$}
  \label{fig_1}
\end{figure}

Referring to Figure 1 of \cite{PCA_survey} and identifying $W$ with the symbol $0$ and $L$ with the symbol $1$, we see that $E_{r',s'}$ can be viewed as a generalization of the \emph{noisy additive PCA}, since $E_{r',s'}$ allows us to introduce randomness into the update $\eta_{t+1}(x)$ that happens conditioned on $\left(\eta_{t}(x),\eta_{t}(x+1)\right) = (0,0)$ or on $\left(\eta_{t}(x), \eta_{t}(x+1)\right) = (1,1)$ (where recall, from \S\ref{subsubsec:general_PCAs} that $\eta_{t}(x)$ is the state of the site $x \in \mathbb{Z}$ at time-step $t$). We may even provide an interpretation to $E_{r',s'}$ that is similar to the interpretation of how we obtain noisy additive PCAs: given $\eta_{t}$, for each $x \in \mathbb{Z}$, we first set the value at $x$ to be $\frac{1}{2}\eta_{t}(x) + \frac{1}{2}\eta_{t}(x+1)$, and then obtain $\eta_{t+1}(x)$ by
\begin{enumerate}
\item flipping a $0$ to a $1$ with probability $(1-s')^{2}$,
\item flipping a $1$ to a $0$ with probability $1-{r'}^{2}$,
\item flipping a $\frac{1}{2}$ to a $1$ with probability $r'(1-s')$ and to a $0$ with probability $1-r'(1-s')$.
\end{enumerate}
Referring to Figure 7 of \cite{PCA_survey}, we see that $E_{r',s'}$ can also be seen as a generalization of the \emph{directed animals PCA}, allowing us to incorporate randomness into $\eta_{t+1}(x)$ when we condition on $\left(\eta_{t}(x), \eta_{t}(x+1)\right)$ being an element from the set $\{(0,1), (1,0), (1,1)\}$. In \cite{directed_animals_1}, the \emph{random gas model} on an \emph{agreeable graph} has been described, in which a site (i.e.\ a vertex of the graph) cannot be occupied if any of its ``children" or ``neighbours" is occupied by a gas particle. Our PCA $E_{r',s'}$ admits a generalization of this model when we let $W$ indicate the state of being unoccupied and $L$ indicate the state of being occupied: assuming that both $r'$ and $s'$ are very small, and considering the underlying graph to be $\mathbb{Z}^{2}$ in which the ``children" or ``neighbours" of $(x,y)$ are $(x+1,y)$ and $(x,y+1)$,
\begin{enumerate}
\item if both $(x+1,y)$ and $(x,y+1)$ are unoccupied, then $(x,y)$ is occupied with probability $(1-s')^{2}$, and unoccupied otherwise,
\item if precisely one of $(x+1,y)$ and $(x,y+1)$ is occupied, then $(x,y)$ is occupied with probability $r'(1-s')$ (which is much lower than the probability of occupation in the first case), and unoccupied otherwise,
\item if both of $(x+1,y)$ and $(x,y+1)$ are occupied, then $(x,y)$ is occupied with probability ${r'}^{2}$ (which, in turn, is much lower than the probability of occupation in the second case), and unoccupied otherwise.
\end{enumerate} 
In some sense, we can see this as a less idealized, more practical version of the random gas model, in which gas particles do repel each other but not so strongly that the occupation of a ``child" or ``neighbour" to a given vertex completely prevents the occupation of that vertex (i.e.\ our PCA $E_{r',s'}$ allows for a random gas model \emph{without hard constraints}), and where the higher the number of occupied ``children" or ``neighbours" of a vertex, the lower the chance of that vertex being occupied.

Very similar motivations justify the importance of studying the PCA $G_{p,q,r}$ as well.

We mention here that \cite{bresler_guo_polyanskiy} shows the existence of a $\delta_{0} > 0$ such that an elementary PCA (whose alphabet comprises the symbols $W$ and $L$) with the NAND function and either a vertex-binary-symmetric-channel (BSC) noise with parameter $\delta$ or an edge-BSC noise with parameter $\delta$ is ergodic for all $\delta \in (0,\delta_{0})$ (to elucidate, the BSC$_{\delta}$ operator applied to a symbol from $\{W,L\}$ keeps it intact with probability $(1-2\delta)$ and switches it to the other symbol with probability $2\delta$).

Finally, given that, of the two main objects we investigate in this paper, one is a class of PCAs and the other a game that is closely tied to the process of percolation, it is only relevant that we draw attention of the reader to a few of the works in the literature dedicated to exploring the connection between PCAs and percolation. In \cite{hartarsky2021generalised}, it is shown that \emph{additive} PCAs can be interpreted as \emph{oriented site percolation} models. In \cite{hartarsky2022bootstrap}, a technique for proving the sharpness of phase transitions in percolation is implemented in the setting of \emph{attractive} PCAs. Moreover, a correspondence between PCAs and \emph{bootstrap percolation} is established in \cite{hartarsky2022bootstrap}, thereby deducing exponential decay of the probability of remaining healthy above criticality for a class of bootstrap percolation models, with implications for related \emph{kinetically constrained models}. It is further demonstrated that this correspondence provides an equivalence between the non-triviality of the phase transition in certain bootstrap percolation models and the stability, with respect to noise, of certain CAs.

\section{Game rules: an analysis of each game via recurrence relations}\label{subsec:recurrence}
Recall, from \S\ref{sec:intro}, our brief allusion to game rules as the bridge between our percolation games on one hand, and our PCA on the other. Deducing these game rules, or, in other words, recurrence relations that govern the games we study, is key to our analysis, and \S\ref{subsec:recurrence} is dedicated to this task.

We partition $\mathbb{Z}^{2}$ into the following three (random) subsets of vertices (the randomness arises from the random assignment of labels as described in \S\ref{sec:formal_defns_main_results}):
\begin{enumerate}
\item we let $W$ consist of all $(x,y) \in \mathbb{Z}^{2}$ such that the game with $(x,y)$ as its initial vertex is won by the player who plays its first round;
\item we let $L$ consist of all $(x,y) \in \mathbb{Z}^{2}$ such that the game with $(x,y)$ as its initial vertex is lost by the player who plays its first round;
\item and we let $D$ comprise all $(x,y) \in \mathbb{Z}^{2}$ such that the game with $(x,y)$ as the initial vertex results in a draw.
\end{enumerate}  
In particular, in case of generalized percolation games, a vertex labeled as a trap is placed in the set $W$, while a vertex labeled as a target is placed in the set $L$. The intuition behind such a convention can be explained as follows. One may imagine an \emph{unseen} round that takes place before the \emph{actual} game begins, during which the player who is supposed to play the second round of the actual game moves the token to the initial vertex $(x,y)$. If $(x,y)$ is a trap, she loses immediately, allowing her opponent, i.e.\ the player who is supposed to play the first round of the actual game, to win even before the game begins, justifying the reason why we include each trap vertex in the set $W$. Likewise, the reason for including target vertices in the set $L$ is justified.

For each $k \in \mathbb{Z}$, we set $S_{k} = \{(x,y) \in \mathbb{Z}^{2}:x+y=k\}$ (i.e.\ $S_{k}$ is the diagonal line, running from top-left to bottom-right, that constitutes all those vertices of $\mathbb{Z}^{2}$ whose coordinates sum to $k$). Right away, we state the following fact that becomes evident from the description of both of our games in \S\ref{sec:formal_defns_main_results}: 
\begin{enumerate}
\item once each vertex on $S_{k+1}$ has been categorized into one of the subsets $W$, $L$, and $D$, 
\item each edge between $S_{k}$ and $S_{k+1}$ has received its label (which is `trap' or `open' when we consider the generalized percolation game, and `trap' or `target' or `open' when we consider the bond percolation game), 
\item and each vertex on $S_{k}$ has received its label (this is relevant only for the generalized percolation game, and the label is any one of `trap', `target' and `open'),
\end{enumerate} 
we are in a position to uniquely determine which of $W$, $L$ and $D$ each vertex on $S_{k}$ belongs to. How this determination happens is revealed via the recurrence relations, or game rules, described in \S\ref{subsec:generalized_game_rules} and \S\ref{subsec:bond_game_rules}.

\subsection{Game rules for the generalized percolation game}\label{subsec:generalized_game_rules}
Fix any $k \in \mathbb{Z}$, any $(x,y) \in S_{k}$, and set $\Out(x,y)=\{(x+1,y),(x,y+1)\}$. We assume that $(x,y)$ serves as the initial vertex for a generalized percolation game, and we let $P_{1}$ denote the player who plays the first round, and $P_{2}$ the player who plays the second round of this game. Recalling the rules of the generalized percolation game described in \S\ref{sec:formal_defns_main_results}, we are able to make the following observations:
\begin{enumerate}
\item When both vertices of $\Out(x,y)$ are in $W$, $P_{2}$ wins no matter how $P_{1}$ moves in the first round, unless $(x,y)$ has been labeled as a trap. Therefore, in this case, $(x,y) \in L$ with probability $(1-p)$ and $(x,y) \in W$ with probability $p$.
\item When $(x+1,y) \in L$ and $(x,y+1) \in W$ (an analogous situation arises when $(x+1,y) \in W$ and $(x,y+1) \in L$), $P_{1}$ wins (by moving the token from $(x,y)$ to $(x+1,y)$ in the first round) if either $(x,y)$ has been labeled as a trap or $(x,y)$ has been labeled open and the directed edge $\big((x,y),(x+1,y)\big)$ is also open, while $P_{2}$ is the winner otherwise. Therefore, in this case, $(x,y) \in W$ with probability $\{p+(1-p-q)(1-r)\}$, and $(x,y) \in L$ with probability $\{q+(1-p-q)r\}$.
\item When both vertices of $\Out(x,y)$ are in $L$, $P_{1}$ wins if $(x,y)$ has been labeled as a trap or $(x,y)$ is open and at least one of the edges $\big((x,y),(x+1,y)\big)$ and $\big((x,y),(x,y+1)\big)$ is also open, while $P_{2}$ wins otherwise. Therefore, in this case, $(x,y) \in W$ with probability $\{p+(1-p-q)(1-r^{2})\}$, and $(x,y) \in L$ with probability $\{q+(1-p-q)r^{2}\}$.
\item When $(x+1,y) \in W$ and $(x,y+1) \in D$ (an analogous situation arises when $(x+1,y) \in D$ and $(x,y+1) \in W$), $P_{1}$ wins if $(x,y)$ has been labeled as a trap, the game results in a draw if $(x,y)$ is open and the edge $\big((x,y),(x,y+1)\big)$ has not been labeled as a trap, and $P_{2}$ wins in all other cases. Therefore, we have $(x,y) \in W$ with probability $p$, $(x,y) \in D$ with probability $(1-p-q)(1-r)$, and $(x,y) \in L$ with probability $\{q+r(1-p-q)\}$. 
\item When $(x+1,y) \in L$ and $(x,y+1) \in D$ (an analogous situation arises when $(x+1,y) \in D$ and $(x,y+1) \in L$), $P_{1}$ wins if $(x,y)$ has been labeled as a trap, or if both the vertex $(x,y)$ and the edge $\big((x,y),(x+1,y)\big)$ are open. The game results in a draw if $(x,y)$ is open, the edge $\big((x,y),(x+1,y)\big)$ has been labeled as a trap and the edge $\big((x,y),(x,y+1)\big)$ is open, while in all other cases, $P_{2}$ wins. Therefore, we have $(x,y) \in W$ with probability $\{p+(1-p-q)(1-r)\}$, $(x,y) \in D$ with probability $(1-p-q)r(1-r)$, and $(x,y) \in L$ with probability $\{q+r^{2}(1-p-q)\}$.
\item Finally, when both vertices of $\Out(x,y)$ are in $D$, $P_{1}$ wins if $(x,y)$ is a trap, the game results in a draw if $(x,y)$ is open and at least one of the edges $\big((x,y),(x+1,y)\big)$ and $\big((x,y),(x,y+1)\big)$ is open, and in all other cases, $P_{2}$ wins. Therefore, $(x,y) \in W$ with probability $p$, $(x,y) \in D$ with probability $(1-p-q)(1-r^{2})$ and $(x,y) \in L$ with probability $\{q+r^{2}(1-p-q)\}$.
\end{enumerate}
For each $k \in \mathbb{Z}$, we identify $S_{k}$ with a copy of the integer line $\mathbb{Z}$, by identifying the vertex $(x,k-x)$ on $S_{k}$ with $x$ on $\mathbb{Z}$. For an arbitrary but fixed $k \in \mathbb{Z}$, let $\omega(x) \in \hat{\mathcal{A}}=\{W,L,D\}$, referred to as the \emph{state} of the vertex $(x,k+1-x)$, denote the subset (out of $W$, $L$ and $D$) to which $(x,k+1-x)$ belongs, for each $x \in \mathbb{Z}$. Conditioned on the infinite tuple $\omega = (\omega(x): x \in \mathbb{Z})$ that specifies the state of \emph{each} vertex on $S_{k+1}$, the \emph{random} state of the vertex $(x,k-x)$, lying on $S_{k}$, equals $\G_{p,q,r}\omega(x)$ for each $x \in \mathbb{Z}$, as is evident from the recurrence relations described above, and from \eqref{general_update_rule_eq}, \eqref{GPCA_rule_1}, \eqref{GPCA_rule_2}, \eqref{GPCA_rule_3}, \eqref{GPCA_rule_4}, \eqref{GPCA_rule_5} and \eqref{GPCA_rule_6}). This is how the game rules governing the generalized percolation game give rise to the PCA $\G_{p,q,r}$.

\subsection{Game rules for the bond percolation game}\label{subsec:bond_game_rules}
The recurrence relations arising from the bond percolation game described in \S\ref{sec:formal_defns_main_results} can be deduced in much the same manner as those derived in \S\ref{subsec:generalized_game_rules} for the generalized percolation game. We let an arbitrarily chosen $(x,y) \in S_{k}$, for some $k \in \mathbb{Z}$, be the initial vertex, and assume, as before, that $P_{1}$ plays the first round of the game while $P_{2}$ plays the second. We then deduce the state of the vertex $(x,y)$, i.e.\ which of the subsets $W$, $L$ and $D$ it belongs to, as follows:
\begin{enumerate}
\item When both vertices of $\Out(x,y)$ are in $W$, $P_{1}$ loses unless at least one of the edges $\big((x,y),(x+1,y)\big)$ and $\big((x,y),(x,y+1)\big)$ has been labeled as a target, so that $(x,y) \in L$ with probability $(1-s')^{2}$ and $(x,y) \in W$ with probability $(2s'-{s'}^{2})$.
\item When $(x+1,y) \in L$ and $(x,y+1) \in W$ (analogously, when $(x+1,y) \in W$ and $(x,y+1) \in L$), $P_{1}$ wins unless the edge $\big((x,y),(x+1,y)\big)$ has been labeled as a trap and the edge $\big((x,y),(x,y+1)\big)$ is not a target. Hence, $(x,y) \in L$ with probability $r'(1-s')$, and $(x,y) \in W$ with probability $\{1-r'(1-s')\}$.
\item When both vertices of $\Out(x,y)$ are in $L$, $P_{1}$ wins unless both the edges $\big((x,y),(x+1,y)\big)$ and $\big((x,y),(x,y+1)\big)$ are traps. Therefore, $(x,y) \in L$ with probability ${r'}^{2}$ and $(x,y) \in W$ with probability $(1-{r'}^{2})$.
\item When $(x+1,y) \in W$ and $(x,y+1) \in D$ (analogously, when $(x+1,y) \in D$ and $(x,y+1) \in W$), $P_{1}$ wins if at least one of the edges $\big((x,y),(x+1,y)\big)$ and $\big((x,y),(x,y+1)\big)$ has been labeled as a target, the game results in a draw if $\big((x,y),(x+1,y)\big)$ is not a target and $\big((x,y),(x,y+1)\big)$ is open, and in every other case, $P_{2}$ wins. Thus, $(x,y) \in W$ with probability $\{1-(1-s')^{2}\}$, $(x,y) \in D$ with probability $(1-s')(1-r'-s')$, and $(x,y) \in L$ with probability $r'(1-s')$.
\item When $(x+1,y) \in L$ and $(x,y+1) \in D$ (analogously, when $(x+1,y) \in D$ and $(x,y+1) \in L$), $P_{1}$ wins as long as the edge $\big((x,y),(x+1,y)\big)$ is not a trap or at least one of the edges $\big((x,y),(x+1,y)\big)$ and $\big((x,y),(x,y+1)\big)$ is a target, the game results in a draw if $\big((x,y),(x+1,y)\big)$ is a trap and $\big((x,y),(x,y+1)\big)$ is open, and in all other cases, $P_{2}$ wins. Thus, $(x,y) \in W$ with probability $\{1-r'(1-s')\}$, $(x,y) \in D$ with probability $r'(1-r'-s')$, and $(x,y) \in L$ with probability ${r'}^{2}$.
\item When both vertices of $\Out(x,y)$ are in $D$, $P_{1}$ wins if at least one of $\big((x,y),(x+1,y)\big)$ and $\big((x,y),(x,y+1)\big)$ is a target, the game results in a draw if neither $\big((x,y),(x+1,y)\big)$ nor $\big((x,y),(x,y+1)\big)$ is a target but not both are traps, and $P_{2}$ wins otherwise. Thus, $(x,y) \in W$ with probability $\{1-(1-s')^{2}\}$, $(x,y) \in D$ with probability $(1-r'-s')(1+r'-s')$, and $(x,y) \in L$ with probability ${r'}^{2}$.
\end{enumerate}
As argued in \S\ref{subsec:generalized_game_rules}, the above-mentioned recurrence relations or game rules governing the bond percolation game are captured precisely by the stochastic update rules of the PCA $\E_{r',s'}$ in \eqref{envelope_PCA_rule_1}, \eqref{envelope_PCA_rule_2}, \eqref{envelope_PCA_rule_3}, \eqref{envelope_PCA_rule_4}, \eqref{envelope_PCA_rule_5} and \eqref{envelope_PCA_rule_6}.

\section{Proofs of Theorems~\ref{thm:three-parameter}, \ref{thm:two-parameter} and \ref{thm:PCAs}, assuming Theorems~\ref{thm:generalized_envelope_PCA_D} and \ref{thm:bond_envelope_PCA_D} to be true}\label{subsec:relation}
Let us consider a generalized percolation game on $\mathbb{Z}^{2}$. For each fixed realization of the random assignment of labels (as described in \S\ref{sec:formal_defns_main_results}) to the vertices and edges of our graph $\mathbb{Z}^{2}$, each vertex $(x,y) \in \mathbb{Z}^{2}$ can be classified uniquely into one of the subsets $W$, $D$ and $L$ (dictated by the outcome of the game whose initial vertex is $(x,y)$), and we refer to this subset as the \emph{state} of $(x,y)$. Let $\pmb{\eta}(x,y)$ denote the \emph{random} state of $(x,y)$ induced by the random assignment of labels to the vertices and edges of $\mathbb{Z}^{2}$. Since \emph{all} vertices of $\mathbb{Z}^{2}$ bear labels that are i.i.d.,\ and \emph{all} edges of $\mathbb{Z}^{2}$ bear labels that are i.i.d.\ as well, the joint law $\mu_{k}$ of the random infinite tuple $\left(\pmb{\eta}(x,y): (x,y) \in S_{k}\right) = \left(\pmb{\eta}(x,k-x): x \in \mathbb{Z}\right)$ is the same for \emph{every} $k \in \mathbb{Z}$, so that we may denote this common law henceforth by $\mu$ (with no dependence on $k$). Referring to the game rules explained in \S\ref{subsec:generalized_game_rules} and how they can be represented by $\G_{p,q,r}$ defined in \S\ref{subsec:formal_PCA}, we can write $\G_{p,q,r}\mu_{k+1}=\mu_{k}$, for each $k \in \mathbb{Z}$. Combining the two observations made above, we conclude that $\G_{p,q,r}\mu=\mu$, thus proving that $\mu$ is a stationary distribution for $\G_{p,q,r}$. Moreover, the i.i.d.\ nature of the assignment of labels to the vertices and edges of $\mathbb{Z}^{2}$ ensures that $\mu$ is both \emph{translation-invariant} and \emph{reflection-invariant}. By these two attributes, we mean the following:
\begin{enumerate}
\item Recall, from the last paragraph of \S\ref{subsec:generalized_game_rules}, that we identify $S_{k}$ with $\mathbb{Z}$ by mapping $(x,k-x)$ onto $x$ for each $x \in \mathbb{Z}$. Consequently, $\mu$ can be thought of as a probability measure on $\Omega = \hat{\mathcal{A}}^{\mathbb{Z}} = \{W,L,D\}^{\mathbb{Z}}$. For any $y \in \mathbb{Z}$, we let $\mathfrak{T}^{y}: \Omega \rightarrow \Omega$ map any configuration $\eta = (\eta(x): x \in \mathbb{Z})$ to the configuration $\mathfrak{T}^{y}\eta = \left(T^{y}\eta(x): x \in \mathbb{Z}\right)$ where $\mathfrak{T}^{y}\eta(x) = \eta(x+y)$ for each $x \in \mathbb{Z}$. A probability measure $\nu$, defined with respect to the $\sigma$-field $\mathcal{F}$ generated by all the cylinder sets of $\Omega$, is said to be \emph{translation-invariant} if $\nu(B) = \nu\left(\mathfrak{T}^{y}B\right)$ for every $B \in \mathbb{D}$, where $\mathfrak{T}^{y}B = \left\{\mathfrak{T}^{y}\eta: \eta \in B\right\}$.

\item Let $\mathfrak{R}: \Omega \rightarrow \Omega$ map any configuration $\eta = (\eta(x): x \in \mathbb{Z})$ to the configuration $\mathfrak{R}\eta = \left(\mathfrak{R}\eta(x): x \in \mathbb{Z}\right)$ where $\mathfrak{R}\eta(x) = \eta(-x)$ for each $x \in \mathbb{Z}$. A probability measure $\nu$, defined with respect to $\mathcal{F}$, is said to be \emph{reflection-invariant} if $\nu(B) = \nu\left(\mathfrak{R}B\right)$ for every $B \in \mathbb{D}$, where $\mathfrak{R}B = \left\{\mathfrak{R}\eta: \eta \in B\right\}$.
\end{enumerate}

We now state what can be considered the most technically challenging result of our paper (the measure $\mu$ in the statement of Theorem~\ref{thm:generalized_envelope_PCA_D} need not be the same as the one alluded to in the previous paragraph):
\begin{theorem}\label{thm:generalized_envelope_PCA_D}
Let $(p,q,r) \in \Theta$ satisfy the constraints stated in Theorem~\ref{thm:three-parameter}, i.e.\ each of $p$, $q$ and $r$ is sufficiently small, the inequality in \eqref{three_cond_universal} holds, and precisely one of \eqref{three_cond_1}, \eqref{three_cond_2}, \eqref{three_cond_3} and \eqref{three_cond_4} is true. If $\mu$ is any translation-invariant and reflection-invariant stationary distribution for $\G_{p,q,r}$, then for any fixed but arbitrary $x \in \mathbb{Z}$, the probability, under $\mu$, that $x$ is assigned the state $D$, abbreviated as $\mu(D)$, equals $0$.  
\end{theorem}
An analogous result is true in case of bond percolation games, and can be stated as follows:
\begin{theorem}\label{thm:bond_envelope_PCA_D}
Let $(r',s') \in \Theta'$ belong to any one of the three regimes described in Theorem~\ref{thm:two-parameter}, i.e.\ one of \ref{bond_regime_1}, \ref{bond_regime_2} and \ref{bond_regime_3}. If $\mu$ is any translation-invariant and reflection-invariant stationary distribution for $\E_{r',s'}$, then for any fixed but arbitrary $x \in \mathbb{Z}$, the probability, under $\mu$, that $x$ is assigned the state $D$, abbreviated simply as $\mu(D)$, equals $0$.  
\end{theorem}
The proofs of Theorems~\ref{thm:generalized_envelope_PCA_D} and \ref{thm:bond_envelope_PCA_D} are accomplished in \S\ref{sec:generalized_weight_function} and \S\ref{sec:bond_weight_function} respectively, employing the technique of \emph{weight functions} or \emph{potential functions} that was first introduced in \cite{holroyd2019percolation} and later explored in \cite{bresler_guo_polyanskiy}, \cite{bhasin2022class} and \cite{bhasin2022ergodicity}.

We have yet to establish a connection between the ergodicity of either of the PCAs $G_{p,q,r}$ and $\G_{p,q,r}$ (respectively, $E_{r',s'}$ and $\E_{r',s'}$) and the event of draw in our generalized percolation game (respectively, bond percolation game). This is where $G_{p,q,r}$ (respectively, $E_{r',s'}$) plays an important role.
\begin{theorem}\label{thm:ergodic_equiv_draw_probab_0}
For each $(p,q,r) \in \Theta$, the generalized percolation game, with underlying parameters $p$, $q$ and $r$, has probability $0$ of resulting in a draw if and only if the PCA $G_{p,q,r}$ is ergodic. Likewise, for each $(r',s') \in \Theta'$, the bond percolation game, with underlying parameters $r'$ and $s'$, has probability $0$ of resulting in a draw if and only if the PCA $E_{r',s'}$ is ergodic.
\end{theorem}
\begin{proof}
The proof of Theorem~\ref{thm:ergodic_equiv_draw_probab_0} happens via the exact same argument as the proof of Proposition 2.2 of \cite{holroyd2019percolation}, and is therefore omitted from this paper. 
\end{proof}
Our next result is Lemma~\ref{lem:ergodicity_equivalence}, and we note that it can be proved in exactly the same way as Proposition 2.1 of \cite{holroyd2019percolation} (this lemma, in fact, gets used in proving one side of the two-way implication asserted upon in Theorem~\ref{thm:ergodic_equiv_draw_probab_0}):
\begin{lemma}\label{lem:ergodicity_equivalence}
The PCA $\G_{p,q,r}$ is ergodic if and only if the PCA $G_{p,q,r}$ is ergodic. Likewise, $\E_{r',s'}$ is ergodic if and only if $E_{r',s'}$ is ergodic.
\end{lemma}

Some discussion is now in order as to the philosophy behind invoking the notion of envelope PCAs. We let $F$ denote a $d$-dimensional PCA endowed with the alphabet $\mathcal{A}_{F} = \{W,L\}$, some neighbourhood-marking set $\mathcal{N}_{F}$, and stochastic update rules given by some stochastic matrix $\varphi_{F}$. Introduced in \cite{buvsic}, the envelope PCA $\F$ corresponding to $F$ is a different PCA whose alphabet $\mathcal{A}_{\F} = \{W,L,D\}$ is obtained by appending $\mathcal{A}_{F}$ with the symbol $D$. The neighbourhood-marking set for $\F$ is, once again, $\mathcal{N}_{F}$, and the stochastic update rules for $\F$ are captured by the stochastic matrix $\varphi_{\F}$ such that, when we restrict $\mathcal{A}_{\F}$ to $\mathcal{A}_{F}$, the corresponding restriction of $\varphi_{\F}$ yields $\varphi_{F}$. The state spaces for $F$ and $\F$ are respectively $\Omega_{F} = \mathcal{A}_{F}^{\mathbb{Z}^{d}}$ and $\Omega_{\F} = \mathcal{A}_{\F}^{\mathbb{Z}^{d}}$. A configuration $\eta = \left(\eta(\mathbf{x}): \mathbf{x} \in \mathbb{Z}^{d}\right)$ in $\Omega_{\F}$ can be thought of as a configuration that actually belongs to $\Omega_{F}$, but with an unknown symbol (which could equal either $W$ or $L$) occupying each coordinate $\mathbf{x}$ of $\mathbb{Z}^{d}$ for which $\eta(\mathbf{x}) = D$. Thus, the symbol $D$ acts as a placeholder whenever $\mathbf{x} \in \mathbb{Z}^{d}$ is occupied by a symbol from the original alphabet $\mathcal{A}_{F}$ that we are not aware of or are uncertain about. Given an ergodic PCA $F$ with the unique stationary distribution $\mu$, the motivation for introducing envelope PCAs lies in coming up with an efficient procedure, known as a \emph{perfect sampling} procedure (see \cite{propp1996exact} and \cite{buvsic}), for generating a configuration from $\mu$. One such algorithm implements the \emph{coupled from the past} (\emph{CFTP}) method proposed in \cite{propp1996exact}, but it is inefficient when the state space $\Omega_{F}$ is large (for instance, in this paper, $\Omega_{G_{p,q,r}}=\Omega_{E_{r',s'}}=\{W,L\}^{\mathbb{Z}}$ is uncountably infinite). Even when the PCA $F$ is ergodic, the corresponding envelope PCA $\F$ may or may not be ergodic, but if it is, then \cite{buvsic} shows that an efficient perfect sampling procedure can be achieved by running $\F$ on a single initial configuration. The envelope PCA $\F$ is used to couple two (random) configurations obtained via (possibly repeated) applications of $F$, and as mentioned above, the symbol $D$ is used to populate the cells of $\mathbb{Z}^{d}$ at which these two coupled configurations may differ from one another.

We are now in a position to prove Theorem~\ref{thm:three-parameter}, assuming the conclusion of Theorem~\ref{thm:generalized_envelope_PCA_D}, as follows:
\begin{proof}[Proof of Theorem~\ref{thm:three-parameter}]
Let $(p,q,r) \in \Theta$ satisfy the constraints stated in Theorem~\ref{thm:three-parameter}. From Theorem~\ref{thm:generalized_envelope_PCA_D}, we may conclude that $\mu\left[\left\{\pmb{\eta}(x,k-x) = D\right\}\right] = 0$ for \emph{any} $x, k \in \mathbb{Z}$, where $\mu$ is the joint law of $(\pmb{\eta}(x,k-x):x \in \mathbb{Z})$ (see the first paragraph of \S\ref{subsec:relation}). Such a conclusion is possible because we showed, just before stating Theorem~\ref{thm:generalized_envelope_PCA_D}, that $\mu$ is a translation-invariant and reflection-invariant stationary distribution for $\G_{p,q,r}$. In other words, for \emph{any} vertex $(x,k-x)$ lying on \emph{any} diagonal $S_{k}$, the probability that it belongs to $D$ equals $0$ under $\mu$. Therefore, almost surely, there is \emph{no} vertex in $\mathbb{Z}^{2}$ which, when serving as the initial vertex for the generalized percolation game, leads to a draw. This, along with Theorem~\ref{thm:ergodic_equiv_draw_probab_0} and Lemma~\ref{lem:ergodicity_equivalence}, yields the conclusion of Theorem~\ref{thm:three-parameter}.
\end{proof} 
The argument for proving Theorem~\ref{thm:two-parameter}, provided Theorem~\ref{thm:bond_envelope_PCA_D} holds, is identical to the above.

\begin{proof}[Proof of Theorem~\ref{thm:PCAs}]
As explained above, once Theorem~\ref{thm:generalized_envelope_PCA_D} has been established, the conclusion of Theorem~\ref{thm:three-parameter} follows. By Theorem~\ref{thm:ergodic_equiv_draw_probab_0}, we then know that the PCA $G_{p,q,r}$ is ergodic whenever the underlying parameter-triple $(p,q,r)$ belongs to $\Theta$ and satisfies the constraints stated in Theorem~\ref{thm:three-parameter}. By Lemma~\ref{lem:ergodicity_equivalence}, we further conclude that the envelope PCA $\G_{p,q,r}$ is ergodic for all such values of $(p,q,r)$ as well. This concludes the proof of the first assertion stated in Theorem~\ref{thm:PCAs}. The proof of the second assertion of Theorem~\ref{thm:PCAs} follows via an identical argument (in which we replace Theorem~\ref{thm:three-parameter} by Theorem~\ref{thm:two-parameter}).
\end{proof}

All we have left to prove now are Theorems~\ref{thm:generalized_envelope_PCA_D} and \ref{thm:bond_envelope_PCA_D}.

\section{The proof of Theorem~\ref{thm:generalized_envelope_PCA_D} by the technique of weight functions}\label{sec:generalized_weight_function}
It was stated in \S\ref{subsec:relation} that the crux of this paper lies in the rather technical proof of Theorem~\ref{thm:generalized_envelope_PCA_D} (and its counterpart in case of bond percolation games, Theorem~\ref{thm:bond_envelope_PCA_D}). Before proceeding with the proof, we summarize here, for the reader's convenience, how the rest of \S\ref{sec:generalized_weight_function} has been organized:
\begin{enumerate}
\item In \S\ref{subsec:central_ideas_weight_functions}, we outline the central ideas relevant to the step-by-step construction of our weight function -- these ideas apply equally well to the construction of the weight functions that allow us to prove Theorem~\ref{thm:bond_envelope_PCA_D}. 
\item In \S\ref{subsec:final_wt_fn_generalized}, we state, without proof, the weight function we come up with for the generalized percolation game, the \emph{weight function inequality} it satisfies, and finally, draw the desired conclusion from this inequality (thus proving Theorem~\ref{thm:generalized_envelope_PCA_D}) .  
\item In \S\ref{sec:generalized_wt_fn_steps}, we show the details of each step leading to the construction of our weight function. \end{enumerate}

\subsection{The principal ideas behind the construction of our weight functions}\label{subsec:central_ideas_weight_functions}
We present to the reader a very broad, but hopefully comprehensive, outline of how we aim to accomplish the proof of Theorem~\ref{thm:generalized_envelope_PCA_D}. We denote by $\mathcal{R}$ the subset of the parameter-space $\Theta$ (defined in \S\ref{sec:formal_defns_main_results}) that consists of all those triples $(p,q,r)$ that satisfy the constraints stated in Theorem~\ref{thm:three-parameter} -- in other words, $p$, $q$ and $r$ are sufficiently small, the inequality in \eqref{three_cond_universal} holds, and precisely one of \eqref{three_cond_1}, \eqref{three_cond_2}, \eqref{three_cond_3} and \eqref{three_cond_4} is true. Although, in the rest of \S\ref{subsec:central_ideas_weight_functions}, we focus on $(p,q,r) \in \mathcal{R}$ and work with the envelope PCA $\G_{p,q,r}$ corresponding to generalized percolation games, this overview is equally pertinent to the construction of the weight functions corresponding to the three different regimes addressed in Theorem~\ref{thm:bond_envelope_PCA_D} for bond percolation games. 

Recall, from \S\ref{subsubsec:specific_PCAs}, that the state space corresponding to our envelope PCA, $\G_{p,q,r}$, is $\Omega = \hat{\mathcal{A}}^{\mathbb{Z}}$, where $\hat{\mathcal{A}}=\{W,L,D\}$ is its alphabet. Let $\mathcal{M}$ denote the space of all reflection-invariant and translation-invariant (recall these definitions from \S\ref{subsec:relation}) probability measures on $\Omega$, defined with respect to the $\sigma$-field $\mathcal{F}$ generated by all cylinder sets of $\Omega$. The goal is to come up with a ``suitable" real-valued function $w$, referred to as a \emph{weight function} or \emph{potential function}, defined on $\mathcal{M}$. As such, a weight function (that serves the purpose outlined below) need not be unique, but as far as this paper is concerned, we construct $w$ as a linear combination of the form
\begin{equation}
w(\mu) = \sum_{i=1}^{t}c_{i}\mu(\mathcal{C}_{i}),\label{gen_form}
\end{equation} 
where $t \in \mathbb{N}$, each of $\mathcal{C}_{1}$, $\ldots$, $\mathcal{C}_{t}$ is a cylinder set (thus belonging to $\mathcal{F}$), and each of $c_{1}$, $\ldots$, $c_{t}$ is a real-valued function (in fact, a polynomial) in the parameters $p$, $q$ and $r$. Our goal is to construct a weight function, of the form \eqref{gen_form}, that satisfies an inequality, henceforth referred to as a \emph{weight function inequality}, of the form
\begin{equation}
w\left(\G_{p,q,r}\mu\right) \leqslant w(\mu) - \sum_{i=1}^{t'}c'_{i}\mu\left(\mathcal{C}'_{i}\right),\label{gen_ineq_form}
\end{equation} 
where $t' \in \mathbb{N}$, and $\mathcal{C}'_{i}$ is a cylinder set and $c'_{i}$ is a real-valued function (once again, a polynomial) in $p$, $q$ and $r$, for each $i \in \{1,2,\ldots,t'\}$. It must be ensured that 
\begin{equation}\label{desired_criterion}
c'_{i} \geqslant 0 \text{ for each } i \in \{1,2,\ldots,t'\} \text{ for each } (p,q,r) \in \mathcal{R}.
\end{equation}
Let $\mathcal{I}$ denote the set of all $i$, with $i \in \{1,2,\ldots,t'\}$, such that $c'_{i} > 0$ whenever $(p,q,r) \in \mathcal{R}$. If $\mu$ is stationary for $\G_{p,q,r}$ (so that $\G_{p,q,r}\mu = \mu$) and $\mathcal{I}$ is non-empty, by \eqref{desired_criterion} we conclude that $\mu\left(\mathcal{C}'_{i}\right) = 0$ for each $i \in \mathcal{I}$ whenever $(p,q,r) \in \mathcal{R}$. If the cylinder sets $\mathcal{C}'_{1}, \ldots, \mathcal{C}'_{t'}$ have been chosen judiciously enough, suitable set-theoretic operations performed on $\mathcal{C}'_{i}$, for $i \in \mathcal{I}$, allow us to conclude that $\mu\left((D)_{0}\right)=0$ whenever $(p,q,r) \in \mathcal{R}$ (here, $(D)_{0}$ indicates the cylinder set that consists of \emph{all} configurations $\eta = (\eta(x): x \in \mathbb{Z})$ with $\eta(0) = D$, i.e.\ in which the $0$-th coordinate is occupied by the symbol $D$). This concludes the primary idea propelling the proof of Theorem~\ref{thm:generalized_envelope_PCA_D}.

From here onward, we make use of some simplified notations, as follows. In its most general form, a cylinder set is indicated by $\left(a_{1},a_{2},\ldots,a_{\ell}\right)_{x_{1},x_{2},\ldots,x_{\ell}}$, where $a_{1}, a_{2}, \ldots, a_{\ell} \in \hat{\mathcal{A}}$ and $x_{1}, x_{2}, \ldots, x_{\ell} \in \mathbb{Z}$, signifying that this cylinder set comprises all tuples $\eta = (\eta(x): x \in \mathbb{Z})$ in which $\eta(x_{i}) = a_{i}$ for each $i \in \{1,2,\ldots,\ell\}$. When $x_{1}, x_{2}, \ldots, x_{\ell}$ are consecutive integers, i.e.\ $x_{i+1}=x_{i}+1$ for each $i \in \{1,2,\ldots,\ell-1\}$, we have 
\begin{equation}
\mu\left(\left(a_{1},a_{2},\ldots,a_{\ell}\right)_{0,1,\ldots,\ell-1}\right) = \mu\left(\left(a_{1},a_{2},\ldots,a_{\ell}\right)_{x_{1},x_{1}+1,\ldots,x_{1}+\ell-1}\right)\nonumber
\end{equation}
as $\mu$ is translation-invariant. This allows us to abbreviate $\mu\left(\left(a_{1},a_{2},\ldots,a_{\ell}\right)_{x_{1},x_{1}+1,\ldots,x_{1}+\ell-1}\right)$ as simply $\mu\left(\left(a_{1},a_{2},\ldots,a_{\ell}\right)\right)$ for \emph{any} $x_{1} \in \mathbb{Z}$. We further remove the commas and the outer parentheses and shorten $\mu\left(\left(a_{1},a_{2},\ldots,a_{\ell}\right)\right)$ to $\mu\left(a_{1}a_{2}\ldots a_{\ell}\right)$. We call a cylinder set, of the form $\left(a_{1},a_{2},\ldots,a_{\ell}\right)_{x_{1},x_{2},\ldots,x_{\ell}}$, \emph{$D$-inclusive} if $a_{i}=D$ for at least one $i \in \{1,2,\ldots,\ell\}$. 

We emphasize here that even proceeding via the approach outlined above, one may be able to come up with \emph{different} weight functions all of which serve the same purpose (i.e.\ satisfy an inequality of the form \eqref{gen_ineq_form}, along with the criterion stated in \eqref{desired_criterion}). Since the ultimate goal is to draw a conclusion about $\mu(D)$, it makes sense
\begin{enumerate*}
\item to consider each cylinder set $\mathcal{C}_{i}$ in \eqref{gen_form} to be $D$-inclusive, and
\item to begin our step-by-step construction of the weight function by setting $\mathcal{C}_{1} = (D)_{0}$ in \eqref{gen_form}.
\end{enumerate*}
The subsequent choices for $\mathcal{C}_{2}$, $\mathcal{C}_{3}$ etc.\ will be motivated in the subsubsections under \S\ref{sec:generalized_wt_fn_steps}. As seen in \S\ref{sec:generalized_wt_fn_steps}, in every step, up to and including the penultimate step, of constructing our weight function, we obtain a weight function inequality that, although of the same form as \eqref{gen_ineq_form}, fails to satisfy \eqref{desired_criterion}. This is what leads us to carry out an ``adjustment" of the weight function obtained until that step. In \S\ref{subsubsec:adjust_gen}, we present to the reader the heuristics of how such an adjustment is implemented, and how it contributes to satisfying a weight function inequality of the form \eqref{gen_ineq_form}. However, the implementation of this idea becomes much clearer when the reader goes through the specific details laid out in the various subsubsections of \S\ref{sec:generalized_wt_fn_steps}.

\subsubsection{How computations of the pushforward measures of various cylinder sets have been carried out in the sequel}\label{subsubsec:pushforward_computations} Letting $\mathcal{C}=(a_{1},a_{2},\ldots,a_{\ell})_{x_{1},x_{1}+1,\ldots,x_{1}+\ell-1}$ be a cylinder set, we denote by $\mu(a_{0}\mathcal{C})$ the probability of the cylinder set $\mathcal{C}'=(a_{0},a_{1},a_{2},\ldots,a_{\ell})_{x_{1}-1,x_{1},x_{1}+1,\ldots,x_{1}+\ell-1}$, and by $\mu(\mathcal{C}a_{\ell+1})$ the probability of the cylinder set $\mathcal{C}''=(a_{1},a_{2},\ldots,a_{\ell},a_{\ell+1})_{x_{1},x_{1}+1,\ldots,x_{1}+\ell-1,x_{1}+\ell}$. We now state the following lemma (whose proof is immediate and has been, therefore, omitted):
\begin{lemma}\label{lem:pushforward_1}
For any cylinder set $\mathcal{C}=(a_{1},a_{2},\ldots,a_{\ell})_{x_{1},x_{1}+1,\ldots,x_{1}+\ell-1}$, and any probability measure $\mu\in\mathcal{M}$ (i.e.\ $\mu$ is translation-invariant, reflection-invariant and measurable with respect to the $\sigma$-field $\mathcal{F}$ generated by all cylinder sets of $\Omega$), we can write
\begin{equation}
\mu(\mathcal{C})=\mu(W\mathcal{C})+\mu(D\mathcal{C})+\mu(L\mathcal{C})=\mu(\mathcal{C}W)+\mu(\mathcal{C}D)+\mu(\mathcal{C}L).\nonumber
\end{equation}
\end{lemma}
Lemma~\ref{lem:pushforward_1} has been used repeatedly in the computations of \S\ref{sec:generalized_wt_fn_steps}. In addition, another general result has been used, which is captured by Lemma~\ref{lem:pushforward_2}. Before stating Lemma~\ref{lem:pushforward_2}, we introduce the notation $\mathcal{C}\in\{W,L,D\}^{i}$, for any $i\in\mathbb{N}$, to mean that $\mathcal{C}$ is a cylinder set of the form $(a_{0},a_{1},\ldots,a_{i-1})_{0,1,\ldots,i-1}$. Given cylinder sets $\mathcal{C}=(a_{0},a_{1},\ldots,a_{i-1})_{0,1,\ldots,i-1}$ and $\mathcal{D}=(b_{0},b_{1},\ldots,b_{i})_{0,1,\ldots,i}$, we denote by $\Prob[\mathcal{C}\big|\mathcal{D}]$ the probability of the event that $\G_{p,q,r}\eta(j)=a_{j}$ for each $j\in\{0,1,\ldots,i-1\}$, conditioned on the event that $\eta(j)=b_{j}$ for each $j\in\{0,1,\ldots,i\}$. 
\begin{lemma}\label{lem:pushforward_2}
For each $i\in\mathbb{N}$ and each $\mathcal{C}\in\{W,L,D\}^{i}$,
\begin{equation}
\G_{p,q,r}\mu(\mathcal{C})=\sum_{\mathcal{D}\in\{W,L,D\}^{i+1}}\Prob[\mathcal{C}\big|\mathcal{D}]\mu(\mathcal{D}).\nonumber
\end{equation}
\end{lemma}
When applying Lemma~\ref{lem:pushforward_2}, such as in the proofs of Lemmas~\ref{lem:pushforward_D} and \ref{lem:pushforward_WD}, we often simply write $\Prob[a_{0}a_{1}\ldots a_{i-1}\big|b_{0}b_{1}\ldots b_{i}]$ in place of $\Prob[\mathcal{C}\big|\mathcal{D}]$.

\subsubsection{How each adjustment to the weight function has been carried out in the sequel}\label{subsubsec:adjust_gen}
Let us denote by $w_{i-1}$ the weight function we have constructed up until the start of the $i$-th adjustment (thus, the very first weight function we begin with is denoted by $w_{0}$, the weight function obtained after the first adjustment is denoted by $w_{1}$ and so on). The weight function inequality at the \emph{beginning} of the $i$-th adjustment will then be of the form 
\begin{equation}
w_{i-1}\left(\G_{p,q,r}\mu\right) \leqslant w_{i-1}(\mu) - \sum_{t=1}^{k_{i-1}}\alpha_{i-1,t}\mu\left(\mathcal{C}_{i-1,t}\right),\label{gen_wt_fn_ineq}
\end{equation}
for some $k_{i-1} \in \mathbb{N}$, coefficients $\alpha_{i-1,1}, \ldots, \alpha_{i-1,k_{i-1}}$ that are polynomials in $p$, $q$ and $r$, and cylinder sets $\mathcal{C}_{i-1,1}, \ldots, \mathcal{C}_{i-1,k_{i-1}}$. However, if there exists at least one coefficient among $\alpha_{i-1,1}, \ldots, \alpha_{i-1,k_{i-1}}$ that is negative when $(p,q,r) \in \mathcal{R}$, the inequality \eqref{gen_wt_fn_ineq} is unable to fulfill \eqref{desired_criterion} when $(p,q,r) \in \mathcal{R}$, thereby necessitating an adjustment to the weight function $w_{i-1}$ derived so far. We now define, for suitably chosen cylinder sets $\mathcal{D}_{i,1}$, $\ldots$, $\mathcal{D}_{i,\ell_{i}}$, and suitably chosen coefficients $\beta_{i,1}$, $\ldots$, $\beta_{i,\ell_{i}}$ that are polynomials in $p$, $q$ and $r$, the updated / adjusted weight function
\begin{equation}
w_{i}(\mu) = w_{i-1}(\mu) - \sum_{t=1}^{\ell_{i}}\beta_{i,t}\mu\left(\mathcal{D}_{i,t}\right).\label{adjust_gen_form}
\end{equation}
When incorporated into \eqref{gen_wt_fn_ineq}, $w_{i}$ yields the updated weight function inequality (obtained at the \emph{end} of the $i$-th adjustment):
\begin{align}
{}&w_{i}\left(\G_{p,q,r}\mu\right) + \sum_{t=1}^{\ell_{i}}\beta_{i,t}\G_{p,q,r}\mu\left(\mathcal{D}_{i,t}\right) \leqslant w_{i}(\mu) + \sum_{t=1}^{\ell_{i}}\beta_{i,t}\mu\left(\mathcal{D}_{i,t}\right) - \sum_{t=1}^{k_{i-1}}\alpha_{i-1,t}\mu\left(\mathcal{C}_{i-1,t}\right)\nonumber\\
\Longleftrightarrow{}& w_{i}\left(\G_{p,q,r}\mu\right) \leqslant w_{i}(\mu) + \sum_{t=1}^{\ell_{i}}\beta_{i,t}\mu\left(\mathcal{D}_{i,t}\right) - \sum_{t=1}^{k_{i-1}}\alpha_{i-1,t}\mu\left(\mathcal{C}_{i-1,t}\right) - \sum_{t=1}^{\ell_{i}}\beta_{i,t}\G_{p,q,r}\mu\left(\mathcal{D}_{i,t}\right).\label{ith_wt_fn_ineq}
\end{align}
Attempt is usually made to select $\mathcal{D}_{i,1}, \ldots, \mathcal{D}_{i,\ell_{i}}$ as subsets of elements from the set 
\begin{equation}\label{non-negative_coefficient_subset}
\left\{\mathcal{C}_{i-1,t}: 1 \leqslant t \leqslant k_{i-1} \text{ and } \alpha_{i-1,t} > 0\right\},
\end{equation}
with the added restriction that if $\mathcal{D}_{i,t} \subseteq \mathcal{C}_{i-1,t'}$ for some $\mathcal{C}_{i-1,t'}$ belonging to the set in \eqref{non-negative_coefficient_subset}, then $\beta_{i,t} \leqslant \alpha_{i-1,t'}$, so that the coefficient of $\mu\left(\mathcal{C}_{i-1,t'}\right)$ in the right side of \eqref{ith_wt_fn_ineq} remains non-positive. Attempt is also made to select $\mathcal{D}_{i,1}, \ldots, \mathcal{D}_{i,\ell_{i}}$ in such a manner that terms from the sum $-\sum_{t=1}^{\ell_{i}}\beta_{i,t}\G_{p,q,r}\mu\left(\mathcal{D}_{i,t}\right)$ are able to aid in negating terms of the form $-\alpha_{i-1,t}\mu\left(\mathcal{C}_{i-1,t}\right)$ in which $\alpha_{i-1,t} < 0$ when $(p,q,r) \in \mathcal{R}$. 

The above gives a broad overview of 
\begin{enumerate*}
\item why adjustments to the weight function, in several rounds, are required, and
\item how they are usually accomplished.
\end{enumerate*}
The above idea is reiterated until the final weight function inequality is of the form given in \eqref{gen_ineq_form} and satisfies the criterion in \eqref{desired_criterion}.

\subsection{The final weight function obtained for the generalized percolation game when $(p,q,r) \in \mathcal{R}$}\label{subsec:final_wt_fn_generalized}
We state here, for the reader's convenience, the final weight functions obtained, and the corresponding weight function inequalities, of the form \eqref{gen_ineq_form} (and obeying the criterion stated in \eqref{desired_criterion}), that these weight functions satisfy when $(p,q,r) \in \mathcal{R}$. The final weight function that aids us in proving Theorem~\ref{thm:generalized_envelope_PCA_D}, and, in turn, Theorem~\ref{thm:three-parameter}, is 
\begin{equation}
w(\mu)=\mu(D)+\mu(WD)+\mu(LWD)\label{w_regime_1}
\end{equation}
when \eqref{three_cond_1} is true, and
\begin{align}\label{w_{1}}
w(\mu)={}&\mu(D)+\mu(WD)+\{4r-2(p+q)+(p+q)^{2}-6r^{2}-3r(p+2q)\}\mu(LD)+(1-p-q)(1-r)\nonumber\\&(-2r+r^{2}+3pr+3qr+r^{3})\mu(LDW)+(1-p-q)(1-r)(-1+p+q+r^{2}+2qr)\mu(LDL)\nonumber\\&-[2-(3+p+q+r-2q^{2}-2r^{2}-5qr+2pq+pr+3pqr+4q^{2}r+4qr^{2}+5pr^{2}-2p^{2}r\nonumber\\&-5pqr^{2}-3p^{2}r^{2}-2q^{2}r^{2})(1-p-q)(1-r)]\mu(LWD)
\end{align}
when any one of \eqref{three_cond_2}, \eqref{three_cond_3} and \eqref{three_cond_4} is true. The weight function inequality that each of the two functions in \eqref{w_regime_1} and \eqref{w_{1}} satisfies (and which is of the form \eqref{gen_ineq_form}) is given by
\begin{align}\label{generalized_final_wt_fn_ineq}
w(\G_{p,q,r}\mu)\leqslant{}&w(\mu)-[2-(1-p-q)(1-r)\{2+2p+q+r-q^{2}-r^{2}-3qr+pq+2q^{2}r+2qr^{2}+pqr+\nonumber\\&2pr^{2}-p^{2}r-2pqr^{2}-p^{2}r^{2}-q^{2}r^{2}\}]\{\mu(WD)-\mu(WWWD)-\mu(DWDD)-\mu(LWD)\}\nonumber\\&-\{1-(1+p+q+pq+pr-q^{2}-qr-p^{2}r+q^{2}r)(1-p-q)(1-r)(1+r)\}\mu(DD)\nonumber\\&-r^{2}(1-r)^{2}(1-p-q)^{2}\mu(WLD)-r^{2}(1-r)^{2}(1-p-q)^{2}\mu(DLD).
\end{align}

Let us now deduce, from \eqref{generalized_final_wt_fn_ineq}, our desired conclusion, i.e.\ that $\mu(D)=0$ whenever $\mu$ is a translation-invariant and reflection-invariant stationary distribution for $\G_{p,q,r}$. To this end, we shall focus on the term involving $\mu(DD)$ in \eqref{generalized_final_wt_fn_ineq}. Expanding its coefficient, we obtain:
\begin{align}
{}&-1+(1+p+q+pq+pr-q^{2}-qr-p^{2}r+q^{2}r)(1-p-q)(1-r)(1+r)\nonumber\\
={}&-p^{3}r^{3}-p^{2}q\{1-r-r^{2}+r^{3}\}-p^{2}r\{2-r-2r^{2}-p\}-pq^{2}r(1-r^{2})-pr^{3}-2q^{2}r^{3}-q^{3}r\{1+r-r^{2}\}\nonumber\\&-qr\{1-pr-2q-r^{2}-2qr\}-q^{2}(1-q)-p(p+q-r)-q^{2}-r^{2}.\label{DD_coeff_expanded}
\end{align}
Since \eqref{three_cond_universal} is assumed to hold no matter which of \eqref{three_cond_1}, \eqref{three_cond_2}, \eqref{three_cond_3} and \eqref{three_cond_4} $(p,q,r)$ satisfies, we have
\begin{align}
{}&2(p+q)-(p+q)^{2}+3r(p+2q) \geqslant 4r-6r^{2}
\Longleftrightarrow p+q-r \geqslant r\left(1-3r-\frac{3p}{2}-3q\right)+\frac{1}{2}(p+q)^{2},\nonumber
\end{align}
and this lower bound on the right side of the inequality above is strictly positive as long as $p$, $q$ and $r$ are sufficiently small but at least one of them is strictly positive. Note, furthermore, that if $p=0$, the term $-p(p+q-r)$ disappears from \eqref{DD_coeff_expanded}, but $-q^{2}-r^{2}$ remains, and this is strictly negative whenever at least one of $q$ and $r$ is non-zero. All these observations show us that when $p$, $q$ and $r$ are sufficiently small and \eqref{three_cond_universal} holds, the expression in \eqref{DD_coeff_expanded} is strictly negative as long as at least one of $p$, $q$ and $r$ is strictly positive (recall, from \S\ref{sec:formal_defns_main_results}, that the condition $p+q+r > 0$ is included in the definition of the parameter-space $\Theta$). Consequently, when $\mu$ is stationary for $\G_{p,q,r}$, we deduce, from \eqref{generalized_final_wt_fn_ineq} and the identity $\G_{p,q,r}\mu=\mu$, that $\mu(DD)=0$.

Since $\G_{p,q,r}\mu=\mu$ for any stationary $\mu$, hence $\mu(DD)=0 \implies \G_{p,q,r}\mu(DD)=0$. Setting $\mathcal{C}=(D,D)_{0,1}$, and considering $\mathcal{D}\in\{(W,D,W)_{0,1,2},(W,D,L)_{0,1,2},(L,D,L)_{0,1,2}\}$, an application of Lemma~\ref{lem:pushforward_2} yields the inequality 
\begin{equation}
(1-p-q)^{2}(1-r)^{2}\{\mu(WDW)+2r\mu(WDL)+r^{2}\mu(LDL)\} \leqslant \G_{p,q,r}\mu(DD),\nonumber
\end{equation}
so that when $r > 0$, we have 
\begin{equation}
\G_{p,q,r}\mu(DD)=0 \implies \mu(WDW)=\mu(WDL)=\mu(LDL)=0.\nonumber
\end{equation}
Moreover, we have $\mu(WDD)=\mu(LDD)=\mu(DDD)=0$ since each of them is bounded above by $\mu(DD)$ (follows from Lemma~\ref{lem:pushforward_1}), and $\mu(DD)=0$. Since $\mu(D)=\mu(WDW)+\mu(LDL)+\mu(DDD)+2\mu(WDL)+2\mu(WDD)+2\mu(LDD)$ (by Lemma~\ref{lem:pushforward_1} and the reflection invariance of $\mu$), we have $\mu(D)=0$, as desired.

Note that, when $r=0$, we are back to the set-up considered in \cite{holroyd2019percolation}, and it has already been established there that the associated PCA, $\G_{p,q,0}$, is ergodic whenever $p+q > 0$.

\subsection{Detailed construction of the weight function for generalized percolation games when $(p,q,r) \in \mathcal{R}$}\label{sec:generalized_wt_fn_steps}
Recall that $\mathcal{R} \subset \Theta$ comprises all those values of $(p,q,r)$ that satisfy the constraints stated in Theorem~\ref{thm:three-parameter}, i.e.\ $p$, $q$ and $r$ are sufficiently small, the inequality in \eqref{three_cond_universal} holds, and precisely one of \eqref{three_cond_1}, \eqref{three_cond_2}, \eqref{three_cond_3} and \eqref{three_cond_4} is true. Throughout \S\ref{sec:generalized_wt_fn_steps}, it is assumed that the probability measure $\mu$ that we deal with is an element of $\mathcal{M}$, i.e.\ it is defined on the state space $\Omega=\hat{\mathcal{A}}^{\mathbb{Z}}=\{W,L,D\}^{\mathbb{Z}}$ of $\G_{p,q,r}$, with respect to the $\sigma$-field $\mathcal{F}$ generated by all cylinder sets of $\Omega$, and that $\mu$ is both reflection-invariant and translation-invariant. Recall, from the discussion in the paragraph preceding \S\ref{subsubsec:adjust_gen}, that beginning the construction of our weight function with the cylinder set $\mathcal{C}_{1}=(D)_{0}$ seems to be a sensible and natural choice. Since, at each step, we must obtain an inequality of the form given by \eqref{gen_ineq_form}, it is crucial that we write down the pushforward measure $\G_{p,q,r}\mu(\mathcal{C}_{i})$ for each cylinder set $\mathcal{C}_{i}$ that we incorporate into our gradually-unfolding weight function. Let us, therefore, begin with the following lemma, which captures $\G_{p,q,r}\mu(\mathcal{C}_{1})$ for $\mathcal{C}_{1}=(D)_{0}$:
\begin{lemma}\label{lem:pushforward_D}
For each $\mu\in\mathcal{M}$, we have
\begin{equation}
\G_{p,q,r}\mu(D)=2(1-p-q)(1-r)\mu(WD)+2(1-p-q)r(1-r)\mu(LD)+(1-p-q)(1-r^{2})\mu(DD).\label{D}
\end{equation}
\end{lemma}
\begin{proof}
We apply Lemma~\ref{lem:pushforward_2} to the cylinder set $(D)_{1}$, obtain the conditional probabilities $\Prob[D\big|WD], \ldots, \Prob[D\big|DD]$ from \eqref{GPCA_rule_4}, \eqref{GPCA_rule_5} and \eqref{GPCA_rule_6}, and use of the assumption that $\mu$ is reflection-invariant, to obtain:
\begin{align}
\G_{p,q,r}\mu(D)={}&(1-p-q)(1-r)\left\{\mu(WD)+\mu(DW)\right\}+(1-p-q)r(1-r)\left\{\mu(LD)+\mu(DL)\right\}\nonumber\\&+(1-p-q)(1-r^{2})\mu(DD)\nonumber\\
={}&2(1-p-q)(1-r)\mu(WD)+2(1-p-q)r(1-r)\mu(LD)+(1-p-q)(1-r^{2})\mu(DD),\nonumber
\end{align}
as claimed.
\end{proof} 

We now explain the intuition behind our choice of the second cylinder set, $\mathcal{C}_{2}$, as we attempt to shape up our weight function in the form given by \eqref{gen_form}. The coefficient $2(1-p-q)r(1-r)$ of $\mu(LD)$ in \eqref{D} is always bounded above by $1$ (in fact, by $1/2$), since $r(1-r) \leqslant 1/4$ and $(1-p-q) \leqslant 1$. The coefficient $(1-p-q)(1-r^{2})$ of $\mu(DD)$ in \eqref{D} is obviously bounded above by $1$. On the other hand, when $p$, $q$ and $r$ are all sufficiently small, the coefficient $2(1-p-q)(1-r)$ of $\mu(WD)$ in \eqref{D} is very close to $2$. We can thus write, using \eqref{D} and the fact that $\mu(D)=\mu(WD)+\mu(LD)+\mu(DD)$, the following inequality: $\G_{p,q,r}\mu(D) \leqslant \mu(D)+\mu(WD)$. Comparing this with \eqref{gen_ineq_form}, it becomes evident that if $c_{1}=1$ and $\mathcal{C}_{1}=(D)_{0}$ in \eqref{gen_form}, a sensible choice for our next cylinder set and the corresponding coefficient would be $\mathcal{C}_{2}=(W,D)_{0,1}$ and $c_{2}=1$. This now requires us to find the pushforward measure $\G_{p,q,r}\mu(WD)$:
\begin{lemma}\label{lem:pushforward_WD}
For any $\mu\in\mathcal{M}$, we have
\begin{align}
\G_{p,q,r}\mu(WD)={}&2p(1-p-q)(1-r)\mu(WD)+p(1-p-q)(1-r)(1+r)\mu(DD)+(1-p-q+r+pr-qr\nonumber\\&-r^{2}+pr^{2}+qr^{2}-r^{3}+pr^{3}+qr^{3})(1-p-q)(1-r)\mu(LD)-r(1-p-q)^{2}(1-r)^{2}\mu(LDW)\nonumber\\&-(1-r)^{2}(1-p-q)^{2}\mu(LDL)+(1-p-q)^{2}(1-r)^{2}\mu(LWD)-r^{2}(1-r)^{2}(1-p-q)^{2}\nonumber\\&\mu(WLD)-r^{2}(1-r)^{2}(1-p-q)^{2}\mu(DLD).\label{WD}
\end{align}
\end{lemma}
\begin{proof}
We apply Lemma~\ref{lem:pushforward_2}, keeping in mind that the only way the event $\{\G_{p,q,r}\eta(1)=D\}$ happens with positive probability is if at least one of $\eta(1)$ and $\eta(2)$ equals $D$. For $\mathcal{D}=(b_{0},b_{1},b_{2})_{0,1,2}\in\{W,L,D\}^{3}$, computing the conditional probabilities 
\begin{equation}
\Prob[WD\big|\mathcal{D}]=\Prob\left[\G_{p,q,r}\eta(0)=W\big|\eta(0)=b_{0},\eta(1)=b_{1}\right]\Prob\left[\G_{p,q,r}\eta(1)=D\big|\eta(1)=b_{1},\eta(2)=b_{2}\right]\nonumber
\end{equation} 
by making use of \eqref{GPCA_rule_1} through \eqref{GPCA_rule_6}, we obtain:
\begin{align}
\G_{p,q,r}\mu(WD)={}&p(1-p-q)(1-r)\mu(WDW)+p(1-p-q)(1-r)\mu(DDW)+\{p+(1-p-q)(1-r)\}\nonumber\\&(1-p-q)(1-r)\mu(LDW)+p(1-p-q)(1-r^{2})\mu(WDD)+p(1-p-q)(1-r^{2})\mu(DDD)\nonumber\\&+\{p+(1-p-q)(1-r)\}(1-p-q)(1-r^{2})\mu(LDD)+p(1-p-q)r(1-r)\mu(WDL)\nonumber\\&+p(1-p-q)r(1-r)\mu(DDL)+\{p+(1-p-q)(1-r)\}(1-p-q)r(1-r)\mu(LDL)\nonumber\\&+p(1-p-q)(1-r)\mu(WWD)+p(1-p-q)(1-r)\mu(DWD)+\{p+(1-p-q)(1-r)\}\nonumber\\&(1-p-q)(1-r)\mu(LWD)+\{p+(1-p-q)(1-r)\}(1-p-q)r(1-r)\mu(WLD)+\nonumber\\&\{p+(1-p-q)(1-r)\}(1-p-q)r(1-r)\mu(DLD)+\{p+(1-p-q)(1-r^{2})\}(1-p-q)\nonumber\\&r(1-r)\mu(LLD)\nonumber\\
={}&p(1-p-q)(1-r)\mu(WDW)+p(1-p-q)(1-r)(2+r)\mu(DDW)+\{1-q-r+qr+2pr\}\nonumber\\&(1-p-q)(1-r)\mu(LDW)+p(1-p-q)(1-r^{2})\mu(DDD)+\{1-q-r^{2}+2pr+qr^{2}+pr^{2}\}\nonumber\\&(1-p-q)(1-r)\mu(DDL)+(1-q-r+pr+qr)(1-p-q)r(1-r)\mu(LDL)\nonumber\\&+p(1-p-q)(1-r)\mu(WWD)+p(1-p-q)(1-r)\mu(DWD)+(1-q-r+pr+qr)\nonumber\\&(1-p-q)(1-r)\mu(LWD)+(1-q-r+pr+qr)(1-p-q)r(1-r)\mu(WLD)+\nonumber\\&(1-q-r+pr+qr)(1-p-q)r(1-r)\mu(DLD)+(1-q-r^{2}+pr^{2}+qr^{2})(1-p-q)\nonumber\\&r(1-r)\mu(LLD) \quad (\text{using reflection-invariance})\nonumber\\
={}&\underbrace{p(1-p-q)(1-r)\mu(WDW)+p(1-p-q)(1-r)\mu(DDW)}_{(1)}\nonumber\\&\underbrace{+p(1-p-q)(1-r)(1+r)\mu(WDD)}_{(2)}\underbrace{+p(1-p-q)(1-r)\mu(LDW)}_{(1)}\nonumber\\&\underbrace{+(1-p-q-r+qr+2pr)(1-p-q)(1-r)\mu(LDW)}_{(3)}\underbrace{+p(1-p-q)(1-r^{2})\mu(DDD)}_{(2)}\nonumber\\&\underbrace{+p(1-p-q)(1-r)(1+r)\mu(LDD)}_{(2)}\underbrace{+(1-p-q-r^{2}+pr+pr^{2}+qr^{2})(1-p-q)(1-r)\mu(LDD)}_{(3)}\nonumber\\&\underbrace{+(1-q-r+pr+qr)(1-p-q)r(1-r)\mu(LDL)}_{(3)}\underbrace{+p(1-p-q)(1-r)\mu(WWD)}_{(4)}\nonumber\\&\underbrace{+p(1-p-q)(1-r)\mu(DWD)+p(1-p-q)(1-r)\mu(LWD)}_{(4)}+(1-p-q)^{2}(1-r)^{2}\mu(LWD)\nonumber\\&\underbrace{+(1-q-r^{2}+pr^{2}+qr^{2})(1-p-q)r(1-r)\mu(WLD)}_{(5)}-r^{2}(1-r)^{2}(1-p-q)^{2}\mu(WLD)\nonumber\\&\underbrace{+(1-q-r^{2}+pr^{2}+qr^{2})(1-p-q)r(1-r)\mu(DLD)}_{(5)}-r^{2}(1-r)^{2}(1-p-q)^{2}\mu(DLD)\nonumber\\&\underbrace{+(1-q-r^{2}+pr^{2}+qr^{2})(1-p-q)r(1-r)\mu(LLD)}_{(5)}\nonumber\\
={}&\underbrace{p(1-p-q)(1-r)\mu(DW)}_{\text{combining terms underbraced by }(1)}\underbrace{+p(1-p-q)(1-r)(1+r)\mu(DD)}_{\text{combining terms underbraced by }(2)}\nonumber\\&\underbrace{+(1-p-q-r^{2}+pr+pr^{2}+qr^{2})(1-p-q)(1-r)\mu(LD)}_{\text{combining terms underbraced by }(3)}-r(1-p-q)^{2}(1-r)^{2}\mu(LDW)\nonumber\\&-(1-r)^{2}(1-p-q)^{2}\mu(LDL)\underbrace{+p(1-p-q)(1-r)\mu(WD)}_{\text{combining terms underbraced by }(4)}+(1-p-q)^{2}(1-r)^{2}\mu(LWD)\nonumber\\&\underbrace{+(1-q-r^{2}+pr^{2}+qr^{2})(1-p-q)r(1-r)\mu(LD)}_{\text{combining terms underbraced by }(5)}-r^{2}(1-r)^{2}(1-p-q)^{2}\mu(WLD)\nonumber\\&-r^{2}(1-r)^{2}(1-p-q)^{2}\mu(DLD)\nonumber\\
={}&2p(1-p-q)(1-r)\mu(WD)+p(1-p-q)(1-r)(1+r)\mu(DD)+(1-p-q+r+pr-qr\nonumber\\&-r^{2}+pr^{2}+qr^{2}-r^{3}+pr^{3}+qr^{3})(1-p-q)(1-r)\mu(LD)-r(1-p-q)^{2}(1-r)^{2}\mu(LDW)\nonumber\\&-(1-r)^{2}(1-p-q)^{2}\mu(LDL)+(1-p-q)^{2}(1-r)^{2}\mu(LWD)-r^{2}(1-r)^{2}(1-p-q)^{2}\nonumber\\&\mu(WLD)-r^{2}(1-r)^{2}(1-p-q)^{2}\mu(DLD),\nonumber
\end{align}
where, in the last step, we combine terms using the reflection-invariance property of $\mu$.
\end{proof}

Recall that, so far, we have set $c_{1}=c_{2}=1$, $\mathcal{C}_{1}=(D)_{0}$ and $\mathcal{C}_{2}=(W,D)_{0,1}$. The weight function, so far, equals $\mu(D)+\mu(WD)$, and the right side of the weight function inequality (of the form \eqref{gen_ineq_form}) is given by 
\begin{align}
{}&\G_{p,q,r}\mu(D)+\G_{p,q,r}\mu(WD)\nonumber\\
={}&2(1-p-q)(1-r)\mu(WD)+2(1-p-q)r(1-r)\mu(LD)+(1-p-q)(1-r^{2})\mu(DD)\nonumber\\&+2p(1-p-q)(1-r)\mu(WD)+p(1-p-q)(1-r)(1+r)\mu(DD)+(1-p-q+r+pr-qr\nonumber\\&-r^{2}+pr^{2}+qr^{2}-r^{3}+pr^{3}+qr^{3})(1-p-q)(1-r)\mu(LD)-r(1-p-q)^{2}(1-r)^{2}\mu(LDW)\nonumber\\&-(1-r)^{2}(1-p-q)^{2}\mu(LDL)+(1-p-q)^{2}(1-r)^{2}\mu(LWD)-r^{2}(1-r)^{2}(1-p-q)^{2}\nonumber\\&\mu(WLD)-r^{2}(1-r)^{2}(1-p-q)^{2}\mu(DLD)\nonumber\\
={}&2(1+p)(1-p-q)(1-r)\mu(WD)+(1+p)(1-p-q)(1-r)(1+r)\mu(DD)+(1-p-q+3r\nonumber\\&+pr-qr-r^{2}+pr^{2}+qr^{2}-r^{3}+pr^{3}+qr^{3})(1-p-q)(1-r)\mu(LD)-r(1-p-q)^{2}(1-r)^{2}\mu(LDW)\nonumber\\&-(1-r)^{2}(1-p-q)^{2}\mu(LDL)+(1-p-q)^{2}(1-r)^{2}\mu(LWD)-r^{2}(1-r)^{2}(1-p-q)^{2}\nonumber\\&\mu(WLD)-r^{2}(1-r)^{2}(1-p-q)^{2}\mu(DLD).\label{intermediate_1_generalized}
\end{align}
We now pause and focus on the coefficient of $\mu(LD)$ in \eqref{intermediate_1_generalized} (we need not focus on the coefficients of $\mu(WD)$ and $\mu(DD)$ since we can readily see that the former is bounded above by $2$ and the latter by $1$). Writing
\begin{align}
pr-qr-r^{2}+pr^{2}+qr^{2}-r^{3}+pr^{3}+qr^{3}=pr-qr-r^{2}(1-p-q)(1+r),\nonumber
\end{align}
we see that the coefficient of $\mu(LD)$ in \eqref{intermediate_1_generalized} is given by
\begin{align}
{}&1+2r-2(p+q)+(p+q)^{2}-3r^{2}-(p+q)r+3(p+q)r^{2}-(p+q)^{2}r\nonumber\\&+pr(1-p-q)(1-r)-qr(1-p-q)(1-r)-r^{2}(1-p-q)^{2}(1-r)^{2}\nonumber\\
={}&1+2r-2(p+q)+(p+q)^{2}-3r^{2}-2qr+pr\{-1+(1-p-q)(1-r)\}\nonumber\\&+r^{2}\{-(1-p-q)^{2}(1-r)^{2}+3p+4q-pq-q^{2}\}-p^{2}r-pqr\nonumber\\
\leqslant{}&1+2r-2(p+q)+(p+q)^{2}-3r^{2}-2qr \quad \text{when } p, q, r \text{ are sufficiently small,}\label{intermediate_2_generalized}
\end{align}
and when \eqref{three_cond_universal} holds, we have 
\begin{align}
2(p+q)-(p+q)^{2}-2r+3r^{2}+2qr\geqslant 4r-6r^{2}-3pr-6qr-2r+3r^{2}+2qr=r(2-3p-4q-3r) \geqslant 0\nonumber
\end{align}
for all $p$ and $r$ sufficiently small, thus proving that the expression in \eqref{intermediate_2_generalized}, and consequently, the coefficient of $\mu(LD)$ in \eqref{intermediate_1_generalized}, is bounded above by $1$. 

We now continue with \eqref{intermediate_1_generalized}, and write (decomposing $\mu(D)$ as in Lemma~\ref{lem:pushforward_1}):
\begin{align}
{}&\G_{p,q,r}\mu(D)+\G_{p,q,r}\mu(WD)\nonumber\\
={}&\mu(D)+\mu(WD)-2\{1-(1+p)(1-p-q)(1-r)\}\mu(WD)-\{1-(1+p)(1-p-q)(1-r)(1+r)\}\mu(DD)\nonumber\\&-\{1-(1-p-q+3r+pr-qr-r^{2}+pr^{2}+qr^{2}-r^{3}+pr^{3}+qr^{3})(1-p-q)(1-r)\}\mu(LD)\nonumber\\&-r(1-p-q)^{2}(1-r)^{2}\mu(LDW)-(1-r)^{2}(1-p-q)^{2}\mu(LDL)+(1-p-q)^{2}(1-r)^{2}\mu(LWD)\nonumber\\&-r^{2}(1-r)^{2}(1-p-q)^{2}\mu(WLD)-r^{2}(1-r)^{2}(1-p-q)^{2}\mu(DLD),\label{intermediate_4_generalized}
\end{align}
and given the observations made above regarding the coefficient of $\mu(LD)$ in \eqref{intermediate_1_generalized}, we conclude that, when $p$, $q$ and $r$ are all sufficiently small and \eqref{three_cond_universal} holds, the only term on the right side of \eqref{intermediate_4_generalized}, other than $\mu(D)$ and $\mu(WD)$, that has a non-negative coefficient is $(1-p-q)^{2}(1-r)^{2}\mu(LWD)$. Moreover, when $p$, $q$ and $r$ are all sufficiently small, we have $(1-p-q)^{2}(1-r)^{2}=\Theta(1)$ (i.e.\ its order of magnitude is the same as the constant $1$ when $p$, $q$ and $r$ approach $0$). Hence, it seems sensible to set $\mathcal{C}_{3}=(L,W,D)_{0,1,2}$ and $c_{3}=1$, and the weight function obtained thus far becomes
\begin{equation}
w_{0}(\mu)=\mu(D)+\mu(WD)+\mu(LWD).\label{w_{0}}
\end{equation}
This necessitates the computation of $\G_{p,q,r}\mu(LWD)$, which has been accomplished via Lemma~\ref{lem:pushforward_LWD}.
\begin{lemma}\label{lem:pushforward_LWD}
For any $\mu\in\mathcal{M}$, we can write $\G_{p,q,r}\mu(LWD)=\sum_{i=1}^{9}A_{i}$, where
\begin{align}
A_{1}\leqslant{}&(1-p)\{p+(1-p-q)(1-r)\}(1-p-q)r(1-r)\mu(LD),\label{bound_A_{1}}\\
A_{2}\leqslant{}&\{q+(1-p-q)r\}p(1-p-q)(1-r)(1+r)\mu(DD),\label{bound_A_{2}}\\
A_{3}\leqslant{}&\{q+(1-p-q)r\}(1-p-q)(1-r)\{1+p-q-r+2pr+qr\}\mu(LWD),\label{bound_A_{3}},\\
A_{4}\leqslant{}&\{q+(1-p-q)r\}\{p+(1-p-q)(1-r)\}(1-p-q)r(1-r)\mu(LDL),\label{bound_A_{4}},\\
A_{5}={}&(1-p)p(1-p-q)r(1-r)\mu(WDL),\label{bound_A_{5}},\\
A_{6}\leqslant{}&\{q+(1-p-q)r\}\{p+(1-p-q)(1-r)\}(1-p-q)(1-r)\mu(DW)\nonumber\\&+\{p(1-p)-(q+r-pr-qr)(1-q-r+pr+qr)\}(1-p-q)(1-r)\mu(WWDW),\label{bound_A_{6}}\\
A_{7}\leqslant{}&p(1-p-q)^{2}(1-r)^{2}\mu(WWWD)+\{q+(1-p-q)r\}p(1-p-q)(1-r)\mu(WD),\label{bound_A_{7}}\\
A_{8}\leqslant{}&(q+r-pr-qr)(1-q-r+pr+qr)(1-p-q)(1-r)(1+r)\mu(LDD),\label{bound_A_{8}}\\
A_{9}\leqslant{}&q(1-q-r+pr+qr)(1-p-q)(1-r)(1+r)\mu(WDD)\nonumber\\&+\{p(1-p)-q(1-q-r+pr+qr)\}(1-p-q)(1-r)(1+r)\mu(WWDD)\nonumber\\&+(pr+qr-q)(1-p-q)^{2}(1-r)(1+r)\mu(DWDD).\label{bound_A_{9}}
\end{align}
\end{lemma}
\begin{proof}
Extending the notation introduced just before stating Lemma~\ref{lem:pushforward_2}, we now let $R_{0}\times R_{1} \times \cdots \times R_{i-1}$, for $i\in\mathbb{N}$ and $R_{j}\subset\{W,L,D\}$ for each $j\in\{0,1,\ldots,i-1\}$, denote the set of all cylinder sets $(a_{0},a_{1},\ldots,a_{i-1})_{0,1,\ldots,i-1}$ such that $a_{j}\in R_{j}$ for each $j\in\{0,1,\ldots,i-1\}$. Applying Lemma~\ref{lem:pushforward_2} to the cylinder set $\mathcal{C}=(L,W,D)_{0,1,2}$, and keeping in mind that in order for the event $\{\G_{p,q,r}\eta(2)=D\}$ to take place with positive probability, at least one of the events $\{\eta(2)=D\}$ and $\{\eta(3)=D\}$ must occur, we can write:
\begin{align}
\G_{p,q,r}\mu(LWD)={}&\sum_{\mathcal{D}\in\{W,L,D\}^{4}}\Prob[LWD\big|\mathcal{D}]\mu(\mathcal{D})=\sum_{i=1}^{9}A_{i},\nonumber
\end{align}
where we set
\begin{align}
A_{1}={}&\sum_{\mathcal{D}\in\{W,L,D\}^{2}\times\{L\}\times\{D\}}\Prob[LWD\big|\mathcal{D}]\mu(\mathcal{D}),\nonumber\\
A_{2}={}&\sum_{\mathcal{D}\in\{W,L,D\}\times\{D\}^{2}\times\{W,L,D\}}\Prob[LWD\big|\mathcal{D}]\mu(\mathcal{D}),\nonumber\\
A_{3}={}&\sum_{\mathcal{D}\in\{W,L,D\}\times\{L\}\times\{W\}\times\{D\}}\Prob[LWD\big|\mathcal{D}]\mu(\mathcal{D})+\sum_{\mathcal{D}\in\{L\}\times\{W\}\times\{D\}\times\{W,L,D\}}\Prob[LWD\big|\mathcal{D}]\mu(\mathcal{D}),\nonumber\\
A_{4}={}&\sum_{\mathcal{D}\in\{W,L,D\}\times\{L\}\times\{D\}\times\{L\}}\Prob[LWD\big|\mathcal{D}]\mu(\mathcal{D}),\nonumber\\
A_{5}={}&\sum_{\mathcal{D}\in\{W,D\}\times\{W\}\times\{D\}\times\{L\}}\Prob[LWD\big|\mathcal{D}]\mu(\mathcal{D}),\nonumber\\
A_{6}={}&\sum_{\mathcal{D}\in\{W,L,D\}\times\{L\}\times\{D\}\times\{W\}}\Prob[LWD\big|\mathcal{D}]\mu(\mathcal{D})+\sum_{\mathcal{D}\in\{W,D\}\times\{W\}\times\{D\}\times\{W\}}\Prob[LWD\big|\mathcal{D}]\mu(\mathcal{D}),\nonumber\\
A_{7}={}&\sum_{\mathcal{D}\in\{W,L,D\}\times\{W,D\}\times\{W\}\times\{D\}}\Prob[LWD\big|\mathcal{D}]\mu(\mathcal{D}),\nonumber\\
A_{8}={}&\sum_{\mathcal{D}\in\{W,L,D\}\times\{L\}\times\{D\}^{2}}\Prob[LWD\big|\mathcal{D}]\mu(\mathcal{D}),\nonumber\\
A_{9}={}&\sum_{\mathcal{D}\in\{W,D\}\times\{W\}\times\{D\}^{2}}\Prob[LWD\big|\mathcal{D}]\mu(\mathcal{D}).\nonumber
\end{align}
In the computation of $A_{i}$ for each $i\in\{1,2,\ldots,9\}$, we make use of the relation 
\begin{align}
\Prob[LWD\big|\mathcal{D}]={}&\Prob\left[\G_{p,q,r}\eta(0)=L\big|\eta(0)=b_{0},\eta(1)=b_{1}\right]\Prob\left[\G_{p,q,r}\eta(1)=W\big|\eta(1)=b_{1},\eta(2)=b_{2}\right]\nonumber\\&\Prob\left[\G_{p,q,r}\eta(2)=D\big|\eta(2)=b_{2},\eta(3)=b_{3}\right],\nonumber
\end{align}
where $\mathcal{D}=(b_{0},b_{1},b_{2},b_{3})_{0,1,2,3}$, and we apply the stochastic update rules of \eqref{GPCA_rule_1} through \eqref{GPCA_rule_6}.

We begin with
\begin{align}
A_{1}={}&(1-p)\{p+(1-p-q)(1-r)\}(1-p-q)r(1-r)\mu(WWLD)+\{q+(1-p-q)r\}\{p+(1-p-q)\nonumber\\&(1-r^{2})\}(1-p-q)r(1-r)\mu(WLLD)+\{q+(1-p-q)r\}\{p+(1-p-q)(1-r)\}(1-p-q)r\nonumber\\&(1-r)\mu(LWLD)+\{q+(1-p-q)r^{2}\}\{p+(1-p-q)(1-r^{2})\}(1-p-q)r(1-r)\mu(LLLD)+\nonumber\\&\{q+(1-p-q)r\}\{p+(1-p-q)(1-r)\}(1-p-q)r(1-r)\mu(WDLD)+\{q+(1-p-q)r\}\nonumber\\&\{p+(1-p-q)(1-r)\}(1-p-q)r(1-r)\mu(DWLD)+\{q+(1-p-q)r^{2}\}\{p+(1-p-q)(1-r)\}\nonumber\\&(1-p-q)r(1-r)\mu(LDLD)+\{q+(1-p-q)r^{2}\}\{p+(1-p-q)(1-r^{2})\}(1-p-q)r(1-r)\nonumber\\&\mu(DLLD)+\{q+(1-p-q)r^{2}\}\{p+(1-p-q)(1-r)\}(1-p-q)r(1-r)\mu(DDLD).\label{contribution_LD}
\end{align}
A quick comparison of the coefficients in the various terms of \eqref{contribution_LD} reveals that the coefficient of each of $\mu(LWLD)$, $\mu(LLLD)$, $\ldots$, $\mu(DDLD)$ is bounded above by the coefficient of $\mu(WLLD)$, and 
\begin{align}
{}&(1-p)\{p+(1-p-q)(1-r)\}-\{q+(1-p-q)r\}\{p+(1-p-q)(1-r^{2})\}\nonumber\\
={}&(1-p)\{p+(1-p-q)(1-r)\}-\{q+(1-p-q)r\}\{p+(1-p-q)(1-r)+r(1-p-q)(1-r)\}\nonumber\\
={}&(1-p-q)(1-r)\{p+(1-p-q)(1-r)\}-\{q+(1-p-q)r\}r(1-p-q)(1-r)\nonumber\\
={}&(1-p-q)(1-r)\{p+(1-p-q)(1-r-r^{2})-qr\}\nonumber
\end{align}
is non-negative whenever $p$, $q$ and $r$ are all sufficiently small, so that the coefficient of $\mu(WWLD)$ exceeds that of $\mu(WLLD)$ in \eqref{contribution_LD}. For each cylinder set $\mathcal{D}\in\{W,L,D\}^{2}\times\{L\}\times\{D\}$, let us denote the corresponding coefficient by $\alpha_{\mathcal{D}}$ \big(for instance, the coefficient of $\mu(DLLD)$ is $\alpha_{DLLD}=\{q+(1-p-q)r^{2}\}\{p+(1-p-q)(1-r^{2})\}(1-p-q)r(1-r)$\big), and let us define
\begin{equation}
\alpha'_{\mathcal{D}}=\alpha_{\mathcal{D}}-\alpha_{WWLD}=\alpha_{\mathcal{D}}-(1-p)\{p+(1-p-q)(1-r)\}(1-p-q)r(1-r)\nonumber
\end{equation}
for each cylinder set $\mathcal{D}\in\{W,L,D\}^{2}\times\{L\}\times\{D\}\setminus\left\{(W,W,L,D)_{0,1,2,3}\right\}$. From our discussion above, it is evident that when $p$, $q$ and $r$ are all small enough, each $\alpha'_{\mathcal{D}} \leqslant 0$, so that we may rewrite \eqref{contribution_LD} as
\begin{align}
A_{1}={}&(1-p)\{p+(1-p-q)(1-r)\}(1-p-q)r(1-r)\mu(LD)+\alpha'_{WLLD}\mu(WLLD)+\alpha'_{LWLD}\mu(LWLD)\nonumber\\&+\alpha'_{LLLD}\mu(LLLD)+\alpha'_{WDLD}\mu(WDLD)+\alpha'_{DWLD}\mu(DWLD)+\alpha'_{LDLD}\mu(LDLD)+\alpha'_{DLLD}\mu(DLLD)\nonumber\\&+\alpha'_{DDLD}\mu(DDLD)\leqslant(1-p)\{p+(1-p-q)(1-r)\}(1-p-q)r(1-r)\mu(LD),\nonumber
\end{align}
which proves the inequality in \eqref{bound_A_{1}}. Next, we have
\begin{align}
A_{2}={}&\{q+(1-p-q)r\}p(1-p-q)(1-r)\mu(WDDW)+\{q+(1-p-q)r^{2}\}p(1-p-q)(1-r)\mu(DDDW)\nonumber\\&+\{q+(1-p-q)r^{2}\}p(1-p-q)(1-r)\mu(LDDW)+\{q+(1-p-q)r\}p(1-p-q)(1-r^{2})\mu(WDDD)\nonumber\\&+\{q+(1-p-q)r^{2}\}p(1-p-q)(1-r^{2})\mu(DDDD)+\{q+(1-p-q)r^{2}\}p(1-p-q)(1-r^{2})\mu(LDDD)\nonumber\\&+\{q+(1-p-q)r\}p(1-p-q)r(1-r)\mu(WDDL)+\{q+(1-p-q)r^{2}\}p(1-p-q)r(1-r)\mu(DDDL)\nonumber\\&+\{q+(1-p-q)r^{2}\}p(1-p-q)r(1-r)\mu(LDDL).\label{contribution_DD}
\end{align}
A quick comparison of the coefficients of the various terms in \eqref{contribution_DD} reveals that the coefficient of $\mu(WDDD)$, namely, $\{q+(1-p-q)r\}p(1-p-q)(1-r)(1+r)$, is the largest. As in the analysis of $A_{1}$, letting $\alpha_{\mathcal{D}}$ denote the coefficient of $\mu(\mathcal{D})$ in \eqref{contribution_DD} for every $\mathcal{D}\in\{W,L,D\}\times\{D\}^{2}\times\{W,L,D\}$, and setting
\begin{equation}
\alpha'_{\mathcal{D}}=\alpha_{\mathcal{D}}-\alpha_{WDDD}=\alpha_{\mathcal{D}}-\{q+(1-p-q)r\}p(1-p-q)(1-r)(1+r)\nonumber
\end{equation}
for each $\mathcal{D}\in\{W,L,D\}\times\{D\}^{2}\times\{W,L,D\}\setminus\left\{(W,D,D,D)_{0,1,2,3}\right\}$, we know, from our observation above, that each $\alpha'_{\mathcal{D}} \leqslant 0$, and we can rewrite \eqref{contribution_DD} as:
\begin{align}
A_{2}={}&\{q+(1-p-q)r\}p(1-p-q)(1-r)(1+r)\mu(DD)+\alpha'_{WDDW}\mu(WDDW)+\alpha'_{DDDW}\mu(DDDW)\nonumber\\&+\alpha'_{LDDW}\mu(LDDW)+\alpha'_{DDDD}\mu(DDDD)+\alpha'_{LDDD}\mu(LDDD)+\alpha'_{WDDL}\mu(WDDL)+\alpha'_{DDDL}\mu(DDDL)\nonumber\\&+\alpha'_{LDDL}\mu(LDDL)\leqslant\{q+(1-p-q)r\}p(1-p-q)(1-r)(1+r)\mu(DD),\nonumber
\end{align}
which proves the bound in \eqref{bound_A_{2}}. From the definition of $A_{3}$, we have
\begin{align}
A_{3}={}&\{q+(1-p-q)r\}\{p+(1-p-q)(1-r)\}(1-p-q)(1-r)\mu(WLWD)+\{q+(1-p-q)r^{2}\}\nonumber\\&\{p+(1-p-q)(1-r)\}(1-p-q)(1-r)\mu(DLWD)+\{q+(1-p-q)r^{2}\}\{p+(1-p-q)(1-r)\}\nonumber\\&(1-p-q)(1-r)\mu(LLWD)+\{q+(1-p-q)r\}p(1-p-q)(1-r)\mu(LWDW)+\{q+(1-p-q)r\}\nonumber\\&p(1-p-q)(1-r^{2})\mu(LWDD)+\{q+(1-p-q)r\}p(1-p-q)r(1-r)\mu(LWDL)\nonumber\\
={}&\{q+(1-p-q)r\}\{p+(1-p-q)(1-r)\}(1-p-q)(1-r)\mu(LWD)+\alpha'_{DLWD}\mu(DLWD)+\alpha'_{LLWD}\nonumber\\&\mu(LLWD)+\{q+(1-p-q)r\}p(1-p-q)(1-r^{2})\mu(LWD)+\alpha'_{LWDW}\mu(LWDW)+\alpha'_{LWDL}\mu(LWDL)\nonumber\\
={}&\{q+(1-p-q)r\}(1-p-q)(1-r)\{1+p-q-r+2pr+qr\}\mu(LWD)+\alpha'_{DLWD}\mu(DLWD)\nonumber\\&+\alpha'_{LLWD}\mu(LLWD)+\alpha'_{LWDW}\mu(LWDW)+\alpha'_{LWDL}\mu(LWDL),\label{contribution_LWD}
\end{align}
where each of the coefficients
\begin{align}
{}&\alpha'_{DLWD}=\alpha'_{LLWD}=-r(1-p-q)^{2}(1-r)^{2}\{p+(1-p-q)(1-r)\},\nonumber\\
{}&\alpha'_{LWDW}=-r\{q+(1-p-q)r\}p(1-p-q)(1-r),\nonumber\\
{}&\alpha'_{LWDL}=-\{q+(1-p-q)r\}p(1-p-q)(1-r)\nonumber
\end{align}
is non-positive, letting \eqref{contribution_LWD} imply the inequality in \eqref{bound_A_{3}}. Next, we have
\begin{align}
A_{4}={}&\{q+(1-p-q)r\}\{p+(1-p-q)(1-r)\}(1-p-q)r(1-r)\mu(WLDL)\nonumber\\&+\{q+(1-p-q)r^{2}\}\{p+(1-p-q)(1-r)\}(1-p-q)r(1-r)\mu(DLDL)\nonumber\\&+\{q+(1-p-q)r^{2}\}\{p+(1-p-q)(1-r)\}(1-p-q)r(1-r)\mu(LLDL)\nonumber\\
={}&\{q+(1-p-q)r\}\{p+(1-p-q)(1-r)\}(1-p-q)r(1-r)\mu(LDL)\nonumber\\&+\alpha'_{DLDL}\mu(DLDL)+\alpha'_{LLDL}\mu(LLDL),\label{contribution_LDL}
\end{align}
where $\alpha'_{DLDL}=\alpha'_{LLDL}=-\{p+(1-p-q)(1-r)\}(1-p-q)^{2}r^{2}(1-r)^{2}$ are both non-positive coefficients, thus letting \eqref{contribution_LDL} imply that the inequality in \eqref{bound_A_{4}} is true. Keeping in mind that we have already taken into account $\mathcal{D}=(L,W,D,L)_{0,1,2,3}$ in $A_{3}$, we obtain, from the definition of $A_{5}$,
\begin{align}
A_{5}={}&(1-p)p(1-p-q)r(1-r)\mu(WWDL)+\{q+(1-p-q)r\}p(1-p-q)r(1-r)\mu(DWDL)\nonumber\\
={}&(1-p)p(1-p-q)r(1-r)\mu(WDL)-pr(1-p-q)^{2}(1-r)^{2}\mu(DWDL)\nonumber\\&-(1-p)p(1-p-q)r(1-r)\mu(LWDL)\leqslant (1-p)p(1-p-q)r(1-r)\mu(WDL),\nonumber
\end{align}
proving \eqref{bound_A_{5}}. Keeping in mind that $\mathcal{D}=(L,W,D,W)_{0,1,2,3}$ has already been taken into account in $A_{3}$, we obtain, from the definition of $A_{6}$:
\begin{align}
A_{6}={}&\{q+(1-p-q)r\}\{p+(1-p-q)(1-r)\}(1-p-q)(1-r)\mu(WLDW)\nonumber\\&+\{q+(1-p-q)r^{2}\}\{p+(1-p-q)(1-r)\}(1-p-q)(1-r)\mu(DLDW)\nonumber\\&+\{q+(1-p-q)r^{2}\}\{p+(1-p-q)(1-r)\}(1-p-q)(1-r)\mu(LLDW)\nonumber\\&+(1-p)p(1-p-q)(1-r)\mu(WWDW)+\{q+(1-p-q)r\}p(1-p-q)(1-r)\mu(DWDW)\nonumber\\
\leqslant{}&\{q+(1-p-q)r\}\{p+(1-p-q)(1-r)\}(1-p-q)(1-r)\mu(DW)+\alpha'_{DLDW}\mu(DLDW)\nonumber\\&+\alpha'_{LLDW}\mu(LLDW)+\{p(1-p)-(q+r-pr-qr)(1-q-r+pr+qr)\}(1-p-q)(1-r)\nonumber\\&\mu(WWDW)+\alpha'_{DWDW}\mu(DWDW),\label{contribution_DW}
\end{align}
(the sum of $\mu(WLDW)$, $\mu(DLDW)$, $\mu(LLDW)$, $\mu(WWDW)$ and $\mu(DWDW)$ is bounded above by $\mu(DW)$, which follows from Lemma~\ref{lem:pushforward_1}) where each of 
\begin{align}
{}&\alpha'_{DLDW}=\alpha'_{LLDW}=-r\{p+(1-p-q)(1-r)\}(1-p-q)^{2}(1-r)^{2},\nonumber\\
{}&\alpha'_{DWDW}=-\{q+(1-p-q)r\}(1-p-q)^{2}(1-r)^{2}\nonumber
\end{align}
is a non-positive, allowing us to conclude, from \eqref{contribution_DW}, that \eqref{bound_A_{6}} is true. Next, we observe that
\begin{align}
A_{7}={}&(1-p)p(1-p-q)(1-r)\mu(WWWD)+\{q+(1-p-q)r\}p(1-p-q)(1-r)\mu(DWWD)\nonumber\\&+\{q+(1-p-q)r\}p(1-p-q)(1-r)\mu(LWWD)+\{q+(1-p-q)r\}p(1-p-q)(1-r)\nonumber\\&\mu(WDWD)+\{q+(1-p-q)r^{2}\}p(1-p-q)(1-r)\mu(DDWD)+\{q+(1-p-q)r^{2}\}\nonumber\\&p(1-p-q)(1-r)\mu(LDWD)\nonumber\\
\leqslant{}&p(1-p-q)^{2}(1-r)^{2}\mu(WWWD)+\{q+(1-p-q)r\}p(1-p-q)(1-r)\mu(WD)\nonumber\\&+\alpha'_{DDWD}\mu(DDWD)+\alpha'_{LDWD}\mu(LDWD),\label{contribution_WD}
\end{align}
where $\alpha'_{DDWD}=\alpha'_{LDWD}=-pr(1-p-q)^{2}(1-r)^{2}$ (the last inequality follows from the fact that the sum of $\mu(WWWD)$, $\mu(DWWD)$, $\mu(LWWD)$, $\mu(WDWD)$, $\mu(DDWD)$ and $\mu(LDWD)$ is bounded above by $\mu(WD)$, which, in turn, follows from Lemma~\ref{lem:pushforward_1}), thus allowing \eqref{contribution_WD} to imply that \eqref{bound_A_{7}} holds. The second last term in the expansion of $\G_{p,q,r}\mu(LWD)$ equals
\begin{align}
A_{8}={}&\{q+(1-p-q)r\}\{p+(1-p-q)(1-r)\}(1-p-q)(1-r)(1+r)\mu(WLDD)\nonumber\\&+\{q+(1-p-q)r^{2}\}\{p+(1-p-q)(1-r)\}(1-p-q)(1-r)(1+r)\mu(DLDD)\nonumber\\&+\{q+(1-p-q)r^{2}\}\{p+(1-p-q)(1-r)\}(1-p-q)(1-r)(1+r)\mu(LLDD)\nonumber\\
={}&(q+r-pr-qr)(1-q-r+pr+qr)(1-p-q)(1-r)(1+r)\mu(LDD)\nonumber\\&-r\{p+(1-p-q)(1-r)\}(1-p-q)^{2}(1-r)^{2}(1+r)\{\mu(DLDD)+\mu(LLDD)\},\nonumber
\end{align}
corroborating the inequality in \eqref{bound_A_{8}}, whereas the final term in the expansion of $\G_{p,q,r}\mu(LWD)$ equals
\begin{align}
A_{9}={}&(1-p)p(1-p-q)(1-r)(1+r)\mu(WWDD)\nonumber\\&+\{q+(1-p-q)r\}p(1-p-q)(1-r)(1+r)\mu(DWDD)\label{contribution_{W,D}_W_D_D}\\
\leqslant{}&q(1-q-r+pr+qr)(1-p-q)(1-r)(1+r)\mu(WDD)\nonumber\\&+\{p(1-p)-q(1-q-r+pr+qr)\}(1-p-q)(1-r)(1+r)\mu(WWDD)\nonumber\\&+(pr+qr-q)(1-p-q)^{2}(1-r)(1+r)\mu(DWDD),\nonumber
\end{align} 
(the second step is obtained by first applying Lemma~\ref{lem:pushforward_1}, then ignoring the term involving $\mu(LWDD)$ since its coefficient is non-positive), corrborating \eqref{bound_A_{9}}. This concludes the proof of Lemma~\ref{lem:pushforward_LWD}.
\end{proof}

The weight function inequality of the form given by \eqref{gen_ineq_form}, corresponding to the weight function in \eqref{w_{0}}, can now be derived as follows, making use of \eqref{intermediate_4_generalized} and Lemma~\ref{lem:pushforward_LWD}:
\begin{align}\label{w_{0}_ineq_1}
w_{0}(\G_{p,q,r}\mu)={}&\G_{p,q,r}\mu(D)+\G_{p,q,r}\mu(WD)+\G_{p,q,r}\mu(LWD)\nonumber\\
={}&\mu(D)+\mu(WD)-2\{1-(1+p)(1-p-q)(1-r)\}\mu(WD)-\{1-(1+p)(1-p-q)(1-r)\nonumber\\&(1+r)\}\mu(DD)-\{1-(1-p-q+3r+pr-qr-r^{2}+pr^{2}+qr^{2}-r^{3}+pr^{3}+qr^{3})(1-p-q)\nonumber\\&(1-r)\}\mu(LD)-r(1-p-q)^{2}(1-r)^{2}\mu(LDW)-(1-r)^{2}(1-p-q)^{2}\mu(LDL)+(1-p-q)^{2}\nonumber\\&(1-r)^{2}\mu(LWD)-r^{2}(1-r)^{2}(1-p-q)^{2}\mu(WLD)-r^{2}(1-r)^{2}(1-p-q)^{2}\mu(DLD)\nonumber\\&+\G_{p,q,r}\mu(LWD)\nonumber\\
={}&\mu(D)+\mu(WD)+\mu(LWD)-2\{1-(1+p)(1-p-q)(1-r)\}\mu(WD)\nonumber\\&-\{1-(1+p)(1-p-q)(1-r)(1+r)\}\mu(DD)-\{1-(1-p-q+3r+pr-qr-r^{2}+pr^{2}\nonumber\\&+qr^{2}-r^{3}+pr^{3}+qr^{3})(1-p-q)(1-r)\}\mu(LD)-r(1-p-q)^{2}(1-r)^{2}\mu(LDW)\nonumber\\&-(1-r)^{2}(1-p-q)^{2}\mu(LDL)-\{1-(1-p-q)^{2}(1-r)^{2}\}\mu(LWD)-r^{2}(1-r)^{2}(1-p-q)^{2}\nonumber\\&\mu(WLD)-r^{2}(1-r)^{2}(1-p-q)^{2}\mu(DLD)+\sum_{i=1}^{9}A_{i}\nonumber\\
\leqslant{}&w_{0}(\mu)-2\{1-(1+p)(1-p-q)(1-r)\}\mu(WD)-\{1-(1+p)(1-p-q)(1-r)(1+r)\}\mu(DD)\nonumber\\&\underbrace{-\{1-(1-p-q+3r+pr-qr-r^{2}+pr^{2}+qr^{2}-r^{3}+pr^{3}+qr^{3})(1-p-q)(1-r)\}\mu(LD)}\nonumber\\&-r(1-p-q)^{2}(1-r)^{2}\mu(LDW)-(1-r)^{2}(1-p-q)^{2}\mu(LDL)-\{1-(1-p-q)^{2}(1-r)^{2}\}\nonumber\\&\mu(LWD)-r^{2}(1-r)^{2}(1-p-q)^{2}\mu(WLD)-r^{2}(1-r)^{2}(1-p-q)^{2}\mu(DLD)\nonumber\\&\underbrace{+(1-p)\{p+(1-p-q)(1-r)\}(1-p-q)r(1-r)\mu(LD)}+\sum_{i=2}^{9}A_{i},
\end{align}
where the last step is obtained by implementing the inequality of \eqref{bound_A_{1}}. Combining the terms involving $\mu(LD)$ (indicated by underbraces) in \eqref{w_{0}_ineq_1}, we obtain:
\begin{align}
{}&-\{1-(1-p-q+3r+pr-qr-r^{2}+pr^{2}+qr^{2}-r^{3}+pr^{3}+qr^{3})(1-p-q)(1-r)\}\mu(LD)\nonumber\\&+(1-p)\{p+(1-p-q)(1-r)\}(1-p-q)r(1-r)\mu(LD)\nonumber\\
={}&[-1+(1-p-q+4r-2qr-2r^{2}+pqr+3pr^{2}+2qr^{2}-r^{3}+pr^{3}+qr^{3}-p^{2}r^{2}-pqr^{2})\nonumber\\&(1-p-q)(1-r)]\mu(LD)\nonumber\\
={}&[-1+\{1-p-q+4r-qr(2-p-2r)-r^{2}(2-3p)-r^{3}(1-p-q)-p^{2}r^{2}-pqr^{2}\}\nonumber\\&(1-p-q)(1-r)]\mu(LD)\label{intermediate_5'}\\
\leqslant{}&[-1+(1-p-q+4r)(1-p-q)(1-r)]\mu(LD) \quad \text{when } p, r \text{ are sufficiently small;}\nonumber\\
={}&[-1+\{1+3r-4r^{2}-2(p+q)+(p+q)^{2}-(p+q)r-(p+q)r(1+p+q-4r)\}]\mu(LD)\nonumber\\
\leqslant{}&\{3r-4r^{2}-2(p+q)+(p+q)^{2}-(p+q)r\}]\mu(LD) \quad \text{when } r \text{ sufficiently small}.\label{intermediate_5}
\end{align}
When $p$, $q$ and $r$ are sufficiently small and the inequality in \eqref{three_cond_universal} holds, we have
\begin{align}
{}&2(p+q)-(p+q)^{2}-3r+4r^{2}+(p+q)r\nonumber\\
\geqslant{}& 4r-6r^{2}-3pr-6qr-3r+4r^{2}+(p+q)r=r(1-2p-5q-2r) \geqslant 0,\nonumber
\end{align}
so that the coefficient of $\mu(LD)$ in \eqref{intermediate_5} is non-positive, and consequently, the coefficient of $\mu(LD)$ in \eqref{intermediate_5'}, is non-positive. Incorporating \eqref{intermediate_5'} into \eqref{w_{0}_ineq_1}, we obtain:
\begin{align}
w_{0}(\G_{p,q,r}\mu)\leqslant{}&w_{0}(\mu)-2\{1-(1+p)(1-p-q)(1-r)\}\mu(WD)-\{1-(1+p)(1-p-q)\nonumber\\&(1-r)(1+r)\}\mu(DD)-[1-\{1-p-q+4r-qr(2-p-2r)-r^{2}(2-3p)\nonumber\\&-r^{3}(1-p-q)-p^{2}r^{2}-pqr^{2}\}(1-p-q)(1-r)]\mu(LD)-r(1-p-q)^{2}(1-r)^{2}\nonumber\\&\mu(LDW)-(1-r)^{2}(1-p-q)^{2}\mu(LDL)-\{1-(1-p-q)^{2}(1-r)^{2}\}\mu(LWD)\nonumber\\&-r^{2}(1-r)^{2}(1-p-q)^{2}\mu(WLD)-r^{2}(1-r)^{2}(1-p-q)^{2}\mu(DLD)+\sum_{i=2}^{9}A_{i}\nonumber\\
\leqslant{}&w_{0}(\mu)-2\{1-(1+p)(1-p-q)(1-r)\}\mu(WD)\nonumber\\&\underbrace{-\{1-(1+p)(1-p-q)(1-r)(1+r)\}\mu(DD)}-[1-\{1-p-q+4r-qr(2-p-2r)\nonumber\\&-r^{2}(2-3p)-r^{3}(1-p-q)-p^{2}r^{2}-pqr^{2}\}(1-p-q)(1-r)]\mu(LD)\nonumber\\&-r(1-p-q)^{2}(1-r)^{2}\mu(LDW)-(1-r)^{2}(1-p-q)^{2}\mu(LDL)-\nonumber\\&\{1-(1-p-q)^{2}(1-r)^{2}\}\mu(LWD)-r^{2}(1-r)^{2}(1-p-q)^{2}\mu(WLD)-\nonumber\\&r^{2}(1-r)^{2}(1-p-q)^{2}\mu(DLD)\underbrace{+\{q+(1-p-q)r\}p(1-p-q)(1-r)(1+r)\mu(DD)}+\sum_{i=3}^{9}A_{i},\label{w_{0}_ineq_2}
\end{align}
where the last step follows from the inequality in \eqref{bound_A_{2}}. Before proceeding further, we observe that
\begin{align}
{}&\{(1-p-q)(1-r)(1+r)\}^{-1}=\{1-p-q-r^{2}+pr^{2}+qr^{2}\}^{-1}=\sum_{j=0}^{\infty}(p+q+r^{2}-pr^{2}-qr^{2})^{j}\nonumber\\
\geqslant{}&1+p+q+r^{2}-pr^{2}-qr^{2}+(p+q+r^{2}-pr^{2}-qr^{2})^{2}\nonumber\\
={}&1+p+q+r^{2}+p^{2}+q^{2}+2pq+p^{2}r^{4}-2p^{2}r^{2}+2pqr^{4}-4pqr^{2}-2pr^{4}+pr^{2}+q^{2}r^{4}\nonumber\\&-2q^{2}r^{2}-2qr^{4}+qr^{2}+r^{4}\label{inverse_1}\\
={}&(1+p+pq+pr-p^{2}r-pqr)+\{(p-r)^{2}+pr\}+p^{2}r(1-2r)+pr^{2}(1-2r^{2})+pqr(1-4r)\nonumber\\&+q^{2}(1-2r^{2})+qr^{2}(1-2r^{2})+q+pq+p^{2}r^{4}+2pqr^{4}+q^{2}r^{4}+r^{4}\nonumber\\
\geqslant{}&1+p+pq+pr-p^{2}r-pqr \text{ when } p,q,r \text{ are all sufficiently small}.\label{intermediate_6_generalized}
\end{align}
Combining the terms involving $\mu(DD)$ in \eqref{w_{0}_ineq_2} (indicated by underbraces), we obtain:
\begin{align}
{}&-\{1-(1+p)(1-p-q)(1-r)(1+r)\}\mu(DD)+\{q+(1-p-q)r\}p(1-p-q)(1-r)(1+r)\mu(DD)\nonumber\\
{}&=[-1+(1+p+pq+pr-p^{2}r-pqr)(1-p-q)(1-r)(1+r)]\mu(DD),\label{intermediate_7_generalized}
\end{align}
and the coefficient of $\mu(DD)$ in \eqref{intermediate_7_generalized} is non-positive by the inequality we have derived in \eqref{intermediate_6_generalized}. Incorporating \eqref{intermediate_7_generalized} into \eqref{w_{0}_ineq_2}, we obtain:
\begin{align}
w_{0}(\G_{p,q,r}\mu)\leqslant{}&w_{0}(\mu)-2\{1-(1+p)(1-p-q)(1-r)\}\mu(WD)-\{1-(1+p+pq+pr-p^{2}r-pqr)\nonumber\\&(1-p-q)(1-r)(1+r)\}\mu(DD)-[1-\{1-p-q+4r-qr(2-p-2r)-r^{2}(2-3p)\nonumber\\&-r^{3}(1-p-q)-p^{2}r^{2}-pqr^{2}\}(1-p-q)(1-r)]\mu(LD)-r(1-p-q)^{2}(1-r)^{2}\nonumber\\&\mu(LDW)-(1-r)^{2}(1-p-q)^{2}\mu(LDL)-\{1-(1-p-q)^{2}(1-r)^{2}\}\mu(LWD)\nonumber\\&-r^{2}(1-r)^{2}(1-p-q)^{2}\mu(WLD)-r^{2}(1-r)^{2}(1-p-q)^{2}\mu(DLD)+\sum_{i=3}^{9}A_{i}\nonumber\\
\leqslant{}&w_{0}(\mu)-2\{1-(1+p)(1-p-q)(1-r)\}\mu(WD)-\{1-(1+p+pq+pr-p^{2}r-pqr)\nonumber\\&(1-p-q)(1-r)(1+r)\}\mu(DD)-[1-\{1-p-q+4r-qr(2-p-2r)-r^{2}(2-3p)\nonumber\\&-r^{3}(1-p-q)-p^{2}r^{2}-pqr^{2}\}(1-p-q)(1-r)]\mu(LD)-r(1-p-q)^{2}(1-r)^{2}\nonumber\\&\mu(LDW)-(1-r)^{2}(1-p-q)^{2}\mu(LDL)\underbrace{-\{1-(1-p-q)^{2}(1-r)^{2}\}\mu(LWD)}\nonumber\\&-r^{2}(1-r)^{2}(1-p-q)^{2}\mu(WLD)-r^{2}(1-r)^{2}(1-p-q)^{2}\mu(DLD)\nonumber\\&\underbrace{+\{q+(1-p-q)r\}(1-p-q)(1-r)\{1+p-q-r+2pr+qr\}\mu(LWD)}+\sum_{i=4}^{9}A_{i},\label{w_{0}_ineq_3}
\end{align}
where the last step follows from \eqref{bound_A_{3}}. Much like the derivation carried out for \eqref{intermediate_6_generalized}, we have
\begin{align}
{}&\{(1-p-q)(1-r)\}^{-1}=(1-p-q-r+pr+qr)^{-1}\nonumber\\
\geqslant{}&1+(p+q+r-pr-qr)+(p+q+r-pr-qr)^{2}\nonumber\\
={}&1+p+q+r+p^{2}+q^{2}+r^{2}+2pq+pr+qr-4pqr-2p^{2}r-2pr^{2}-2q^{2}r-2qr^{2}\nonumber\\&+p^{2}r^{2}+q^{2}r^{2}+2pqr^{2}\label{inverse_2}\\
={}&(1-p-q^{2}-r^{2}-2qr+pq+pr+2pqr+3pr^{2}+2qr^{2}+2q^{2}r-p^{2}r-3pqr^{2}\nonumber\\&-q^{2}r^{2}-2p^{2}r^{2})+2p+q+r+2q^{2}+pq+qr(3-6p-4q-4r)+r^{2}(2-5p)\nonumber\\&+p^{2}(1-r)+3p^{2}r^{2}+2q^{2}r^{2}+5pqr^{2}\nonumber\\
\geqslant{}&(1-p-q^{2}-r^{2}-2qr+pq+pr+2pqr+3pr^{2}+2qr^{2}+2q^{2}r-p^{2}r-3pqr^{2}\nonumber\\&-q^{2}r^{2}-2p^{2}r^{2}) \text{ for all } p,q,r \text{ sufficiently small}.\label{intermediate_8_generalized}
\end{align}
Combining the terms involving $\mu(LWD)$ in \eqref{w_{0}_ineq_3} (indicated by underbraces), we obtain:
\begin{align}
{}&\{q+(1-p-q)r\}(1-p-q)(1-r)\{1+p-q-r+2pr+qr\}\mu(LWD)\nonumber\\&-\{1-(1-p-q)^{2}(1-r)^{2}\}\mu(LWD)\nonumber\\
={}&[-1+\{(q+r-pr-qr)(1+p-q-r+2pr+qr)+(1-p-q)(1-r)\}\nonumber\\&(1-p-q)(1-r)]\mu(LWD)\nonumber\\
={}&[-1+(1-p-q^{2}-r^{2}-2qr+pq+pr+2pqr+3pr^{2}+2qr^{2}+2q^{2}r-p^{2}r-3pqr^{2}\nonumber\\&-q^{2}r^{2}-2p^{2}r^{2})(1-p-q)(1-r)]\mu(LWD),\label{intermediate_9_generalized}
\end{align}
and by \eqref{intermediate_8_generalized}, the coefficient of $\mu(LWD)$ in \eqref{intermediate_9_generalized} is non-positive. Incorporating \eqref{intermediate_9_generalized} into \eqref{w_{0}_ineq_3}, we get:
\begin{align}
w_{0}(\G_{p,q,r}\mu)\leqslant{}&w_{0}(\mu)-2\{1-(1+p)(1-p-q)(1-r)\}\mu(WD)-\{1-(1+p+pq+pr-p^{2}r-pqr)\nonumber\\&(1-p-q)(1-r)(1+r)\}\mu(DD)-[1-\{1-p-q+4r-qr(2-p-2r)-r^{2}(2-3p)\nonumber\\&-r^{3}(1-p-q)-p^{2}r^{2}-pqr^{2}\}(1-p-q)(1-r)]\mu(LD)-r(1-p-q)^{2}(1-r)^{2}\nonumber\\&\mu(LDW)-(1-r)^{2}(1-p-q)^{2}\mu(LDL)-[1-(1-p-q^{2}-r^{2}-2qr+pq+pr+2pqr\nonumber\\&+3pr^{2}+2qr^{2}+2q^{2}r-p^{2}r-3pqr^{2}-q^{2}r^{2}-2p^{2}r^{2})(1-p-q)(1-r)]\mu(LWD)\nonumber\\&-r^{2}(1-r)^{2}(1-p-q)^{2}\mu(WLD)-r^{2}(1-r)^{2}(1-p-q)^{2}\mu(DLD)+\sum_{i=4}^{9}A_{i}\nonumber\\
\leqslant{}&w_{0}(\mu)-2\{1-(1+p)(1-p-q)(1-r)\}\mu(WD)-\{1-(1+p+pq+pr-p^{2}r-pqr)\nonumber\\&(1-p-q)(1-r)(1+r)\}\mu(DD)-[1-\{1-p-q+4r-qr(2-p-2r)-r^{2}(2-3p)\nonumber\\&-r^{3}(1-p-q)-p^{2}r^{2}-pqr^{2}\}(1-p-q)(1-r)]\mu(LD)-r(1-p-q)^{2}(1-r)^{2}\nonumber\\&\mu(LDW)\underbrace{-(1-r)^{2}(1-p-q)^{2}\mu(LDL)}-[1-(1-p-q^{2}-r^{2}-2qr+pq+pr+2pqr\nonumber\\&+3pr^{2}+2qr^{2}+2q^{2}r-p^{2}r-3pqr^{2}-q^{2}r^{2}-2p^{2}r^{2})(1-p-q)(1-r)]\mu(LWD)\nonumber\\&-r^{2}(1-r)^{2}(1-p-q)^{2}\mu(WLD)-r^{2}(1-r)^{2}(1-p-q)^{2}\mu(DLD)\nonumber\\&\underbrace{+\{q+(1-p-q)r\}\{p+(1-p-q)(1-r)\}(1-p-q)r(1-r)\mu(LDL)}+\sum_{i=5}^{9}A_{i},\label{w_{0}_ineq_4}
\end{align}
where the last step follows from the inequality in \eqref{bound_A_{4}}. Combining the terms involving $\mu(LDL)$ in \eqref{w_{0}_ineq_4} (indicated by underbraces), we obtain:
\begin{align}
{}&-(1-r)^{2}(1-p-q)^{2}\mu(LDL)+\{q+(1-p-q)r\}\{p+(1-p-q)(1-r)\}(1-p-q)r(1-r)\mu(LDL)\nonumber\\
{}&=-(1-p-q)(1-r)[(1-r)(1-p-q)-\{q+(1-p-q)r\}\{p+(1-p-q)(1-r)\}r]\mu(LDL)\nonumber\\
{}&=-(1-p-q)(1-r)[(1-r)(1-p-q)\{1-qr-r^{2}(1-p-q)\}-pqr-pr^{2}(1-p-q)]\mu(LDL),\label{intermediate_10}
\end{align}
which is non-positive when $p$, $q$ and $r$ are sufficiently small. Incorporating \eqref{intermediate_10} into \eqref{w_{0}_ineq_4}, we obtain:
\begin{align}
w_{0}(\G_{p,q,r}\mu)\leqslant{}&w_{0}(\mu)-2\{1-(1+p)(1-p-q)(1-r)\}\mu(WD)-\{1-(1+p+pq+pr-p^{2}r-pqr)\nonumber\\&(1-p-q)(1-r)(1+r)\}\mu(DD)-[1-\{1-p-q+4r-qr(2-p-2r)-r^{2}(2-3p)\nonumber\\&-r^{3}(1-p-q)-p^{2}r^{2}-pqr^{2}\}(1-p-q)(1-r)]\mu(LD)-r(1-p-q)^{2}(1-r)^{2}\nonumber\\&\mu(LDW)-(1-p-q)(1-r)[(1-r)(1-p-q)\{1-qr-r^{2}(1-p-q)\}-pqr\nonumber\\&-pr^{2}(1-p-q)]\mu(LDL)-[1-(1-p-q^{2}-r^{2}-2qr+pq+pr+2pqr\nonumber\\&+3pr^{2}+2qr^{2}+2q^{2}r-p^{2}r-3pqr^{2}-q^{2}r^{2}-2p^{2}r^{2})(1-p-q)(1-r)]\mu(LWD)\nonumber\\&-r^{2}(1-r)^{2}(1-p-q)^{2}\mu(WLD)-r^{2}(1-r)^{2}(1-p-q)^{2}\mu(DLD)+\sum_{i=5}^{9}A_{i}\nonumber\\
\leqslant{}&w_{0}(\mu)-2\{1-(1+p)(1-p-q)(1-r)\}\mu(WD)-\{1-(1+p+pq+pr-p^{2}r-pqr)\nonumber\\&(1-p-q)(1-r)(1+r)\}\mu(DD)-[1-\{1-p-q+4r-qr(2-p-2r)-r^{2}(2-3p)\nonumber\\&-r^{3}(1-p-q)-p^{2}r^{2}-pqr^{2}\}(1-p-q)(1-r)]\mu(LD)\underbrace{-r(1-p-q)^{2}(1-r)^{2}\mu(LDW)}\nonumber\\&-(1-p-q)(1-r)[(1-r)(1-p-q)\{1-qr-r^{2}(1-p-q)\}-pqr\nonumber\\&-pr^{2}(1-p-q)]\mu(LDL)-[1-(1-p-q^{2}-r^{2}-2qr+pq+pr+2pqr\nonumber\\&+3pr^{2}+2qr^{2}+2q^{2}r-p^{2}r-3pqr^{2}-q^{2}r^{2}-2p^{2}r^{2})(1-p-q)(1-r)]\mu(LWD)\nonumber\\&-r^{2}(1-r)^{2}(1-p-q)^{2}\mu(WLD)-r^{2}(1-r)^{2}(1-p-q)^{2}\mu(DLD)\nonumber\\&\underbrace{+(1-p)p(1-p-q)r(1-r)\mu(WDL)}+\sum_{i=6}^{9}A_{i},\label{w_{0}_ineq_5}
\end{align}
where the last step follows from \eqref{bound_A_{5}}. Combining the terms involving $\mu(WDL)$ and $\mu(LDW)$ in \eqref{w_{0}_ineq_5} (indicated by underbraces) and using reflection invariance, we obtain:
\begin{align}
{}&-r(1-p-q)^{2}(1-r)^{2}\mu(LDW)+(1-p)p(1-p-q)r(1-r)\mu(WDL)\nonumber\\
={}&-r(1-p-q)(1-r)\{(1-p-q)(1-r)-p(1-p)\}\mu(WDL),\label{intermediate_11}
\end{align}
and the coefficient of $\mu(WDL)$ in \eqref{intermediate_11} is assuredly non-positive when $p$, $q$ and $r$ are sufficiently small. Incorporating \eqref{intermediate_11} into \eqref{w_{0}_ineq_5}, we obtain:
\begin{align}
w_{0}(\G_{p,q,r}\mu)\leqslant{}&w_{0}(\mu)-2\{1-(1+p)(1-p-q)(1-r)\}\mu(WD)-\{1-(1+p+pq+pr-p^{2}r-pqr)\nonumber\\&(1-p-q)(1-r)(1+r)\}\mu(DD)-[1-\{1-p-q+4r-qr(2-p-2r)-r^{2}(2-3p)\nonumber\\&-r^{3}(1-p-q)-p^{2}r^{2}-pqr^{2}\}(1-p-q)(1-r)]\mu(LD)-r(1-p-q)(1-r)\nonumber\\&\{(1-p-q)(1-r)-p(1-p)\}\mu(LDW)-(1-p-q)(1-r)[(1-r)(1-p-q)\nonumber\\&\{1-qr-r^{2}(1-p-q)\}-pqr-pr^{2}(1-p-q)]\mu(LDL)-[1-(1-p-q^{2}-r^{2}-2qr\nonumber\\&+pq+pr+2pqr+3pr^{2}+2qr^{2}+2q^{2}r-p^{2}r-3pqr^{2}-q^{2}r^{2}-2p^{2}r^{2})(1-p-q)(1-r)]\nonumber\\&\mu(LWD)-r^{2}(1-r)^{2}(1-p-q)^{2}\mu(WLD)-r^{2}(1-r)^{2}(1-p-q)^{2}\mu(DLD)+\sum_{i=6}^{9}A_{i}\nonumber\\
\leqslant{}&w_{0}(\mu)\underbrace{-2\{1-(1+p)(1-p-q)(1-r)\}\mu(WD)}-\{1-(1+p+pq+pr-p^{2}r-pqr)\nonumber\\&(1-p-q)(1-r)(1+r)\}\mu(DD)-[1-\{1-p-q+4r-qr(2-p-2r)-r^{2}(2-3p)\nonumber\\&-r^{3}(1-p-q)-p^{2}r^{2}-pqr^{2}\}(1-p-q)(1-r)]\mu(LD)-r(1-p-q)(1-r)\nonumber\\&\{(1-p-q)(1-r)-p(1-p)\}\mu(LDW)-(1-p-q)(1-r)[(1-r)(1-p-q)\nonumber\\&\{1-qr-r^{2}(1-p-q)\}-pqr-pr^{2}(1-p-q)]\mu(LDL)-[1-(1-p-q^{2}-r^{2}-2qr\nonumber\\&+pq+pr+2pqr+3pr^{2}+2qr^{2}+2q^{2}r-p^{2}r-3pqr^{2}-q^{2}r^{2}-2p^{2}r^{2})(1-p-q)(1-r)]\nonumber\\&\mu(LWD)-r^{2}(1-r)^{2}(1-p-q)^{2}\mu(WLD)-r^{2}(1-r)^{2}(1-p-q)^{2}\mu(DLD)\nonumber\\&\underbrace{+\{q+(1-p-q)r\}\{p+(1-p-q)(1-r)\}(1-p-q)(1-r)\mu(DW)}\nonumber\\&+\{p(1-p)-(q+r-pr-qr)(1-q-r+pr+qr)\}(1-p-q)(1-r)\mu(WWDW)\nonumber\\&+p(1-p-q)^{2}(1-r)^{2}\mu(WWWD)\underbrace{+\{q+(1-p-q)r\}p(1-p-q)(1-r)\mu(WD)}+\sum_{i=8}^{9}A_{i},\label{w_{0}_ineq_6}
\end{align}
where the last step follows from the inequalities in \eqref{bound_A_{6}} and \eqref{bound_A_{7}}. From \eqref{inverse_2}, we have 
\begin{align}
{}&2\{(1-p-q)(1-r)\}^{-1}\geqslant 2+2p+2q+2r+2p^{2}+2q^{2}+2r^{2}+4pq+2pr+2qr-8pqr-4p^{2}r\nonumber\\&-4pr^{2}-4q^{2}r-4qr^{2}+2p^{2}r^{2}+2q^{2}r^{2}+4pqr^{2}\nonumber\\
={}&\{2+2p+q+r-q^{2}-r^{2}-3qr+pq+2q^{2}r+2qr^{2}+pqr+2pr^{2}-p^{2}r-2pqr^{2}-p^{2}r^{2}-q^{2}r^{2}\}\nonumber\\&+q+r+2p^{2}+3q^{2}+3r^{2}+3pq(1-3r)+pr(2-3p-6r)+qr(5-6q-6r)+3p^{2}r^{2}\nonumber\\&+3q^{2}r^{2}+6pqr^{2}\nonumber\\
\geqslant{}&\{2+2p+q+r-q^{2}-r^{2}-3qr+pq+2q^{2}r+2qr^{2}+pqr+2pr^{2}-p^{2}r-2pqr^{2}-p^{2}r^{2}-q^{2}r^{2}\}\label{intermediate_12}
\end{align}
whenever $p$, $q$ and $r$ are sufficiently small. Combining the terms involving $\mu(WD)$ and $\mu(DW)$ in \eqref{w_{0}_ineq_6} (indicated by underbraces), and making use of the reflection-invariance of $\mu$, we obtain:
\begin{align}
{}&-2\{1-(1+p)(1-p-q)(1-r)\}\mu(WD)+\{q+(1-p-q)r\}\{p+(1-p-q)(1-r)\}\nonumber\\&(1-p-q)(1-r)\mu(DW)+\{q+(1-p-q)r\}p(1-p-q)(1-r)\mu(WD)\nonumber\\
={}&[-2+(1-p-q)(1-r)\{2(1+p)+(q+r-pr-qr)(1-q-r+pr+qr)\nonumber\\&+(q+r-pr-qr)p\}]\mu(WD)\nonumber\\
={}&[-2+(1-p-q)(1-r)\{2+2p+q+r-q^{2}-r^{2}-3qr+pq+2q^{2}r+2qr^{2}\nonumber\\&+pqr+2pr^{2}-p^{2}r-2pqr^{2}-p^{2}r^{2}-q^{2}r^{2}\}]\mu(WD),\label{intermediate_13}
\end{align}
and by the inequality derived in \eqref{intermediate_12}, it is immediate that the coefficient of $\mu(WD)$ in \eqref{intermediate_13} is non-positive. Incorporating \eqref{intermediate_13} into \eqref{w_{0}_ineq_6}, we obtain:
\begin{align}
w_{0}(\G_{p,q,r}\mu)\leqslant{}&w_{0}(\mu)-[2-(1-p-q)(1-r)\{2+2p+q+r-q^{2}-r^{2}-3qr+pq+2q^{2}r+2qr^{2}\nonumber\\&+pqr+2pr^{2}-p^{2}r-2pqr^{2}-p^{2}r^{2}-q^{2}r^{2}\}]\mu(WD)-\{1-(1+p+pq+pr-p^{2}r-pqr)\nonumber\\&(1-p-q)(1-r)(1+r)\}\mu(DD)-[1-\{1-p-q+4r-qr(2-p-2r)-r^{2}(2-3p)\nonumber\\&-r^{3}(1-p-q)-p^{2}r^{2}-pqr^{2}\}(1-p-q)(1-r)]\mu(LD)-r(1-p-q)(1-r)\nonumber\\&\{(1-p-q)(1-r)-p(1-p)\}\mu(LDW)-(1-p-q)(1-r)[(1-r)(1-p-q)\nonumber\\&\{1-qr-r^{2}(1-p-q)\}-pqr-pr^{2}(1-p-q)]\mu(LDL)-[1-(1-p-q^{2}-r^{2}-2qr\nonumber\\&+pq+pr+2pqr+3pr^{2}+2qr^{2}+2q^{2}r-p^{2}r-3pqr^{2}-q^{2}r^{2}-2p^{2}r^{2})(1-p-q)(1-r)]\nonumber\\&\mu(LWD)-r^{2}(1-r)^{2}(1-p-q)^{2}\mu(WLD)-r^{2}(1-r)^{2}(1-p-q)^{2}\mu(DLD)\nonumber\\&+\{p(1-p)-(q+r-pr-qr)(1-q-r+pr+qr)\}(1-p-q)(1-r)\mu(WWDW)\nonumber\\&+p(1-p-q)^{2}(1-r)^{2}\mu(WWWD)+\sum_{i=8}^{9}A_{i}\nonumber\\
\leqslant{}&w_{0}(\mu)-[2-(1-p-q)(1-r)\{2+2p+q+r-q^{2}-r^{2}-3qr+pq+2q^{2}r+2qr^{2}\nonumber\\&+pqr+2pr^{2}-p^{2}r-2pqr^{2}-p^{2}r^{2}-q^{2}r^{2}\}]\mu(WD)-\{1-(1+p+pq+pr-p^{2}r-pqr)\nonumber\\&(1-p-q)(1-r)(1+r)\}\mu(DD)-[1-\{1-p-q+4r-qr(2-p-2r)-r^{2}(2-3p)\nonumber\\&-r^{3}(1-p-q)-p^{2}r^{2}-pqr^{2}\}(1-p-q)(1-r)]\mu(LD)-r(1-p-q)(1-r)\nonumber\\&\{(1-p-q)(1-r)-p(1-p)\}\mu(LDW)-(1-p-q)(1-r)[(1-r)(1-p-q)\nonumber\\&\{1-qr-r^{2}(1-p-q)\}-pqr-pr^{2}(1-p-q)]\mu(LDL)-[1-(1-p-q^{2}-r^{2}-2qr\nonumber\\&+pq+pr+2pqr+3pr^{2}+2qr^{2}+2q^{2}r-p^{2}r-3pqr^{2}-q^{2}r^{2}-2p^{2}r^{2})(1-p-q)(1-r)]\nonumber\\&\mu(LWD)-r^{2}(1-r)^{2}(1-p-q)^{2}\mu(WLD)-r^{2}(1-r)^{2}(1-p-q)^{2}\mu(DLD)\nonumber\\&+\{p(1-p)-(q+r-pr-qr)(1-q-r+pr+qr)\}(1-p-q)(1-r)\mu(WWDW)\nonumber\\&+p(1-p-q)^{2}(1-r)^{2}\mu(WWWD)+(q+r-pr-qr)(1-q-r+pr+qr)(1-p-q)\nonumber\\&(1-r)(1+r)\mu(LDD)+q(1-q-r+pr+qr)(1-p-q)(1-r)(1+r)\mu(WDD)\nonumber\\&+\{p(1-p)-q(1-q-r+pr+qr)\}(1-p-q)(1-r)(1+r)\mu(WWDD)\nonumber\\&+(pr+qr-q)(1-p-q)^{2}(1-r)(1+r)\mu(DWDD),\label{w_{0}_ineq_7}
\end{align}
where the last step follows from \eqref{bound_A_{8}} and \eqref{bound_A_{9}}.

The idea, now, is to split the term involving $\mu(LDD)$ in \eqref{w_{0}_ineq_7} into two suitable parts, one of which is to be negated by the term involving $\mu(DD)$ in \eqref{w_{0}_ineq_7}, and the other is to be negated by the term involving $\mu(LD)$ in \eqref{w_{0}_ineq_7}. We decompose the term involving $\mu(LDD)$ in \eqref{w_{0}_ineq_7} into the following two parts:
\begin{align}
{}&(q+r-pr-qr)(1-q-r+pr+qr)(1-p-q)(1-r)(1+r)\mu(LDD)\nonumber\\
={}&q(1-q-r+pr+qr)(1-p-q)(1-r)(1+r)\mu(LDD)\nonumber\\&+r(1-q-r+pr+qr)(1-p-q)^{2}(1-r)(1+r)\mu(LDD).\label{decomp_LDD}
\end{align}
At this point, we observe, from \eqref{inverse_1}, that:
\begin{align}
{}&\{(1-p-q)(1-r)(1+r)\}^{-1}\geqslant 1+p+q+r^{2}+p^{2}+q^{2}+2pq+p^{2}r^{4}-2p^{2}r^{2}+2pqr^{4}\nonumber\\&-4pqr^{2}-2pr^{4}+pr^{2}+q^{2}r^{4}-2q^{2}r^{2}-2qr^{4}+qr^{2}+r^{4}\nonumber\\
={}&(1+p+q+pq+pr-q^{2}-qr-p^{2}r+q^{2}r)+(p-r)^{2}+2q^{2}+pq(1-4r^{2}+2r^{4})\nonumber\\&+qr(1-q+r-2qr-2r^{3}+qr^{3})+pr(1+p+r-2pr-2r^{3}+pr^{3})+r^{4}\nonumber\\
\geqslant{}&(1+p+q+pq+pr-q^{2}-qr-p^{2}r+q^{2}r) \text{ when } p,q,r \text{ are sufficiently small}.\label{intermediate_14}
\end{align}
We now combine the first of the two terms in \eqref{decomp_LDD} with the terms involving $\mu(WDD)$ and $\mu(DD)$ in \eqref{w_{0}_ineq_7}, and use the inequality $\mu(WDD)+\mu(LDD)\leqslant \mu(DD)$ (follows from Lemma~\ref{lem:pushforward_1}), to obtain:
\begin{align}
{}&-\{1-(1+p+pq+pr-p^{2}r-pqr)(1-p-q)(1-r)(1+r)\}\mu(DD)\nonumber\\&+q(1-q-r+pr+qr)(1-p-q)(1-r)(1+r)\mu(LDD)\nonumber\\&+q(1-q-r+pr+qr)(1-p-q)(1-r)(1+r)\mu(WDD)\nonumber\\
{}&\leqslant[-1+(1+p+q+pq+pr-q^{2}-qr-p^{2}r+q^{2}r)(1-p-q)(1-r)(1+r)]\mu(DD).\label{intermediate_15}
\end{align}  
By \eqref{intermediate_14}, the coefficient of $\mu(DD)$ in \eqref{intermediate_15} is non-positive. Next, we combine the second term appearing in \eqref{decomp_LDD} with the term involving $\mu(LD)$ in \eqref{w_{0}_ineq_7}, and use Lemma~\ref{lem:pushforward_1}, to obtain:
\begin{align}
{}&-[1-\{1-p-q+4r-qr(2-p-2r)-r^{2}(2-3p)-r^{3}(1-p-q)-p^{2}r^{2}-pqr^{2}\}(1-p-q)(1-r)]\nonumber\\&\mu(LD)+r(1-q-r+pr+qr)(1-p-q)^{2}(1-r)(1+r)\mu(LDD)\nonumber\\
={}&[-1+\{1-p-q+4r-2qr+pqr+2qr^{2}-2r^{2}+3pr^{2}-r^{3}(1-p-q)-p^{2}r^{2}-pqr^{2}\}(1-p-q)(1-r)]\nonumber\\&\mu(LD)+r(1-q-r+pr+qr)(1-p-q)^{2}(1-r)(1+r)[\mu(LD)-\mu(LDL)-\mu(LDW)]\nonumber\\
\leqslant{}&[-1+\{1-p-q+4r-2qr+pqr+2qr^{2}-2r^{2}+3pr^{2}+r(1-q-r+pr+qr)(1-p-q)(1+r)\}\nonumber\\&(1-p-q)(1-r)]\mu(LD)-r(1-q-r+pr+qr)(1-p-q)^{2}(1-r)(1+r)[\mu(LDL)+\mu(LDW)] \nonumber\\
{}& (\text{when } p, q, r \text{ are sufficiently small})\nonumber\\
={}&[-1+(1-p-q+5r-2r^{2}-pr-4qr-r^{3}+2qr^{2}+4pr^{2}+2pqr+q^{2}r-p^{2}r^{3}-p^{2}r^{2}-2pqr^{3}-pqr^{2}\nonumber\\&+2pr^{3}-q^{2}r^{3}+2qr^{3})(1-p-q)(1-r)]\mu(LD)-r(1-q-r+pr+qr)(1-p-q)^{2}(1-r)(1+r)\nonumber\\&[\mu(LDL)+\mu(LDW)]\nonumber\\
\leqslant{}&[-1+\{1-p-q+5r-2r^{2}-pr-4qr+2qr^{2}+4pr^{2}+2pqr+q^{2}r-r^{3}(1-2p-2q)\}(1-p-q)\nonumber\\&(1-r)]\mu(LD)-r(1-q-r+pr+qr)(1-p-q)^{2}(1-r)(1+r)[\mu(LDL)+\mu(LDW)]\nonumber\\
\leqslant{}&[4r-2p-2q+p^{2}+2pq+q^{2}-7r^{2}-4pr-7qr+5pqr+12pr^{2}+4q^{2}r+13qr^{2}+2r^{3}-2p^{2}qr(1-r)\nonumber\\&-p^{2}r^{2}(5-4r)-3pq^{2}r(1-r)-pqr^{2}(13-6r)-6pr^{3}-q^{3}r(1-r)-q^{2}r^{2}(7-2r)-4qr^{3}]\mu(LD)\nonumber\\&-r(1-q-r+pr+qr)(1-p-q)^{2}(1-r)(1+r)[\mu(LDL)+\mu(LDW)]\nonumber\\
{}& (\text{when } p, q \text{ are sufficiently small -- in particular, when } p+q \leqslant 1/2)\nonumber\\
\leqslant{}&[4r-2(p+q)+(p+q)^{2}-6r^{2}-3pr-6qr-r^{2}(1-2r-6p-4q)-pr(1-5q-6r)-qr(1-4q-9r)]\nonumber\\&\mu(LD)-r(1-q-r+pr+qr)(1-p-q)^{2}(1-r)(1+r)[\mu(LDL)+\mu(LDW)]\nonumber\\
{}& (\text{when } p, q, r \text{ are sufficiently small})\nonumber\\
\leqslant{}&[4r-2(p+q)+(p+q)^{2}-6r^{2}-3pr-6qr]\mu(LD)-r(1-q-r+pr+qr)(1-p-q)^{2}(1-r)(1+r)\nonumber\\
{}&[\mu(LDL)+\mu(LDW)] \quad (\text{when } p, q, r \text{ are sufficiently small}),\label{intermediate_16}
\end{align}
and the coefficient of $\mu(LD)$ in \eqref{intermediate_16} is non-positive whenever \eqref{three_cond_universal} holds. Combining the terms involving $\mu(LDL)$ and $\mu(LDW)$ of \eqref{intermediate_16} with those of \eqref{w_{0}_ineq_7}, we obtain:
\begin{align}
{}&-r(1-p-q)(1-r)\{(1-p-q)(1-r)-p(1-p)\}\mu(LDW)-(1-p-q)(1-r)[(1-r)(1-p-q)\nonumber\\&\{1-qr-r^{2}(1-p-q)\}-pqr-pr^{2}(1-p-q)]\mu(LDL)-r(1-q-r+pr+qr)(1-p-q)^{2}(1-r)\nonumber\\&(1+r)[\mu(LDL)+\mu(LDW)]\nonumber\\
={}&(1-p-q)(1-r)\{-2r+r^{2}+3pr+3qr+r^{3}-p^{2}r(1-r-r^{2})-pqr(1-r-2r^{2})-2pr^{3}-2pr^{2}\nonumber\\&-q^{2}r(1-r^{2})-2qr^{3}-qr^{2}\}\mu(LDW)+(1-p-q)(1-r)\{-1+p+q+r^{2}+2qr-2q^{2}r(1-r)\nonumber\\&-3qr^{2}(1-p)-pr^{2}(2-p)-pqr\}\mu(LDL)\nonumber\\
\leqslant{}&(1-p-q)(1-r)[(-2r+r^{2}+3pr+3qr+r^{3})\mu(LDW)+(-1+p+q+r^{2}+2qr)\mu(LDL)],\label{intermediate_16'}
\end{align}
for all $p$, $q$ and $r$ sufficiently small. The coefficient of each of $\mu(LDW)$ and $\mu(LDL)$ in \eqref{intermediate_16'} is non-positive when $p$, $q$ and $r$ are sufficiently small. Incorporating \eqref{intermediate_15}, \eqref{intermediate_16} and \eqref{intermediate_16'} into \eqref{w_{0}_ineq_7}, we obtain:
\begin{align}
w_{0}(\G_{p,q,r}\mu)\leqslant{}&w_{0}(\mu)-[2-(1-p-q)(1-r)\{2+2p+q+r-q^{2}-r^{2}-3qr+pq+2q^{2}r+2qr^{2}\nonumber\\&+pqr+2pr^{2}-p^{2}r-2pqr^{2}-p^{2}r^{2}-q^{2}r^{2}\}]\mu(WD)-\{1-(1+p+q+pq+pr\nonumber\\&-q^{2}-qr-p^{2}r+q^{2}r)(1-p-q)(1-r)(1+r)\}\mu(DD)+\{4r-2(p+q)+(p+q)^{2}\nonumber\\&-6r^{2}-3r(p+2q)\}\mu(LD)+(1-p-q)(1-r)(-2r+r^{2}+3pr+3qr+r^{3})\mu(LDW)\nonumber\\&+(1-p-q)(1-r)(-1+p+q+r^{2}+2qr)\mu(LDL)-[1-(1-p-q^{2}-r^{2}-2qr\nonumber\\&+pq+pr+2pqr+3pr^{2}+2qr^{2}+2q^{2}r-p^{2}r-3pqr^{2}-q^{2}r^{2}-2p^{2}r^{2})(1-p-q)(1-r)]\nonumber\\&\mu(LWD)-r^{2}(1-r)^{2}(1-p-q)^{2}\mu(WLD)-r^{2}(1-r)^{2}(1-p-q)^{2}\mu(DLD)\nonumber\\&+\{p(1-p)-(q+r-pr-qr)(1-q-r+pr+qr)\}(1-p-q)(1-r)\mu(WWDW)\nonumber\\&+p(1-p-q)^{2}(1-r)^{2}\mu(WWWD)+\{p(1-p)-q(1-q-r+pr+qr)\}(1-p-q)\nonumber\\&(1-r)(1+r)\mu(WWDD)+(pr+qr-q)(1-p-q)^{2}(1-r)(1+r)\mu(DWDD).\label{w_{0}_ineq_8}
\end{align}

We recall, for the reader's convenience, that the ultimate goal for us is to come up with a weight function inequality that satisfies \eqref{desired_criterion}, and although the weight function inequality in \eqref{w_{0}_ineq_8} is of the same form as \eqref{gen_ineq_form}, it does not satisfy \eqref{desired_criterion} because of the last few terms. In particular, we observe that:
\begin{enumerate}
\item the coefficient $\mu(WWWD)$ in \eqref{w_{0}_ineq_8} is non-negative;
\item the coefficient of $\mu(DWDD)$ in \eqref{w_{0}_ineq_8} may or may not be non-negative;  
\item \label{for_f(x)} the coefficient of $\mu(WWDW)$ in \eqref{w_{0}_ineq_8} is positive whenever $p(1-p) > (q+r-pr-qr)(1-q-r+pr+qr)$, which happens when \eqref{three_cond_4} is true;
\item \label{another_for_f(x)} and the coefficient of $\mu(WWDD)$ in \eqref{w_{0}_ineq_8} is positive whenever $p(1-p) > q(1-q-r+pr+qr)$, which happens when one of \eqref{three_cond_2}, \eqref{three_cond_3} and \eqref{three_cond_4} is true. 
\end{enumerate}
Here, we provide a brief explanation as to why the first inequality of \eqref{three_cond_4} ensures, when $p$, $q$ and $r$ are sufficiently small, that $p(1-p) > (q+r-pr-qr)(1-q-r+pr+qr)$, and this, in turn, implies that $p(1-p) > q(1-q-r+pr+qr)$. The function $f(x)=x(1-x)$, for $x\in[0,1]$, is strictly increasing on $[0,1/2)$ and strictly decreasing on $(1/2,1]$. When $p$ is sufficiently small (in particular, $p<1/2$), and the first inequality of \eqref{three_cond_4} holds, we can, therefore, write 
\begin{equation}
\frac{1}{2}>p>q+r(1-p-q) \implies f(p)>f(q+r(1-p-q)) \Longleftrightarrow p(1-p)>(q+r-pr-qr)(1-q-r+pr+qr),\nonumber
\end{equation}
which explains the claim made in \eqref{for_f(x)}, and since $q+r(1-p-q)\geqslant q$ (since each of $r$ and $(1-p-q)$ is non-negative), the inequality above further implies that $p(1-p) > q(1-q-r+pr+qr)$, which explains the claim stated in \eqref{another_for_f(x)}.

At this point, our first attempt is to try to negate, to as large an extent as possible, the terms involving $\mu(WWWD)$ and $\mu(DWDD)$ in \eqref{w_{0}_ineq_8}, using the term involving $\mu(WD)$ in \eqref{w_{0}_ineq_8}. Here, we make use of Lemma~\ref{lem:pushforward_1} and reflection-invariance to write: 
\begin{align}
\mu(WD)=\mu(WWD)+\mu(DWD)+\mu(LWD)\geqslant\mu(WWWD)+\mu(DWDD)+\mu(LWD).\label{WD_split}
\end{align}
The inequality in \eqref{WD_split} is applied to the term involving $\mu(WD)$ in \eqref{w_{0}_ineq_8}, thereby obtaining three new terms: one involving $\mu(WWWD)$, one involving $\mu(DWDD)$, and the third involving $\mu(LWD)$. The second of these is combined with the term involving $\mu(DWDD)$ that is \emph{already present} in \eqref{w_{0}_ineq_8}, to get:
\begin{align}
{}&-[2-(1-p-q)(1-r)\{2+2p+q+r-q^{2}-r^{2}-3qr+pq+2q^{2}r+2qr^{2}+pqr+2pr^{2}-p^{2}r\nonumber\\&-2pqr^{2}-p^{2}r^{2}-q^{2}r^{2}\}]\mu(DWDD)+(pr+qr-q)(1-p-q)^{2}(1-r)(1+r)\mu(DWDD)\nonumber\\
={}&[-2+\{2+2p+r-r^{2}-3qr+2pq+pr+3pr^{2}-2p^{2}r+2q^{2}r+3qr^{2}-2p^{2}r^{2}-2q^{2}r^{2}-4pqr^{2}\}\nonumber\\&(1-p-q)(1-r)]\mu(DWDD).\label{DWDD_left_1}
\end{align}
In order to draw conclusions about whether the coefficient of $\mu(DWDD)$ in \eqref{DWDD_left_1} is non-positive or not, we observe, from \eqref{inverse_2}, that
\begin{align}
{}&2\{(1-p-q)(1-r)\}^{-1} \geqslant 2+2p+2q+2r+2p^{2}+2q^{2}+2r^{2}+4pq+2pr+2qr-8pqr-4p^{2}r-4pr^{2}\nonumber\\&-4q^{2}r-4qr^{2}+2p^{2}r^{2}+2q^{2}r^{2}+4pqr^{2}\nonumber\\
={}&(2+2p+r-r^{2}-3qr+2pq+pr+3pr^{2}-2p^{2}r+2q^{2}r+3qr^{2}-2p^{2}r^{2}-2q^{2}r^{2}-4pqr^{2})\nonumber\\&+2q+r+2p^{2}+2q^{2}+3r^{2}+5qr+2pq+pr-7pr^{2}-2p^{2}r-6q^{2}r-7qr^{2}-8pqr+4p^{2}r^{2}+4q^{2}r^{2}\nonumber\\&+8pqr^{2}\nonumber\\
={}&(2+2p+r-r^{2}-3qr+2pq+pr+3pr^{2}-2p^{2}r+2q^{2}r+3qr^{2}-2p^{2}r^{2}-2q^{2}r^{2}-4pqr^{2})+\nonumber\\&2q+r+2p^{2}+2q^{2}+3r^{2}+qr(5-6q-7r+3qr)+2pq(1-4r+4r^{2})+pr(1-7r-2p+4pr)\nonumber\\
\geqslant{}&2+2p+r-r^{2}-3qr+2pq+pr+3pr^{2}-2p^{2}r+2q^{2}r+3qr^{2}-2p^{2}r^{2}-2q^{2}r^{2}-4pqr^{2}\label{intermediate_17}
\end{align}
when $p$, $q$ and $r$ are sufficiently small. It now becomes evident that the coefficient of $\mu(DWDD)$ in \eqref{DWDD_left_1} is non-positive because of the inequality established in \eqref{intermediate_17}. The new term involving $\mu(LWD)$, obtained by applying \eqref{WD_split} to the term involving $\mu(WD)$ in \eqref{w_{0}_ineq_8}, is now combined with the term involving $\mu(LWD)$ that is \emph{already present} in \eqref{w_{0}_ineq_8}, to get:
\begin{align}
{}&-[2-(1-p-q)(1-r)\{2+2p+q+r-q^{2}-r^{2}-3qr+pq+2q^{2}r+2qr^{2}+pqr+2pr^{2}-p^{2}r-2pqr^{2}\nonumber\\&-p^{2}r^{2}-q^{2}r^{2}\}]\mu(LWD)-[1-(1-p-q^{2}-r^{2}-2qr+pq+pr+2pqr+3pr^{2}+2qr^{2}+2q^{2}r-p^{2}r\nonumber\\&-3pqr^{2}-q^{2}r^{2}-2p^{2}r^{2})(1-p-q)(1-r)]\mu(LWD)\nonumber\\
={}&-[3-(3+p+q+r-2q^{2}-2r^{2}-5qr+2pq+pr+3pqr+4q^{2}r+4qr^{2}+5pr^{2}-2p^{2}r\nonumber\\&-5pqr^{2}-3p^{2}r^{2}-2q^{2}r^{2})(1-p-q)(1-r)]\mu(LWD).\label{LWD_left_1}
\end{align}
Recall that \eqref{intermediate_8_generalized} allowed us to conclude that the coefficient of $\mu(LWD)$ in \eqref{intermediate_9_generalized} is non-positive, and that the inequality in \eqref{intermediate_12} allowed us to conclude that the coefficient of $\mu(WD)$ in \eqref{intermediate_13} is non-positive. Combining these conclusions, it follows that the coefficient of $\mu(LWD)$ in \eqref{LWD_left_1} is non-positive as well. 

As mentioned above, the inequality in \eqref{WD_split} is applied to the term involving $\mu(WD)$ in \eqref{w_{0}_ineq_8}, thus giving rise to a new term involving $\mu(WWWD)$, which we now combine with the existing term involving $\mu(WWWD)$ that is already present in \eqref{w_{0}_ineq_8}, to get:
\begin{align}
{}&-[2-(1-p-q)(1-r)\{2+2p+q+r-q^{2}-r^{2}-3qr+pq+2q^{2}r+2qr^{2}+pqr+2pr^{2}-p^{2}r\nonumber\\&-2pqr^{2}-p^{2}r^{2}-q^{2}r^{2}\}]\mu(WWWD)+p(1-p-q)^{2}(1-r)^{2}\mu(WWWD)\nonumber\\
={}&[-2(1-p-q)(1-r)\{(1-p-q)(1-r)\}^{-1}+(1-p-q)(1-r)\{2+2p+q+r-q^{2}-r^{2}-3qr\nonumber\\&+pq+2q^{2}r+2qr^{2}+pqr+2pr^{2}-p^{2}r-2pqr^{2}-p^{2}r^{2}-q^{2}r^{2}\}+p(1-p-q)^{2}(1-r)^{2}]\mu(WWWD)\nonumber\\
\leqslant{}&(1-p-q)(1-r)[-2\{1+p+q+r+p^{2}+q^{2}+r^{2}+2pq+pr+qr-4pqr-2p^{2}r-2pr^{2}-2q^{2}r\nonumber\\&-2qr^{2}+p^{2}r^{2}+q^{2}r^{2}+2pqr^{2}\}+\{2+2p+q+r-q^{2}-r^{2}-3qr+pq+2q^{2}r+2qr^{2}+pqr+2pr^{2}\nonumber\\&-p^{2}r-2pqr^{2}-p^{2}r^{2}-q^{2}r^{2}\}+p(1-p-q)(1-r)]\mu(WWWD)\nonumber\\
{}&(\text{making use of the inequality in \eqref{inverse_2}})\nonumber\\
\leqslant{}&(1-p-q)(1-r)[p-q-r-3p^{2}-3q^{2}-3r^{2}-4pq-3pr-5qr+4p^{2}r+6pr^{2}+6q^{2}r+6qr^{2}\nonumber\\&+10pqr]\mu(WWWD).\label{WWWD_left_1}
\end{align}
The coefficient of $\mu(WWWD)$ in \eqref{WWWD_left_1} may or may not be non-positive.

Incorporating \eqref{DWDD_left_1}, \eqref{LWD_left_1} and \eqref{WWWD_left_1} into \eqref{w_{0}_ineq_8}, and eliminating the term involving $\mu(DWDD)$ since its coefficient in \eqref{DWDD_left_1} has already been proven above to be non-positive, via the inequality in \eqref{intermediate_17}, when $p$, $q$ and $r$ are sufficiently small, we obtain:
\begin{align}
w_{0}(\G_{p,q,r}\mu)\leqslant{}&w_{0}(\mu)-[2-(1-p-q)(1-r)\{2+2p+q+r-q^{2}-r^{2}-3qr+pq+2q^{2}r+2qr^{2}+pqr\nonumber\\&+2pr^{2}-p^{2}r-2pqr^{2}-p^{2}r^{2}-q^{2}r^{2}\}]\{\mu(WD)-\mu(WWWD)-\mu(DWDD)-\mu(LWD)\}\nonumber\\&-\{1-(1+p+q+pq+pr-q^{2}-qr-p^{2}r+q^{2}r)(1-p-q)(1-r)(1+r)\}\mu(DD)\nonumber\\&+\{4r-2(p+q)+(p+q)^{2}-6r^{2}-3r(p+2q)\}\mu(LD)+(1-p-q)(1-r)(-2r+r^{2}\nonumber\\&+3pr+3qr+r^{3})\mu(LDW)+(1-p-q)(1-r)(-1+p+q+r^{2}+2qr)\mu(LDL)-\nonumber\\&[3-(3+p+q+r-2q^{2}-2r^{2}-5qr+2pq+pr+3pqr+4q^{2}r+4qr^{2}+5pr^{2}-2p^{2}r\nonumber\\&-5pqr^{2}-3p^{2}r^{2}-2q^{2}r^{2})(1-p-q)(1-r)]\mu(LWD)-r^{2}(1-r)^{2}(1-p-q)^{2}\mu(WLD)\nonumber\\&-r^{2}(1-r)^{2}(1-p-q)^{2}\mu(DLD)+\{p(1-p)-(q+r-pr-qr)(1-q-r+pr+qr)\}\nonumber\\&(1-p-q)(1-r)\mu(WWDW)+(1-p-q)(1-r)\{p-q-r-3p^{2}-3q^{2}-3r^{2}-4pq\nonumber\\&-3pr-5qr+4p^{2}r+6pr^{2}+6q^{2}r+6qr^{2}+10pqr\}\mu(WWWD)+\{p(1-p)-q(1-q\nonumber\\&-r+pr+qr)\}(1-p-q)(1-r)(1+r)\mu(WWDD).\label{w_{0}_ineq_9}
\end{align}

\subsubsection{When \eqref{three_cond_1} of Theorem~\ref{thm:three-parameter} is true}\label{subsubsec:gen_regime_1} Here, we assume that $(p,q,r)$ satisfies \eqref{three_cond_universal} and \eqref{three_cond_1}. The latter allows us to deduce the following upper bound on the coefficient of $\mu(WWWD)$ in \eqref{WWWD_left_1} (which is also the coefficient of $\mu(WWWD)$ in \eqref{w_{0}_ineq_9}):
\begin{align}
{}&(1-p-q)(1-r)\{p-q-r-3p^{2}-3q^{2}-3r^{2}-4pq-3pr-5qr+4p^{2}r+6pr^{2}+6q^{2}r+6qr^{2}+10pqr\}\nonumber\\
={}& (1-p-q)(1-r)\{p(1-p)-q-r-2p^{2}-3q^{2}-3r^{2}-4pq-3pr-5qr+4p^{2}r+6pr^{2}+6q^{2}r+6qr^{2}\nonumber\\&+10pqr\}\nonumber\\
\leqslant{}&(1-p-q)(1-r)\{q-q^{2}-qr(1-p-q)-q-r-2p^{2}-3q^{2}-3r^{2}-4pq-3pr-5qr+4p^{2}r+6pr^{2}\nonumber\\& +6q^{2}r+6qr^{2}+10pqr\}\nonumber\\
={}&(1-p-q)(1-r)\{-r-2p^{2}-4q^{2}-3r^{2}-pq(4-11r)-pr(3-4p-6r)-qr(6-6r-7q)\}\nonumber
\end{align}
which is non-positive for all $p$, $q$ and $r$ sufficiently small. Thus, the coefficient of $\mu(WWWD)$ in \eqref{WWWD_left_1} is non-positive when $p$, $q$ and $r$ are sufficiently small and $(p,q,r)$ belongs to the regime described in \eqref{three_cond_1}. Since \eqref{three_cond_1} holds, the coefficient of $\mu(WWDD)$ in \eqref{w_{0}_ineq_9} is non-positive, and furthermore, we have $p(1-p) \leqslant q\{1-q-r(1-p-q)\} \leqslant \{q+r(1-p-q)\}\{1-q-r(1-p-q)\}$, proving that the coefficient of $\mu(WWDW)$ in \eqref{w_{0}_ineq_9} is non-positive as well. Consequently, \eqref{w_{0}_ineq_9} is of the form given by \eqref{gen_ineq_form} and satisfies \eqref{desired_criterion}, thus fulfilling the objective we set out to achieve when $(p,q,r)$ satisfies \eqref{three_cond_universal} and \eqref{three_cond_1}. Therefore, the construction of our desired weight function comes to an end at this step when \eqref{three_cond_1} holds.

Since the coefficient of each of $\mu(LD)$, $\mu(LDW)$, $\mu(LDL)$, $\mu(LWD)$, $\mu(WWDW)$, $\mu(WWDD)$ and $\mu(WWWD)$ in \eqref{w_{0}_ineq_9} is non-positive in this case, we can remove these terms from the right side of \eqref{w_{0}_ineq_9}, and the resulting inequality would be exactly what appears in \eqref{generalized_final_wt_fn_ineq} (once we replace the notation $w_{0}$ by $w$).

\subsubsection{When either \eqref{three_cond_2} or \eqref{three_cond_3} holds}\label{subsubsec:gen_regime_2,3}
The reason for considering these two possibilities together is that the way we proceed, to a large extent, is the same for both of them (in fact, much of this approach is also common to the scenario where $(p,q,r)$ satisfies \eqref{three_cond_4}). 

The very first observation to make is that the coefficient of $\mu(WWDW)$ in \eqref{w_{0}_ineq_9} is non-positive when either of \eqref{three_cond_2} and \eqref{three_cond_3} is true. Next, recall, from the discussion in \eqref{for_f(x)}, that the inequality $p(1-p) \leqslant \{q+r(1-p-q)\}\{1-q-r(1-p-q)\}$, which is a part of both \eqref{three_cond_2} and \eqref{three_cond_3}, becomes equivalent to $p \leqslant q+r(1-p-q)$ when $p \leqslant 1/2$ and $q+r(1-p-q) \leqslant 1/2$ (in other words, when $p$, $q$ and $r$ are sufficiently small). Applying the latter inequality, the coefficient of $\mu(WWWD)$ in \eqref{w_{0}_ineq_9} can be bounded above as follows:
\begin{align}
{}&(1-p-q)(1-r)\{p-q-r-3p^{2}-3q^{2}-3r^{2}-4pq-3pr-5qr+4p^{2}r+6pr^{2}+6q^{2}r+6qr^{2}+10pqr\}\nonumber\\
\leqslant{}&(1-p-q)(1-r)\{q+r(1-p-q)-q-r-3p^{2}-3q^{2}-3r^{2}-4pq-3pr-5qr+4p^{2}r+6pr^{2}+6q^{2}r\nonumber\\&+6qr^{2}+10pqr\}\nonumber\\
={}&(1-p-q)(1-r)\{-3p^{2}-3q^{2}-3r^{2}-4pq-4pr-6qr+4p^{2}r+6pr^{2}+6q^{2}r+6qr^{2}+10pqr\}\nonumber\\
={}&(1-p-q)(1-r)\{-3p^{2}-3q^{2}-3r^{2}-2pq(2-5r)-2pr(2-2p-3r)-6qr(1-q-r)\},\label{WWWD_nonpositive_regime_2}
\end{align}
which is non-positive for all $p$, $q$ and $r$ sufficiently small. This shows us that the coefficient of $\mu(WWWD)$ in \eqref{w_{0}_ineq_9} is non-positive in this case. Only the coefficient of $\mu(WWDD)$ in \eqref{w_{0}_ineq_9} remains non-negative.

This is where the need for an adjustment to our weight function, along the lines of the idea outlined in \S\ref{subsubsec:adjust_gen}, arises. There are quite a few terms with non-positive coefficients available to us on the right side of \eqref{w_{0}_ineq_9}, but we should take into account the following aspects while deciding which of them to use for negating the existing (possibly) non-negative terms on the right side of \eqref{w_{0}_ineq_9}:
\begin{enumerate}
\item which of the cylinder sets (namely, $(D,D)_{0,1}$, $(L,D)_{0,1}$, $(L,D,W)_{0,1,2}$ etc.)\ appearing on the right side of \eqref{w_{0}_ineq_9}, with non-positive coefficients, can boast a significant contribution from $\mu(WWD)$, $\mu(WWWD)$, $\mu(WWDW)$ and $\mu(WWDD)$ when we consider their probabilities under the pushforward measure (i.e.\ when we consider the expansion of $\G_{p,q,r}\mu(DD)$, $\G_{p,q,r}\mu(LD)$, $\G_{p,q,r}\mu(LDW)$ etc.); 
\item the order of magnitude of the coefficient of each of these terms assuming $p$, $q$ and $r$ to be sufficiently small (for instance, the coefficient of $\mu(LDW)$ in \eqref{w_{0}_ineq_9} is of the order of $r$, that of $\mu(LDL)$ is of the order of $1$, and so on).
\end{enumerate}
To this end, we make use of the following lemma:
\begin{lemma}\label{lem:pushforward_inequalities}
For any $\mu\in\mathcal{M}$, the following inequalities are true:
\begin{align} 
{}&(1-p)(1-p-q)(1-r)\mu(WWD)\leqslant \G_{p,q,r}\mu(LD),\label{i}\\
{}&(1-p)(1-p-q)(1-r)p\{\mu(WWDW)+\mu(WWDD)\}\leqslant \G_{p,q,r}\mu(LDW),\label{ii}\\
{}&(1-p)(1-p-q)(1-r)\{q+(1-p-q)r\}\mu(WWDW)+(1-p)(1-p-q)(1-r)\nonumber\\&\{q+(1-p-q)r^{2}\}\mu(WWDD)\leqslant \G_{p,q,r}\mu(LDL),\label{iii}\\
{}&(1-p)p(1-p-q)(1-r)\mu(WWDW)+(1-p)p(1-p-q)(1-r)(1+r)\mu(WWDD)\nonumber\\&+(1-p)p(1-p-q)(1-r)\mu(WWWD)\leqslant \G_{p,q,r}\mu(LWD).\label{iv}
\end{align}
\end{lemma}
\begin{proof}
The proof of Lemma~\ref{lem:pushforward_inequalities} follows from Lemma~\ref{lem:pushforward_2} and the stochastic update rules given by \eqref{GPCA_rule_1} through \eqref{GPCA_rule_6}. The reason why we obtain inequalities instead of equalities is because, when computing $\G_{p,q,r}\mu(\mathcal{C})$ where $\mathcal{C}$ is any of the cylinder sets $(L,D)_{0,1}$, $(L,D,W)_{0,1,2}$, $(L,D,L)_{0,1,2}$ and $(L,W,D)_{0,1,2}$, we do not consider \emph{all} possible cylinder sets $\mathcal{D}$ for which $\Prob[\mathcal{C}\big|\mathcal{D}]>0$. The details of the proof are as follows:
\begin{enumerate}
\item For the lower bound on $\G_{p,q,r}\mu(LD)$, we focus on the cylinder set $\mathcal{C}=(L,D)_{0,1}$, and we consider $\mathcal{D}=(W,W,D)_{0,1,2}$. Noting that $\Prob[\mathcal{C}\big|\mathcal{D}]$ equals $(1-p)(1-p-q)(1-r)$, we now obtain \eqref{i} as a consequence of Lemma~\ref{lem:pushforward_2}.
\item For the lower bound on $\G_{p,q,r}\mu(LDW)$, we focus on the cylinder set $\mathcal{C}=(L,D,W)_{0,1,2}$, and we consider $\mathcal{D}\in\{W\}^{2}\times\{D\}\times\{W,D\}$, and noting that $\Prob[\mathcal{C}\big|\mathcal{D}]$ equals $(1-p)(1-p-q)(1-r)p$ for each $\mathcal{D}$ here, we obtain \eqref{ii} as a consequence of Lemma~\ref{lem:pushforward_2}.
\item For the lower bound on $\G_{p,q,r}\mu(LDL)$, we focus on the cylinder set $\mathcal{C}=(L,D,L)_{0,1,2}$, and we consider $\mathcal{D}\in\{W\}^{2}\times\{D\}\times\{W,D\}$. Noting that $\Prob[\mathcal{C}\big|\mathcal{D}]$ equals $(1-p)(1-p-q)(1-r)\{q+(1-p-q)r\}$ when $\mathcal{D}=(W,W,D,W)_{0,1,2,3}$, whereas it equals $(1-p)(1-p-q)(1-r)\{q+(1-p-q)r^{2}\}$ when $\mathcal{D}=(W,W,D,D)_{0,1,2,3}$, we obtain \eqref{iii} as a consequence of Lemma~\ref{lem:pushforward_2}.
\item Finally, for the lower bound on $\G_{p,q,r}\mu(LWD)$, we focus on the cylinder set $\mathcal{C}=(L,W,D)_{0,1,2}$, and we consider $\mathcal{D}\in\left(\{W\}^{2}\times\{D\}\times\{W,D\}\right)\cup\left(\{W\}^{3}\times\{D\}\right)$. Noting that $\Prob[\mathcal{C}\big|\mathcal{D}]$ equals $(1-p)p(1-p-q)(1-r)$ when $\mathcal{D}\in\left\{(W,W,D,W)_{0,1,2,3},(W,W,W,D)_{0,1,2,3}\right\}$, and it equals $(1-p)p(1-p-q)(1-r)(1+r)$ when $\mathcal{D}=(W,W,D,D)_{0,1,2,3}$, we obtain \eqref{iv} as a consequence of Lemma~\ref{lem:pushforward_2}.\qedhere
\end{enumerate}
\end{proof}

We now define the adjusted / updated weight function, $w_{1}$, to be exactly as defined in \eqref{w_{1}}. For the reader's convenience, we write here its rather long and involved expression, recalling $w_{0}$ from \eqref{w_{0}}:
\begin{align}
w_{1}(\mu)={}&w_{0}(\mu)+\{4r-2(p+q)+(p+q)^{2}-6r^{2}-3r(p+2q)\}\mu(LD)+(1-p-q)(1-r)(-2r+r^{2}+3pr\nonumber\\&+3qr+r^{3})\mu(LDW)+(1-p-q)(1-r)(-1+p+q+r^{2}+2qr)\mu(LDL)-[3-(3+p+q+r\nonumber\\&-2q^{2}-2r^{2}-5qr+2pq+pr+3pqr+4q^{2}r+4qr^{2}+5pr^{2}-2p^{2}r-5pqr^{2}-3p^{2}r^{2}-2q^{2}r^{2})\nonumber\\&(1-p-q)(1-r)]\mu(LWD).\label{w_{1}_rewritten}
\end{align}
Following the argument outlined in \S\ref{subsubsec:adjust_gen}, the updated weight function in \eqref{w_{1}_rewritten} transforms the current weight function inequality, in \eqref{w_{0}_ineq_9}, as follows (the same way as we deduce \eqref{ith_wt_fn_ineq}):
\begin{align}
w_{1}(\G_{p,q,r}\mu)\leqslant{}&w_{1}(\mu)-[2-(1-p-q)(1-r)\{2+2p+q+r-q^{2}-r^{2}-3qr+pq+2q^{2}r+2qr^{2}+pqr+\nonumber\\&2pr^{2}-p^{2}r-2pqr^{2}-p^{2}r^{2}-q^{2}r^{2}\}]\{\mu(WD)-\mu(WWWD)-\mu(DWDD)-\mu(LWD)\}\nonumber\\&-\{1-(1+p+q+pq+pr-q^{2}-qr-p^{2}r+q^{2}r)(1-p-q)(1-r)(1+r)\}\mu(DD)\nonumber\\&-r^{2}(1-r)^{2}(1-p-q)^{2}\mu(WLD)-r^{2}(1-r)^{2}(1-p-q)^{2}\mu(DLD)+\{p(1-p)-\nonumber\\&(q+r-pr-qr)(1-q-r+pr+qr)\}(1-p-q)(1-r)\mu(WWDW)+(1-p-q)(1-r)\nonumber\\&\{p-q-r-3p^{2}-3q^{2}-3r^{2}-4pq-3pr-5qr+4p^{2}r+6pr^{2}+6q^{2}r+6qr^{2}+10pqr\}\nonumber\\&\mu(WWWD)+\{p(1-p)-q(1-q-r+pr+qr)\}(1-p-q)(1-r)(1+r)\mu(WWDD)\nonumber\\&+\{4r-2(p+q)+(p+q)^{2}-6r^{2}-3r(p+2q)\}\G_{p,q,r}\mu(LD)+(1-p-q)(1-r)(-2r+\nonumber\\&r^{2}+3pr+3qr+r^{3})\G_{p,q,r}\mu(LDW)+(1-p-q)(1-r)(-1+p+q+r^{2}+2qr)\nonumber\\&\G_{p,q,r}\mu(LDL)-[3-(3+p+q+r-2q^{2}-2r^{2}-5qr+2pq+pr+3pqr+4q^{2}r+4qr^{2}\nonumber\\&+5pr^{2}-2p^{2}r-5pqr^{2}-3p^{2}r^{2}-2q^{2}r^{2})(1-p-q)(1-r)]\G_{p,q,r}\mu(LWD)\nonumber\\
\leqslant{}&w_{1}(\mu)-[2-(1-p-q)(1-r)\{2+2p+q+r-q^{2}-r^{2}-3qr+pq+2q^{2}r+2qr^{2}+pqr+\nonumber\\&2pr^{2}-p^{2}r-2pqr^{2}-p^{2}r^{2}-q^{2}r^{2}\}]\{\mu(WD)-\mu(WWWD)-\mu(DWDD)-\mu(LWD)\}\nonumber\\&-\{1-(1+p+q+pq+pr-q^{2}-qr-p^{2}r+q^{2}r)(1-p-q)(1-r)(1+r)\}\mu(DD)\nonumber\\&-r^{2}(1-r)^{2}(1-p-q)^{2}\mu(WLD)-r^{2}(1-r)^{2}(1-p-q)^{2}\mu(DLD)+\{p(1-p)-\nonumber\\&(q+r-pr-qr)(1-q-r+pr+qr)\}(1-p-q)(1-r)\mu(WWDW)+(1-p-q)(1-r)\nonumber\\&\{p-q-r-3p^{2}-3q^{2}-3r^{2}-4pq-3pr-5qr+4p^{2}r+6pr^{2}+6q^{2}r+6qr^{2}+10pqr\}\nonumber\\&\mu(WWWD)+\{p(1-p)-q(1-q-r+pr+qr)\}(1-p-q)(1-r)(1+r)\mu(WWDD)\nonumber\\&+\{4r-2(p+q)+(p+q)^{2}-6r^{2}-3r(p+2q)\}(1-p)(1-p-q)(1-r)\mu(WWD)+\nonumber\\&(1-p-q)(1-r)(-2r+r^{2}+3pr+3qr+r^{3})(1-p)(1-p-q)(1-r)p\{\mu(WWDW)\nonumber\\&+\mu(WWDD)\}+(1-p-q)(1-r)(-1+p+q+r^{2}+2qr)[(1-p)(1-p-q)(1-r)\nonumber\\&\{q+(1-p-q)r\}\mu(WWDW)+(1-p)(1-p-q)(1-r)\{q+(1-p-q)r^{2}\}\mu(WWDD)]\nonumber\\&-[3-(3+p+q+r-2q^{2}-2r^{2}-5qr+2pq+pr+3pqr+4q^{2}r+4qr^{2}+5pr^{2}-2p^{2}r\nonumber\\&-5pqr^{2}-3p^{2}r^{2}-2q^{2}r^{2})(1-p-q)(1-r)]\{(1-p)p(1-p-q)(1-r)\mu(WWDW)+\nonumber\\&(1-p)p(1-p-q)(1-r)(1+r)\mu(WWDD)+(1-p)p(1-p-q)(1-r)\mu(WWWD)\}\label{intermediate_18}
\end{align}
where the inequality in the last step of the derivation of \eqref{intermediate_18} is obtained by making use of \eqref{i}, \eqref{ii}, \eqref{iii} and \eqref{iv}. 

Let us write down all those terms of \eqref{intermediate_18} that we now need to combine suitably:
\begin{enumerate}
\item \label{item_1} $\{p(1-p)-(q+r-pr-qr)(1-q-r+pr+qr)\}(1-p-q)(1-r)\mu(WWDW)$,
\item \label{item_2}$(1-p-q)(1-r)\{p-q-r-3p^{2}-3q^{2}-3r^{2}-4pq-3pr-5qr+4p^{2}r+6pr^{2}+6q^{2}r+6qr^{2}+10pqr\}\mu(WWWD)$,
\item \label{item_3} $\{p(1-p)-q(1-q-r+pr+qr)\}(1-p-q)(1-r)(1+r)\mu(WWDD)$,
\item \label{item_4} $\{4r-2(p+q)+(p+q)^{2}-6r^{2}-3r(p+2q)\}(1-p)(1-p-q)(1-r)\mu(WWD)$,
\item \label{item_5} $(1-p-q)(1-r)(-2r+r^{2}+3pr+3qr+r^{3})(1-p)(1-p-q)(1-r)p\{\mu(WWDW)+\mu(WWDD)\}$,
\item \label{item_6} $(1-p-q)(1-r)(-1+p+q+r^{2}+2qr)[(1-p)(1-p-q)(1-r)\{q+(1-p-q)r\}\mu(WWDW)+(1-p)(1-p-q)(1-r)\{q+(1-p-q)r^{2}\}\mu(WWDD)]$,
\item \label{item_7} $-[3-(3+p+q+r-2q^{2}-2r^{2}-5qr+2pq+pr+3pqr+4q^{2}r+4qr^{2}+5pr^{2}-2p^{2}r-5pqr^{2}-3p^{2}r^{2}-2q^{2}r^{2})(1-p-q)(1-r)]\{(1-p)p(1-p-q)(1-r)\mu(WWDW)+(1-p)p(1-p-q)(1-r)(1+r)\mu(WWDD)+(1-p)p(1-p-q)(1-r)\mu(WWWD)\}$.
\end{enumerate}
Recall that when $(p,q,r)$ satisfies the constraints of \eqref{three_cond_2} or \eqref{three_cond_3}, it is only the coefficient of $\mu(WWDD)$ that is non-positive on the right side of \eqref{w_{0}_ineq_9} (other than $w_{0}(\mu)$ itself). Consequently, for this case, we need only focus on the terms involving $\mu(WWDD)$ in the above-mentioned list (appearing in \eqref{item_3}, \eqref{item_5}, \eqref{item_6} and \eqref{item_7}), and the term involving $\mu(WWD)$ (appearing in \eqref{item_4}). Combining the terms involving $\mu(WWDD)$ from \eqref{item_3}, \eqref{item_5}, \eqref{item_6} and \eqref{item_7}, we get the following coefficient of $\mu(WWDD)$ (or rather, an upper bound on it), as long as $p$, $q$ and $r$ are sufficiently small:
\begin{align}
{}&(1-p-q)(1-r)[\{p(1-p)-q(1-q-r+pr+qr)\}(1+r)+(1-p-q)(1-r)(-2r+r^{2}+3pr+3qr\nonumber\\&+r^{3})(1-p)p+(-1+p+q+r^{2}+2qr)(1-p)(1-p-q)(1-r)\{q+(1-p-q)r^{2}\}-\{3-(3+p+q\nonumber\\&+r-2q^{2}-2r^{2}-5qr+2pq+pr+3pqr+4q^{2}r+4qr^{2}+5pr^{2}-2p^{2}r-5pqr^{2}-3p^{2}r^{2}-2q^{2}r^{2})\nonumber\\&(1-p-q)(1-r)\}(1-p)p(1+r)]\nonumber\\
={}&[-p^{5}r(2+3r-2r^{2}-3r^{3})-p^{4}qr^{2}(10+r-8r^{2})-11p^{4}r^{4}-4p^{4}r^{3}-7p^{3}q^{2}r^{2}(1+r-r^{2})-20p^{3}qr^{4}\nonumber\\&-p^{3}q(1+3r-15r^{2}-9r^{3}-2p-pr)-p^{3}r(11+3r+r^{2}-13r^{3}-r^{4}-8p-6pr)-2p^{2}q^{3}r^{3}(2-r)\nonumber\\&-2p^{2}q^{3}(1-2r)-p^{2}q^{2}r(13-7p-2r-15r^{2}+9r^{3})-p^{2}qr^{3}(17-11r-2r^{2})-p^{2}q(1-5q-5r)\nonumber\\&-3p^{2}r^{5}-4p^{2}r^{4}-p^{2}r^{2}(1-7r)-p^{2}(1-8r-p-p^{2})-3pq^{3}r(1+r-r^{2})-pq^{2}r^{4}(3-r)-8pq^{2}r^{3}\nonumber\\&-pq^{2}(7-6r-3q-11r^{2})-4pqr^{5}-pqr^{2}(8-11r-5r^{2})-3pqr-pr^{3}(5+2r-3r^{2})-2q^{3}r^{4}\nonumber\\&-q^{3}(1+r-3r^{2}-r^{3})-q^{2}r^{5}-q^{2}r^{2}(7-5r^{2})-2qr^{4}(2-r)-r^{2}(1-r-2p-5q-r^{2}+2qr+r^{3})\nonumber\\&+\{-2p^{2}+pq-3pr+p+3q^{2}+qr-2q\}](1-p-q)(1-r)\nonumber\\
\leqslant{}& (p-2q+pq+qr-3pr-2p^{2}+3q^{2})(1-p-q)(1-r).\label{coeff_WWDD_1}
\end{align}

The first scenario to consider is where the final expression in \eqref{coeff_WWDD_1} is non-positive, i.e.\ the second inequality of \eqref{three_cond_2} holds. The current weight function inequality, obtained from \eqref{intermediate_18} by updating the coefficient of $\mu(WWDD)$ to the expression in \eqref{coeff_WWDD_1}, and by combining all the terms involving $\mu(WWDW)$ (as listed in \eqref{item_1}, \eqref{item_5}, \eqref{item_6} and \eqref{item_7}) and all the terms involving $\mu(WWWD)$ (as listed in \eqref{item_2} and \eqref{item_7}) that were present in the final step of \eqref{intermediate_18} (this, in fact, \emph{further} reduces the already non-positive coefficients of $\mu(WWDW)$ and $\mu(WWWD)$), is given by:
\begin{align}
w_{1}(\G_{p,q,r}\mu)\leqslant{}&w_{1}(\mu)-[2-(1-p-q)(1-r)\{2+2p+q+r-q^{2}-r^{2}-3qr+pq+2q^{2}r+2qr^{2}+pqr+\nonumber\\&2pr^{2}-p^{2}r-2pqr^{2}-p^{2}r^{2}-q^{2}r^{2}\}]\{\mu(WD)-\mu(WWWD)-\mu(DWDD)-\mu(LWD)\}\nonumber\\&-\{1-(1+p+q+pq+pr-q^{2}-qr-p^{2}r+q^{2}r)(1-p-q)(1-r)(1+r)\}\mu(DD)\nonumber\\&-r^{2}(1-r)^{2}(1-p-q)^{2}\mu(WLD)-r^{2}(1-r)^{2}(1-p-q)^{2}\mu(DLD)+(1-p-q)(1-r)\nonumber\\&[\{p(1-p)-(q+r-pr-qr)(1-q-r+pr+qr)\}+(1-p-q)(1-r)(-2r+r^{2}+3pr\nonumber\\&+3qr+r^{3})(1-p)p+(-1+p+q+r^{2}+2qr)(1-p)(1-p-q)(1-r)\{q+(1-p-q)r\}\nonumber\\&-\{3-(3+p+q+r-2q^{2}-2r^{2}-5qr+2pq+pr+3pqr+4q^{2}r+4qr^{2}+5pr^{2}-2p^{2}r\nonumber\\&-5pqr^{2}-3p^{2}r^{2}-2q^{2}r^{2})(1-p-q)(1-r)\}(1-p)p]\mu(WWDW)+(1-p-q)(1-r)\nonumber\\&[p-q-r-3p^{2}-3q^{2}-3r^{2}-4pq-3pr-5qr+4p^{2}r+6pr^{2}+6q^{2}r+6qr^{2}+10pqr-\nonumber\\&\{3-(3+p+q+r-2q^{2}-2r^{2}-5qr+2pq+pr+3pqr+4q^{2}r+4qr^{2}+5pr^{2}-2p^{2}r\nonumber\\&-5pqr^{2}-3p^{2}r^{2}-2q^{2}r^{2})(1-p-q)(1-r)\}(1-p)p]\mu(WWWD)+(1-p-q)(1-r)\nonumber\\&(p-2q+pq+qr-3pr-2p^{2}+3q^{2})\mu(WWDD)+\{4r-2(p+q)+(p+q)^{2}-6r^{2}\nonumber\\&-3r(p+2q)\}(1-p)(1-p-q)(1-r)\mu(WWD).\label{w_{0}_ineq_10}
\end{align}
The inequality in \eqref{w_{0}_ineq_10} is evidently of the form given by \eqref{gen_ineq_form}, and satisfies \eqref{desired_criterion}. Therefore, our construction of the desired weight function comes to an end here in this case. The final weight function is as given by \eqref{w_{1}_rewritten} (which is the same as \eqref{w_{1}}), and the final weight function inequality is \eqref{w_{0}_ineq_10}. Since the coefficient of each of $\mu(WWDW)$, $\mu(WWDD)$, $\mu(WWWD)$ and $\mu(WWD)$ in \eqref{w_{0}_ineq_10} is non-positive when \eqref{three_cond_universal} and \eqref{three_cond_2} hold, we may remove these terms from \eqref{w_{0}_ineq_10}, which would yield the exact same weight function inequality as that given by \eqref{generalized_final_wt_fn_ineq} (once we replace the notation $w_{1}$ by simply $w$).

The second scenario to consider is where the final expression in \eqref{coeff_WWDD_1} is positive, i.e.\ if the second inequality of \eqref{three_cond_3} holds instead. In this case, we combine the coefficient of $\mu(WWD)$ in \eqref{item_4} with the expression in \eqref{coeff_WWDD_1} (making use of the inequality $\mu(WWDD) \leqslant \mu(WWD)$ -- we can use this since the coefficient of $\mu(WWD)$ in \eqref{item_4} is non-positive due to \eqref{three_cond_universal}), to get the following upper bound on the coefficient of $\mu(WWDD)$:
\begin{align}
{}&(p-2q+pq+qr-3pr-2p^{2}+3q^{2})(1-p-q)(1-r)+\{4r-2(p+q)+(p+q)^{2}-6r^{2}-3r(p+2q)\}\nonumber\\&(1-p)(1-p-q)(1-r)\nonumber\\
={}&(1-p-q)(1-r)\{-p^{3}-2p^{2}q+3p^{2}r+p^{2}-pq^{2}+6pqr+5pq+6pr^{2}-10pr-p+4q^{2}-5qr\nonumber\\&-4q-6r^{2}+4r\}\nonumber\\
\leqslant{}&(1-p-q)(1-r)\{-pr(1-3p-6q-6r)+p^{2}+5pq-9pr-p+4q^{2}-5qr-4q-6r^{2}+4r\}\nonumber\\
\leqslant{}&(1-p-q)(1-r)(-p-4q+4r+p^{2}+4q^{2}+5pq-9pr-5qr-6r^{2}),\label{coeff_WWDD_regime_2}
\end{align}
as long as $p$, $q$ and $r$ are sufficiently small, and this final expression is non-positive due to the third inequality of \eqref{three_cond_3}. Once incorporated into \eqref{w_{0}_ineq_10}, the updated coefficient of $\mu(WWDD)$ from \eqref{coeff_WWDD_regime_2} transforms the weight function inequality into one that satisfies \eqref{desired_criterion} (since now, each term on the right side of \eqref{w_{0}_ineq_10}, apart from $w_{0}(\mu)$, has a non-positive coefficient). Since the coefficient of each of $\mu(WWDW)$, $\mu(WWDD)$ and $\mu(WWWD)$ in the updated weight function inequality is non-positive, we may remove all of these terms, leaving us with exactly the same weight function inequality as that given by \eqref{generalized_final_wt_fn_ineq} (once we replace the notation $w_{1}$ by simply $w$).

\subsubsection{When $(p,q,r)$ satisfies \eqref{three_cond_4} of Theorem~\ref{thm:three-parameter}}\label{subsubsec:regime_4} The hardest case to consider is where \eqref{three_cond_4} is true. Most of the approach adopted so far also applies to this case, i.e.\ the weight function is updated as in \eqref{w_{1}_rewritten}, the coefficient of $\mu(WWDD)$ is updated to the expression given by \eqref{coeff_WWDD_1}, and the weight function inequality is updated to \eqref{w_{0}_ineq_10}. Observe that in the regime given by \eqref{three_cond_4}, there is no guarantee that the coefficient of any of $\mu(WWDW)$, $\mu(WWDD)$ and $\mu(WWWD)$ in \eqref{w_{0}_ineq_10} is non-positive. 

First, we simplify, by providing suitable upper bounds, the coefficients of $\mu(WWDW)$ and $\mu(WWWD)$ in \eqref{w_{0}_ineq_10}, since otherwise, the expressions are too cluttered due of the presence of very small terms (third or higher order terms, i.e.\ terms of the form $p^{i}q^{j}r^{k}$ with $i+j+k \geqslant 3$, make little difference to us as $p$, $q$ and $r$ are assumed to be sufficiently small). We deduce the following upper bound on the coefficient of $\mu(WWDW)$ in \eqref{w_{0}_ineq_10}, provided $p$, $q$ and $r$ are small enough:
\begin{align}
{}&(1-p-q)(1-r)[\{p(1-p)-(q+r-pr-qr)(1-q-r+pr+qr)\}+(1-p-q)(1-r)(-2r+r^{2}+3pr\nonumber\\&+3qr+r^{3})(1-p)p+(-1+p+q+r^{2}+2qr)(1-p)(1-p-q)(1-r)\{q+(1-p-q)r\}-\{3-(3+p+\nonumber\\&q+r-2q^{2}-2r^{2}-5qr+2pq+pr+3pqr+4q^{2}r+4qr^{2}+5pr^{2}-2p^{2}r-5pqr^{2}-3p^{2}r^{2}-2q^{2}r^{2})\nonumber\\&(1-p-q)(1-r)\}(1-p)p]\nonumber\\
={}&[-p^{5}r(2+r-3r^{2})-p^{4}qr^{2}(9-8r)-p^{4}qr-11p^{4}r^{3}-7p^{3}q^{2}r^{2}(2-r)-20p^{3}qr^{3}-p^{3}qr(4-25r-7q)\nonumber\\&-p^{3}q(1-2p)-p^{3}r^{2}(7-14r-4p)-p^{3}r(8-6p)-2p^{2}q^{3}r^{2}(3-r)-2p^{2}q^{3}(1-3r)-9p^{2}q^{2}r^{3}\nonumber\\&-p^{2}q^{2}r(19-23r)-p^{2}qr^{2}(23-11r-r^{2})-p^{2}q(1-12r)-p^{2}r^{4}-6p^{2}r^{3}-p^{2}(1-p-5r-p^{2}-8r^{2})\nonumber\\&-3pq^{3}r(2-r)-pq^{2}r^{3}(4-r)-5pq^{2}r^{2}-pq^{2}(7-15r-5p-3q)-3pqr^{4}-pqr(13-9r-6r^{2})\nonumber\\&-pr^{2}(6+r-2r^{2})-2q^{3}r^{3}-q^{3}(1-3r^{2})-q^{2}r^{4}-3q^{2}r^{2}(1-2r)-5q^{2}r-qr^{3}(5-2r)-2qr^{2}-r^{4}\nonumber\\&+\{-2p^{2}+pq+pr+p+3q^{2}+7qr-2q+r^{3}+2r^{2}-2r\}](1-p-q)(1-r)\nonumber\\
\leqslant{}&(p-2q-2r+pq+pr+7qr-2p^{2}+3q^{2}+2r^{2}+r^{3})(1-p-q)(1-r).\label{coeff_WWDW_1}
\end{align}
Next, we deduce the following upper bound on the coefficient of $\mu(WWWD)$ in \eqref{w_{0}_ineq_10}:
\begin{align}
{}&(1-p-q)(1-r)\{p-q-r-3p^{2}-3q^{2}-3r^{2}-4pq-3pr-5qr+4p^{2}r+6pr^{2}+6q^{2}r+6qr^{2}+10pqr\}\nonumber\\&-\{3-(3+p+q+r-2q^{2}-2r^{2}-5qr+2pq+pr+3pqr+4q^{2}r+4qr^{2}+5pr^{2}-2p^{2}r-5pqr^{2}-3p^{2}r^{2}\nonumber\\&-2q^{2}r^{2})(1-p-q)(1-r)\}(1-p)p(1-p-q)(1-r)\nonumber\\
={}&(1-p-q)(1-r)[-p^{5}r(2+r-3r^{2})-p^{4}qr^{2}(9-8r)-p^{4}qr-11p^{4}r^{3}-7p^{3}q^{2}r^{2}(2-r)-22p^{3}qr^{3}\nonumber\\&-p^{3}qr(6-30r-7q)-2p^{3}q(1-p)-p^{3}r(4+12r-15r^{2}-6pr-4p)-2p^{2}q^{3}r^{2}(3-r)-2p^{2}q^{3}(1-3r)\nonumber\\&-13p^{2}q^{2}r^{3}-p^{2}q^{2}r(19-29r)-p^{2}qr^{2}(33-20r)-2pq^{3}r^{3}-6pq^{3}r(1-r)-pq^{2}r^{2}(15-6r)\nonumber\\&-pq^{2}(3-12r-2q)-p^{2}(1-4r-2q-p-p^{2}-3q^{2}-11qr-10r^{2})-pr(1-3r-2r^{2}+9pr^{2}-6q-4p)\nonumber\\&-qr(1-6q-6r-12pr)+\{p-q-r-4p^{2}-3q^{2}-3r^{2}-6pq-4pr-4qr\}]\nonumber\\
\leqslant{}&(1-p-q)(1-r)(p-q-r-4p^{2}-3q^{2}-3r^{2}-6pq-4pr-4qr)\label{coeff_WWWD_1}
\end{align}
when $p$, $q$ and $r$ are sufficiently small. Incorporating \eqref{coeff_WWDW_1} and \eqref{coeff_WWWD_1} into \eqref{w_{0}_ineq_10}, we obtain:
\begin{align}
w_{1}(\G_{p,q,r}\mu)\leqslant{}&w_{1}(\mu)-[2-(1-p-q)(1-r)\{2+2p+q+r-q^{2}-r^{2}-3qr+pq+2q^{2}r+2qr^{2}+pqr\nonumber\\&+2pr^{2}-p^{2}r-2pqr^{2}-p^{2}r^{2}-q^{2}r^{2}\}]\{\mu(WD)-\mu(WWWD)-\mu(DWDD)-\mu(LWD)\}\nonumber\\&-\{1-(1+p+q+pq+pr-q^{2}-qr-p^{2}r+q^{2}r)(1-p-q)(1-r)(1+r)\}\mu(DD)\nonumber\\&-r^{2}(1-r)^{2}(1-p-q)^{2}\mu(WLD)-r^{2}(1-r)^{2}(1-p-q)^{2}\mu(DLD)+(p-2q-2r+pq\nonumber\\&+pr+7qr-2p^{2}+3q^{2}+2r^{2}+r^{3})(1-p-q)(1-r)\mu(WWDW)+(1-p-q)(1-r)\nonumber\\&(p-q-r-4p^{2}-3q^{2}-3r^{2}-6pq-4pr-4qr)\mu(WWWD)+(1-p-q)(1-r)(p-2q\nonumber\\&+pq+qr-3pr-2p^{2}+3q^{2})\mu(WWDD)+\{4r-2(p+q)+(p+q)^{2}-6r^{2}-3r(p+2q)\}\nonumber\\&(1-p)(1-p-q)(1-r)\mu(WWD).\label{w_{0}_ineq_11}
\end{align}
At this point, we could further subdivide our analysis into various cases depending on which of these three coefficients, i.e.\ the coefficients of $\mu(WWDW)$, $\mu(WWWD)$ and $\mu(WWDD)$, is / are non-negative, which would allow us coverage of a \emph{marginally} better regime, while \emph{significantly} lengthening and complicating the computations that follow. We, therefore, adopt the following idea irrespective of which among the coefficients of $\mu(WWDW)$, $\mu(WWDD)$ and $\mu(WWWD)$ in \eqref{w_{0}_ineq_11} is / are non-negative: this idea is to make use of the yet-untouched term involving $\mu(WWD)$, appearing in \eqref{item_4}, in reducing further each of the coefficients of $\mu(WWDW)$, $\mu(WWDD)$ and $\mu(WWWD)$ in \eqref{w_{0}_ineq_11}. Specifically, letting $c_{WWDW}$, $c_{WWDD}$, $c_{WWWD}$ and $c_{WWD}$ denote the coefficients of $\mu(WWDW)$, $\mu(WWDD)$, $\mu(WWWD)$ and $\mu(WWD)$, respectively, in \eqref{w_{0}_ineq_11}, and keeping in mind that $c_{WWD}\leqslant 0$ due to \eqref{three_cond_universal}, we write:
\begin{align}
{}&c_{WWD}\mu(WWD)+c_{WWDW}\mu(WWDW)+c_{WWDD}\mu(WWDD)+c_{WWWD}\mu(WWWD)\nonumber\\
{}&=\left\{\frac{1}{2}c_{WWD}\mu(WWD)+c_{WWWD}\mu(WWWD)\right\}+\left\{\frac{1}{2}c_{WWD}\mu(WWD)+c_{WWDW}\mu(WWDW)+c_{WWDD}\mu(WWDD)\right\}\nonumber\\
{}&\leqslant\left\{\frac{1}{2}c_{WWD}+c_{WWWD}\right\}\mu(WWWD)+\left\{\frac{1}{2}c_{WWD}+c_{WWDW}\right\}\mu(WWDW)+\left\{\frac{1}{2}c_{WWD}+c_{WWDD}\right\}\mu(WWDD),\label{combine_WWD_WWDW_WWDD_WWWD}
\end{align}
where we make use of the inequalities $\mu(WWWD) \leqslant \mu(WWD)$ and $\mu(WWDW)+\mu(WWDD)\leqslant \mu(WWD)$ (which follows from Lemma~\ref{lem:pushforward_1}). Combining half of the term involving $\mu(WWD)$ in \eqref{w_{0}_ineq_11} with the term involving $\mu(WWWD)$ in \eqref{w_{0}_ineq_11} the same way as shown in \eqref{combine_WWD_WWDW_WWDD_WWWD}, and keeping $p$, $q$ and $r$ sufficiently small wherever needed (for instance, in the third step of the derivation below), we obtain:
\begin{align}
{}&\big[(1-p-q)(1-r)(p-q-r-4p^{2}-3q^{2}-3r^{2}-6pq-4pr-4qr)+\frac{1}{2}\{4r-2(p+q)+(p+q)^{2}\nonumber\\&-6r^{2}-3r(p+2q)\}(1-p)(1-p-q)(1-r)\big]\mu(WWWD) \nonumber\\
={}&\Big\{-\frac{p^{3}}{2}-p^{2}q-\frac{3pr}{2}(1-p-2q-2r)-\frac{pq^{2}}{2}-6pr-\frac{5p^{2}}{2}-4pq-\frac{5q^{2}}{2}-7qr-2q-6r^{2}+r\Big\}\nonumber\\&(1-p-q)(1-r)\mu(WWWD)\nonumber\\
\leqslant{}&\Big(-6pr-\frac{5p^{2}}{2}-4pq-\frac{5q^{2}}{2}-7qr-2q-6r^{2}+r\Big)(1-p-q)(1-r)\mu(WWWD)\nonumber\\
={}&\frac{1}{2}\Big\{\Big(-5pr-\frac{p^{2}}{2}+3pq+\frac{7q^{2}}{2}-2qr-3q-3r^{2}+2r\Big)-7pr-\frac{9p^{2}}{2}-11pq-\frac{17q^{2}}{2}-12qr\nonumber\\&-q-9r^{2}\Big\}(1-p-q)(1-r)\mu(WWWD)\nonumber\\
\leqslant{}&\frac{1}{2}\Big(-5pr-\frac{p^{2}}{2}+3pq+\frac{7q^{2}}{2}-2qr-3q-3r^{2}+2r\Big)(1-p-q)(1-r)\mu(WWWD).\label{coeff_WWWD_2}
\end{align}
Likewise, combining the remaining half of the term involving $\mu(WWD)$ in \eqref{w_{0}_ineq_11} with the terms involving $\mu(WWDW)$ and $\mu(WWDD)$ in \eqref{w_{0}_ineq_11} the same way as shown in \eqref{combine_WWD_WWDW_WWDD_WWWD}, and keeping $p$, $q$ and $r$ sufficiently small wherever needed (such as in obtaining the inequality in the last step of deriving \eqref{coeff_WWDD_WWDW_2}), we obtain:
\begin{align}
{}&\big[(p-2q-2r+pq+pr+7qr-2p^{2}+3q^{2}+2r^{2}+r^{3})(1-p-q)(1-r)+\frac{1}{2}\{4r-2(p+q)+(p+q)^{2}\nonumber\\&-6r^{2}-3r(p+2q)\}(1-p)(1-p-q)(1-r)\big]\mu(WWDW)+\big[(1-p-q)(1-r)(p-2q+pq+qr-3pr\nonumber\\&-2p^{2}+3q^{2})+\frac{1}{2}\{4r-2(p+q)+(p+q)^{2}-6r^{2}-3r(p+2q)\}(1-p)(1-p-q)(1-r)\big]\mu(WWDD)\nonumber\\
={}&\Big\{-\frac{p^{3}}{2}-p^{2}q-\frac{pq^{2}}{2}-\frac{3pr}{2}(1-p-2q-2r)-5pr-\frac{p^{2}}{2}+3pq+\frac{7q^{2}}{2}-2qr-3q-3r^{2}+2r\Big\}\nonumber\\&(1-p-q)(1-r)\mu(WWDD)+\Big\{-\frac{p^{3}}{2}-p^{2}q-\frac{pq^{2}}{2}-\frac{3pr}{2}(1-p-2q-2r)-pr-\frac{p^{2}}{2}+\frac{7q^{2}}{2}\nonumber\\&+4qr-3q-r^{2}+3pq+r^{3}\Big\}(1-p-q)(1-r)\mu(WWDW)\nonumber\\
\leqslant{}&\Big(-5pr-\frac{p^{2}}{2}+3pq+\frac{7q^{2}}{2}-2qr-3q-3r^{2}+2r\Big)(1-p-q)(1-r)\mu(WWDD)\nonumber\\&+\Big(-pr-\frac{p^{2}}{2}+\frac{7q^{2}}{2}+4qr-3q-r^{2}+3pq+r^{3}\Big)(1-p-q)(1-r)\mu(WWDW)\nonumber\\
={}&\Big(-5pr-\frac{p^{2}}{2}+3pq+\frac{7q^{2}}{2}-2qr-3q-3r^{2}+2r\Big)(1-p-q)(1-r)\mu(WWDD)+(1-p-q)(1-r)\nonumber\\&\Big\{\Big(-5pr-\frac{p^{2}}{2}+3pq+\frac{7q^{2}}{2}-2qr-3q-3r^{2}+2r\Big)-r(2-4p-6q-2r-r^{2})\Big\}\mu(WWDW)\nonumber\\
\leqslant{}&\Big(-5pr-\frac{p^{2}}{2}+3pq+\frac{7q^{2}}{2}-2qr-3q-3r^{2}+2r\Big)(1-p-q)(1-r)\{\mu(WWDD)+\mu(WWDW)\}.\label{coeff_WWDD_WWDW_2}
\end{align}
The final coefficient of each of $\mu(WWWD)$, $\mu(WWDD)$ and $\mu(WWDW)$ in \eqref{coeff_WWWD_2} and \eqref{coeff_WWDD_WWDW_2} is non-positive due to the last inequality of \eqref{three_cond_4}. 

Incorporating the updated coefficients of $\mu(WWDW)$, $\mu(WWDD)$ and $\mu(WWWD)$ from \eqref{coeff_WWWD_2} and \eqref{coeff_WWDD_WWDW_2} into \eqref{w_{0}_ineq_10}, we obtain:
\begin{align}
w_{1}(\G_{p,q,r}\mu)\leqslant{}&w_{1}(\mu)-[2-(1-p-q)(1-r)\{2+2p+q+r-q^{2}-r^{2}-3qr+pq+2q^{2}r+2qr^{2}+pqr+\nonumber\\&2pr^{2}-p^{2}r-2pqr^{2}-p^{2}r^{2}-q^{2}r^{2}\}]\{\mu(WD)-\mu(WWWD)-\mu(DWDD)-\mu(LWD)\}\nonumber\\&-\{1-(1+p+q+pq+pr-q^{2}-qr-p^{2}r+q^{2}r)(1-p-q)(1-r)(1+r)\}\mu(DD)\nonumber\\&-r^{2}(1-r)^{2}(1-p-q)^{2}\mu(WLD)-r^{2}(1-r)^{2}(1-p-q)^{2}\mu(DLD)+\nonumber\\&\left(-5pr-\frac{p^{2}}{2}+3pq+\frac{7q^{2}}{2}-2qr-3q-3r^{2}+2r\right)(1-p-q)(1-r)\nonumber\\&\left\{\frac{1}{2}\mu(WWWD)+\mu(WWDD)+\mu(WWDW)\right\}\nonumber\\
\leqslant{}&w_{1}(\mu)-[2-(1-p-q)(1-r)\{2+2p+q+r-q^{2}-r^{2}-3qr+pq+2q^{2}r+2qr^{2}+pqr+\nonumber\\&2pr^{2}-p^{2}r-2pqr^{2}-p^{2}r^{2}-q^{2}r^{2}\}]\{\mu(WD)-\mu(WWWD)-\mu(DWDD)-\mu(LWD)\}\nonumber\\&-\{1-(1+p+q+pq+pr-q^{2}-qr-p^{2}r+q^{2}r)(1-p-q)(1-r)(1+r)\}\mu(DD)\nonumber\\&-r^{2}(1-r)^{2}(1-p-q)^{2}\mu(WLD)-r^{2}(1-r)^{2}(1-p-q)^{2}\mu(DLD),\nonumber
\end{align}
where the last step is obtained by eliminating the terms involving $\mu(WWDW)$, $\mu(WWDD)$ and $\mu(WWWD)$ from the previous step, which is allowed because the coefficient in each of these terms is non-positive when \eqref{three_cond_4} is true. We draw the reader's attention to the fact that the final inequality deduced above is the same as the inequality stated in \eqref{generalized_final_wt_fn_ineq} (once we replace the notation $w_{1}$ by simply $w$), and \eqref{generalized_final_wt_fn_ineq} is of the same form as \eqref{gen_ineq_form} and satisfies \eqref{desired_criterion}. This completes the construction of our weight function (the final form of which is as given by \eqref{w_{1}}), for the generalized percolation games, when the underlying parameter-triple $(p,q,r)$ satisfies the constraints given by \eqref{three_cond_4} of Theorem~\ref{thm:three-parameter}.

\section{The proof of Theorem~\ref{thm:bond_envelope_PCA_D} by the technique of weight functions}\label{sec:bond_weight_function}
The principal ideas involved in constructing our weight function for each of the three regimes stated in Theorem~\ref{thm:two-parameter} are the same as those outlined in \S\ref{subsec:central_ideas_weight_functions} (including, as outlined in \S\ref{subsubsec:adjust_gen}, the ideas implemented for the iterative tweaking of the weight function until it satisfies an inequality of the form given by \eqref{gen_ineq_form}, along with the criterion stated in \eqref{desired_criterion}). However, as becomes evident when one goes through the steps leading to the construction of the weight function for the proof of Theorem~\ref{thm:generalized_envelope_PCA_D}, outlined in \S\ref{sec:generalized_wt_fn_steps}, the details involved are rather \emph{ad hoc} and \emph{specific} to the regime of values of the parameter-pair $(r',s')$ under consideration. Consequently, we dedicate \S\ref{sec:bond_1_wt_fn_steps} to the detailed construction of the weight function for the regime given by \ref{bond_regime_1}, \S\ref{sec:bond_2_wt_fn_steps} to that for the regime given by \ref{bond_regime_2}, and \S\ref{sec:bond_3_wt_fn_steps} to that for the regime given by \ref{bond_regime_3}.

\subsection{The final weight functions obtained for the bond percolation game when $(r',s')$ belongs to one of the regimes given by \ref{bond_regime_1}, \ref{bond_regime_2} and \ref{bond_regime_3}}\label{subsec:final_wt_fn_bond}
As in \S\ref{subsec:final_wt_fn_generalized}, we state here, for the reader's convenience, the final weight function we have come up with, and the corresponding weight function inequality it satisfies, for $(r',s')$ belonging to each of the regimes given by \ref{bond_regime_1}, \ref{bond_regime_2} and \ref{bond_regime_3}. When $(r',s')$ satisfies the constraints of \ref{bond_regime_1}, the weight function is given by
\begin{align}\label{final_wt_fn_q=0}
w(\mu)={}&\mu(D)+\mu(WD)+\Big[1-(1-r'-s')\Big\{2(2s'-{s'}^{2})(1-s')+r'+(2s'-{s'}^{2})^{2}(1-s')+\frac{2{r'}^{2}}{(1-s')}\nonumber\\&-(2s'-{s'}^{2})^{2}r'-\frac{5(2s'-{s'}^{2}){r'}^{2}}{(1-s')}\Big\}\Big]\mu(LWD)-(1-r'-s')\Big\{-2(2s'-{s'}^{2})(1-s')-4r'\nonumber\\&+(2s'-{s'}^{2})^{2}(1-s')+\frac{3{r'}^{2}}{(1-s')}+2(2s'-{s'}^{2})r'-2(2s'-{s'}^{2})^{2}r'-\frac{6(2s'-{s'}^{2}){r'}^{2}}{(1-s')}\Big\}\mu(LD)-\nonumber\\&(1-r'-s')\Big\{3r'+2(2s'-{s'}^{2})^{2}(1-s')+\frac{2{r'}^{2}}{(1-s')}-(2s'-{s'}^{2})r'-2(2s'-{s'}^{2})^{2}r'-\frac{4(2s'-{s'}^{2}){r'}^{2}}{(1-s')}\nonumber\\&-\frac{{r'}^{3}}{(1-s')^{2}}\Big\}\mu(LDW)-(1-s')(1-r'-s')\Big\{(1-s')^{2}-{r'}^{2}+(2s'-{s'}^{2}){r'}^{2}\Big\}\mu(LDL),
\end{align}
and the corresponding weight function inequality is given by
\begin{align}
w(\E_{r',s'}\mu)\leqslant{}&w(\mu)-\frac{(1-r'-s')\{(1-s')(2s'-{s'}^{2})-r'\}^{2}}{1-s'}\mu(DD).\label{final_wt_fn_ineq_q=0}
\end{align}
We now deduce, from \eqref{final_wt_fn_ineq_q=0}, the conclusion we desire, i.e.\ $\mu(D)=0$ for any translation-invariant and reflection-invariant probability distribution $\mu$ that is stationary for $\E_{r',s'}$. As seen in \S\ref{subsubsec:bond_to_generalized}, $\E_{r',s'}$ is equivalent to $\G_{p,q,r}$ if we set $q=0$ and let $p$ and $r$ equal suitable functions of $r'$ and $s'$, as given by \eqref{transform_1} and \eqref{transform_2}. From \eqref{transform_2}, we see that $r'=0 \implies r=0$, so that we are back to the set-up of \cite{holroyd2019percolation}, with $q=0$ and $p=2s'-{s'}^{2}$, when $r'=0$. Our desired conclusion , therefore, follows from \cite{holroyd2019percolation} when $r'=0$. It thus suffices for us to consider $r' > 0$. 

Since $r'$ and $s'$ are sufficiently small, we have $(1-r'-s') > 0$, and because of \eqref{two_regime_1_eq}, we have $(1-s')(2s'-{s'}^{2})-r' > 0$ as well, so that the coefficient of $\mu(DD)$ in \eqref{final_wt_fn_ineq_q=0} is strictly positive. Consequently, when $\mu$ is stationary for $\E_{r',s'}$, the parameters $r'$ and $s'$ are sufficiently small and they satisfy the inequality in \eqref{two_regime_1_eq}, we have $\mu(DD)=0$, which, in turn, implies that $\E_{r',s'}\mu(DD)=0$. By an application of Lemma~\ref{lem:pushforward_2} with $\G_{p,q,r}$ replaced by $\E_{r',s'}$ (where we set $\mathcal{C}=(D,D)_{0,1}$ and consider $\mathcal{D}\in\left\{(W,D,W)_{0,1,2},(W,D,L)_{0,1,2},(L,D,L)_{0,1,2}\right\}$) and using \eqref{envelope_PCA_rule_1} through \eqref{envelope_PCA_rule_6}, we see that 
\begin{equation}
(1-s')^{2}(1-r'-s')^{2}\mu(WDW)+2r'(1-s')(1-r'-s')^{2}\mu(WDL)+{r'}^{2}(1-r'-s')^{2}\mu(LDL) \leqslant \E_{r',s'}\mu(DD),\nonumber
\end{equation}
so that, for $r'$ and $s'$ sufficiently small, $r' > 0$, and $(r',s')$ satisfying \eqref{two_regime_1_eq}, we have $\mu(WDW)=\mu(WDL)=\mu(LDW)=\mu(LDL)=0$ when $\mu$ is any translation-invariant and reflection-invariant stationary distribution for $\E_{r',s'}$. Moreover, we already know that $\mu(WDD)=\mu(DDW)=\mu(LDD)=\mu(DDL)=\mu(DDD)=0$ as each of these is bounded above by $\mu(DD)$. Combining all of these, we obtain $\mu(D)=0$, as desired.

When $(r',s')$ belongs to the regime given by \ref{bond_regime_2}, the weight function takes various forms depending on the subset of $(0.157175,1]$ that the parameter $r'$ belongs to. When $s'=0$ and $r' \in [0.4564,1]$, it suffices to consider the weight function
\begin{equation}
w(\mu) = \mu(D) + \mu(WD) + \mu(LWD) - 2r'\mu(WD),\label{adjust_1}
\end{equation}
and the corresponding weight function inequality is given by
\begin{align}
w(\E_{r',0}\mu) \leqslant{}& w(\mu) - {r'}^{2}\mu(DD) + (2r'-6{r'}^{2}+4{r'}^{3}-{r'}^{4})\mu(LD) - (3r'-3{r'}^{2}+{r'}^{3})\mu(LWD)\nonumber\\& - (1-r')^{2}(1-r'+2{r'}^{2})\mu(LDL) - 2{r'}^{2}(1-r')^{2}\mu(LDD) - r'(1-r')^{2}(1+r')\mu(LDW) \nonumber\\& - r'(1-r')^{4}\mu(LLD) - r'(1-r')^{3}\mu(DLD).\label{wt_ineq_1}
\end{align}
When $s'=0$ and $r' \in (0.201383,0.4564)$, the weight function that yields the desired conclusion is
\begin{equation}\label{adjust_2}
w(\mu) = \mu(D) + \mu(WD) + \mu(LWD) - 2r'\mu(WD) - (3r'-3{r'}^{2}+{r'}^{3})\mu(LWD),
\end{equation}
and the corresponding weight function inequality is given by
\begin{align}
w(\E_{r',0}\mu) \leqslant{}& w(\mu) - {r'}^{2}\mu(DD) + (2r'-14{r'}^{2}+23{r'}^{3}-14{r'}^{4}+3{r'}^{6}-{r'}^{7})\mu(LD)\nonumber\\&- (1-r')^{2}(1-r'-3{r'}^{2}+3{r'}^{3}-{r'}^{4})\mu(LDL) - r'(1-r')^{2}(1-r'-3{r'}^{2}+3{r'}^{3}-{r'}^{4})\mu(LDW) \nonumber\\&- r'(1-r')^{2}({r'}^{5}-2{r'}^{4}-{r'}^{3}+7{r'}^{2}-5r'+1)\mu(LLD) - r'(1-r')^{6}\mu(DLD) \nonumber\\&- (3r'-3{r'}^{2}+{r'}^{3})r'(1-r')^{2}\mu(LWD).\label{wt_ineq_2}
\end{align}
For the final sub-regime, i.e.\ where $s'=0$ and $r' \in (0.157175,0.201383]$, the weight function is given by
\begin{align}
w(\mu) ={}&\mu(D) + \mu(WD) + \mu(LWD) - 2r'\mu(WD) - (3r'-3{r'}^{2}+{r'}^{3})\mu(LWD) - r'(1-r')^{6}\mu(LLWD) \nonumber\\&- r'(1-r')^{6}\mu(LLDW) - {r'}^{2}\mu(WDD) - (3r'-3{r'}^{2}+{r'}^{3})r'(1-r')^{2}\mu(LWD) \nonumber\\&- r'(1-r')^{4}(1-3r'+{r'}^{2}+{r'}^{3}-3{r'}^{4}+3{r'}^{5}-{r'}^{6})\mu(LWLD),\label{adjust_4}
\end{align}
and the corresponding weight function inequality is of the form
\begin{align}
w(\E_{r',0}\mu) \leqslant{}& w(\mu) + r'(2-16r'+20{r'}^{2}+7{r'}^{3}-42{r'}^{4}+51{r'}^{5}-35{r'}^{6}+7{r'}^{7}+13{r'}^{8}-14{r'}^{9}+6{r'}^{10}-{r'}^{11})\nonumber\\&\mu(LD) - {r'}^{2}\mu(DDD).\label{wt_ineq_4}
\end{align}

When $\mu$ is a reflection-invariant and translation-invariant stationary measure for $\E_{r',0}$, and $(r',s')$ belongs to any one of the sub-regimes mentioned above, the relevant inequality out of \eqref{wt_ineq_1}, \eqref{wt_ineq_2} and \eqref{wt_ineq_4} yields $\mu(LD) = 0$. As $\mu$ is stationary for $\E_{r',0}$, this implies $\E_{r',0}\mu(LD)=0$ as well. Setting $\mathcal{C}=(L,D)_{0,1}$ and $\mathcal{D}\in\{W,L,D\}\times\{D\}\times\{W,L,D\}$, and applying Lemma~\ref{lem:pushforward_2} with $\G{p,q,r}$ replaced by $\E_{r',s'}$, we obtain:
\begin{align}
{}&r'(1-r')\mu(WDW)+r'(1-r')(1+r')\mu(WDD)+{r'}^{2}(1-r')\mu(WDL)+{r'}^{2}(1-r')\mu(DDW)\nonumber\\&+{r'}^{2}(1-r')(1+r')\mu(DDD)+{r'}^{3}(1-r')\mu(DDL)+{r'}^{2}(1-r')\mu(LDW)+{r'}^{2}(1-r')(1+r')\mu(LDD)\nonumber\\&+{r'}^{3}(1-p)\mu(LDL) \leqslant \E_{r',0}\mu(LD).\nonumber
\end{align}
Since $\E_{r',0}\mu(LD)=0$, we conclude that each of $\mu(WDW)$, $\mu(WDD)$, $\mu(WDL)$, $\mu(DDW)$, $\mu(DDD)$, $\mu(DDL)$, $\mu(LDW)$, $\mu(LDD)$ and $\mu(LDL)$ equals $0$. Adding these, we obtain $\mu(D)=0$, as desired.

For the regime described in \ref{bond_regime_3}, our weight function is given by
\begin{align}
    w(\mu) &= \mu(D) + \mu(WD) - (2{r'}^2 - 5{r'}^3 + 2{r'}^4) \mu(WDW) - (3{r'} - 10{r'}^2 + 10{r'}^3 - 4{r'}^4) \mu(WDL) \nonumber \\
    &\enspace - (7{r'}^2 - 7{r'}^3 + 2{r'}^4) \mu(DDW) - (2r' - 2{r'}^2 + 5{r'}^3 - 2{r'}^4) \mu(DDL)  \nonumber\\
    &\enspace - (1 - 4r' + 6{r'}^2 - 5{r'}^3 + 2{r'}^4) \mu(LDL) - (4{r'}^2 - {r'}^3 + 2{r'}^4) \mu(WLD)  \nonumber\\
    &\enspace - (3{r'}^2 + 3{r'}^3 -2{r'}^4) \mu(LLD) - (2r' + 8{r'}^2 - 9{r'}^3 + 2{r'}^4) \mu(WWD),\label{weight_fn_p=q}
\end{align}
and the corresponding weight function inequality can be written as follows: 
\begin{align}
    w(\E_{r',r'}\mu) &\leqslant w(\mu) + (1 - 10r' + 7{r'}^2 + 64{r'}^4 - 292{r'}^5 + 583{r'}^6 - 663{r'}^7 + 447{r'}^8 -168{r'}^9 + 28{r'}^{10}) \mu(LWD).\label{weight_fn_ineq_p=q}
\end{align}
The coefficient of $\mu(LWD)$ in \eqref{weight_fn_ineq_p=q} is negative whenever $r'=s' > 0.10883$, so that, when $\mu$ is a translation-invariant and reflection-invariant stationary measure for $\E_{r',r'}$, we conclude, from \eqref{weight_fn_ineq_p=q}, that $\mu(LWD) = 0$. This, in turn, leads to $\E_{r',r'}\mu(LWD) = 0$ since $\mu$ is stationary for $\E_{r',r'}$. Setting $\mathcal{C}=(L,W,D)_{0,1,2}$, considering $\mathcal{D}\in\{W,L,D\}^{3}\times\{D\}$, applying Lemma~\ref{lem:pushforward_2} with $\G_{p,q,r}$ replaced by $\E_{r',s'}$, and using \eqref{envelope_PCA_rule_1} through \eqref{envelope_PCA_rule_6}, we obtain: 
\begin{align}
    \E_{r', r'}\mu(LWD) &\geqslant (1-r')^2(2r'-{r'}^2)(1-r')(1-2r') \mu(WWWD) + r'(1-r')(2r'-{r'}^2)(1-r')(1-2r') \mu(LWWD) \nonumber\\&+ r'(1-r')(2r'-{r'}^2)(1-r')(1-2r') \mu(DWWD) + r'(1-r')(1-r'+{r'}^2)(1-r')(1-2r') \mu(WLWD) \nonumber\\&+ {r'}^2(1-r'+{r'}^2)(1-r')(1-2r')\mu(LLWD) + {r'}^2(1-r'+{r'}^2)(1-r')(1-2r')\mu(DLWD) \nonumber\\&+ r'(1-r')(2r'-{r'}^2)(1-r')(1-2r') \mu(WDWD) + {r'}^2(2r'-{r'}^2)(1-r')(1-2r') \mu(LDWD) \nonumber\\&+ {r'}^2(2r'-{r'}^2)(1-r')(1-2r')\mu(DDWD) + (1-r')^2(1-r'+{r'}^2)r'(1-2r') \mu(WWLD) \nonumber\\&+ r'(1-r')(1-r'+{r'}^2)r'(1-2r') \mu(LWLD)+ r'(1-r')(1-r'+{r'}^2)r'(1-2r') \mu(DWLD) \nonumber\\&+ r'(1-r')(1-{r'}^2)r'(1-2r') \mu(WLLD) + {r'}^2(1-{r'}^2)r'(1-2r') \mu(LLLD) \nonumber\\&+ {r'}^2(1-{r'}^2)r'(1-2r') \mu(DLLD) + r'(1-r')(1-r'+{r'}^2)r'(1-2r') \mu(WDLD) \nonumber\\&+ {r'}^2(1-r'+{r'}^2)r'(1-2r') \mu(LDLD) + {r'}^2(1-r'+{r'}^2)r'(1-2r') \mu(DDLD) \nonumber\\&+ (1-r')^2(2r'-{r'}^2)(1-2r') \mu(WWDD) + r'(1-r')(2r'-{r'}^2)(1-2r') \mu(LWDD) \nonumber\\&+ r'(1-r')(2r'-{r'}^2)(1-2r') \mu(DWDD) + r'(1-r')(1-r'+{r'}^2)(1-2r') \mu(WLDD) \nonumber\\&+ {r'}^2(1-r'+{r'}^2)(1-2r') \mu(LLDD) + {r'}^2(1-r'+{r'}^2)(1-2r') \mu(DLDD) \nonumber\\&+ r'(1-r')(2r'-{r'}^2)(1-2r') \mu(WDDD) + {r'}^2(2r'-{r'}^2)(1-2r') \mu(LDDD) \nonumber\\&+ {r'}^2(2r'-{r'}^2)(1-2r') \mu(DDDD).\label{F_{p,p}mu(LWD)}
    \end{align}
Since we consider $0 < r' < 0.5$ (when $r'=s'$, it makes sense to consider $r' \leqslant 0.5$ since $r'+s' \leqslant 1$, and when $r'+s'=1$, each edge of our graph would be labeled either a trap or a target, bringing the game to an end right after the first round), the coefficient of each term in the right side of \eqref{F_{p,p}mu(LWD)} is strictly positive. Since $\E_{r',r'}\mu(LWD)=0$ when $\mu$ is stationary for $\E_{r',r'}$, we obtain, from \eqref{F_{p,p}mu(LWD)}, that $\mu(\mathcal{D})=0$ for each cylinder set $\mathcal{D}\in\{W,L,D\}^{3}\times\{D\}$. Adding them all, we obtain $\mu(D)=0$, as desired.

\subsection{Detailed construction of the weight function for the bond percolation game when $(r',s')$ belongs to the regime given by \ref{bond_regime_1}}\label{sec:bond_1_wt_fn_steps} As in the case of \S\ref{sec:generalized_wt_fn_steps}, it is assumed, throughout \S\ref{sec:bond_1_wt_fn_steps}, that $\mu$ is a translation-invariant and reflection-invariant measure on the state space $\Omega=\hat{\mathcal{A}}^{\mathbb{Z}}=\{W,L,D\}^{\mathbb{Z}}$ of $\E_{r',s'}$, measurable with respect to the $\sigma$-field $\mathcal{F}$ generated by all cylinder sets of $\Omega$. Recall, from \S\ref{subsubsec:bond_to_generalized}, that the envelope PCA $\E_{r',s'}$, corresponding to the bond percolation game with underlying parameter-pair $(r',s')$, identical to the envelope PCA $\G_{p,q,r}$, corresponding to the generalized percolation game with underlying parameter-triple $(p,q,r)$, provided we set (see \eqref{transform_1} and \eqref{transform_2})
\begin{equation}
p=2s'-{s'}^{2},\quad q=0 \quad \text{and} \quad r=\frac{r'}{1-s'}.\label{transform_combined}
\end{equation} 
With these values of $p$ and $r$, the inequality in \eqref{two_regime_1_eq} boils down to 
\begin{equation}
\frac{3p^{2}}{2}+8pr+\frac{11r^{2}}{2}\geqslant 2r+\frac{9p^{3}}{2}+16p^{2}r+\frac{43pr^{2}}{2}+5r^{3},\label{two_regime_1_eq_transformed}
\end{equation}
and proving the claim that $\nu(D)=0$ for any translation-invariant and reflection-invariant stationary distribution $\nu$ for the PCA $\E_{r',s'}$, whenever $(r',s')$ satisfies the constraints of \ref{bond_regime_1}, is equivalent to showing, instead, that $\mu(D)=0$ for any translation-invariant and reflection-invariant stationary distribution $\mu$ for $\G_{p,0,r}$, where $p$ and $r$ are as given by \eqref{transform_combined} and $(p,r)$ satisfies the inequality in \eqref{two_regime_1_eq_transformed}. 

Recall that our computations in \S\ref{sec:generalized_wt_fn_steps} were performed \emph{without} making the assumption that $q=0$, and hence we may proceed in the same manner as in \S\ref{sec:generalized_wt_fn_steps}, up to and including the derivation of \eqref{w_{0}_ineq_7}. Setting $w_{0}(\mu)=\mu(D)+\mu(WD)+\mu(LWD)$, the weight function inequality obtained can be written as 
\begin{align}
w_{0}(\G_{p,0,r}\mu)\leqslant{}&w_{0}(\mu)-[2-(1-p)(1-r)\{2+2p+r-r^{2}+2pr^{2}-p^{2}r-p^{2}r^{2}\}]\mu(WD)\nonumber\\&-\{1-(1+p+pr-p^{2}r)(1-p)(1-r)(1+r)\}\mu(DD)-[1-\{1-p+4r\nonumber\\&-r^{2}(2-3p)-r^{3}(1-p)-p^{2}r^{2}\}(1-p)(1-r)]\mu(LD)-r(1-p)^{2}(1-r)\nonumber\\&(1-p-r)\mu(LDW)-(1-p)^{2}(1-r)[(1-r)\{1-r^{2}(1-p)\}-pr^{2}]\mu(LDL)\nonumber\\&-[1-(1-p-r^{2}+pr+3pr^{2}-p^{2}r-2p^{2}r^{2})(1-p)(1-r)]\mu(LWD)\nonumber\\&-r^{2}(1-r)^{2}(1-p)^{2}\mu(WLD)-r^{2}(1-r)^{2}(1-p)^{2}\mu(DLD)+\{p-r+r^{2}(1-p)\}\nonumber\\&(1-p)^{2}(1-r)\mu(WWDW)+p(1-p)^{2}(1-r)^{2}\mu(WWWD)+A_{8}+A_{9},\label{w_{0}_ineq_7_q=0}
\end{align}  
where $A_{8}$ and $A_{9}$ are as defined in the proof of Lemma~\ref{lem:pushforward_LWD}. Substituting $q=0$ in \eqref{bound_A_{8}}, we have $A_{8}\leqslant r(1-r+pr)(1-p)^{2}(1-r)(1+r)\mu(LDD)$, and combining this bound with the term involving $\mu(LD)$ in \eqref{w_{0}_ineq_7_q=0}, we get
\begin{align}
{}&-[1-\{1-p+4r-r^{2}(2-3p)-r^{3}(1-p)-p^{2}r^{2}\}(1-p)(1-r)]\mu(LD)\nonumber\\&+r(1-r+pr)(1-p)^{2}(1-r)(1+r)\mu(LDD)\nonumber\\
={}&[-(1-p)(1-r)\{(1-p)(1-r)\}^{-1}+\{1-p+4r-r^{2}(2-3p)-r^{3}(1-p)-p^{2}r^{2}\}(1-p)(1-r)]\nonumber\\&\mu(LD)+r(1-r+pr)(1-p)^{2}(1-r)(1+r)\mu(LDD)\nonumber\\
\leqslant{}&[-(1-p)(1-r)\{1+p+r+p^{2}+r^{2}+pr-2p^{2}r-2pr^{2}+p^{2}r^{2}\}+(1-p)(1-r)\{1-p+4r\nonumber\\&-r^{2}(2-3p)-r^{3}(1-p)-p^{2}r^{2}\}]\mu(LD)+r(1-r+pr)(1-p)^{2}(1-r)(1+r)\{\mu(LD)\nonumber\\&-\mu(LDW)-\mu(LDL)\}\nonumber\\
={}&(1-p)(1-r)\{-2p+4r-p^{2}-3r^{2}-2pr+2p^{2}r+6pr^{2}-r^{3}(2-3p)-3p^{2}r^{2}-p^{2}r^{3}\}\mu(LD)\nonumber\\&-r(1-r+pr)(1-p)^{2}(1-r)(1+r)\{\mu(LDW)+\mu(LDL)\}\nonumber\\
\leqslant{}&(1-p)(1-r)\{-2p+4r-p^{2}-3r^{2}-2pr+2p^{2}r+6pr^{2}\}\mu(LD)\nonumber\\&-r(1-r+pr)(1-p)^{2}(1-r)(1+r)\{\mu(LDW)+\mu(LDL)\}\label{q=0_intermediate_2}
\end{align}
when $p$ is sufficiently small (it suffices for $p$ to be bounded above by $2/3$). Note, here, that when \eqref{two_regime_1_eq_transformed} holds, we have 
\begin{align}
{}&-2p+4r-p^{2}-3r^{2}-2pr+2p^{2}r+6pr^{2}\nonumber\\
\leqslant{}&-2p+3p^{2}+16pr+11r^{2}-9p^{3}-32p^{2}r-43pr^{2}-10r^{3}-p^{2}-3r^{2}-2pr+2p^{2}r+6pr^{2}\nonumber\\
={}&-2p+2p^{2}+14pr+8r^{2}-9p^{3}-30p^{2}r-37pr^{2}-10r^{3}\nonumber\\
\leqslant{}&-2p+2p^{2}+14pr+\frac{1}{2}\left(3p^{2}+16pr+11r^{2}\right)^{2}-9p^{3}-30p^{2}r-37pr^{2}-10r^{3}\nonumber\\
={}&-2p+2p^{2}+14pr+9p^{4}+256p^{2}r^{2}+121r^{4}+96p^{3}r+66p^{2}r^{2}+352pr^{3}-9p^{3}-30p^{2}r-37pr^{2}-10r^{3}\nonumber\\
={}&-p(2-2p-14r+9p^{2}-9p^{3})-p^{2}r(30-96p-161r)-pr^{2}(37-352r-161p)-r^{3}(10-121r),\nonumber
\end{align}
showing us that the coefficient of $\mu(LD)$ in \eqref{q=0_intermediate_2} is non-positive when $p$ and $r$ are sufficiently small and \eqref{two_regime_1_eq_transformed} holds. Setting $q=0$ in the computation of $A_{9}$ in \eqref{contribution_{W,D}_W_D_D}, we get
\begin{align}
A_{9}={}&p(1-p)^{2}(1-r)(1+r)\mu(WWDD)+pr(1-p)^{2}(1-r)(1+r)\mu(DWDD).\label{contribution_WDD_q=0}
\end{align}
Since $\mu(WWDD)+\mu(DWDD)\leqslant \mu(WDD) \leqslant \mu(WD)$ (by Lemma~\ref{lem:pushforward_1}), we combine the term involving $\mu(WD)$ in \eqref{w_{0}_ineq_7_q=0} with the terms in \eqref{contribution_WDD_q=0} to write:
\begin{align}
{}&-[2-(1-p)(1-r)\{2+2p+r-r^{2}+2pr^{2}-p^{2}r-p^{2}r^{2}\}]\mu(WD)+p(1-p)^{2}(1-r)(1+r)\mu(WWDD)\nonumber\\&+pr(1-p)^{2}(1-r)(1+r)\mu(DWDD)\nonumber\\
={}&[-2(1-p)(1-r)\{(1-p)(1-r)\}^{-1}+(1-p)(1-r)\{2+2p+r-r^{2}+2pr^{2}-p^{2}r-p^{2}r^{2}\}]\mu(WD)\nonumber\\&+p(1-p)^{2}(1-r)(1+r)\mu(WWDD)+pr(1-p)^{2}(1-r)(1+r)\mu(DWDD)\nonumber\\
\leqslant{}&[-2(1-p)(1-r)\{1+p+r+p^{2}+r^{2}+pr-2p^{2}r-2pr^{2}+p^{2}r^{2}\}+(1-p)(1-r)\{2+2p+r-r^{2}\nonumber\\&+2pr^{2}-p^{2}r-p^{2}r^{2}\}]\mu(WD)+p(1-p)^{2}(1-r)(1+r)\mu(WWDD)+pr(1-p)^{2}(1-r)(1+r)\nonumber\\&\mu(DWDD)\quad (\text{substituting } q=0 \text{ in \eqref{inverse_2}})\nonumber\\
={}&(1-p)(1-r)\{-r-2p^{2}-3r^{2}-2pr+3p^{2}r+6pr^{2}-3p^{2}r^{2}\}\mu(WD)+p(1-p)^{2}(1-r)(1+r)\nonumber\\&\mu(WWDD)+pr(1-p)^{2}(1-r)(1+r)\mu(DWDD)\nonumber\\
={}&(1-p)(1-r)\{-r-2p^{2}-3r^{2}-2pr+3p^{2}r+6pr^{2}-3p^{2}r^{2}\}\{\mu(WD)-\mu(WWDD)-\mu(DWDD)\}\nonumber\\&+(1-p)(1-r)(p-r-3p^{2}-3r^{2}-pr+2p^{2}r+6pr^{2}-3p^{2}r^{2})\mu(WWDD)\nonumber\\&+(1-p)(1-r)(-r-pr-2p^{2}-3r^{2}+7pr^{2}+2p^{2}r-4p^{2}r^{2})\mu(DWDD)\nonumber\\
\leqslant{}&(1-p)(1-r)\{-r-2p^{2}-3r^{2}-2pr+3p^{2}r+6pr^{2}-3p^{2}r^{2}\}\{\mu(WD)-\mu(WWDD)-\mu(DWDD)\}\nonumber\\&+(1-p)(1-r)(p-r-3p^{2}-3r^{2}-pr+2p^{2}r+6pr^{2})\mu(WWDD)\label{q=0_intermediate_1}
\end{align}
(we can remove the term involving $\mu(DWDD)$ in the last step since the coefficient of $\mu(DWDD)$ is evidently non-positive for $p$ and $r$ sufficiently small). Since $\mu(WWDD)+\mu(DWDD)+\mu(WDL)\leqslant \mu(WD)$ (follows from Lemma~\ref{lem:pushforward_1}), we can combine the terms involving $\mu(LDW)$ and $\mu(LDL)$ in \eqref{w_{0}_ineq_7_q=0}, the terms involving $\mu(LDW)$ and $\mu(LDL)$ in \eqref{q=0_intermediate_2}, and the term involving $\mu(WD)-\mu(WWDD)-\mu(DWDD)$ in \eqref{q=0_intermediate_1} to get:
\begin{align}
{}&-r(1-p)^{2}(1-r)(1-p-r)\mu(LDW)-(1-p)^{2}(1-r)[(1-r)\{1-r^{2}(1-p)\}-pr^{2}]\mu(LDL)-r(1-r\nonumber\\&+pr)(1-p)^{2}(1-r)(1+r)\{\mu(LDW)+\mu(LDL)\}+(1-p)(1-r)\{-r-2p^{2}-3r^{2}-2pr+3p^{2}r+6pr^{2}\nonumber\\&-3p^{2}r^{2}\}\{\mu(WD)-\mu(WWDD)-\mu(DWDD)\}\nonumber\\
={}&(1-p)(1-r)\{-3r-2p^{2}-2r^{2}+pr+2p^{2}r+4pr^{2}+r^{3}-pr^{3}(2-p)-2p^{2}r^{2}\}\mu(LDW)\nonumber\\&-(1-p)^{2}(1-r)(1-r^{2}+pr^{2})\mu(LDL)+(1-p)(1-r)\{-r-2p^{2}-3r^{2}-2pr+3p^{2}r+6pr^{2}-3p^{2}r^{2}\}\nonumber\\&\{\mu(WD)-\mu(WWDD)-\mu(DWDD)-\mu(WDL)\}\nonumber\\
\leqslant{}&(1-p)(1-r)\{-3r-2p^{2}-2r^{2}+pr+2p^{2}r+4pr^{2}+r^{3}\}\mu(LDW)-(1-p)^{2}(1-r)(1-r^{2}+pr^{2})\mu(LDL)\nonumber\\&+(1-p)(1-r)\{-r-2p^{2}-3r^{2}-2pr+3p^{2}r+6pr^{2}-3p^{2}r^{2}\}\{\mu(WD)-\mu(WWDD)-\mu(DWDD)\nonumber\\&-\mu(WDL)\}.\label{q=0_intermediate_3}
\end{align}
Finally, the coefficient of $\mu(LWD)$ in \eqref{w_{0}_ineq_7_q=0} can be bounded above by the expression in \eqref{inverse_2}, with $q=0$:
\begin{align}
{}&-[1-(1-p-r^{2}+pr+3pr^{2}-p^{2}r-2p^{2}r^{2})(1-p)(1-r)]\nonumber\\
={}&-(1-p)(1-r)\{(1-p)(1-r)\}^{-1}+(1-p)(1-r)(1-p-r^{2}+pr+3pr^{2}-p^{2}r-2p^{2}r^{2})\nonumber\\
\leqslant{}&(1-p)(1-r)\{1+p+r+p^{2}+r^{2}+pr-2p^{2}r-2pr^{2}+p^{2}r^{2}\}\nonumber\\&+(1-p)(1-r)(1-p-r^{2}+pr+3pr^{2}-p^{2}r-2p^{2}r^{2})\nonumber\\
={}&(1-p)(1-r)(-2p-r-p^{2}-2r^{2}+p^{2}r+5pr^{2}-3p^{2}r^{2})\nonumber\\
\leqslant{}&-(1-p)(1-r)\{2p+r+p^{2}+2r^{2}-p^{2}r-5pr^{2}\}.\label{q=0_intermediate_4}
\end{align}
Incorporating the term involving $\mu(LD)$ from \eqref{q=0_intermediate_2}, the term involving $\mu(WWDD)$ from \eqref{q=0_intermediate_1}, the terms from \eqref{q=0_intermediate_3}, and the updated coefficient of $\mu(LWD)$ from \eqref{q=0_intermediate_4} into \eqref{w_{0}_ineq_7_q=0}, we obtain:
\begin{align}
w_{0}(\G_{p,0,r}\mu)\leqslant{}&w_{0}(\mu)+(1-p)(1-r)\{-r-2p^{2}-3r^{2}-2pr+3p^{2}r+6pr^{2}-3p^{2}r^{2}\}\{\mu(WD)-\mu(WWDD)\nonumber\\&-\mu(DWDD)-\mu(WDL)\}-\{1-(1+p+pr-p^{2}r)(1-p)(1-r)(1+r)\}\mu(DD)-(1-p)\nonumber\\&(1-r)\{2p-4r+p^{2}+3r^{2}+2pr-2p^{2}r-6pr^{2}\}\mu(LD)-(1-p)(1-r)\{3r+2p^{2}+2r^{2}-pr\nonumber\\&-2p^{2}r-4pr^{2}-r^{3}\}\mu(LDW)-(1-p)^{2}(1-r)(1-r^{2}+pr^{2})\mu(LDL)-(1-p)(1-r)\nonumber\\&\{2p+r+p^{2}+2r^{2}-p^{2}r-5pr^{2}\}\mu(LWD)-r^{2}(1-r)^{2}(1-p)^{2}\mu(WLD)-r^{2}(1-r)^{2}(1-p)^{2}\nonumber\\&\mu(DLD)+\{p-r+r^{2}(1-p)\}(1-p)^{2}(1-r)\mu(WWDW)+p(1-p)^{2}(1-r)^{2}\mu(WWWD)\nonumber\\&+(1-p)(1-r)(p-r-3p^{2}-3r^{2}-pr+2p^{2}r+6pr^{2})\mu(WWDD).\label{w_{0}_ineq_8_q=0}
\end{align}
The non-negative coefficients on the right side of \eqref{w_{0}_ineq_8_q=0}, other than $w_{0}(\mu)$ itself, are those that correspond to $\mu(WWDW)$, $\mu(WWWD)$ and $\mu(WWDD)$.  

We now update the weight function (the same way as we did in \eqref{w_{1}_rewritten}) $w_{0}$ as follows:
\begin{align}
w_{1}(\mu)={}&w_{0}(\mu)-(1-p)(1-r)\{2p-4r+p^{2}+3r^{2}+2pr-2p^{2}r-6pr^{2}\}\mu(LD)-(1-p)(1-r)\nonumber\\&\{3r+2p^{2}+2r^{2}-pr-2p^{2}r-4pr^{2}-r^{3}\}\mu(LDW)-(1-p)^{2}(1-r)(1-r^{2}+pr^{2})\mu(LDL)\nonumber\\&-(1-p)(1-r)\{2p+r+p^{2}+2r^{2}-p^{2}r-5pr^{2}\}\mu(LWD),\label{w_{1}_q=0}
\end{align}
which is exactly the same as what we stated in \eqref{final_wt_fn_q=0} (once we replace $w_{1}$ by $w$ and set $p$ and $r$ to be the functions of $r'$ and $s'$ as given by \eqref{transform_combined}). In order to understand how the update in \eqref{w_{1}_q=0} impacts the weight function inequality in \eqref{w_{0}_ineq_8_q=0}, we need to compute $\G_{p,0,r}\mu(LD)$, $\G_{p,0,r}\mu(LDW)$, $\G_{p,0,r}\mu(LDL)$ and $\G_{p,0,r}\mu(LDW)$, but only partially, just as we did in \eqref{i}, \eqref{ii}, \eqref{iii} and \eqref{iv}:
\begin{align}\label{i_q=0}
{}&\bullet (1-p)^{2}(1-r)\mu(WWD) \leqslant \G_{p,0,r}\mu(LD)\nonumber\\
\Longleftrightarrow{}& -(1-p)(1-r)\{2p-4r+p^{2}+3r^{2}+2pr-2p^{2}r-6pr^{2}\}\G_{p,0,r}\mu(LD) \leqslant -(1-p)(1-r)\nonumber\\&\{2p-4r+p^{2}+3r^{2}+2pr-2p^{2}r-6pr^{2}\}(1-p)^{2}(1-r)\mu(WWD)\nonumber\\
\Longleftrightarrow{}& -(1-p)(1-r)\{2p-4r+p^{2}+3r^{2}+2pr-2p^{2}r-6pr^{2}\}\G_{p,0,r}\mu(LD) \leqslant -(1-p)^{3}(1-r)^{2}\nonumber\\&\{2p-4r+p^{2}+3r^{2}+2pr-2p^{2}r-6pr^{2}\}\mu(WWD)
\end{align}
(the second step of the derivation of \eqref{i_q=0} is true because the coefficient $-(1-p)(1-r)\{2p-4r+p^{2}+3r^{2}+2pr-2p^{2}r-6pr^{2}\}$ is non-positive when \eqref{two_regime_1_eq_transformed} holds and $p$ and $r$ are both sufficiently small);
\begin{align}\label{ii_q=0}
{}&\bullet p(1-p)^{2}(1-r)\{\mu(WWDW)+\mu(WWDD)\}\leqslant \G_{p,0,r}\mu(LDW)\nonumber\\
\Longleftrightarrow{}&-(1-p)(1-r)\{3r+2p^{2}+2r^{2}-pr-2p^{2}r-4pr^{2}-r^{3}\}\G_{p,0,r}\mu(LDW) \leqslant -(1-p)(1-r)\nonumber\\&\{3r+2p^{2}+2r^{2}-pr-2p^{2}r-4pr^{2}-r^{3}\}p(1-p)^{2}(1-r)\{\mu(WWDW)+\mu(WWDD)\}\nonumber\\
\Longleftrightarrow{}&-(1-p)(1-r)\{3r+2p^{2}+2r^{2}-pr-2p^{2}r-4pr^{2}-r^{3}\}\G_{p,0,r}\mu(LDW) \leqslant -p(1-p)^{3}(1-r)^{2}\nonumber\\&\{3r+2p^{2}+2r^{2}-pr-2p^{2}r-4pr^{2}-r^{3}\}\{\mu(WWDW)+\mu(WWDD)\},
\end{align}
\begin{align}\label{iii_q=0}
{}&\bullet r(1-p)^{3}(1-r)\mu(WWDW)+r^{2}(1-p)^{3}(1-r)\mu(WWDD)\leqslant \G_{p,0,r}\mu(LDL)\nonumber\\
\Longleftrightarrow{}&-(1-p)^{2}(1-r)(1-r^{2}+pr^{2})\G_{p,0,r}\mu(LDL) \leqslant -(1-p)^{2}(1-r)(1-r^{2}+pr^{2})\{r(1-p)^{3}(1-r)\nonumber\\&\mu(WWDW)+r^{2}(1-p)^{3}(1-r)\mu(WWDD)\}\nonumber\\
\Longleftrightarrow{}&-(1-p)^{2}(1-r)(1-r^{2}+pr^{2})\G_{p,0,r}\mu(LDL) \leqslant -r(1-p)^{5}(1-r)^{2}(1-r^{2}+pr^{2})\mu(WWDW)\nonumber\\&-r^{2}(1-p)^{5}(1-r)^{2}(1-r^{2}+pr^{2})\mu(WWDD);
\end{align}
and finally,
\begin{align}\label{iv_q=0}
{}&\bullet p(1-p)^{2}(1-r)\{\mu(WWDW)+\mu(WWWD)\}+p(1-p)^{2}(1-r)(1+r)\mu(WWDD)\leqslant \G_{p,0,r}\mu(LWD)\nonumber\\
\Longleftrightarrow{}&-(1-p)(1-r)\{2p+r+p^{2}+2r^{2}-p^{2}r-5pr^{2}\}\G_{p,0,r}\mu(LWD) \leqslant -(1-p)(1-r)\{2p+r+p^{2}+2r^{2}\nonumber\\&-p^{2}r-5pr^{2}\}[p(1-p)^{2}(1-r)\{\mu(WWDW)+\mu(WWWD)\}+p(1-p)^{2}(1-r)(1+r)\mu(WWDD)]\nonumber\\
\Longleftrightarrow{}&-(1-p)(1-r)\{2p+r+p^{2}+2r^{2}-p^{2}r-5pr^{2}\}\G_{p,0,r}\mu(LWD) \leqslant -p(1-p)^{3}(1-r)^{2}\{2p+r+p^{2}\nonumber\\&+2r^{2}-p^{2}r-5pr^{2}\}\{\mu(WWDW)+\mu(WWWD)\}-p(1-p)^{3}(1-r)^{2}(1+r)\{2p+r+p^{2}+2r^{2}\nonumber\\&-p^{2}r-5pr^{2}\}\mu(WWDD).
\end{align}
Implementing the same idea as that used to derive \eqref{ith_wt_fn_ineq}, the update in \eqref{w_{1}_q=0} transforms the weight function inequality in \eqref{w_{0}_ineq_8_q=0} as follows, using the bounds obtained in \eqref{i_q=0}, \eqref{ii_q=0}, \eqref{iii_q=0} and \eqref{iv_q=0}:
\begin{align}
w_{1}(\G_{p,0,r}\mu)\leqslant{}&w_{1}(\mu)+(1-p)(1-r)\{-r-2p^{2}-3r^{2}-2pr+3p^{2}r+6pr^{2}-3p^{2}r^{2}\}\{\mu(WD)-\mu(WWDD)\nonumber\\&-\mu(DWDD)-\mu(WDL)\}-\{1-(1+p+pr-p^{2}r)(1-p)(1-r)(1+r)\}\mu(DD)-r^{2}(1-r)^{2}\nonumber\\&(1-p)^{2}\mu(WLD)-r^{2}(1-r)^{2}(1-p)^{2}\mu(DLD)+\{p-r+r^{2}(1-p)\}(1-p)^{2}(1-r)\mu(WWDW)\nonumber\\&+p(1-p)^{2}(1-r)^{2}\mu(WWWD)+(1-p)(1-r)(p-r-3p^{2}-3r^{2}-pr+2p^{2}r+6pr^{2})\nonumber\\&\mu(WWDD)-(1-p)(1-r)\{2p-4r+p^{2}+3r^{2}+2pr-2p^{2}r-6pr^{2}\}\G_{p,0,r}\mu(LD)-(1-p)\nonumber\\&(1-r)\{3r+2p^{2}+2r^{2}-pr-2p^{2}r-4pr^{2}-r^{3}\}\G_{p,0,r}\mu(LDW)-(1-p)^{2}(1-r)(1-r^{2}+pr^{2})\nonumber\\&\G_{p,0,r}\mu(LDL)-(1-p)(1-r)\{2p+r+p^{2}+2r^{2}-p^{2}r-5pr^{2}\}\G_{p,0,r}\mu(LWD)\nonumber\\
\leqslant{}&w_{1}(\mu)+(1-p)(1-r)\{-r-2p^{2}-3r^{2}-2pr+3p^{2}r+6pr^{2}-3p^{2}r^{2}\}\{\mu(WD)-\mu(WWDD)\nonumber\\&-\mu(DWDD)-\mu(WDL)\}-\{1-(1+p+pr-p^{2}r)(1-p)(1-r)(1+r)\}\mu(DD)-r^{2}(1-r)^{2}\nonumber\\&(1-p)^{2}\mu(WLD)-r^{2}(1-r)^{2}(1-p)^{2}\mu(DLD)+\{p-r+r^{2}(1-p)\}(1-p)^{2}(1-r)\mu(WWDW)\nonumber\\&+p(1-p)^{2}(1-r)^{2}\mu(WWWD)+(1-p)(1-r)(p-r-3p^{2}-3r^{2}-pr+2p^{2}r+6pr^{2})\nonumber\\&\mu(WWDD)-(1-p)^{3}(1-r)^{2}\{2p-4r+p^{2}+3r^{2}+2pr-2p^{2}r-6pr^{2}\}\mu(WWD)-p(1-p)^{3}\nonumber\\&(1-r)^{2}\{3r+2p^{2}+2r^{2}-pr-2p^{2}r-4pr^{2}-r^{3}\}\{\mu(WWDW)+\mu(WWDD)\}-r(1-p)^{5}\nonumber\\&(1-r)^{2}(1-r^{2}+pr^{2})\mu(WWDW)-r^{2}(1-p)^{5}(1-r)^{2}(1-r^{2}+pr^{2})\mu(WWDD)-p(1-p)^{3}\nonumber\\&(1-r)^{2}\{2p+r+p^{2}+2r^{2}-p^{2}r-5pr^{2}\}\{\mu(WWDW)+\mu(WWWD)\}-p(1-p)^{3}(1-r)^{2}\nonumber\\&(1+r)\{2p+r+p^{2}+2r^{2}-p^{2}r-5pr^{2}\}\mu(WWDD).\label{w_{0}_ineq_9_q=0}
\end{align}
We make use of an inequality analogous to \eqref{combine_WWD_WWDW_WWDD_WWWD}, so that
\begin{enumerate}
\item the coefficient of $\mu(WWDD)$ is updated to
\begin{align}
{}&(1-p)(1-r)(p-r-3p^{2}-3r^{2}-pr+2p^{2}r+6pr^{2})-\frac{1}{2}(1-p)^{3}(1-r)^{2}\{2p-4r+p^{2}+3r^{2}+2pr-2p^{2}r\nonumber\\&-6pr^{2}\}-p(1-p)^{3}(1-r)^{2}\{3r+2p^{2}+2r^{2}-pr-2p^{2}r-4pr^{2}-r^{3}\}-r^{2}(1-p)^{5}(1-r)^{2}(1-r^{2}+pr^{2})\nonumber\\&-p(1-p)^{3}(1-r)^{2}(1+r)\{2p+r+p^{2}+2r^{2}-p^{2}r-5pr^{2}\}\nonumber\\
={}&(1-p)(1-r)\Big[-p^{5}r^{4}(1-r)-p^{5}r^{3}-p^{5}r^{2}-p^{5}(3-5r)-p^{4}r^{3}-\frac{5p^{4}r}{2}(3-4r)-p^{3}r(4+12r-4r^{2}-r^{3}\nonumber\\&-10r^{2})-10p^{2}r^{5}-\frac{p^{2}r^{2}}{2}(13-3r-6r^{2})-pr^{3}(6+4r-5r^{2})-r^{5}+\frac{7p^{4}}{2}+p^{3}+\frac{29p^{2}r}{2}-\frac{7p^{2}}{2}+20pr^{2}\nonumber\\&-9pr+r^{4}+\frac{5r^{3}}{2}-\frac{15r^{2}}{2}+r\Big]\nonumber\\
\leqslant{}&\frac{1}{2}(1-p)(1-r)[2r-7p^{2}-15r^{2}-18pr+29p^{2}r+40pr^{2}+2p^{3}+5r^{3}+7p^{4}+2r^{4}]\nonumber\\
={}&\frac{1}{2}(1-p)(1-r)\Big[\left\{\frac{9p^{3}}{2}+16p^{2}r-\frac{3p^{2}}{2}+\frac{43pr^{2}}{2}-8pr+5r^{3}-\frac{11r^{2}}{2}+2r\right\}-\frac{11p^{2}}{2}-\frac{19r^{2}}{2}-10pr\nonumber\\&+13p^{2}r+\frac{37pr^{2}}{2}-\frac{5p^{3}}{2}+7p^{4}+2r^{4}\Big]\nonumber\\
\leqslant{}&\frac{1}{2}(1-p)(1-r)\left\{\frac{9p^{3}}{2}+16p^{2}r-\frac{3p^{2}}{2}+\frac{43pr^{2}}{2}-8pr+5r^{3}-\frac{11r^{2}}{2}+2r\right\}\label{WWDD_coeff_final_q=0}
\end{align}
for $p$ and $r$ sufficiently small, and this final upper bound is non-positive because of \eqref{two_regime_1_eq_transformed};
\item the coefficient of $\mu(WWDW)$ is updated to
\begin{align}
{}&(1-p)^{2}(1-r)\{p-r+r^{2}(1-p)\}-\frac{1}{2}(1-p)^{3}(1-r)^{2}\{2p-4r+p^{2}+3r^{2}+2pr-2p^{2}r-6pr^{2}\}\nonumber\\&-p(1-p)^{3}(1-r)^{2}\{3r+2p^{2}+2r^{2}-pr-2p^{2}r-4pr^{2}-r^{3}\}-r(1-p)^{5}(1-r)^{2}(1-r^{2}+pr^{2})\nonumber\\&-p(1-p)^{3}(1-r)^{2}\{2p+r+p^{2}+2r^{2}-p^{2}r-5pr^{2}\}\nonumber\\
={}&(1-p)^{2}(1-r)\Big[-p^{4}r^{4}-p^{4}r(6-3r-r^{2})-p^{3}r^{2}(11-5r-4r^{2})-\frac{p^{3}}{2}(1-5r-6p)-5p^{2}r^{4}-3p^{2}r^{3}\nonumber\\&-\frac{p^{2}}{2}(3-11r-8r^{2})-\frac{pr^{3}}{2}(11-6r)-\frac{pr}{2}(6-11r)-\frac{r^{2}}{2}(3-5r+2r^{2})\Big]\label{WWDW_coeff_final_q=0}
\end{align}
which is evidently non-positive for $p$ and $r$ sufficiently small;
\item the coefficient of $\mu(WWWD)$ is updated to 
\begin{align}
{}&p(1-p)^{2}(1-r)^{2}-\frac{1}{2}(1-p)^{3}(1-r)^{2}\{2p-4r+p^{2}+3r^{2}+2pr-2p^{2}r-6pr^{2}\}\nonumber\\&-p(1-p)^{3}(1-r)^{2}\{2p+r+p^{2}+2r^{2}-p^{2}r-5pr^{2}\}\nonumber\\
={}&(1-p)^{2}(1-r)^{2}\left\{2r-\frac{3p^{2}}{2}-\frac{3r^{2}}{2}-4pr+3p^{2}r+\frac{5pr^{2}}{2}+\frac{3p^{3}}{2}+4p^{2}r^{2}-5p^{3}r^{2}+p^{4}-p^{4}r\right\}\nonumber\\
={}&-p^{6}r^{3}-3p^{6}r(1-r)-\frac{p^{5}}{2}(1-2p-6r)-5p^{5}r^{4}-\frac{p^{5}r^{2}}{2}(19-24r)\nonumber\\&-p^{4}r^{3}(26-14r)-\frac{p^{4}}{2}(7-18r-13r^{2})-\frac{21p^{3}r^{4}}{2}-p^{3}r(19-14r-11r^{2})+\frac{9p^{3}}{2}-\frac{5p^{2}r^{4}}{2}\nonumber\\&-p^{2}r^{2}(30-18r)+16p^{2}r-\frac{3p^{2}}{2}-\frac{pr^{3}}{2}(38-11r)+\frac{43pr^{2}}{2}-8pr-\frac{3r^{4}}{2}+5r^{3}-\frac{11r^{2}}{2}+2r\nonumber\\
\leqslant{}&\frac{9p^{3}}{2}+16p^{2}r-\frac{3p^{2}}{2}+\frac{43pr^{2}}{2}-8pr+5r^{3}-\frac{11r^{2}}{2}+2r\label{WWWD_coeff_final_q=0}
\end{align}
which is non-positive precisely because of the inequality in \eqref{two_regime_1_eq_transformed}.
\end{enumerate} 

A final step that we perform now is the deduction of a suitable upper bound for the coefficient of $\mu(DD)$ in \eqref{w_{0}_ineq_9_q=0} (using \eqref{inverse_2} with $q=0$):
\begin{align}
{}&-\{1-(1+p+pr-p^{2}r)(1-p)(1-r)(1+r)\}=-(1-p)(1-r)\{(1-p)(1-r)\}^{-1}\nonumber\\&+(1+p+pr-p^{2}r)(1-p)(1-r)(1+r)\nonumber\\
\leqslant{}&-(1-p)(1-r)(1+p+r+p^{2}+r^{2}+pr-2p^{2}r-2pr^{2}+p^{2}r^{2})\nonumber\\&+(1-p)(1-r)(1+p+r+2pr-p^{2}r+pr^{2}-p^{2}r^{2})\nonumber\\
={}&-(1-p)(1-r)(p^{2}+r^{2}-pr-p^{2}r-3pr^{2}+2p^{2}r^{2})\nonumber\\
={}&-(1-p)(1-r)\{p^{2}+r^{2}-2pr+pr(1-p-3r+2pr)\} \leqslant -(1-p)(1-r)(p-r)^{2},\label{updated_DD_coeff_q=0}
\end{align}
where the last inequality is true for $p$ and $r$ sufficiently small. Incorporating \eqref{WWDD_coeff_final_q=0}, \eqref{WWDW_coeff_final_q=0}, \eqref{WWWD_coeff_final_q=0} and \eqref{updated_DD_coeff_q=0} into \eqref{w_{0}_ineq_9_q=0}, we get the updated weight function inequality
\begin{align}
w_{1}(\G_{p,0,r}\mu)\leqslant{}&w_{1}(\mu)+(1-p)(1-r)\{-r-2p^{2}-3r^{2}-2pr+3p^{2}r+6pr^{2}-3p^{2}r^{2}\}\{\mu(WD)\nonumber\\&-\mu(WWDD)-\mu(DWDD)-\mu(WDL)\}-(1-p)(1-r)(p-r)^{2}\mu(DD)\nonumber\\&-r^{2}(1-r)^{2}(1-p)^{2}\mu(WLD)-r^{2}(1-r)^{2}(1-p)^{2}\mu(DLD)\nonumber\\&+\left(\frac{9p^{3}}{2}+16p^{2}r-\frac{3p^{2}}{2}+\frac{43pr^{2}}{2}-8pr+5r^{3}-\frac{11r^{2}}{2}+2r\right)\mu(WWWD)\nonumber\\&+\frac{1}{2}(1-p)(1-r)\left(\frac{9p^{3}}{2}+16p^{2}r-\frac{3p^{2}}{2}+\frac{43pr^{2}}{2}-8pr+5r^{3}-\frac{11r^{2}}{2}+2r\right)\mu(WWDD)\nonumber\\&(1-p)^{2}(1-r)\Big[-p^{4}r^{4}-p^{4}r(6-3r-r^{2})-p^{3}r^{2}(11-5r-4r^{2})-\frac{p^{3}}{2}(1-5r-6p)\nonumber\\&-5p^{2}r^{4}-3p^{2}r^{3}-\frac{p^{2}}{2}(3-11r-8r^{2})-\frac{pr^{3}}{2}(11-6r)-\frac{pr}{2}(6-11r)-\frac{r^{2}}{2}(3-5r+2r^{2})\Big]\nonumber\\&\mu(WWDW).\nonumber
\end{align} 
At this point, we remove all terms, except the one involving $\mu(DD)$, from the right side of the inequality above, since the coefficient of each of these terms is non-positive as long as $p$ and $r$ are sufficiently small and \eqref{two_regime_1_eq_transformed} holds. The final weight function inequality obtained is exactly what we have in \eqref{final_wt_fn_ineq_q=0} once we set $p$, $q$ and $r$ to be as in \eqref{transform_combined}, and we replace $w_{1}$ by $w$. This concludes the construction of our weight function for the envelope PCA $\E_{r',s'}$ (correspondingly, for $\G_{p,q,r}$ with $p$, $q$ and $r$ are as given by \eqref{transform_combined}), when $(r',s')$ belong to the regime in \ref{bond_regime_1}.

\subsection{Detailed construction of the weight function for the bond percolation game when $(r',s')$ belongs to the regime given by \ref{bond_regime_2}}\label{sec:bond_2_wt_fn_steps} In the regime described by \ref{bond_regime_2}, we set $s'=0$ and $r' > 0$. We begin again, as in \S\ref{sec:generalized_wt_fn_steps} and \S\ref{sec:bond_1_wt_fn_steps}, with $c_{1}=1$ and $\mathcal{C}_{1}=(D)_{0}$, followed by $c_{2}=1$ and $\mathcal{C}_{2}=(W,D)_{0,1}$. We skip some of the details in the computation of the pushforward measures (i.e.\ $\E_{r',0}\mu(D)$, $\E_{r',0}\mu(WD)$ etc.)\ as these are special cases (obtained by setting $s'=0$) of what we have already computed in \S\ref{sec:bond_1_wt_fn_steps}. In each computation that follows, we indicate by underbraces, wherever necessary, which terms are being combined together to proceed to the next step:
\begin{align}
\E_{r',0}\mu(D) ={}& 2(1-r')\mu(WD) + (1+r')(1-r')\mu(DD) + 2r'(1-r')\mu(LD),\label{pushforward_D_q=0}\\
\E_{r',0}\mu(WD)={}&\underbrace{(1-r')^{2}\mu(LDW)+r'(1-r')^{2}\mu(LDL)+(1-r')^{2}(1+r')\mu(LDD)}_{\text{combine using } \mu(LD)=\mu(LDW)+\mu(LDL)+\mu(LDD)}+(1-r')^{2}\mu(LWD)\nonumber\\&+\underbrace{r'(1-r')^{2}\mu(WLD)+r'(1-r')^{2}(1+r')\mu(LLD)+r'(1-r')^{2}\mu(DLD)}_{\text{combine using } \mu(LD) = \mu(WLD)+\mu(LLD)+\mu(DLD)}\label{pushforward_WD_alternate}\\
={}& \underbrace{(1-r')^{2}(1+r')\mu(LD)-r'(1-r')^{2}\mu(LDW)-(1-r')^{2}\mu(LDL)}_{\text{obtained from the first combination}}+(1-r')^{2}\mu(LWD)\nonumber\\&+\underbrace{r'(1+r')(1-r')^{2}\mu(LD)-r'^{2}(1-r')^{2}\mu(WLD)-r'^{2}(1-r')^{2}\mu(DLD)}_{\text{obtained from the second combination}}\nonumber\\
={}&(1+r')^{2}(1-r')^{2}\mu(LD)+(1-r')^{2}\mu(LWD)-r'(1-r')^{2}\mu(LDW)-(1-r')^{2}\mu(LDL)\nonumber\\&-r'^{2}(1-r')^{2}\mu(WLD)-r'^{2}(1-r')^{2}\mu(DLD),\label{pushforward_WD_q=0}
\end{align}
and we can write $\E_{r',0}\mu(LWD)=C_{1}-C_{2}$, where
\begin{align}
C_{1} ={}& r'(1-r')^{2}\mu(LDW)+r'(1-r')^{2}(1+r')\mu(LDD)+r'^{2}(1-r')^{2}\mu(LDL)+r'(1-r')^{2}\mu(LWD)\nonumber\\&+r'(1-r')^{2}\mu(WLD)+r'^{2}(1-r')^{2}(1+r')\mu(LLD)+r'^{2}(1-r')^{2}\mu(DLD)\label{C_{1}}
\end{align}
and 
\begin{align}
C_{2} ={}& r'(1-r')^{3}\mu(DLDW) + r'(1-r')^{3}\mu(LLDW) + r'(1-r')^{3}(1+r')\mu(DLDD) + r'(1-r')^{3}(1+r')\mu(LLDD) \nonumber\\&+ r'^{2}(1-r')^{3}\mu(DLDL) + r'^{2}(1-r')^{3}\mu(LLDL) + r'(1-r')^{3}\mu(LLWD) + r'(1-r')^{3}\mu(DLWD) \nonumber\\&+ r'(1-r')^{3}\mu(LWLD) + r'(1-r')^{3}\mu(DWLD) + r'^{2}(1-r')^{3}(1+r')\mu(LLLD) + r'^{2}(1-r')^{3}(1+r')\mu(DLLD) \nonumber\\&+ r'^{2}(1-r')^{3}\mu(LDLD) + r'^{2}(1-r')^{3}\mu(DDLD).\label{C_{2}}
\end{align}
It is worthwhile to note here that \emph{each} term of \emph{each} of $C_{1}$ and $C_{2}$ is non-negative. Another crucial observation is that, in $\E_{r',0}\mu(LWD) = C_{1}-C_{2}$, the coefficient corresponding to \emph{each} term is of the order of $r'$ (or smaller), so that when $r'$ is small, each term ought to be negligible. We now set our initial choice of the weight function to be (just as in \eqref{w_{0}} of \S\ref{sec:generalized_wt_fn_steps}, as well as our initial choice in \S\ref{sec:bond_1_wt_fn_steps}) $w_{0}(\mu)=\mu(D)+\mu(WD)+\mu(LWD)$. This yields the weight function inequality (which, in this case, is an identity):
\begin{align}
w_{0}\left(\E_{r',0}\mu\right) ={}& \E_{r',0}\mu(D) + \E_{r',0}\mu(WD) + \E_{r',0}\mu(LWD)\nonumber\\
={}& 2(1-r')\mu(WD) + (1+r')(1-r')\mu(DD) + 2r'(1-r')\mu(LD) + (1+r')^{2}(1-r')^{2}\mu(LD) \nonumber\\&+ (1-r')^{2}\mu(LWD) - r'(1-r')^{2}\mu(LDW) - (1-r')^{2}\mu(LDL) - r'^{2}(1-r')^{2}\mu(WLD) \nonumber\\&- r'^{2}(1-r')^{2}\mu(DLD) + r'(1-r')^{2}\mu(LDW)+r'(1-r')^{2}(1+r')\mu(LDD)+r'^{2}(1-r')^{2}\nonumber\\&\mu(LDL)+r'(1-r')^{2}\mu(LWD)+r'(1-r')^{2}\mu(WLD)+r'^{2}(1-r')^{2}(1+r')\mu(LLD)\nonumber\\&+r'^{2}(1-r')^{2}\mu(DLD)-C_{2}\nonumber\\
={}& \underbrace{2\mu(WD)} - 2r'\mu(WD) + \underbrace{\mu(DD)} - r'^{2}\mu(DD) + \underbrace{\mu(LD)} + (2r'-4r'^{2}+r'^{4})\mu(LD) \nonumber\\&+ \underbrace{\mu(LWD)} - (r'+r'^{2}-r'^{3})\mu(LWD) - (1-r')^{3}(1+r')\mu(LDL) + r'(1-r')^{3}\mu(WLD)\nonumber\\&+r'(1-r')^{2}(1+r')\mu(LDD)+r'^{2}(1-r')^{2}(1+r')\mu(LLD)-C_{2}\nonumber\\
={}& \underbrace{w_{0}(\mu)}_{\text{combining underbraced terms above}} - 2r'\mu(WD) - r'^{2}\mu(DD) + \underbrace{(2r'-4r'^{2}+r'^{4})\mu(LD)} \nonumber\\&- (r'+r'^{2}-r'^{3})\mu(LWD) - (1-r')^{3}(1+r')\mu(LDL) + \underbrace{r'(1-r')^{3}\mu(WLD)}\nonumber\\&+\underbrace{r'(1-r')^{2}(1+r')\mu(LDD)}+\underbrace{r'^{2}(1-r')^{2}(1+r')\mu(LLD)}-C_{2}.\label{wt_ineq_0}  
\end{align}
While of the same form as \eqref{gen_ineq_form}, the inequality in \eqref{wt_ineq_0} does not satisfy \eqref{desired_criterion}: in the underbraced terms in the final expression of \eqref{wt_ineq_0}, the coefficients $r'(1-r')^{3}$, $r'(1-r')^{2}(1+r')$ and $r'^{2}(1-r')^{2}(1+r')$ are non-negative (in fact, strictly positive for $r' \in (0,1)$), and the coefficient $(2r'-4r'^{2}+r'^{4})$ is also strictly positive for $r' \in (0,0.53918)$. Therefore, further adjustments are necessary to our initial weight function.

Note, here, that $-2r'\mu(WD)$ is one of the non-positive terms in the right side of \eqref{wt_ineq_0}, and it is worthwhile to check whether $-2r'\E_{r',0}\mu(WD)$ can be used to negate some of the positive terms (indicated by underbraces) in the final expression of \eqref{wt_ineq_0}. From \eqref{pushforward_WD_alternate}, we see that $-2r'\E_{r',0}\mu(WD)$ contributes the terms
\begin{equation}
-2r'(1-r')^{2}\mu(LDW)-2r'^{2}(1-r')^{2}\mu(LDL)-2r'(1-r')^{2}(1+r')\mu(LDD),\nonumber
\end{equation}
as well as the terms
\begin{align}
{}&-2r'^{2}(1-r')^{2}\mu(WLD)-2r'^{2}(1-r')^{2}(1+r')\mu(LLD)-2r'^{2}(1-r')^{2}\mu(DLD)\nonumber\\
={}& -2r'^{2}(1-r')^{2}\mu(LD) - 2r'^{3}(1-r')^{2}\mu(LLD),\label{combination_1}
\end{align}
and these can be used to negate, partially, the term $(2r'-4r'^{2}+r'^{4})\mu(LD)$ on the right side of \eqref{wt_ineq_0} when the coefficient $(2r'-4r'^{2}+r'^{4})$ is positive. This serves as a good motivation for us to set the first adjustment to be $w_{1}(\mu)= w_{0}(\mu)-2r'\mu(WD)$. This adjustment updates the weight function inequality of \eqref{wt_ineq_0} to:
\begin{align}
w_{1}\left(\E_{r',0}\mu\right)={}& w_{1}(\mu) + 2r'\mu(WD) - 2r'\mu(WD) - r'^{2}\mu(DD) + (2r'-4r'^{2}+r'^{4})\mu(LD) - (r'+r'^{2}-r'^{3})\mu(LWD) \nonumber\\&- (1-r')^{3}(1+r')\mu(LDL) + r'(1-r')^{3}\mu(WLD) + r'(1-r')^{2}(1+r')\mu(LDD) \nonumber\\&+ r'^{2}(1-r')^{2}(1+r')\mu(LLD) - C_{2} - 2r(1-r')^{2}\mu(LDW) - 2r'^{2}(1-r')^{2}\mu(LDL) \nonumber\\&- 2r'(1-r')^{2}(1+r')\mu(LDD) - 2r'(1-r')^{2}\mu(LWD) \underbrace{- 2r'^{2}(1-r')^{2}\mu(WLD)} \nonumber\\&\underbrace{- 2r'^{2}(1-r')^{2}(1+r')\mu(LLD) - 2r'^{2}(1-r')^{2}\mu(DLD)}\nonumber\\
={}& w_{1}(\mu) - r'^{2}\mu(DD) \underbrace{+ (2r'-4r'^{2}+r'^{4})\mu(LD)} \underbrace{- (r'+r'^{2}-r'^{3})\mu(LWD)}_{\text{(i)}} \underbrace{- (1-r')^{3}(1+r')\mu(LDL)}_{\text{(ii)}} \nonumber\\&+ r'(1-r')^{3}\mu(WLD) \underbrace{+ r'(1-r')^{2}(1+r')\mu(LDD)}_{\text{(iii)}} \underbrace{+ r'^{2}(1-r')^{2}(1+r')\mu(LLD)} - C_{2} \nonumber\\&- 2r'(1-r')^{2}\mu(LDW) \underbrace{- 2r'^{2}(1-r')^{2}\mu(LDL)}_{\text{(ii)}} \underbrace{- 2r'(1-r')^{2}(1+r')\mu(LDD)}_{\text{(iii)}} \underbrace{- 2r'(1-r')^{2}\mu(LWD)}_{\text{(i)}}\nonumber\\& \underbrace{- 2r'^{2}(1-r')^{2}\mu(LD) - 2r'^{3}(1-r')^{2}\mu(LLD)}_{\text{combining the underbraced terms in the previous step using \eqref{combination_1}}}\nonumber\\
={}& w_{1}(\mu) - r'^{2}\mu(DD) \underbrace{+ (2r'-6r'^{2}+4r'^{3}-r'^{4})\mu(LD)}_{\text{combining underbraced terms involving } \mu(LD)} \underbrace{- (3r'-3r'^{2}+r'^{3})\mu(LWD)}_{\text{combining underbraced terms labeled (i)}}\nonumber\\& \underbrace{-(1-r')^{2}(1+r'^{2})\mu(LDL)}_{\text{combining underbraced terms labeled (ii)}} + r'(1-r')^{3}\mu(WLD) \underbrace{-r'(1-r')^{2}(1+r')\mu(LDD)}_{\text{combining underbraced terms labeled (iii)}}\nonumber\\& \underbrace{+r'^{2}(1-r')^{3}\mu(LLD)}_{\text{combining underbraced terms involving } \mu(LLD)} - 2r'(1-r')^{2}\mu(LDW) - C_{2}\nonumber\\
={}& w_{1}(\mu) - r'^{2}\mu(DD) + (2r'-6r'^{2}+4r'^{3}-r'^{4})\mu(LD) - (3r'-3r'^{2}+r'^{3})\mu(LWD)\nonumber\\& \underbrace{- (1-r')^{2}(1+r'^{2})\mu(LDL) - r'(1-r')^{2}(1+r')\mu(LDD) - 2r'(1-r')^{2}\mu(LDW)}_{\text{obtain } - r'(1-r')^{3}\mu(LD) \text{ from these}} \nonumber\\& \underbrace{+ r'(1-r')^{3}\mu(WLD) + r'^{2}(1-r')^{3}\mu(LLD)}_{\text{obtain } r'(1-r')^{3}\mu(LD) \text{ from these}} - C_{2}\nonumber\\
={}& w_{1}(\mu) - r'^{2}\mu(DD) + (2r'-6r'^{2}+4r'^{3}-r'^{4})\mu(LD) - (3r'-3r'^{2}+r'^{3})\mu(LWD)\nonumber\\& \underbrace{- r'(1-r')^{3}\mu(LD)} - (1-r')^{2}(1-r'+2r'^{2})\mu(LDL) - 2r'^{2}(1-r')^{2}\mu(LDD) \nonumber\\&- r'(1-r')^{2}(1+r')\mu(LDW) \underbrace{+ r'(1-r')^{3}\mu(LD)} - r'(1-r')^{4}\mu(LLD) \nonumber\\&- r'(1-r')^{3}\mu(DLD) - C_{2}\nonumber\\ 
={}& w_{1}(\mu) - r'^{2}\mu(DD) \underbrace{+ (2r'-6r'^{2}+4r'^{3}-r'^{4})\mu(LD)} - (3r'-3r'^{2}+r'^{3})\mu(LWD)\nonumber\\& - (1-r')^{2}(1-r'+2r'^{2})\mu(LDL) - 2r'^{2}(1-r')^{2}\mu(LDD) - r'(1-r')^{2}(1+r')\mu(LDW) \nonumber\\& - r'(1-r')^{4}\mu(LLD) - r'(1-r')^{3}\mu(DLD) - C_{2}.\label{wt_ineq_1_explained}
\end{align}
One can see that \eqref{wt_ineq_1_explained} becomes the inequality in \eqref{wt_ineq_1} if $C_{2}$, which consists of non-negative summands (as remarked right after \eqref{C_{2}}), is removed. The only problematic term on the right side of \eqref{wt_ineq_1_explained}, which \emph{may} prevent it from satisfying \eqref{desired_criterion}, has been indicated by an underbrace: the coefficient $(2r'-6r'^{2}+4r'^{3}-r'^{4})$ of $\mu(LD)$ is strictly negative for $r' \in [0.4564,1]$, so that if we were to consider the regime where $s'=0$ and $r' \in [0.4564,1]$, no further adjustments to our weight function would be required. Since, however, we \emph{do} consider $r' < 0.4564$ in the regime described by \ref{bond_regime_2}, we choose to make use of the term $-(3r'-3r'^{2}+r'^{3})\mu(LWD)$, on the right side of \eqref{wt_ineq_1_explained}, to negate, \emph{partially}, the problematic, underbraced term $(2r'-6r'^{2}+4r'^{3}-r'^{4})\mu(LD)$, following the idea outlined in \S\ref{subsubsec:adjust_gen}.

We define our second adjustment to the weight function by setting $w_{2}(\mu) = w_{1}(\mu) - (3r'-3r'^{2}+r'^{3})\mu(LWD)$, which is exactly the same as \eqref{adjust_2}. Recall that $\E_{r',0}\mu(LWD) = C_{1}-C_{2}$, with $C_{1}$ and $C_{2}$ as defined in \eqref{C_{1}} and \eqref{C_{2}}. As explained in \eqref{ith_wt_fn_ineq}, incorporating the new adjustment into \eqref{wt_ineq_1_explained} yields
\begin{align}
w_{2}\left(\E_{r',0}\mu\right) ={}& w_{2}(\mu) \underbrace{+ (3r'-3r'^{2}+r'^{3})\mu(LWD)} - r'^{2}\mu(DD) + (2r'-6r'^{2}+4r'^{3}-r'^{4})\mu(LD) \nonumber\\&\underbrace{- (3r'-3r'^{2}+r'^{3})\mu(LWD)} - (1-r')^{2}(1-r'+2r'^{2})\mu(LDL) - 2r'^{2}(1-r')^{2}\mu(LDD) \nonumber\\& - r'(1-r')^{2}(1+r')\mu(LDW) - r'(1-r')^{4}\mu(LLD) - r'(1-r')^{3}\mu(DLD) \underbrace{- C_{2}}_{\text{(i)}} \nonumber\\& - (3r'-3r'^{2}+r'^{3})C_{1} \underbrace{+ (3r'-3r'^{2}+r'^{3})C_{2}}_{\text{(i)}}\nonumber\\
={}& w_{2}(\mu) - r'^{2}\mu(DD) + (2r'-6r'^{2}+4r'^{3}-r'^{4})\mu(LD) \underbrace{- (1-r')^{2}(1-r'+2r'^{2})\mu(LDL)}_{\text{(i)}} \nonumber\\& \underbrace{- 2r'^{2}(1-r')^{2}\mu(LDD)}_{\text{(ii)}} \underbrace{- r'(1-r')^{2}(1+r')\mu(LDW)}_{\text{(iii)}} \underbrace{- r'(1-r')^{4}\mu(LLD)}_{\text{(iv)}} \underbrace{- r'(1-r')^{3}\mu(DLD)}_{\text{(v)}} \nonumber\\& \underbrace{- (3r'-3r'^{2}+r'^{3})r'(1-r')^{2}\mu(LDW)}_{\text{(iii)}} \underbrace{- (3r'-3r'^{2}+r'^{3})r'(1-r')^{2}(1+r')\mu(LDD)}_{\text{(ii)}} \nonumber\\& \underbrace{- (3r'-3r'^{2}+r'^{3})r'^{2}(1-r')^{2}\mu(LDL)}_{\text{(i)}} - (3r'-3r'^{2}+r'^{3})r'(1-r')^{2}\mu(LWD) \nonumber\\&- (3r'-3r'^{2}+r'^{3})r'(1-r')^{2}\mu(WLD) \underbrace{- (3r'-3r'^{2}+r'^{3})r'^{2}(1-r')^{2}(1+r')\mu(LLD)}_{\text{(iv)}} \nonumber\\& \underbrace{- (3r'-3r'^{2}+r'^{3})r'^{2}(1-r')^{2}\mu(DLD)}_{\text{(v)}} \underbrace{- (1-r')^{3}C_{2}}_{\text{combining underbraced terms labeled (i)}} \ \text{(using \eqref{C_{1}})}\nonumber\\
={}& w_{2}(\mu) - r'^{2}\mu(DD) + (2r'-6r'^{2}+4r'^{3}-r'^{4})\mu(LD) \nonumber\\&\underbrace{- (1-r')^{2}(r'^{5}-3r'^{4}+3r'^{3}+2r'^{2}-r'+1)\mu(LDL)}_{\text{combining underbraced terms labeled (i)}} \underbrace{- r'^{2}(1-r')^{2}(r'^{3}-2r'^{2}+5)\mu(LDD)}_{\text{combining underbraced terms labeled (ii)}}\nonumber\\& \underbrace{- r'(1-r')^{2}(r'^{3}-3r'^{2}+4r'+1)\mu(LDW)}_{\text{combining underbraced terms labeled (iii)}} \underbrace{- r'(1-r')^{2}(r'^{5}-2r'^{4}+4r'^{2}-2r'+1)\mu(LLD)}_{\text{combining underbraced terms labeled (iv)}}\nonumber\\& \underbrace{- r'(1-r')^{2}(r'^{4}-3r'^{3}+3r'^{2}-r'+1)\mu(DLD)}_{\text{combining underbraced terms labeled (v)}} - (3r'-3r'^{2}+r'^{3})r'(1-r')^{2}\mu(LWD) \nonumber\\&- (3r'-3r'^{2}+r'^{3})r'(1-r')^{2}\mu(WLD) - (1-r')^{3}C_{2}\nonumber\\
={}& w_{2}(\mu) - r'^{2}\mu(DD) + (2r'-6r'^{2}+4r'^{3}-r'^{4})\mu(LD) \nonumber\\&\underbrace{- (1-r')^{2}(r'^{5}-3r'^{4}+3r'^{3}+2r'^{2}-r'+1)\mu(LDL) - r'^{2}(1-r')^{2}(r'^{3}-2r'^{2}+5)\mu(LDD)}_{\text{(vi)}}\nonumber\\& \underbrace{- r'(1-r')^{2}(r'^{3}-3r'^{2}+4r'+1)\mu(LDW)}_{\text{(vi)}} \underbrace{- r'(1-r')^{2}(r'^{5}-2r'^{4}+4r'^{2}-2r'+1)\mu(LLD)}_{\text{(vii)}}\nonumber\\& \underbrace{- r'(1-r')^{2}(r'^{4}-3r'^{3}+3r'^{2}-r'+1)\mu(DLD) - (3r'-3r'^{2}+r'^{3})r'(1-r')^{2}\mu(WLD)}_{\text{(vii)}} \nonumber\\&- (3r'-3r'^{2}+r'^{3})r'(1-r')^{2}\mu(LWD) - (1-r')^{3}C_{2}.\label{intermediate_1} 
\end{align}
Combining the terms that are grouped together using underbraces labeled (vi) in the final expression for \eqref{intermediate_1}, we obtain:
\begin{align}
{}&- (1-r')^{2}(r'^{5}-3r'^{4}+3r'^{3}+2r'^{2}-r'+1)\mu(LDL) - r'^{2}(1-r')^{2}(r'^{3}-2r'^{2}+5)\mu(LDD) \nonumber\\&- r'(1-r')^{2}(r'^{3}-3r'^{2}+4r'+1)\mu(LDW)\nonumber\\
={}& - r'^{2}(1-r')^{2}(r'^{3}-2r'^{2}+5)\mu(LD) \nonumber\\&- \left\{(1-r')^{2}(r'^{5}-3r'^{4}+3r'^{3}+2r'^{2}-r'+1) - r'^{2}(1-r')^{2}(r'^{3}-2r'^{2}+5)\right\}\mu(LDL) \nonumber\\&- \left\{r'(1-r')^{2}(r'^{3}-3r'^{2}+4r'+1) - r'^{2}(1-r')^{2}(r'^{3}-2r'^{2}+5)\right\}\mu(LDW)\nonumber\\ 
={}& - r'^{2}(1-r')^{2}(r'^{3}-2r'^{2}+5)\mu(LD) - (1-r')^{2}(1-r'-3r'^{2}+3r'^{3}-r'^{4})\mu(LDL)\nonumber\\& - r'(1-r')^{2}(1-r'-3r'^{2}+3r'^{3}-r'^{4})\mu(LDW).\label{intermediate_2}
\end{align}
Combining the terms that are grouped together using underbraces labeled (vii) in the final expression for \eqref{intermediate_1}, we obtain:
\begin{align}
{}&- r'(1-r')^{2}(r'^{5}-2r'^{4}+4r'^{2}-2r'+1)\mu(LLD) - r'(1-r')^{2}(r'^{4}-3r'^{3}+3r'^{2}-r'+1)\mu(DLD) \nonumber\\&- (3r'-3r'^{2}+r'^{3})r'(1-r')^{2}\mu(WLD)\nonumber\\
={}& - (3r'-3r'^{2}+r'^{3})r'(1-r')^{2}\mu(LD) \nonumber\\&- \left\{r'(1-r')^{2}(r'^{5}-2r'^{4}+4r'^{2}-2r'+1) - (3r'-3r'^{2}+r'^{3})r'(1-r')^{2}\right\}\mu(LLD) \nonumber\\&- \left\{r'(1-r')^{2}(r'^{4}-3r'^{3}+3r'^{2}-r'+1) - (3r'-3r'^{2}+r'^{3})r'(1-r')^{2}\right\}\mu(DLD)\nonumber\\
={}& - (3r'-3r'^{2}+r'^{3})r'(1-r')^{2}\mu(LD) - r'(1-r')^{2}(r'^{5}-2r'^{4}-r'^{3}+7r'^{2}-5r'+1)\mu(LLD)\nonumber\\& - r'(1-r')^{6}\mu(DLD).\label{intermediate_3}
\end{align}
Incorporating \eqref{intermediate_2} and \eqref{intermediate_3} into \eqref{intermediate_1} yields:
\begin{align}
w_{2}\left(\E_{r',0}\mu\right) ={}& w_{2}(\mu) - r'^{2}\mu(DD) \underbrace{+ (2r'-6r'^{2}+4r'^{3}-r'^{4})\mu(LD) - r'^{2}(1-r')^{2}(r'^{3}-2r'^{2}+5)\mu(LD)}\nonumber\\& - (1-r')^{2}(1-r'-3r'^{2}+3r'^{3}-r'^{4})\mu(LDL) - r'(1-r')^{2}(1-r'-3r'^{2}+3r'^{3}-r'^{4})\mu(LDW) \nonumber\\&\underbrace{- (3r'-3r'^{2}+r'^{3})r'(1-r')^{2}\mu(LD)} - r'(1-r')^{2}(r'^{5}-2r'^{4}-r'^{3}+7r'^{2}-5r'+1)\mu(LLD)\nonumber\\& - r'(1-r')^{6}\mu(DLD) - (3r'-3r'^{2}+r'^{3})r'(1-r')^{2}\mu(LWD) - (1-r')^{3}C_{2}\nonumber\\
={}& w_{2}(\mu) - r'^{2}\mu(DD) \underbrace{+ (2r'-14r'^{2}+23r'^{3}-14r'^{4}+3r'^{6}-r'^{7})\mu(LD)}_{\text{combining the underbraced terms above}} \nonumber\\&- (1-r')^{2}(1-r'-3r'^{2}+3r'^{3}-r'^{4})\mu(LDL) - r'(1-r')^{2}(1-r'-3r'^{2}+3r'^{3}-r'^{4})\mu(LDW) \nonumber\\&- r'(1-r')^{2}(r'^{5}-2r'^{4}-r'^{3}+7r'^{2}-5r'+1)\mu(LLD) - r'(1-r')^{6}\mu(DLD) \nonumber\\&- (3r'-3r'^{2}+r'^{3})r'(1-r')^{2}\mu(LWD) - (1-r')^{3}C_{2}.\label{wt_ineq_2_explained}
\end{align}
Observe that if we were to remove the term $-(1-r')^{3}C_{2}$, keeping in mind that $C_{2}$ consists of non-negative summands (as mentioned after \eqref{C_{2}}), then \eqref{wt_ineq_2_explained} yields the inequality in \eqref{adjust_2}. Next, we note that 
\begin{enumerate}
\item the polynomial $2r'-14r'^{2}+23r'^{3}-14r'^{4}+3r'^{6}-r'^{7}$ is strictly negative for $r' \in (0.201383,1]$,
\item the polynomial $1-r'-3r'^{2}+3r'^{3}-r'^{4}$ is strictly positive for all $r' \in [0,0.52779)$,
\item the polynomial $r'^{5}-2r'^{4}-r'^{3}+7r'^{2}-5r'+1$ is strictly positive for all $r' \geqslant 0$,
\end{enumerate}
so that \eqref{wt_ineq_2_explained} satisfies the criterion in \eqref{desired_criterion} as long as we have $r' \in (0.201383,0.52779)$. However, we wish to go beyond this range of values of $r'$, as indicated by \ref{bond_regime_2}. 

When $r' \leqslant 0.201382$, the only term we need be concerned about is the one indicated by an underbrace in the final expression of \eqref{wt_ineq_2_explained}, i.e.\ $(2r'-14r'^{2}+23r'^{3}-14r'^{4}+3r'^{6}-r'^{7})\mu(LD)$. We now motivate our approach for partially negating this term, though, as emphasized earlier to the reader, there is, truly, no \emph{unique} way of proceeding. It is quite plausible that other, \emph{similar} approaches yield somewhat better bounds, but intuitively, we do not see how our bounds can be \emph{significantly} improved. It is the order of magnitude of the coefficient of $\mu(LD)$ in \eqref{wt_ineq_2_explained}, i.e.\ $2r'$ (when $r'$ is small), that prevents us from covering a regime larger than that in \eqref{bond_regime_2}, since the existing negative terms on the right side of \eqref{wt_ineq_2_explained} are collectively unable to negate this coefficient when $r'$ is ``too small". We explain this in the following paragraph.

On the right side of \eqref{wt_ineq_2_explained}, we have, on one hand, negative terms involving $\mu(LDW)$ and $\mu(LDL)$, and on the other, negative terms involving $\mu(DLD)$ and $\mu(LLD)$. At a glance, what seem to be missing are negative terms involving $\mu(LDD)$ and $\mu(WLD)$ (had there been a couple of negative terms involving $\mu(LDD)$ and $\mu(WLD)$, with coefficients whose order of magnitude is $r'$, for $r'$ is small, on the right side of \eqref{wt_ineq_2_explained}, these could have helped take care of $(2r'-14r'^{2}+23r'^{3}-14r'^{4}+3r'^{6}-r'^{7})\mu(LD)$, since $\mu(LDW)+\mu(LDL)+\mu(LDD)=\mu(DLD)+\mu(LLD)+\mu(WLD)=\mu(LD)$). Next, among the terms of $C_{2}$ in \eqref{C_{2}}, we have
\begin{enumerate}
\item the terms $2r'(1-r')^{3}\mu(DWLD)$ and $r'(1-r')^{3}\mu(LWLD)$, but no term involving $\mu(WWLD)$, which is why we cannot directly obtain from $-(1-r')^{3}C_{2}$ a negative term involving $\mu(WLD)$;
\item the terms $r'(1-r')^{3}(1+r')\mu(DLDD)$ and $r'(1-r')^{3}(1+r')\mu(LLDD)$, but no term involving $\mu(WLDD)$, which is why we cannot directly obtain from $-(1-r')^{3}C_{2}$ a negative term involving $\mu(LDD)$.
\end{enumerate}
We shall, therefore, avail the aid of $-r'^{2}(3-3r'+r'^{2})(1-r')^{2}\mu(LWD)$ and $-r'^{2}\mu(DD)$ that are present in the right side of \eqref{wt_ineq_2_explained}, and the terms $-r'(1-r')^{6}\mu(LLWD)$ and $-r'(1-r')^{6}\mu(LLDW)$ from the expansion of $-(1-r')^{3}C_{2}$, to define the third adjustment to our weight function. It can be checked that the remaining terms of $C_{2}$, namely those involving $\mu(DLDW)$, $\mu(DLDL)$, $\mu(DDLD)$, $\mu(DDLD)$, $\mu(LLDL)$, $\mu(LDLD)$ and $\mu(LLLD)$, are either of no use in generating (via the probabilities of the corresponding cylinder sets induced by the pushforward measure $\E_{r',0}\mu$) the terms that are missing above, or generate the missing terms with coefficients that are of an order of magnitude far too small to be of any significant use to us. 

We set the new, adjusted weight function to be
\begin{align}\label{adjust_3}
w_{3}(\mu) ={}& w_{2}(\mu) - r'(1-r')^{6}\{\mu(LLWD)+\mu(LLDW)\} - r'^{2}\mu(WDD) - (3r'-3r'^{2}+r'^{3})r'(1-r')^{2}\mu(LWD).
\end{align}
To understand how \eqref{adjust_3} impacts the weight function inequality in \eqref{wt_ineq_2_explained}, we need to compute, \emph{partially}, the pushforward measures $\E_{r',0}\mu(LLWD)$, $\E_{r',0}\mu(LLDW)$ and $\E_{r',0}\mu(WDD)$. Recall that $\E_{r',0}\mu(LLWD)$ equals the probability of the event $\{\E_{r',0}\eta(0)=\E_{r',0}\eta(1)=L,\E_{r',0}\eta(2)=W,\E_{r',0}\eta(3)=D\}$ when $\eta$ is a random configuration with law $\mu$, and we consider the contributions to $\E_{r',0}\mu(LLWD)$ arising from the following events \emph{only}:
\begin{enumerate}
\item the contribution of the event $\{\eta(0)=\eta(1)=W,\eta(2)=L,\eta(3)=D\}$ to $\E_{r',0}\mu(LLWD)$ is
\begin{align}
{}&r(1-r')^{2}\mu(WWLDW)+r'(1-r')^{2}(1+r')\mu(WWLDD)+r'^{2}(1-r')^{2}\mu(WWLDL)\nonumber\\
={}& r'(1-r')^{2}\mu(WWLD)+r'^{2}(1-r')^{2}\mu(WWLDD)-r'(1-r')^{3}\mu(WWLDL);\label{LLWD_contribution_1}
\end{align}
\item the contribution of the event $\{\eta(1)=\eta(2)=W,\eta(3)=L,\eta(4)=D\}$ to $\E_{r',0}\mu(LLWD)$ is
\begin{align}
{}& r'(1-r')^{2}\mu(WWWLD)+r'^{2}(1-r')^{2}\mu(LWWLD)+r'^{2}(1-r')^{2}\mu(DWWLD)\nonumber\\
={}& r'^{2}(1-r')^{2}\mu(WWLD)+r'(1-r')^{3}\mu(WWWLD);\label{LLWD_contribution_2}
\end{align}
\item the contribution of the event $\{\eta(0) \in \{D,L\}, \eta(1)=W, \eta(2)=L, \eta(3)=\eta(4)=D\}$ to $\E_{r',0}\mu(LLWD)$ is
\begin{align}
{}&r'^{2}(1-r')^{2}(1+r')\mu(DWLDD)+r'^{2}(1-r')^{2}(1+r')\mu(LWLDD).\label{LLWD_contribution_3}
\end{align}
\end{enumerate}
Combining \eqref{LLWD_contribution_1}, \eqref{LLWD_contribution_2} and \eqref{LLWD_contribution_3}, and letting $C'_{LLWD}$ denote the contribution from all cases \emph{not} considered above, we obtain:
\begin{align}
\E_{r',0}\mu(LLWD) ={}& \underbrace{r'(1-r')^{2}\mu(WWLD)}_{\text{(i)}}+\underbrace{r'^{2}(1-r')^{2}\mu(WWLDD)}_{\text{(ii)}}-r'(1-r')^{3}\mu(WWLDL)\nonumber\\&+\underbrace{r'^{2}(1-r')^{2}\mu(WWLD)}_{\text{(i)}}+r'(1-r')^{3}\mu(WWWLD)+\underbrace{r'^{2}(1-r')^{2}(1+r')\mu(DWLDD)}_{\text{(ii)}}\nonumber\\&+\underbrace{r'^{2}(1-r')^{2}(1+r')\mu(LWLDD)}_{\text{(ii)}}+C'_{LLWD}\nonumber\\
={}& \underbrace{r'(1-r')^{2}(1+r')\mu(WWLD)}_{\text{summing underbraced terms labeled (i)}} - r'(1-r')^{3}\mu(WWLDL) \nonumber\\&+ \underbrace{r'^{2}(1-r')^{2}(1+r')\mu(WLDD) - r'^{3}(1-r')^{2}\mu(WWLDD)}_{\text{summing underbraced terms labeled (ii)}} + C_{LLWD},\label{pushforward_LLWD_q=0}
\end{align}
where we set $C_{LLWD} = r'(1-r')^{3}\mu(WWWLD) + C'_{LLWD}$. Next, for computing $\E_{r',0}\mu(LLDW)$, which is the probability of the event $\{\eta(0)=\eta(1)=L,\eta(2)=D,\eta(3)=W\}$ where $\eta$ follows the distribution $\mu$, we consider only the contribution arising from the event $\{\eta(0)=\eta(1)=W,\eta(2)=L,\eta(3)=D,\eta(4)=L\}$, allowing us to write
\begin{align}
\E_{r',0}\mu(LLDW) ={}& r'^{2}(1-r')^{2}\mu(WWLDL)+C_{LLDW},\label{pushforward_LLDW_q=0}
\end{align}
where $C_{LLDW}$ denotes the contribution from all the cases \emph{not} considered above. For the computation of $\E_{r',0}\mu(WDD)$, which is the probability of the event $\{\E_{r',0}\eta(0)=W,\E_{r',0}\eta(1)=\E_{r',0}\eta(2)=D\}$ where $\eta$ follows the distribution $\mu$, we consider  
\begin{enumerate}
\item the contribution from the event $\{\eta(0)=L,\eta(1)=\eta(2)=D\}$, given by
\begin{align}
{}&(1-r')^{3}(1+r)\mu(LDDW)+(1-r')^{3}(1+r')^{2}\mu(LDDD)+r'(1-r')^{3}(1+r')\mu(LDDL)\nonumber\\
={}& r'(1-r')^{3}(1+r')\mu(LDD)+(1-r')^{4}(1+r')\mu(LDDW)+(1-r')^{3}(1+r')\mu(LDDD);\label{WDD_contribution_1}
\end{align}
\item and the contribution from the event $\{\eta(1)=L,\eta(2)=\eta(3)=D\}$, given by
\begin{align}
{}& r'(1-r')^{3}(1+r')\mu(WLDD)+r'(1-r')^{3}(1+r')\mu(DLDD)+r'(1-r')^{3}(1+r')^{2}\mu(LLDD)\nonumber\\
={}& r'(1-r')^{3}(1+r')\mu(LDD)+r'^{2}(1-r')^{3}(1+r')\mu(LLDD).\label{WDD_contribution_2}
\end{align}
\end{enumerate}
Combining \eqref{WDD_contribution_1} and \eqref{WDD_contribution_2}, we obtain
\begin{align}
\E_{r',0}\mu(WDD) ={}& r'(1-r')^{3}(1+r')\mu(LDD)+(1-r')^{4}(1+r')\mu(LDDW)+(1-r')^{3}(1+r')\mu(LDDD) \nonumber\\&+ r'(1-r')^{3}(1+r')\mu(LDD)+r'^{2}(1-r')^{3}(1+r')\mu(LLDD) + C'_{WDD}\nonumber\\
={}& 2r'(1-r')^{3}(1+r')\mu(LDD)+C_{WDD},\label{pushforward_WDD_q=0}
\end{align}
where $C'_{WDD}$ is the contribution to $\E_{r',0}\mu(WDD)$ of all the cases \emph{not} considered above, and we set $C_{WDD} = C'_{WDD}+(1-r')^{4}(1+r')\mu(LDDW)+(1-r')^{3}(1+r')\mu(LDDD)+r'^{2}(1-r')^{3}(1+r')\mu(LLDD)$. Finally, we write
\begin{align}\label{LWD_pushforward_2_q=0}
\E_{r',0}\mu(LWD) ={}& r'(1-r')^{2}(1+r')\mu(WLDD)+r'(1-r')^{2}\mu(WWLD)+C_{LWD},
\end{align}
where $C_{LWD}$ captures the rest of the expression for $\E_{r',0}\mu(LWD)$ that we have already explicitly computed earlier. It is important to keep in mind here that each term of $C_{LWD}$ is non-negative.

From \eqref{wt_ineq_2_explained}, \eqref{adjust_3}, \eqref{pushforward_LLWD_q=0}, \eqref{pushforward_LLDW_q=0}, \eqref{pushforward_WDD_q=0} and \eqref{LWD_pushforward_2_q=0}, setting $C_{3}=(1-r')^{3}C_{2} - r'(1-r')^{6}\mu(LLWD) - r'(1-r')^{6}\mu(LLDW) - 2r'(1-r')^{6}\mu(DWLD) - r'(1-r')^{6}\mu(LWLD) - r'(1-r')^{6}(1+r')\mu(DLDD) - r'(1-r')^{6}(1+r')\mu(LLDD)$, and implementing the same idea as that used to derive \eqref{ith_wt_fn_ineq}, we obtain:
\begin{align}
w_{3}\left(\E_{r',0}\mu\right) ={}& w_{3}(\mu) \underbrace{+ r'(1-r')^{6}\mu(LLWD) + r'(1-r')^{6}\mu(LLDW)}_{\text{(i)}} \underbrace{+ r'^{2}\mu(WDD)}_{\text{(ii)}} \nonumber\\&\underbrace{+ (3r'-3r'^{2}+r'^{3})r'(1-r')^{2}\mu(LWD)}_{\text{(iii)}} \underbrace{- r'^{2}\mu(DD)}_{\text{(ii)}} + (2r'-14r'^{2}+23r'^{3}-14r'^{4}+3r'^{6}-r'^{7})\mu(LD) \nonumber\\&- (1-r')^{2}(1-r'-3r'^{2}+3r'^{3}-r'^{4})\mu(LDL) - r'(1-r')^{2}(1-r'-3r'^{2}+3r'^{3}-r'^{4})\mu(LDW) \nonumber\\&- r'(1-r')^{2}(r'^{5}-2r'^{4}-r'^{3}+7r'^{2}-5r'+1)\mu(LLD) - r'(1-r')^{6}\mu(DLD) \nonumber\\&\underbrace{- (3r'-3r'^{2}+r'^{3})r'(1-r')^{2}\mu(LWD)}_{\text{(iii)}} \underbrace{- r'(1-r')^{6}\mu(LLWD) - r'(1-r')^{6}\mu(LLDW)}_{\text{(i)}} \nonumber\\&- 2r'(1-r')^{6}\mu(DWLD) - r'(1-r')^{6}\mu(LWLD) - r'(1-r')^{6}(1+r')\mu(DLDD) \nonumber\\&- r'(1-r')^{6}(1+r')\mu(LLDD) - C_{3} - r'(1-r')^{6}\E_{r',0}\mu(LLWD) - r'(1-r')^{6}\E_{r',0}\mu(LLDW) \nonumber\\&- r'^{2}\E_{r',0}\mu(WDD) - (3r'-3r'^{2}+r'^{3})r'(1-r')^{2}\E_{r',0}\mu(LWD)\nonumber\\
={}& w_{3}(\mu) + (2r'-14r'^{2}+23r'^{3}-14r'^{4}+3r'^{6}-r'^{7})\mu(LD) \underbrace{- r'^{2}\mu(LDD) - r'^{2}\mu(DDD)}_{\text{combining underbraced terms labeled (ii)}}\nonumber\\&\underbrace{- (1-r')^{2}(1-r'-3r'^{2}+3r'^{3}-r'^{4})\mu(LDL) - r'(1-r')^{2}(1-r'-3r'^{2}+3r'^{3}-r'^{4})\mu(LDW)}_{\text{(iv)}} \nonumber\\&\underbrace{- r'(1-r')^{2}(r'^{5}-2r'^{4}-r'^{3}+7r'^{2}-5r'+1)\mu(LLD) - r'(1-r')^{6}\mu(DLD)}_{\text{(v)}} \nonumber\\&\underbrace{- 2r'(1-r')^{6}\mu(DWLD) - r'(1-r')^{6}\mu(LWLD)}_{\text{(v)}} \underbrace{- r'(1-r')^{6}(1+r')\mu(DLDD)}_{\text{(iv)}} \nonumber\\&\underbrace{- r'(1-r')^{6}(1+r')\mu(LLDD)}_{\text{(iv)}} \underbrace{- r'^{2}(1-r')^{8}(1+r')\mu(WWLD)}_{\text{(v)}} \underbrace{+ r'^{2}(1-r')^{9}\mu(WWLDL)}_{\text{(iv)}} \nonumber\\&\underbrace{- r'^{3}(1-r')^{8}(1+r')\mu(WLDD)}_{\text{(iv)}} + r'^{4}(1-r')^{8}\mu(WWLDD) - r'(1-r')^{6}C_{LLWD} \nonumber\\&\underbrace{- r'^{3}(1-r')^{8}\mu(WWLDL)}_{\text{(iv)}} - r'(1-r')^{6}C_{LLDW} \underbrace{- 2r'^{3}(1-r')^{3}(1+r')\mu(LDD)}_{\text{(iv)}} - r'^{2}C_{WDD} \nonumber\\&\underbrace{- r'^{3}(3-3r'+r'^{2})(1-r')^{4}(1+r')\mu(WLDD)}_{\text{(iv)}} \underbrace{- r'^{3}(3-3r'+r'^{2})(1-r')^{4}\mu(WWLD)}_{\text{(v)}} \nonumber\\&- r'^{2}(3-3r'+r'^{2})(1-r')^{2}C_{LWD} - C_{3},\label{intermediate_6}
\end{align}
where the last step is obtained by adding terms grouped by underbraces labeled (i) and (iii) in the previous step, and substituting from \eqref{pushforward_LLWD_q=0}, \eqref{pushforward_LLDW_q=0}, \eqref{pushforward_WDD_q=0} and \eqref{LWD_pushforward_2_q=0}. Adding the terms indicated by underbraces labeled (iv) in the final expression of \eqref{intermediate_6}, we obtain
\begin{align}
{}& - (1-r')^{2}(1-r'-3r'^{2}+3r'^{3}-r'^{4})\mu(LDL) - r'(1-r')^{2}(1-r'-3r'^{2}+3r'^{3}-r'^{4})\mu(LDW) \nonumber\\&\underbrace{- r'(1-r')^{6}(1+r')\mu(DLDD) - r'(1-r')^{6}(1+r')\mu(LLDD) - r'^{3}(1-r')^{8}(1+r')\mu(WLDD)}_{\text{(a)}} \nonumber\\&- 2r'^{3}(1-r')^{3}(1+r')\mu(LDD) \underbrace{- r'^{3}(3-3r'+r'^{2})(1-r')^{4}(1+r')\mu(WLDD)}_{\text{(a)}} \underbrace{+ r'^{2}(1-r')^{9}\mu(WWLDL)}_{\text{(b)}} \nonumber\\&\underbrace{- r'^{3}(1-r')^{8}\mu(WWLDL)}_{\text{(b)}} \nonumber\\
={}& - (1-r')^{2}(1-r'-3r'^{2}+3r'^{3}-r'^{4})\mu(LDL) - r'(1-r')^{2}(1-r'-3r'^{2}+3r'^{3}-r'^{4})\mu(LDW) \nonumber\\&\underbrace{-r'^{3}(1-r')^{4}(1+r')(4-7r'+7r'^{2}-4r'^{3}+r'^{4})\mu(LDD)}_{\text{combining underbraced terms labeled (a)}} \nonumber\\&\underbrace{- r'(1-r')^{4}(1+r')(1-2r'-3r'^{2}+7r'^{3}-7r'^{4}+4r'^{5}-r'^{6})\{\mu(DLDD)+\mu(LLDD)\}}_{\text{combining underbraced terms labeled (a)}} \nonumber\\& - 2r'^{3}(1-r')^{3}(1+r')\mu(LDD) \underbrace{+ r'^{2}(1-r')^{8}(1-2r')\mu(WWLDL)}_{\text{combining underbraced terms labeled (b)}}\nonumber\\
={}& - (1-r')^{2}(1-r'-3r'^{2}+3r'^{3}-r'^{4})\mu(LDL) - r'(1-r')^{2}(1-r'-3r'^{2}+3r'^{3}-r'^{4})\mu(LDW) \nonumber\\&\underbrace{- r'^{3}(1-r')^{3}(1+r')(2-r')(3-4r'+5r'^{2}-3r'^{3}+r'^{4})\mu(LDD)}_{\text{summing the two terms involving } \mu(LDD) \text{ above}}\nonumber\\& - r'(1-r')^{4}(1+r')(1-2r'-3r'^{2}+7r'^{3}-7r'^{4}+4r'^{5}-r'^{6})\{\mu(DLDD)+\mu(LLDD)\}\nonumber\\& + r'^{2}(1-r')^{8}(1-2r')\mu(WWLDL)\nonumber\\
={}& \underbrace{- r'^{3}(1-r')^{3}(1+r')(2-r')(3-4r'+5r'^{2}-3r'^{3}+r'^{4})\mu(LD)}\nonumber\\&
\underbrace{- (1-r')^{2}(1-r'-3r'^{2}-3r'^{3}+10r'^{4}-8r'^{5}+9r'^{7}-10r'^{8}+5r'^{9}-r'^{10})\mu(LDL)}\nonumber\\&
\underbrace{- r'(1-r')^{2}(1-r'-9r'^{2}+14r'^{3}-9r'^{4}+9r'^{6}-10r'^{7}+5r'^{8}-r'^{9})\mu(LDW)}\nonumber\\&- r'(1-r')^{4}(1+r')(1-2r'-3r'^{2}+7r'^{3}-7r'^{4}+4r'^{5}-r'^{6})\{\mu(DLDD)+\mu(LLDD)\}\nonumber\\& + r'^{2}(1-r')^{8}(1-2r')\mu(WWLDL)\nonumber\\
{}&(\text{where the underbraced terms are obtained by combining the terms involving }\nonumber\\& \mu(LDL), \mu(LDW) \text{ and } \mu(LDD) \text{ from the previous step})\nonumber\\ 
={}& - r'^{3}(1-r')^{3}(1+r')(2-r')(3-4r'+5r'^{2}-3r'^{3}+r'^{4})\mu(LD) - \alpha_{1}\mu(LDL) - \alpha_{2}\mu(LDW) \nonumber\\&- \alpha_{3}\{\mu(DLDD)+\mu(LLDD)\} + r'^{2}(1-r')^{8}(1-2r')\mu(WWLDL),\label{intermediate_7} 
\end{align}
where
\begin{enumerate}
\item the polynomial $\alpha_{1}=(1-r')^{2}(1-r'-3r'^{2}-3r'^{3}+10r'^{4}-8r'^{5}+9r'^{7}-10r'^{8}+5r'^{9}-r'^{10})$ is strictly positive for all $r' \in [0,0.435029)$,
\item the polynomial $\alpha_{2}=r'(1-r')^{2}(1-r'-9r'^{2}+14r'^{3}-9r'^{4}+9r'^{6}-10r'^{7}+5r'^{8}-r'^{9})$ is strictly positive for all $r' \in [0,0.35678)$,
\item and the polynomial $\alpha_{3}=r'(1-r')^{4}(1+r')(1-2r'-3r'^{2}+7r'^{3}-7r'^{4}+4r'^{5}-r'^{6})$ is strictly positive for all $r' \in [0,0.410819)$.
\end{enumerate}
It is thus evident that, when we consider $0 < r' \leqslant 0.201382$, each $\alpha_{i} > 0$. Next, combining the terms indicated by underbraces labeled (v) in \eqref{intermediate_6}, we obtain:
\begin{align}
{}& - r'(1-r')^{2}(r'^{5}-2r'^{4}-r'^{3}+7r'^{2}-5r'+1)\mu(LLD) - r'(1-r')^{6}\mu(DLD) \underbrace{- 2r'(1-r')^{6}\mu(DWLD)}_{\text{(c)}} \nonumber\\&\underbrace{- r'(1-r')^{6}\mu(LWLD) - r'^{2}(1-r')^{8}(1+r')\mu(WWLD) - r'^{3}(3-3r'+r'^{2})(1-r')^{4}\mu(WWLD)}_{\text{(c)}}\nonumber\\
={}& -r'(1-r')^{2}(r'^{5}-2r'^{4}-r'^{3}+7r'^{2}-5r'+1)\mu(LLD) - r'(1-r')^{6}\mu(DLD) \nonumber\\& \underbrace{- r'^{2}(1-r')^{4}(1-r'^{2}+3r'^{3}-3r'^{4}+r'^{5})\mu(WLD)}_{\text{combining underbraced terms labeled (c)}}\nonumber\\& \underbrace{- r'(1-r')^{4}(2-5r'+2r'^{2}+r'^{3}-3r'^{4}+3r'^{5}-r'^{6})\mu(DWLD)}_{\text{combining underbraced terms labeled (c)}}\nonumber\\& \underbrace{- r'(1-r')^{4}(1-3r'+r'^{2}+r'^{3}-3r'^{4}+3r'^{5}-r'^{6})\mu(LWLD)}_{\text{combining underbraced terms labeled (c)}}\nonumber\\ 
={}& \underbrace{- r'(1-r')^{2}(r'^{5}-2r'^{4}-r'^{3}+7r'^{2}-5r'+1)\mu(LLD) - r'(1-r')^{6}\mu(DLD)}_{\text{(d)}} \nonumber\\& \underbrace{- r'^{2}(1-r')^{4}(1-r'^{2}+3r'^{3}-3r'^{4}+r'^{5})\mu(WLD)}_{\text{(d)}} - \beta_{1}\mu(DWLD) - \beta_{2}\mu(LWLD)\nonumber\\
={}& \underbrace{- r'^{2}(1-r')^{4}(1-r'^{2}+3r'^{3}-3r'^{4}+r'^{5})\mu(LD)}_{\text{combining underbraced terms labeled (d)}}\nonumber\\& \underbrace{- r'(1-r')^{2}(1-6r'+9r'^{2}-r'^{3}-7r'^{4}+11r'^{5}-10r'^{6}+5r'^{7}-r'^{8})\mu(LLD)}_{\text{combining underbraced terms labeled (d)}}\nonumber\\& \underbrace{- r'(1-r')^{4}(1-3r'+r'^{2}+r'^{3}-3r'^{4}+3r'^{5}-r'^{6})\mu(DLD)}_{\text{combining underbraced terms labeled (d)}} - \beta_{1}\mu(DWLD) - \beta_{2}\mu(LWLD)\nonumber\\
={}& -r'^{2}(1-r')^{4}(1-r'^{2}+3r'^{3}-3r'^{4}+r'^{5})\mu(LD) - \beta_{3}\mu(LLD) - \beta_{2}\{\mu(DLD)+\mu(LWLD)\} \nonumber\\&- \beta_{1}\mu(DWLD),\label{intermediate_8}
\end{align}
in which
\begin{enumerate}
\item the polynomial $\beta_{1}=r'(1-r')^{4}(2-5r'+2r'^{2}+r'^{3}-3r'^{4}+3r'^{5}-r'^{6})$ is strictly positive for all $r' \in (0,0.505225)$, 
\item \label{beta_{2}} the polynomial $\beta_{2}=r'(1-r')^{4}(1-3r'+r'^{2}+r'^{3}-3r'^{4}+3r'^{5}-r'^{6})$ is strictly positive for all $r' \in (0,0.387969)$, 
\item the polynomial $\beta_{3}=r'(1-r')^{2}(1-6r'+9r'^{2}-r'^{3}-7r'^{4}+11r'^{5}-10r'^{6}+5r'^{7}-r'^{8})$ is strictly positive for all $r' \in (0,0.265137)$.
\end{enumerate}
Therefore, for $0 < r' \leqslant 0.201382$, each $\beta_{i} > 0$. Incorporating \eqref{intermediate_7} and \eqref{intermediate_8} into \eqref{intermediate_6}, we obtain
\begin{align}
w_{3}(\E_{r',0}\mu) ={}& w_{3}(\mu) + (2r'-14r'^{2}+23r'^{3}-14r'^{4}+3r'^{6}-r'^{7})\mu(LD) \underbrace{- r'^{2}\mu(LDD)} - r'^{2}\mu(DDD) \nonumber\\&- r'^{3}(1-r')^{3}(1+r')(2-r')(3-4r'+5r'^{2}-3r'^{3}+r'^{4})\mu(LD) \underbrace{- \alpha_{1}\mu(LDL)}\nonumber\\& \underbrace{- \alpha_{2}\mu(LDW)} - \alpha_{3}\{\mu(DLDD)+\mu(LLDD)\} + r'^{2}(1-r')^{8}(1-2r')\mu(WWLDL) \nonumber\\&- r'^{2}(1-r')^{4}(1-r'^{2}+3r'^{3}-3r'^{4}+r'^{5})\mu(LD) - \beta_{3}\mu(LLD) - \beta_{2}\{\mu(DLD)+\mu(LWLD)\} \nonumber\\&- \beta_{1}\mu(DWLD) + r'^{4}(1-r')^{8}\mu(WWLDD) - r'(1-r')^{6}C_{LLWD} - r'(1-r')^{6}C_{LLDW} \nonumber\\& - r'^{2}C_{WDD} - r'^{2}(3-3r'+r'^{2})(1-r')^{2}C_{LWD} - C_{3}\nonumber\\
={}& w_{3}(\mu) + (2r'-14r'^{2}+23r'^{3}-14r'^{4}+3r'^{6}-r'^{7})\mu(LD) \underbrace{- r'^{2}\mu(LD)} - r'^{2}\mu(DDD) \nonumber\\&- r'^{3}(1-r')^{3}(1+r')(2-r')(3-4r'+5r'^{2}-3r'^{3}+r'^{4})\mu(LD)\nonumber\\& \underbrace{- (1-3r'-r'^{2}+2r'^{3}+13r'^{4}-31r'^{5}+26r'^{6}+r'^{7}-28r'^{8}+34r'^{9}-21r'^{10}+7r'^{11}-r'^{12})\mu(LDL)}\nonumber\\& \underbrace{- r'(1-4r'-6r'^{2}+31r'^{3}-46r'^{4}+32r'^{5}-28r'^{7}+34r'^{8}-21r'^{9}+7r'^{10}-r'^{11})\mu(LDW)} \nonumber\\&- \alpha_{3}\{\mu(DLDD)+\mu(LLDD)\} + r'^{2}(1-r')^{8}(1-2r')\mu(WWLDL) \nonumber\\&- r'^{2}(1-r')^{4}(1-r'^{2}+3r'^{3}-3r'^{4}+r'^{5})\mu(LD) - \beta_{3}\mu(LLD) - \beta_{2}\{\mu(DLD)+\mu(LWLD)\} \nonumber\\&- \beta_{1}\mu(DWLD) + r'^{4}(1-r')^{8}\mu(WWLDD) - r'(1-r')^{6}C_{LLWD} - r'(1-r')^{6}C_{LLDW} - r'^{2}C_{WDD} \nonumber\\&- r'^{2}(3-3r'+r'^{2})(1-r')^{2}C_{LWD} - C_{3}\nonumber\\
{}&(\text{the underbraced terms in the step above are obtained by combining the underbraced terms}\nonumber\\&\text{from the previous step})\nonumber\\
={}& w_{3}(\mu) \underbrace{+ (2r'-14r'^{2}+23r'^{3}-14r'^{4}+3r'^{6}-r'^{7})\mu(LD) - r'^{2}\mu(LD)} - r'^{2}\mu(DDD) \nonumber\\&\underbrace{- r'^{3}(1-r')^{3}(1+r')(2-r')(3-4r'+5r'^{2}-3r'^{3}+r'^{4})\mu(LD)} - \gamma_{1}\mu(LDL) - \gamma_{2}\mu(LDW) \nonumber\\&- \alpha_{3}\{\mu(DLDD)+\mu(LLDD)\} \underbrace{- r'^{2}(1-r')^{4}(1-r'^{2}+3r'^{3}-3r'^{4}+r'^{5})\mu(LD)} \nonumber\\&- \beta_{3}\mu(LLD) - \beta_{2}\mu(DLD) - \beta_{1}\mu(DWLD) - \beta_{2}\mu(LWLD) + r'^{4}(1-r')^{8}\mu(WWLDD) \nonumber\\&+ r'^{2}(1-r')^{8}(1-2r')\mu(WWLDL) - r'(1-r')^{6}C_{LLWD} - r'(1-r')^{6}C_{LLDW} - r'^{2}C_{WDD} \nonumber\\&- r'^{2}(3-3r'+r'^{2})(1-r')^{2}C_{LWD} - C_{3}\nonumber\\
={}& w_{3}(\mu) \underbrace{+ r'(2-16r'+21r'^{2}+4r'^{3}-39r'^{4}+50r'^{5}-35r'^{6}+7r'^{7}+13r'^{8}-14r'^{9}+6r'^{10}-r'^{11})\mu(LD)} \nonumber\\& - r'^{2}\mu(DDD) - \gamma_{1}\mu(LDL) - \gamma_{2}\mu(LDW) - \alpha_{3}\{\mu(DLDD)+\mu(LLDD)\} - \beta_{3}\mu(LLD) - \beta_{2}\{\mu(DLD)\nonumber\\&+\mu(LWLD)\} - \beta_{1}\mu(DWLD) \underbrace{+ r'^{4}(1-r')^{8}\mu(WWLDD) + r'^{2}(1-r')^{8}(1-2r')\mu(WWLDL)} \nonumber\\&- r'(1-r')^{6}C_{LLWD} - r'(1-r')^{6}C_{LLDW} - r'^{2}C_{WDD} - r'^{2}(3-3r'+r'^{2})(1-r')^{2}C_{LWD} - C_{3}\nonumber\\
{}&(\text{once again, the underbraced terms in the step above are obtained by combining the}\nonumber\\&\text{underbraced terms from the previous step})\label{wt_ineq_3}
\end{align}
where
\begin{enumerate}
\item the polynomial $\gamma_{1}=(1-3r'-r'^{2}+2r'^{3}+13r'^{4}-31r'^{5}+26r'^{6}+r'^{7}-28r'^{8}+34r'^{9}-21r'^{10}+7r'^{11}-r'^{12})$ is strictly positive for all $r' \in [0,0.345627)$,
\item the polynomial $\gamma_{2}=r'(1-4r'-6r'^{2}+31r'^{3}-46r'^{4}+32r'^{5}-28r'^{7}+34r'^{8}-21r'^{9}+7r'^{10}-r'^{11})$ is strictly positive for all $r' \in (0,0.238556)$,
\item and the polynomial $r'(2-16r'+21r'^{2}+4r'^{3}-39r'^{4}+50r'^{5}-35r'^{6}+7r'^{7}+13r'^{8}-14r'^{9}+6r'^{10}-r'^{11})$ is strictly negative for all $r' > 0.157175$.
\end{enumerate}
Among the problematic terms, indicated by underbraces, on the right side of the final expression of \eqref{wt_ineq_3}, there is $r'^{4}(1-r')^{8}\mu(WWLDD)$, which we now take care of by introducing the adjustment $w_{4}(\mu) = w_{3}(\mu) - \beta_{2}\mu(LWLD)$. It can be seen, by expanding $w_{4}(\mu)$ and substituting the expressions for $w_{3}(\mu)$ and $\beta_{2}$ from \eqref{adjust_3} and \eqref{beta_{2}} respectively, that the resulting weight function is exactly the same as that defined in \eqref{adjust_4}.

We partially compute $\E_{r',0}\mu(LWLD)$, which is the probability of the event $\{\E_{r',0}\eta(0)=L,\E_{r',0}\eta(1)=W,\E_{r',0}\eta(2)=L,\E_{r',0}\eta(3)=D\}$:
\begin{enumerate}
\item Suppose $\eta(0)=\eta(1)=W$, $\eta(2)=L$ and $\eta(3)=\eta(4)=D$. Then the event $\{\E_{r',0}\eta(0)=L\}$ happens with probability $1$, the event $\{\E_{r',0}\eta(1)=W\}$ happens with probability $(1-r')$, the event $\{\E_{r',0}\eta(2)=L\}$ happens with probability $r'^{2}$, and the event $\{\E_{r',0}\eta(3)=D\}$ happens with probability $(1-r')(1+r')$. Thus, the contribution of this case to $\E_{r',0}\mu(LWLD)$ is $r'^{2}(1-r')^{2}(1+r')\mu(WWLDD)$.
\end{enumerate} 
Therefore, letting $C_{LWLD}$ denote the contribution from the cases not considered above, we obtain
\begin{align}\label{pushforward_LWLD_q=0}
\E_{r',0}\mu(LWLD) = r'^{2}(1-r')^{2}(1+r')\mu(WWWLD) + C_{LWLD}.
\end{align}
Incorporating the new adjustment into \eqref{wt_ineq_3}, and emulating \eqref{ith_wt_fn_ineq}, we obtain:
\begin{align}
w_{4}\left(\E_{r',0}\mu\right) ={}& w_{4}(\mu) + \beta_{2}\mu(LWLD) + r'(2-16r'+21r'^{2}+4r'^{3}-39r'^{4}+50r'^{5}-35r'^{6}+7r'^{7}+13r'^{8}-14r'^{9}\nonumber\\&+6r'^{10}-r'^{11})\mu(LD) - r'^{2}\mu(DDD) - \gamma_{1}\mu(LDL) - \gamma_{2}\mu(LDW) - \alpha_{3}\{\mu(DLDD)+\mu(LLDD)\} \nonumber\\&- \beta_{3}\mu(LLD) - \beta_{2}\mu(DLD) - \beta_{1}\mu(DWLD) - \beta_{2}\mu(LWLD) \underbrace{+ r'^{4}(1-r')^{8}\mu(WWLDD)} \nonumber\\&+ r'^{2}(1-r')^{8}(1-2r')\mu(WWLDL) - r'(1-r')^{6}C_{LLWD} - r'(1-r')^{6}C_{LLDW} - r'^{2}C_{WDD} \nonumber\\&- r'^{2}(3-3r'+r'^{2})(1-r')^{2}C_{LWD} - C_{3} \underbrace{- \beta_{2}r'^{2}(1-r')^{2}(1+r')\mu(WWLDD)} - \beta_{2}C_{LWLD}\nonumber\\
={}& w_{4}(\mu) + r'(2-16r'+21r'^{2}+4r'^{3}-39r'^{4}+50r'^{5}-35r'^{6}+7r'^{7}+13r'^{8}-14r'^{9}+6r'^{10}-r'^{11})\nonumber\\&\mu(LD) - r'^{2}\mu(DDD) - \gamma_{1}\mu(LDL) - \gamma_{2}\mu(LDW) - \alpha_{3}\{\mu(DLDD)+\mu(LLDD)\} - \beta_{3}\mu(LLD) \nonumber\\&- \beta_{2}\mu(DLD) - \beta_{1}\mu(DWLD) \underbrace{- r'^{3}(1-r')^{6}(1-3r'+r'^{3}-2r'^{4}+2r'^{6}-r'^{7})\mu(WWLDD)}\nonumber\\&+ r'^{2}(1-r')^{8}(1-2r')\mu(WWLDL) - r'(1-r')^{6}C_{LLWD} - r'(1-r')^{6}C_{LLDW} - r'^{2}C_{WDD} \nonumber\\&- r'^{2}(3-3r'+r'^{2})(1-r')^{2}C_{LWD} - C_{3} - \beta_{2}C_{LWLD}\nonumber\\
={}& w_{4}(\mu) + r'(2-16r'+21r'^{2}+4r'^{3}-39r'^{4}+50r'^{5}-35r'^{6}+7r'^{7}+13r'^{8}-14r'^{9}+6r'^{10}-r'^{11})\nonumber\\&\mu(LD) - r'^{2}\mu(DDD) \underbrace{- \gamma_{1}\mu(LDL)} - \gamma_{2}\mu(LDW) - \alpha_{3}\{\mu(DLDD)+\mu(LLDD)\} - \beta_{3}\mu(LLD) \nonumber\\&- \beta_{2}\mu(DLD) - \beta_{1}\mu(DWLD) - \gamma_{3}\mu(WWLDD) \underbrace{+ r'^{2}(1-r')^{8}(1-2r')\mu(WWLDL)}_{\text{use } \mu(WWLDL) \leqslant \mu(LDL) \text{ in the next step}} \nonumber\\&- r'(1-r')^{6}C_{LLWD} - r'(1-r')^{6}C_{LLDW} - r'^{2}C_{WDD} - r'^{2}(3-3r'+r'^{2})(1-r')^{2}C_{LWD} - C_{3} \nonumber\\&- \beta_{2}C_{LWLD}\nonumber\\
\leqslant{}& w_{4}(\mu) + r'(2-16r'+21r'^{2}+4r'^{3}-39r'^{4}+50r'^{5}-35r'^{6}+7r'^{7}+13r'^{8}-14r'^{9}+6r'^{10}-r'^{11})\nonumber\\&\mu(LD) - r'^{2}\mu(DDD) \underbrace{- (1-3r'-2r'^{2}+12r'^{3}-31r'^{4}+81r'^{5}-156r'^{6}+197r'^{7}-168r'^{8}+98r'^{9}}\nonumber\\&\underbrace{-38r'^{10}+9r'^{11}-r'^{12})\mu(LDL)} - \gamma_{2}\mu(LDW) - \alpha_{3}\{\mu(DLDD)+\mu(LLDD)\} - \beta_{3}\mu(LLD)\nonumber\\& - \beta_{2}\mu(DLD) - \beta_{1}\mu(DWLD) - \gamma_{3}\mu(WWLDD) - r'(1-r')^{6}C_{LLWD} - r'(1-r')^{6}C_{LLDW} \nonumber\\&- r'^{2}C_{WDD} - r'^{2}(3-3r'+r'^{2})(1-r')^{2}C_{LWD} - C_{3} - \beta_{2}C_{LWLD}\nonumber\\
={}& w_{4}(\mu) + r'(2-16r'+21r'^{2}+4r'^{3}-39r'^{4}+50r'^{5}-35r'^{6}+7r'^{7}+13r'^{8}-14r'^{9}+6r'^{10}-r'^{11})\nonumber\\&\mu(LD) - r'^{2}\mu(DDD) - \gamma_{4}\mu(LDL)  - \gamma_{2}\mu(LDW) - \alpha_{3}\{\mu(DLDD)+\mu(LLDD)\} - \beta_{3}\mu(LLD) \nonumber\\&- \beta_{2}\mu(DLD) - \beta_{1}\mu(DWLD) - \gamma_{3}\mu(WWLDD) - r'(1-r')^{6}C_{LLWD} - r'(1-r')^{6}C_{LLDW} \nonumber\\&- r'^{2}C_{WDD} - r'^{2}(3-3r'+r'^{2})(1-r')^{2}C_{LWD} - C_{3} - \beta_{2}C_{LWLD},\label{wt_ineq_4_explained}
\end{align}
where
\begin{enumerate}
\item the polynomial $\gamma_{4}=1-3r'-2r'^{2}+12r'^{3}-31r'^{4}+81r'^{5}-156r'^{6}+197r'^{7}-168r'^{8}+98r'^{9}-38r'^{10}+9r'^{11}-r'^{12}$ is strictly positive for all $r' \in [0,0.345094)$,
\item and the polynomial $\gamma_{3} = r'^{3}(1-r')^{6}(1-3r'+r'^{3}-2r'^{4}+2r'^{6}-r'^{7})$ is strictly positive for all $r' \in (0,0.338338)$.
\end{enumerate}
We now have our weight function inequality for $0.157175 < r' \leqslant 0.201382$, since the coefficient of each term in the right side of \eqref{wt_ineq_4_explained}, apart from $w_{4}(\mu)$, is strictly negative for $r' \in (0.157175,0.201382]$, thereby allowing \eqref{wt_ineq_4_explained} to satisfy \eqref{desired_criterion}. Note that \eqref{wt_ineq_4_explained} reduces to the inequality in \eqref{wt_ineq_4} if we remove, apart from $w_{4}(\mu)$, $- r'^{2}\mu(DDD)$ and $r'(2-16r'+20r'^{2}+7r'^{3}-42r'^{4}+51r'^{5}-35r'^{6}+7r'^{7}+13r'^{8}-14r'^{9}+6r'^{10}-r'^{11})\mu(LD)$, all other terms from the right side of \eqref{wt_ineq_4_explained}. This brings us to the end of our construction of the weight functions for the regime described in \ref{bond_regime_2}.

\subsection{Detailed construction of the weight function for the bond percolation game when $(r',s')$ belongs to the regime given by \ref{bond_regime_3}}\label{sec:bond_3_wt_fn_steps}
We now demonstrate the step-by-step construction of our weight function for the regime given by \ref{bond_regime_3}, i.e.\ for all $(r',s')$ that satisfy the constraints $r' = s' > 0.10883$. The approach is the same the general idea outlined in \S\ref{subsec:central_ideas_weight_functions}, and the steps, while differing from those presented in \S\ref{sec:bond_2_wt_fn_steps} in their details, follow a similar line of reasoning. As in each of \S\ref{sec:bond_1_wt_fn_steps} and \S\ref{sec:bond_2_wt_fn_steps}, we begin by setting $c_{1}=1$, $\mathcal{C}_{1}=(D)_{0}$, $c_{2}=1$ and $\mathcal{C}_{2}=(W,D)_{0,1}$ in \eqref{gen_form}. For $r'=s'$, we obtain, using reflection-invariance,
\begin{align}
\E_{r',r'}\mu(D) ={}& (1-r')(1-2r')\{\mu(WD)+\mu(DW)\} + (1-2r')\mu(DD) + r'(1-2r')\{\mu(LD)+\mu(DL)\}\nonumber\\
={}& 2(1-r')(1-2r')\mu(WD)+(1-2r')\mu(DD)+2r'(1-2r')\mu(LD)\label{F_{p,p}mu(D)}
\end{align}
and
\begin{align}
    \E_{r', r'}\mu(WD) ={}& (2r' - r'^2) (1-r') (1-2r') \mu(WWD) + (1-r'+r'^2)r'(1-2r') \mu(WLD) \nonumber\\& + (1-r'+r'^2)(1-r')(1-2r') \mu(LWD) + (1-r'^2)r'(1-2r') \mu(LLD) \nonumber\\&+ (2r'-r'^2)(1-2r') \mu(WDD) + (2r'-r'^2)(1-r')(1-2r') \mu(DWD) \nonumber\\&+ (2r'-r'^2)(1-2r') \mu(DDD) + (1-r'+r'^2)(1-2r') \mu(LDD) \nonumber\\& + (1-r'+r'^2)r'(1-2r') \mu(DLD) + (2r'-r'^2)(1-r')(1-2r') \mu(WDW) \nonumber\\&+ (2r'-r'^2)r'(1-2r') \mu(WDL) + (1-r'+r'^2)(1-r')(1-2r') \mu(LDW) \nonumber\\& + (1-r'+r'^2)r'(1-2r') \mu(LDL) + (2r'-r'^2)(1-r')(1-2r') \mu(DDW) \nonumber\\& + (2r'-r'^2)r'(1-2r') \mu(DDL).\label{F_{p,p}mu(WD)}
\end{align}
Starting with our initial guess for the weight function, namely $w_0(\mu) = \mu(D) + \mu(WD)$, we obtain, using \eqref{F_{p,p}mu(D)} and \eqref{F_{p,p}mu(WD)}, the following weight function inequality:
\begin{align}
    w_0(\E_{r', r'}\mu)={}& w_0(\mu) \underbrace{- \mu(D)  + \E_{r', r'}\mu(D)}_{\text{use \eqref{F_{p,p}mu(D)} and simplify}} - \mu(WD) + \E_{r', r'}\mu(WD) \nonumber\\
    ={}& w_0(\mu) - (3r' - 2r'^2) \mu(DW) - 2r' \mu(DD) - (1-r'+2r'^2) \mu(DL) + (1-r')(1-2r') \mu(WD) \nonumber\\
    & + r'(1-2r') \mu(LD) - \mu(WD)  \underbrace{+ \E_{r', r'}\mu(WD)}_{\text{substitute using \eqref{F_{p,p}mu(WD)}}} \nonumber\\
    &= w_0(\mu) \underbrace{- (3r' - 2r'^2) \mu(DW)}_{\text{substitute } \mu(DW) = \mu(WD)} - 2r' \mu(DD) \underbrace{- (1-r' + 2r'^2) \mu(DL)}_{\text{(i)}} \underbrace{+ (1-r')(1-2r') \mu(WD)}_{\text{(ii)}} \nonumber \\
    & \underbrace{+ r'(1-2r') \mu(LD)}_{\text{(i)}} \underbrace{- \mu(WD)}_{\text{(ii)}} + (2r' - r'^2) (1-r') (1-2r') \mu(WWD) \nonumber\\
    &  + (1-r'+r'^2)r'(1-2r') \mu(WLD) + (1-r'+r'^2)(1-r')(1-2r') \mu(LWD)\nonumber \\
    &  + (1-r'^2)r'(1-2r') \mu(LLD) + (2r'-r'^2)(1-2r') \mu(WDD) + (2r'-r'^2)(1-r')(1-2r') \mu(DWD) \nonumber\\
    &  + (2r'-r'^2)(1-2r') \mu(DDD) + (1-r'+r'^2)(1-2r') \mu(LDD) + (1-r'+r'^2)r'(1-2r') \mu(DLD) \nonumber\\
    & + (2r'-r'^2)(1-r')(1-2r') \mu(WDW) + (2r'-r'^2)r'(1-2r') \mu(WDL) \nonumber\\
    & + (1-r'+r'^2)(1-r')(1-2r') \mu(LDW) + (1-r'+r'^2)r'(1-2r') \mu(LDL) \nonumber\\
    & + (2r'-r'^2)(1-r')(1-2r') \mu(DDW) + (2r'-r'^2)r'(1-2r') \mu(DDL)\nonumber \\
    ={}& w_0(\mu)  \underbrace{- (3r' - 2r'^2) \mu(WD)} - 2r' \mu(DD) \underbrace{- (1-2r' + 4r'^2) \mu(LD)}_{\text{combining terms labelled (i)}}   \underbrace{- (3r'-2r'^2) \mu(WD)}_{\text{combining terms labelled (ii)}} \nonumber\\
    & + (2r' - r'^2) (1-r') (1-2r') \mu(WWD) + (1-r'+r'^2)r'(1-2r') \mu(WLD) \nonumber\\
    & + (1-r'+r'^2)(1-r')(1-2r') \mu(LWD) + (1-r'^2)r'(1-2r') \mu(LLD) \nonumber\\
    & + (2r'-r'^2)(1-2r') \mu(WDD) + (2r'-r'^2)(1-r')(1-2r') \mu(DWD)\nonumber \\
    & + (2r'-r'^2)(1-2r') \mu(DDD) + (1-r'+r'^2)(1-2r') \mu(LDD)  + (1-r'+r'^2)r'(1-2r') \mu(DLD) \nonumber\\
    & + (2r'-r'^2)(1-r')(1-2r') \mu(WDW) + (2r'-r'^2)r'(1-2r') \mu(WDL)\nonumber \\
    & + (1-r'+r'^2)(1-r')(1-2r') \mu(LDW) + (1-r'+r'^2)r'(1-2r') \mu(LDL) \nonumber\\
    & + (2r'-r'^2)(1-r')(1-2r') \mu(DDW) + (2r'-r'^2)r'(1-2r') \mu(DDL)\nonumber \\
    ={}& w_0(\mu) \underbrace{- (2r'-r'^2)(1-2r') \mu(WD)}_{\text{expand using }\mu(WD) = \mu(WDW) + \mu(WDL) + \mu(WDD)\text{ }} \underbrace{- 2r' \mu(DD)}_{\text{ expand using }\mu(DD) = \mu(DDW) + \mu(DDL) + \mu(DDD)} \nonumber \\
    &\underbrace{- (1-r'+r'^2)(1-2r') \mu(LD)}_{\text{expand using }\mu(LD) = \mu(LDW) + \mu(LDL) + \mu(LDD)}  \underbrace{- ((1-2r' +4r'^2) - (1-r'+r'^2)(1-2r')) \mu(LD)}_{\text{expand using }\mu(LD) = \mu(WLD) + \mu(LLD) + \mu(DLD)} \nonumber\\
    &\underbrace{- (6r'-4r'^2-(2r'-r'^2)(1-2r')) \mu(WD)}_{\text{expand using }\mu(WD) = \mu(WWD) + \mu(LWD) + \mu(DWD)} + (2r' - r'^2) (1-r') (1-2r') \mu(WWD) \nonumber\\
    &+ (1-r'+r'^2)r'(1-2r') \mu(WLD)  + (1-r'+r'^2)(1-r')(1-2r') \mu(LWD) \nonumber\\
    &+ (1-r'^2)r'(1-2r') \mu(LLD) + (2r'-r'^2)(1-2r') \mu(WDD) \nonumber\\
    &+ (2r'-r'^2)(1-r')(1-2r') \mu(DWD) + (2r'-r'^2)(1-2r') \mu(DDD) \nonumber\\
    &+ (1-r'+r'^2)(1-2r') \mu(LDD) + (1-r'+r'^2)r'(1-2r') \mu(DLD) \nonumber\\
    & + (2r'-r'^2)(1-r')(1-2r') \mu(WDW) + (2r'-r'^2)r'(1-2r') \mu(WDL)\nonumber \\
    & + (1-r'+r'^2)(1-r')(1-2r') \mu(LDW) + (1-r'+r'^2)r'(1-2r') \mu(LDL) \nonumber\\
    & + (2r'-r'^2)(1-r')(1-2r') \mu(DDW) + (2r'-r'^2)r'(1-2r') \mu(DDL)\nonumber\\
    ={}& w_0(\mu) \underbrace{- (2r'-r'^2)(1-2r') (\mu(WDW) + \mu(WDL) + \mu(WDD))} \nonumber\\
    & \underbrace{-  2r' (\mu(DDW) + \mu(DDL) + \mu(DDD))} \underbrace{- (1-r'+r'^2)(1-2r') (\mu(LDW) + \mu(LDL) + \mu(LDD))} \nonumber\\
    & \underbrace{- ((1-2r' +4r'^2) - (1-r'+r'^2)(1-2r')) (\mu(WLD) + \mu(LLD) + \mu(DLD))}\nonumber \\
    & \underbrace{- (6r'-4r'^2-(2r'-r'^2)(1-2r')) (\mu(WWD) + \mu(LWD) + \mu(DWD))} \nonumber\\
    & + (2r' - r'^2) (1-r') (1-2r') \mu(WWD) + (1-r'+r'^2)r'(1-2r') \mu(WLD) \nonumber\\
    & + (1-r'+r'^2)(1-r')(1-2r') \mu(LWD) + (1-r'^2)r'(1-2r') \mu(LLD) \nonumber\\
    & + (2r'-r'^2)(1-2r') \mu(WDD) + (2r'-r'^2)(1-r')(1-2r') \mu(DWD) + (2r'-r'^2)(1-2r') \mu(DDD)\nonumber \\
    & + (1-r'+r'^2)(1-2r') \mu(LDD) + (1-r'+r'^2)r'(1-2r') \mu(DLD)\nonumber \\
    & + (2r'-r'^2)(1-r')(1-2r') \mu(WDW) + (2r'-r'^2)r'(1-2r') \mu(WDL) \nonumber\\
    & + (1-r'+r'^2)(1-r')(1-2r') \mu(LDW) + (1-r'+r'^2)r'(1-2r') \mu(LDL)\nonumber \\
    & + (2r'-r'^2)(1-r')(1-2r') \mu(DDW) + (2r'-r'^2)r'(1-2r') \mu(DDL)\label{Fpp_inital_1}\\
    ={}&w_0(\mu) - (2r'^2-5r'^3+2r'^4) \mu(WDW) \underbrace{- (2r' - 7r'^2 + 7r'^3 - 2r'^4) \mu(WDL)} \nonumber \\
    & - (7r'^2 - 7r'^3 +2r'^4) \mu(DDW) - (2r'-2r'^2+5r'^3-2r'^4) \mu(DDL) - (5r'^2 - 2r'^3) \mu(DDD) \nonumber\\
    & \underbrace{- (r' - 3r'^2 + 3r'^3 - 2r'^4) \mu(LDW)}_{\text{substituting }\mu(LDW) = \mu(WDL)} - (1 - 4r' + 6r'^2 - 5r'^3 + 2r'^4) \mu(LDL) \nonumber\\
    & - (4r'^2 - r'^3 + 2r'^4) \mu(WLD) - (3r'^2 + 3r'^3 -2r'^4)\mu(LLD) - (4r'^2 - r'^3 + 2r'^4)\mu(DLD) \nonumber\\
    & - (2r' + 8r'^2 - 9r'^3 + 2r'^4) \mu(WWD) - (-1 + 8r' - 5r'^2 + 3r'^3 - 2r'^4) \mu(LWD) \nonumber\\
    & - (2r' + 8r'^2 - 9r'^3 + 2r'^4) \mu(DWD)\label{Fpp_initial_2}\\
    ={}& w_0(\mu) - (2r'^2-5r'^3+2r'^4) \mu(WDW) \underbrace{- (3r' - 10r'^2 + 10r'^3 - 4r'^4) \mu(WDL)}_{\text{combining the highlighted terms}} \nonumber \\
    & - (7r'^2 - 7r'^3 +2r'^4) \mu(DDW) - (2r'-2r'^2+5r'^3-2r'^4) \mu(DDL) - (5r'^2 - 2r'^3) \mu(DDD) \nonumber\\
    & - (1 - 4r' + 6r'^2 - 5r'^3 + 2r'^4) \mu(LDL) - (4r'^2 - r'^3 + 2r'^4) \mu(WLD) - (3r'^2 + 3r'^3 -2r'^4)\mu(LLD)\nonumber \\
    & - (4r'^2 - r'^3 + 2r'^4)\mu(DLD) - (2r' + 8r'^2 - 9r'^3 + 2r'^4) \mu(WWD) \nonumber\\
    & - (-1 + 8r' - 5r'^2 + 3r'^3 - 2r'^4) \mu(LWD) - (2r' + 8r'^2 - 9r'^3 + 2r'^4) \mu(DWD).\label{wt_fn_ineq_1_regime_2}
\end{align}
The only problematic term in the above expression is $-(-1 + 8r' - 5r'^2 + 3r'^3 - 2r'^4) \mu(LWD)$, since the coefficient of this term is strictly positive for sufficiently small values of $r'$. In order to take care of this term, we implement the idea outlined in \S\ref{subsubsec:adjust_gen} by updating our weight function as follows:
\begin{align}
    w_1(\mu) ={}& w_0(\mu) - (2r'^2 - 5r'^3 + 2r'^4) \mu(WDW) - (3r' - 10r'^2 + 10r'^3 - 4r'^4) \mu(WDL) \nonumber\\&- (7r'^2 - 7r'^3 + 2r'^4) \mu(DDW) - (2r' - 2r'^2 + 5r'^3 - 2r'^4) \mu(DDL) \nonumber\\&- (1 - 4r' + 6r'^2 - 5r'^3 + 2r'^4) \mu(LDL) - (4r'^2 - r'^3 + 2r'^4) \mu(WLD) - (3r'^2 + 3r'^3 -2r'^4) \mu(LLD) \nonumber\\&- (2r' + 8r'^2 - 9r'^3 + 2r'^4) \mu(WWD), \label{adjust_1_p=q}
\end{align}
which is exactly what we stated in \eqref{weight_fn_p=q}. 

In order to incorporate this adjustment into our current weight function inequality, following the idea used in the deduction of \eqref{ith_wt_fn_ineq}, we need to compute the pushforward probabilities $\E_{r', r'}\mu(WDW)$, $\E_{r', r'}\mu(WDL)$, $\E_{r', r'}\mu(DDW)$, $\E_{r', r'}\mu(DDL)$, $\E_{r', r'}\mu(LDL)$, $\E_{r', r'}\mu(WLD)$, $\E_{r', r'}\mu(LLD)$ and $\E_{r', r'}\mu(WWD)$. However, we need only \emph{partially} compute these expressions, since we only care about the contributions arising from cylinder sets that are able to negate, to as large an extent as possible, the term $-(-1 + 8r' - 5r'^2 + 3r'^3 - 2r'^4) \mu(LWD)$. For instance, $\E_{r', r'}\mu(WDW)$ is the probability of the event $\{\E_{r', r'}\eta(0) = W, \E_{r', r'}\eta(1) = D, \E_{r', r'}\eta(2) = W\}$ where $\eta$ is a random configuration with law $\mu$, and we need only consider the contributions to this pushforward probability that arise from the event $\{\eta(1)=D, \eta(2)=W, \eta(3)=L\}$. In fact, this is precisely the event whose contributions we consider to each of the pushforward probabilities mentioned above, and we thereby obtain the following lower bounds: 
\begin{align}
    \E_{r', r'}\mu(WDW) \geqslant{}& (1-r'+r'^2)(1-r')(1-2r')\{(2r'-r'^2) \mu(WDWL) + (1-r'+r'^2)\mu(LDWL) \nonumber\\&+ (2r'-r'^2) \mu(DDWL)\} \nonumber\\
    \geqslant{}& (1-r'+r'^2)(1-r')(1-2r')(2r'-r'^2) \mu(DWL) \nonumber \\
    ={}& \underbrace{(1-r'+r'^2)(1-r')(1-2r')(2r'-r'^2) \mu(LWD)}_{\text{putting $\mu(DWL) = \mu(LWD)$, using reflection-invariance}};\label{Fpp_WDW}
\end{align}
\begin{align}
    \E_{r', r'}\mu(WDL) \geqslant&{}  (1-r')(1-2r')(r'-r'^2)\{(2r'-r'^2) \mu(WDWL) + (1-r'+r'^2) \mu(LDWL) \nonumber\\&+ (2r'-r'^2) \mu(DDWL)\} 
    \geqslant (1-r')(1-2r')(r'-r'^2)(2r'-r'^2) \mu(LWD);\label{Fpp_WDL}
\end{align}
\begin{align}
    \E_{r', r'}\mu(DDW)  \geqslant{}& (1-r'+r'^2)(1-r')(1-2r')\{(1-r')(1-2r') \mu(WDWL) + (r'-2r'^2) \mu(LDWL) \nonumber\\& + (1-2r') \mu(DDWL)\}
     \geqslant (1-r'+r'^2)(1-r')(1-2r')(r'-2r'^2) \mu(LWD);\label{Fpp_DDW}
\end{align}
\begin{align}
    \E_{r', r'}\mu(DDL)  \geqslant{}& (r'-r'^2)(1-r')(1-2r')\{(1-r')(1-2r') \mu(WDWL) + (r'-2r'^2) \mu(LDWL) \nonumber\\&+ (1-2r') \mu(DDWL)\} 
     \geqslant (r'-r'^2)(1-r')(1-2r')(r'-2r'^2) \mu(LWD);\label{Fpp_DDL} 
\end{align}
\begin{align}
    \E_{r', r'}\mu(LDL)  \geqslant{}& (r'-r'^2)(1-r')(1-2r')\{(r'-r'^2) \mu(WDWL) + r'^2 \mu(LDWL) + r'^2 \mu(DDWL)\} \nonumber \\
     \geqslant{}& (r'-r'^2)(1-r')(1-2r')r'^2 \mu(LWD);\label{Fpp_LDL}
\end{align}
\begin{align}
    \E_{r', r'}\mu(WLD)  \geqslant{}& (1-r')(1-2r')(r'-r'^2)\{(1-r'+r'^2) \mu(WLWD) + (1-r'^2) \mu(LLWD) \nonumber\\&+ (1-r'+r'^2) \mu(DLWD)\}
     \geqslant (1-r')(1-2r')(r'-r'^2)(1-r'+r'^2) \mu(LWD);\label{Fpp_WLD}
\end{align}
\begin{align}
    \E_{r', r'}\mu(WWD)  \geqslant{}& (1-r')(1-2r')(1-r'+r'^2)\{(1-r'+r'^2) \mu(WLWD) + (1-r'^2) \mu(LLWD) \nonumber \\
    &+ (1-r'+r'^2) \mu(DLWD)\}
     \geqslant (1-r')(1-2r')(1-r'+r'^2)^2 \mu(LWD);\label{Fpp_WWD}
\end{align}
and finally, 
\begin{align}
    \E_{r', r'}\mu(LLD) \geqslant{}& (1-r')(1-2r')(r'-r'^2)\{(r'-r'^2) \mu(WLWD) + r'^2 \mu(LLWD) + r'^2 \mu(DLWD)\} \nonumber\\
     \geqslant{}& (1-r')(1-2r')(r'-r'^2)r'^2 \mu(LWD). \label{Fpp_LLD}
\end{align}
Incorporating the adjustment introduced in \eqref{adjust_1_p=q} into the weight function inequality in \eqref{wt_fn_ineq_1_regime_2} via the idea demonstrated in \eqref{ith_wt_fn_ineq}, and using the lower bounds obtained in \eqref{Fpp_WDW}, \eqref{Fpp_WDL}, \eqref{Fpp_DDW}, \eqref{Fpp_DDL}, \eqref{Fpp_LDL}, \eqref{Fpp_WLD}, \eqref{Fpp_WWD}, \eqref{Fpp_LLD}, we obtain
\begin{align}
    w_1(\E_{r', r'}\mu) =&{} w_1(\mu) \underbrace{+ (w_0(\E_{r', r'}\mu) - w_0(\mu))}_{\text{substitute using \eqref{wt_fn_ineq_1_regime_2}}} - (2r'^2 - 5r'^3 + 2r'^4) \E_{r', r'}\mu(WDW) \nonumber \\
    & - (3r' - 10r'^2 + 10r'^3 - 4r'^4) \E_{r', r'}\mu(WDL) - (7r'^2 - 7r'^3 + 2r'^4) \E_{r', r'}\mu(DDW) \nonumber \\
    & - (2r' - 2r'^2 + 5r'^3 - 2r'^4) \E_{r', r'}\mu(DDL) - (1 - 4r' + 6r'^2 - 5r'^3 + 2r'^4) \E_{r', r'}\mu(LDL) \nonumber \\
    &- (4r'^2 - r'^3 + 2r'^4) \E_{r', r'}\mu(WLD)  - (3r'^2 + 3r'^3 -2r'^4) \E_{r', r'}\mu(LLD) \nonumber \\
    &- (2r' + 8r'^2 - 9r'^3 + 2r'^4) \E_{r', r'}\mu(WWD) + (2r'^2 - 5r'^3 + 2r'^4) \mu(WDW) \nonumber \\
    & + (3r' - 10r'^2 + 10r'^3 - 4r'^4) \mu(WDL) + (7r'^2 - 7r'^3 + 2r'^4) \mu(DDW) \nonumber \\
    &+ (2r' - 2r'^2 + 5r'^3 - 2r'^4) \mu(DDL) + (1 - 4r' + 6r'^2 - 5r'^3 + 2r'^4) \mu(LDL) \nonumber \\
    & + (4r'^2 - r'^3 + 2r'^4) \mu(WLD) + (3r'^2 + 3r'^3 -2r'^4) \mu(LLD) \nonumber \\
    &+ (2r' + 8r'^2 - 9r'^3 + 2r'^4) \mu(WWD) \nonumber \\ 
    =&{} w_1(\mu) \underbrace{- (2r'^2-5r'^3+2r'^4) \mu(WDW)}_\text{(i)} \underbrace{- (3r' - 10r'^2 + 10r'^3 - 4r'^4) \mu(WDL)}_\text{(ii)} \nonumber\\
    & \underbrace{- (7r'^2 - 7r'^3 +2r'^4) \mu(DDW)}_{\text{(iii)}} \underbrace{- (2r'-2r'^2+5r'^3-2r'^4) \mu(DDL)}_{\text{(iv)}} \nonumber \\
    & - (5r'^2 - 2r'^3) \mu(DDD) \underbrace{- (1 - 4r' + 6r'^2 - 5r'^3 + 2r'^4) \mu(LDL)}_\text{(v)}  \nonumber\\
    &\underbrace{- (4r'^2 - r'^3 + 2r'^4) \mu(WLD)}_{\text{(vi)}} \underbrace{- (3r'^2 + 3r'^3 -2r'^4)\mu(LLD)}_{\text{(vii)}} \nonumber\\
    & - (4r'^2 - r'^3 + 2r'^4)\mu(DLD) \underbrace{- (2r' + 8r'^2 - 9r'^3 + 2r'^4) \mu(WWD)}_{\text{(viii)}} \nonumber \\
    &  - (-1 + 8r' - 5r'^2 + 3r'^3 - 2r'^4) \mu(LWD) - (2r' + 8r'^2 - 9r'^3 + 2r'^4) \mu(DWD) \nonumber \\
    & - (2r'^2 - 5r'^3 + 2r'^4) \E_{r', r'}\mu(WDW) - (3r' - 10r'^2 + 10r'^3 - 4r'^4) \E_{r', r'}\mu(WDL) \nonumber \\
    & - (7r'^2 - 7r'^3 + 2r'^4) \E_{r', r'}\mu(DDW) - (2r' - 2r'^2 + 5r'^3 - 2r'^4) \E_{r', r'}\mu(DDL) \nonumber \\
    & - (1 - 4r' + 6r'^2 - 5r'^3 + 2r'^4) \E_{r', r'}\mu(LDL) - (4r'^2 - r'^3 + 2r'^4) \E_{r', r'}\mu(WLD) \nonumber \\
    & - (3r'^2 + 3r'^3 -2r'^4) \E_{r', r'}\mu(LLD) - (2r' + 8r'^2 - 9r'^3 + 2r'^4) \E_{r', r'}\mu(WWD) \nonumber\\
    & \underbrace{+ (2r'^2 - 5r'^3 + 2r'^4) \mu(WDW)}_{\text{(i)}} \underbrace{+ (3r' - 10r'^2 + 10r'^3 - 4r'^4) \mu(WDL)}_{\text{(ii)}} \nonumber \\
    &\underbrace{+ (7r'^2 - 7r'^3 + 2r'^4) \mu(DDW)}_{\text{(iii)}} \underbrace{+ (2r' - 2r'^2 + 5r'^3 - 2r'^4) \mu(DDL)}_{\text{(iv)}} \nonumber \\
    & \underbrace{+ (1 - 4r' + 6r'^2 - 5r'^3 + 2r'^4) \mu(LDL)}_\text{(v)} \underbrace{+ (4r'^2 - r'^3 + 2r'^4) \mu(WLD)}_{\text{(vi)}} \nonumber \\
    & \underbrace{+ (3r'^2 + 3r'^3 -2r'^4) \mu(LLD)}_{\text{(vii)}} \underbrace{+ (2r' + 8r'^2 - 9r'^3 + 2r'^4) \mu(WWD)}_{\text{(viii)}} \nonumber \\
    =&{} w_1(\mu) - (5r'^2 - 2r'^3) \mu(DDD) - (4r'^2 - r'^3 + 2r'^4)\mu(DLD) \nonumber \\
    & - (-1 + 8r' - 5r'^2 + 3r'^3 - 2r'^4) \mu(LWD) - (2r' + 8r'^2 - 9r'^3 + 2r'^4) \mu(DWD) \nonumber \\
    & \underbrace{- (2r'^2 - 5r'^3 + 2r'^4) \E_{r', r'}\mu(WDW)}_{\text{use \eqref{Fpp_WDW}}} \underbrace{- (3r' - 10r'^2 + 10r'^3 - 4r'^4) \E_{r', r'}\mu(WDL)}_{\text{use \eqref{Fpp_WDL}}} \nonumber \\
    & \underbrace{- (7r'^2 - 7r'^3 + 2r'^4) \E_{r', r'}\mu(DDW)}_{\text{use \eqref{Fpp_DDW}}} \underbrace{- (2r' - 2r'^2 + 5r'^3 - 2r'^4) \E_{r', r'}\mu(DDL)}_{\text{use \eqref{Fpp_DDL}}} \nonumber \\
    & \underbrace{- (1 - 4r' + 6r'^2 - 5r'^3 + 2r'^4) \E_{r', r'}\mu(LDL)}_{\text{use \eqref{Fpp_LDL}}} \underbrace{- (4r'^2 - r'^3 + 2r'^4) \E_{r', r'}\mu(WLD)}_{\text{use \eqref{Fpp_WLD}}}  \nonumber\\
    & \underbrace{- (3r'^2 + 3r'^3 -2r'^4) \E_{r', r'}\mu(LLD)}_{\text{use \eqref{Fpp_LLD}}} \underbrace{- (2r' + 8r'^2 - 9r'^3 + 2r'^4) \E_{r', r'}\mu(WWD)}_{\text{use \eqref{Fpp_WWD}}} \nonumber \\
    \leqslant&{} w_1(\mu) - (5r'^2 - 2r'^3) \mu(DDD) - (4r'^2 - r'^3 + 2r'^4)\mu(DLD) \nonumber \\
    & \underbrace{- (-1 + 8r' - 5r'^2 + 3r'^3 - 2r'^4) \mu(LWD)} - (2r' + 8r'^2 - 9r'^3 + 2r'^4) \mu(DWD) \nonumber \\
    &\underbrace{- (2r'^2 - 5r'^3 + 2r'^4) ((1-r'+r'^2)(1-r')(1-2r')(2r'-r'^2) \mu(LWD))} \nonumber \\
    &\underbrace{- (3r' - 10r'^2 + 10r'^3 - 4r'^4) ((1-r')(1-2r')(r'-r'^2)(2r'-r'^2) \mu(LWD))} \nonumber \\
    & \underbrace{- (7r'^2 - 7r'^3 + 2r'^4) ((1-r'+r'^2)(1-r')(1-2r')(r'-2r'^2) \mu(LWD))} \nonumber \\
    & \underbrace{- (2r' - 2r'^2 + 5r'^3 - 2r'^4) ((r'-r'^2)(1-r')(1-2r')(r'-2r'^2) \mu(LWD))} \nonumber \\
    & \underbrace{- (1 - 4r' + 6r'^2 - 5r'^3 + 2r'^4) ((r'-r'^2)(1-r')(1-2r')r'^2 \mu(LWD))} \nonumber \\
    & \underbrace{- (4r'^2 - r'^3 + 2r'^4) ((1-r')(1-2r')(r'-r'^2)(1-r'+r'^2) \mu(LWD))} \nonumber \\
    & \underbrace{- (3r'^2 + 3r'^3 -2r'^4) ((1-r')(1-2r')(r'-r'^2)r'^2 \mu(LWD))} \nonumber \\
    & \underbrace{- (2r' + 8r'^2 - 9r'^3 + 2r'^4) ((1-r')(1-2r')(1-r'+r'^2)^2 \mu(LWD))} \nonumber \\
    =&{} w_1(\mu) - (5r'^2 - 2r'^3) \mu(DDD) - (4r'^2 - r'^3 + 2r'^4)\mu(DLD) - (2r' + 8r'^2 - 9r'^3 + 2r'^4) \mu(DWD) \nonumber \\
    & \underbrace{+ (1 - 10r' + 7r'^2 + 64r'^4 - 292r'^5 + 583r'^6 - 663r'^7 + 447r'^8 - 168r'^9 + 28r'^{10}) \mu(LWD)}_{\text{obtained by combining the terms highlighted above}} \nonumber \\
    \leqslant &{} w_1(\mu) + (1 - 10r' + 7r'^2 + 64r'^4 - 292r'^5 + 583r'^6 - 663r'^7 + 447r'^8 - 168r'^9 + 28r'^{10}) \mu(LWD),\nonumber
\end{align}
which is exactly what we stated in \eqref{weight_fn_ineq_p=q}. This brings us to the conclusion of the construction of our weight function for $(r',s')$ belonging to \eqref{bond_regime_3}.

\bibliography{Percolation_games_bibliography}
\end{document}